\documentclass[12pt,reqno] {amsart}

 \usepackage{amsmath, wasysym}
 \usepackage{bbm, dsfont}
 \usepackage{amssymb}
\usepackage{graphicx}%
\usepackage{hyperref}
\usepackage{cleveref}
\usepackage{comment}
 \usepackage{mathrsfs}
 \usepackage{float}
 \usepackage{tikz, tikz-cd}
  \usetikzlibrary{arrows.meta, knots, calc}
 \usepackage{bm}
 \usepackage{mathtools} 
\usepackage{extarrows} 
\usepackage{bbm}

 \usepackage{enumerate}
  \usepackage{mathdots}
\usetikzlibrary{decorations.markings}
\usetikzlibrary{decorations.pathreplacing}

\usepackage{pgfplots}

\usepgfplotslibrary{fillbetween}

\pgfplotsset{compat=1.18}

\theoremstyle{plain}
\newtheorem{theorem}{Theorem}[section]

\newtheorem{corollary}[theorem]{Corollary}
\newtheorem{lemma}[theorem]{Lemma}
\newtheorem{proposition}[theorem]{Proposition}
\newtheorem*{ac}{Acknowledgement}

\newtheorem{remark}[theorem]{Remark}

\newtheorem{definition}[theorem]{Definition}

\newtheorem{example}[theorem]{Example}

\newcommand{\leftLineHalfLoops}[3]{%
    \vcenter{\hbox{\begin{tikzpicture}[scale=0.65]
        \draw [blue] (-1.1, 2.3)--(-1.1, -0.8);
        \begin{scope}[shift={(-1.1, 0.8)}]
            \draw [fill=white] (-0.3, -0.3) rectangle (0.3, 0.3);
            \node at (0, 0) {\tiny $#1$};  
        \end{scope}
        
        \begin{scope}[shift={(0,1.5)}]
            \draw [blue] (-0.5, 0.8)--(-0.5, 0) .. controls +(0, -0.6) and +(0,-0.6).. (0.5, 0)--(0.5, 0.8); 
            \begin{scope}[shift={(0.5, 0.3)}]
                \draw [fill=white] (-0.3, -0.3) rectangle (0.3, 0.3);
                \node at (0, 0) {\tiny $#2$};
            \end{scope}
        \end{scope}
        
        \draw [blue] (-0.5, -0.8)--(-0.5, 0) .. controls +(0, 0.6) and +(0,0.6).. (0.5, 0)--(0.5, -0.8);
        \begin{scope}[shift={(0.5, -0.3)}]
            \draw [fill=white] (-0.3, -0.3) rectangle (0.3, 0.3);
            \node at (0, 0) {\tiny $#3$};
        \end{scope}
    \end{tikzpicture}}}%
}
\tikzset{->-/.style={decoration={
markings,
mark=at position #1 with {\arrow{>}}},postaction={decorate}}}
\tikzset{-<-/.style={decoration={
markings,
mark=at position #1 with {\arrow{<}}},postaction={decorate}}}
\tikzset{
partial ellipse/.style args={#1:#2:#3}{
    insert path={+ (#1:#3) arc (#1:#2:#3)}
}
}

\newcommand{\leftLineOnlyBox}[1]{%
    \vcenter{\hbox{\begin{tikzpicture}[scale=0.65]
        \draw [blue] (-1.1, 2.3)--(-1.1, -0.8);
        \begin{scope}[shift 
={(-1.1, 0.8)}]
            \draw [fill=white] (-0.3, -0.3) rectangle (0.3, 0.3);
            \node at (0, 0) {\tiny $#1$};  
        \end{scope}
        
        \begin{scope}[shift={(0,1.5)}]
            \draw [blue] (-0.5, 0.8)--(-0.5, 0) .. controls +(0, -0.6) and +(0,-0.6).. (0.5, 0)--(0.5, 0.8); 
        \end{scope}
        
        \draw [blue] (-0.5, -0.8)--(-0.5, 0) .. controls +(0, 0.6) and +(0,0.6).. (0.5, 0)--(0.5, -0.8);
    \end{tikzpicture}}}%
}

\newcommand{\cM}{\mathcal{M}}
\newcommand{\cN}{\mathcal{N}}
\newcommand{\cL}{\mathcal{L}}
\newcommand{\cZ}{\mathcal{Z}}
\newcommand{\cJ}{\mathcal{J}}
\newcommand{\cR}{\mathcal{R}}
\newcommand{\cA}{\mathcal{A}}
\newcommand{\cP}{\mathcal{P}}
\newcommand{\cH}{\mathcal{H}}
\newcommand{\sI}{\mathscr{I}}
\newcommand{\sT}{\mathscr{T}}
\newcommand{\bC}{\mathbb{C}}
\newcommand{\bZ}{\mathbb{Z}}
\newcommand{\bE}{\mathbb{E}}
\newcommand{\bN}{\mathbb{N}}
\newcommand{\bR}{\mathbb{R}}
\newcommand{\bfX}{\mathbf{X}}
\newcommand{\bfY}{\mathbf{Y}}
\newcommand{\id}{\mathrm{Id}}
\newcommand{\Ker}{\mathrm{Ker}}
\newcommand{\Irr}{\mathrm{Irr}}
\newcommand{\K}{\mathbf{K}}
\newcommand{\oM}{\widetilde{\mathcal{M}^{\oplus}}}
\newcommand{\CS}{\mathfrak{CS}}
\newcommand{\fF}{\mathfrak{F}}
\newcommand{\Div}{\mathrm{Div}}
\newcommand{\Ran}{\mathrm{Ran }}
\newcommand{\Dom}{\mathrm{Dom }}
\newcommand{\aPhi}{\Phi^{(a)}}
\newcommand{\diag}{\mathrm{Diag }}
\newcommand{\Tr}{\mathrm{Tr}}
\newcommand{\grad}{\mathrm{grad}}
\newcommand{\ric}{\mathrm{Ric}}
\newcommand{\fix}{\mathscr{M}}
\newcommand{\nfix}{\mathscr{N}}
\newcommand{\jinsong}{\textcolor{red}}
\newcommand{\CBE}{\mathcal{CBE}}
\newcommand{\BE}{\mathcal{BE}}

\title{Intertwining Properties for Bimodule Quantum Markov Semigroups}
\author{Chunlan Jiang}
\address{Chunlan Jiang, Hebei Normal University}
\email{cljiang@hebtu.edu.cn }

\author{Jincheng Wan}
\address{Jincheng Wan, Tsinghua University, Beijing}
\email{wanjc23@mails.tsinghua.edu.cn }

\author{Jinsong Wu}
\address{Jinsong Wu, Beijing Institute of Mathematical Sciences and Applications, Beijing, 101408, China}
\email{wjs@bimsa.cn}

\date{}

\begin{document}

\begin{abstract}
In this paper, we study the Bakry-\'{E}mery estimates for GNS- and KMS-symmetric semigroups in terms of the Fourier multiplier of the gradient form and the iterated gradient form in the framework of quantum Fourier analysis. 
We also systematically investigate the intertwining properties for bimodule GNS- and KMS-symmetric quantum Markov semigroups and compare with the Bakry-\'{E}mery estimates. 
A number of examples of GNS- and KMS-symmetric semigroups satisfying these intertwining properties are presented.
\end{abstract}

\maketitle

\section{Introduction}

The lower Ricci curvature bound is an important concept in classical differential geometry.
An interesting question is how to characterize this bound when the underlying manifold is noncommutative. 
The investigation of lower Ricci curvature bounds in noncommutative geometry was initiated by Carlen and Maas\cite{CarMaa17,  CarMaa20} and subsequently developed by many others\cite{MitMie17, LJL20, WirZha21, WirZha21b}. 
In the commutative setting, lower Ricci curvature can be characterized either by the convexity of entropy on Wasserstein space or by the Bakry-\'{E}mery $\Gamma_2$-criterion\cite{BakEme85}. 
Following the work of Lott, Sturm, and Villani\cite{Stu06a, Stu06b, LotVil09}, Carlen and Maas\cite{CarMaa17} introduced a noncommutative analog of the 2-Wasserstein metric and defined noncommutative lower Ricci curvature via the geodesic semi-convexity of entropy for GNS symmetric quantum Markov semigroups.

In \cite{LJL20}, Li, Junge, and LaRacuente studied the complete version of lower Ricci curvature for tracially symmetric quantum Markov semigroups. 
This was further investigated by Li \cite{Li20} and by Brannan, Gao, and Junge \cite{BGJ22, BGJ22b}. 
In \cite{WirZha21b}, Wirth and Zhang obtained a noncommutative version of curvature-dimension bounds for tracially symmetric quantum Markov semigroups, leading to further functional inequalities and showed that the intertwining properties introduced by Carlen and Maas\cite{CarMaa17} imply the stronger Bakry - Émery estimate and gradient estimate based on dimension for tracially symmetric quantum Markov semigroups.

An alternative way to investigate quantum Markov semigroups is to impose the structure of quantum symmetries. 
In \cite{WuZha25}, the third author and Zhao set up the framework of bimodule quantum Markov semigroups in the language of quantum Fourier analysis \cite{JJLRW20}. 
They completely characterize bimodule GNS quantum Markov semigroups and obtain the modified logarithmic Sobolev inequality for bimodule GNS symmetric quantum Markov semigroups. 
In \cite{JWW25}, the authors systematically investigate bimodule KMS symmetric quantum Markov semigroups, which are quite different from the GNS symmetric quantum Markov semigroups. 
We not only set up the gradient flow for KMS symmetric quantum Markov semigroups with finite quantum symmetries, but also obtain the modified logarithmic Sobolev inequalities for the KMS symmetric semigroup by introducing a modified intertwining property for KMS symmetric quantum groups.

In this paper, we interpret the gradient form and the iterated gradient form as 3-box elements within the framework of quantum Fourier analysis\cite{JJLRW20}. 
By taking advantage of this perspective, the $\Gamma_2$ -criterion reduces to comparing pictorial representations, and we obtain a number of properties of the Bakry-Émery estimate. We find that the Bakry-\'{E}mery estimate for KMS-symmetric quantum Markov semigroups differs significantly from that for GNS-symmetric ones. 
We also give a pictorial representation of the intertwining property, which helps us convert an analytic property into a topological one. Under the intertwining property of GNS- (or KMS-)symmetric quantum Markov semigroups, we obtain some noncommutative versions of curvature–dimension bounds for these semigroups, which partially answers the question raised by Wirth and Zhang. 
By studying the quantum Markov semigroups for quantum symmetries, we find that the commute relation for planar algebra and the flatness of $\alpha$-induced connection imply the intertwining property of the corresponding semigroups.
We also provide a number of examples of GNS- and KMS-symmetric quantum Markov semigroups and obtain the modified logarithmic Sobolev inequalities. 
Finally, by noting the topological interpretation of the intertwining property, we speculate that the optimal value for the intertwining properties is close to the topological properties of the semigroup.

The paper is organized as follows.
In Section 2, we recall the basics of $\lambda$-extensions, Fourier multipliers of bimodule quantum channels and bimodule quantum Markov semigroups, as well as the gradient form and iterated gradient form.
In Section 3, we characterize the Bakry-\'{E}mery estimate for bimodule quantum Markov semigroups and compute the estimate for dephasing semigroups and for a KMS symmetric quantum Markov semigroup on $2\times 2$ matrices.
In Section 4, we recall the intertwining property for (bimodule) GNS-symmetric semigroups and obtain improved Bakry-\'{E}mery and gradient estimates in this setting.
In Section 5, we recall the intertwining property for (bimodule) KMS-symmetric semigroups and obtain improved Bakry-\'{E}mery and gradient estimates in this setting.
In Section 6, we present a number of (bimodule) GNS- or KMS-symmetric quantum Markov semigroups arising from irreducible inclusions and characterize their intertwining properties. Consequently, we obtain the corresponding modified logarithmic Sobolev inequalities and Bakry-\'{E}mery estimates.
In Section 7, we present a number of (bimodule) GNS- or KMS-symmetric quantum Markov semigroups on fermion algebras and characterize their intertwining properties.

\begin{ac}
J. Wu was supported by grants from Beijing Institute of Mathematical Sciences and Applications.
J.~W. was supported by NSFC (Grant no. 12371124). 
C. ~Jiang was supported by Hebei Natural Science Foundation (Grant No. A2023205045) and NSFC (Grant no. 12471120).
The authors would like to thank Zhengwei Liu for discussion.
\end{ac}

\section{Preliminaries}

\subsection{$\lambda$-extensions}

Let $\cN\subset \cM$ be a unital inclusion of finite von Neumann algebras and $\tau$ a normal faithful tracial state on $\cM$. 
The Hilbert space $L^2(\cM, \tau)$ is the Gelfand-Naimark-Segal (GNS) Hilbert space of $\tau$, with cyclic separating vector $\Omega$ and modular conjugation $J$ given by $Jx\Omega = x^*\Omega$ for all $x\in\cM$. 
We denote by $e_1$ the Jones projection and $\bE_{\cN}$ the $\tau$-preserving conditional expectation of $\cM$ onto $\cN$.  
The basic construction $\cM_1=\langle \cM, e_1\rangle$ is the von Neumann algebra generated by $\cM$ and $e_1$. 
The inclusion $\cN\subset\cM$ is called finite if $\cM_1$ is a finite von Neumann algebra, and irreducible if $\cN'\cap \cM=\mathbb{C}$.  
For a finite inclusion, we have $\cM_1 = J\cN'J$, where $\cN'$ is the commutant of $\cN$ on $L^2(\cM,\tau)$. 

Suppose $\tau_1$ is a faithful normal trace on $\cM_1$ extending $\tau$, and let $\mathbb{E}_{\cM}$ be the trace-preserving conditional expectation onto $\cM$. 
The pair $(\cM_1, \tau_1)$ is called a $\lambda$-extension of $\cN\subset\cM$ if $\tau_1|_{\cM}=\tau$ and $\bE_{\cM}(e_1)=\lambda$ for some positive constant $\lambda$. 
The index of the extension is defined as $[\cM:\cN] = \lambda^{-1}$. 
We denote by $\Omega_1$ the cyclic and separating vector in $L^2(\cM_1, \tau_1)$, and by $e_2$ the Jones projection onto $L^2(\cM,\tau_1)$. 
The $\lambda$-extension is called extremal if for all $x\in\cN'\cap \cM$, $\tau_1(x) = \tau_1(Jx^*J)$. 
In this note, $\lambda$-extensions are always assumed to be extremal. 
Let $\mathbb{E}_{\cN'}$ be the $\tau_1$-preserving conditional expectation from $\cM_1$ onto $\cN'\cap \cM_1$. 
Let $\cM_2 = \langle \cM_1,e_2\rangle$ be the basic construction of the inclusion $\cM\subset \cM_1$, with a normal faithful trace $\tau_2$ extending $\tau_1$. 
We assume $\cM_1\subset \cM_2$ is a $\lambda$-extension of $\cM\subset\cM_1$, i.e. $\mathbb{E}_{\cM_1}(e_2) = \lambda$, where $\mathbb{E}_{\cM_1}$ is the $\tau_2$-preserving conditional expectation onto $\cM_1$. 
Denote by $\mathbb{E}_{\cM'}$ the $\tau_2$-preserving conditional expectation from $\cM_2$ onto $\cM'\cap \cM_2$.

For a finite inclusion, there exists a finite set $\{\eta_j\}_{j=1}^m$ of operators in $\cM$ called Pimsner-Popa basis for $\cN\subset \cM$, which satisfies $\displaystyle x=\sum_{j=1}^m  \bE_{\mathcal{N}}(x\eta^*_j)\eta_j$, for all $x\in\cM$\cite{PimPop86}. 
In terms of Jones projection, this condition is expressed as $\displaystyle \sum_{j=1}^m\eta_j^*e_1\eta_j =1$. 
This implies that any operator in $\cM_1$ is a finite sum of operators of the form $ae_1b$ with $a,b\in \cM$. 
As a consequence, for any $y\in \cM_1$, there is a unique $x\in\cM$ such that $ye_1 = xe_1$.
This indicates that $x=\lambda^{-1}\bE_{\cM}(ye_1)$.
That $\cN\subset\cM$ is a $\lambda$-extension implies $\displaystyle \sum_{j=1}^m\eta_j^*\eta_j =\lambda^{-1}$. 
We shall assume that the basis is orthogonal, that is $\bE_{\cN}(\eta_k\eta_j^*)=0$ for $k\neq j$. 
The conditional expectation $\bE^{\cN'}_{\cM'}$ from $\cN' = J\cM_1 J$ onto $\cM' = J\cM J$ can be written as 
\begin{align*}
    \bE^{\cN'}_{\cM'} (x) = \lambda\sum^m_{j=1} \eta^*_j x \eta_j,\quad x\in \cM_2.
\end{align*}
Note that this implies that $\mathbb{E}_{\cM'}(e_1) = \lambda$. 
We also have $\mathbb{E}_{\cM'}(yx) = \mathbb{E}_{\cM'}(xy)$ for all $y\in \cM$ and $x\in \cM_2$. 

The Pimsner-Popa inequality \cite{PimPop86} for the inclusion states that $\bE_{\cN}(x)\geq \lambda_{\cN\subset\cM} x$ for any $0\leq x\in \cM$, where $\lambda_{\cN\subset\cM}$ is the Pimsner-Popa constant. 

The basic construction from a $\lambda$-extension is assumed to be iterated to produce the Jones tower 
\begin{align*}
\cN\subset\cM\subset\cM_1\subset\cM_2\subset\cdots.
\end{align*}
The sequence of higher relative commutants consists the standard invariant of the inclusion initial. 
The standard invariants are axiomatized by planar algebras in \cite{Jones2021}. 

\subsection{Fourier Transforms and Fourier Multipliers}
The Fourier transform $\mathfrak{F}: \cN'\cap \cM_1 \to \cM'\cap \cM_2$ is
\begin{align}\label{eqn:: Fourier transform}
    \mathfrak{F}(x)=&\lambda^{-3/2}\bE_{\cM'}(xe_2e_1), \quad x\in \cN'\cap \cM_1.
\end{align}
The convolution between $x,y\in \cM'\cap \cM_2$ is
\begin{align*}
    x*y =& \mathfrak{F}^{-1}(\mathfrak{F}(y)\mathfrak{F}(x))\\
    =& \lambda^{-9/2}\mathbb{E}_{\cM'}(e_1e_2\mathbb{E}_{\cM_1}(e_2e_1y)\mathbb{E}_{\cM_1}(e_2e_1x)). 
\end{align*}
The shift $\gamma_{1, +}: \cM_1'\cap \cM_3 \to \cN'\cap \cM_1$ is an isomorphism given by
\begin{align*}
\gamma_{1, +}(x)e_3=\lambda^{-2}e_3e_2e_1 x e_1e_2 e_3, \quad x\in \cM_1'\cap \cM_3.
\end{align*}
The inverse $\gamma_{1, +}^{-1}: \cN'\cap \cM_1\to \cM_1'\cap \cM_3$ is given by
\begin{align*}
\gamma_{1, +}^{-1}(x)e_1=\lambda^{-2}e_1e_2 e_3 x e_3e_2e_1, \quad x\in \cN'\cap \cM_1.
\end{align*}
Moreover, we have that $\gamma_{k, +}: \cM_1'\cap \cM_{2k+1}\to \cN'\cap \cM_{2k-1}$ defined as 
\begin{align*}
\gamma_{k, +}(x) e_{2k+1} =\lambda^{-2k}e_{2k+1} e_{2k}\cdots e_1 x e_1 \cdots e_{2k} e_{2k+1}, \quad x\in \cM'\cap \cM_{2k+1}.
\end{align*}
For more details on this string Fourier transform, we refer to \cite{Liu16,JLW16,JLW2019,JJLRW20} etc.

A linear map $\Phi:\cM\rightarrow \cM$ is called positive if it preserves the positive cone $\cM_+$. 
The map $\Phi$ is called completely positive if $\Phi\otimes id_n$ is positive on $\cM\otimes M_n(\mathbb{C})$ for all $n\geq 1$, and completely bounded if $\displaystyle \sup_{n \geq 1}\Vert \Phi\otimes id_n\Vert$ is finite. 
The map $\Phi$ is unital if $\Phi(1)=1$.
A quantum channel is a normal unital completely positive map on $\cM$. 

For a finite inclusion $\cN\subset \cM$ of finite von Neumann algebras, a linear map $\Phi$ is said to be $\cN$-bimodule, if 
\begin{align*}
\Phi(y_1 x y_2) = y_1\Phi(x)y_2,
\end{align*}
for all $y_1, y_2\in\cN$ and $x\in\cM$. 
The Fourier multiplier $\widehat{\Phi}$ is the unique element in $\cM'\cap\cM_2$ such that 
\begin{equation}\label{eqn:: bilinear form induced by Fourier multiplier}
    \langle \widehat{\Phi}(x_1e_1y_1\Omega_1), x_2e_1y_2\Omega_1\rangle = \lambda^{3/2}\tau(y^*_2\Phi(x^*_2x_1)y_1),\quad x_1, x_2,y_1, y_2\in \cM.
\end{equation}
The bimodule map $\Phi$ can be written in terms of the Fourier multiplier $\widehat{\Phi}$ as follows:
\begin{align}\label{eq:multiplierphi}
    \Phi(x)=\lambda^{-5/2} \bE_{\mathcal{M}}(e_2e_1\widehat{\Phi} x e_1e_2).
\end{align}
We shall denote $\lambda^{-5/2} \bE_{\mathcal{M}}(e_2e_1\widehat{\Phi} x e_1e_2)$ by $x*\widehat{\Phi}$ for simplicity.
Recall that $\Phi$ is completely positive if and only if $\widehat{\Phi} \geq 0$.
If $\cN \subset \cM$ admits a downward basic construction, then $\Phi$ is positive if and only if $\Phi$ is completely positive.

The limit $\displaystyle \mathbb{E}_{\Phi} = \lim_{m\rightarrow \infty}\frac{1}{m}\sum^m_{k=1}\Phi^{k}$ exists as a $\cN$-bimodule quantum channel, with the property that $\mathbb{E}^2_{\Phi} = \mathbb{E}_{\Phi}$. 
Moreover, the image of $\mathbb{E}_{\Phi}$ is $\fix(\Phi)$. 
Taking Fourier multiplier gives: 
\begin{align*}
    \widehat{\mathbb{E}}_{\Phi} = \lim_{m \rightarrow \infty}\frac{1}{m}\sum^m _{k=1} \widehat{\Phi}^{(*k)},
\end{align*}
with $\widehat{\mathbb{E}}_{\Phi}*\widehat{\mathbb{E}}_{\Phi} =  \widehat{\mathbb{E}}_{\Phi}$. 
Therefore, $\mathbb{E}_{\Phi}$ is an idempotent with positive Fourier multiplier. 
For more details on bimodule quantum channels, we refer to \cite{HJLW23,HJLW24}.

\subsection{Extensions}
Recall that the associated Jones projections for the inclusion $\cN\subset \cM_1$ are $\widetilde{e}_1=\lambda^{-1} e_2e_1e_3e_2$ and $\widetilde{e}_2=\lambda^{-1} e_4e_3e_5e_4$.

\begin{proposition}\label{prop:extension1}
Suppose that $\Phi: \cM\to \cM$ and $\Psi: \cM_1 \to \cM_1$ are bimodule maps.
Then $\Psi|_{\cM}=\Phi$ if and only if 
\begin{align}\label{eq:extension1}
   \lambda^{-9/2} \bE_{\cM_4}(e_5e_4e_3e_2\widehat{\Psi}e_2e_3e_4e_5)=\widehat{\Phi}.
\end{align}
Pictorially, we have that 
\begin{align*}
    \vcenter{\hbox{
    \begin{tikzpicture}
       \draw [blue] (-0.45, -0.3).. controls+(0, -0.3) and +(0, -0.3).. (-0.75, -0.3)--(-0.75, 0.3) .. controls +(0, 0.3) and +(0, 0.3)..(-0.45, 0.3) (-0.15, -0.6)--(-0.15, 0.6) (0.15, -0.6)--(0.15, 0.6) (0.45, -0.6)--(0.45, 0.6);
       \draw[fill=white] (-0.6, -0.3) rectangle (0.6, 0.3);
       \node at (0, 0) {\tiny $\widehat{\Psi}$};
    \end{tikzpicture}
    }}
    =  \vcenter{\hbox{
    \begin{tikzpicture}
       \draw [blue]   (-0.15, -0.6)--(-0.15, 0.6) (0.15, -0.6)--(0.15, 0.6) (0.45, -0.6)--(0.45, 0.6);
       \draw[fill=white] (-0.35, -0.3) rectangle (0.35, 0.3);
       \node at (0, 0) {\tiny $\widehat{\Phi}$};
    \end{tikzpicture}
    }}.
\end{align*}
\end{proposition}
\begin{proof}
Suppose that $\Psi|_{\cM}=\Phi$.
For any $x\in \cM$, we have that 
\begin{align*}
   \lambda^{-3} \widetilde{e}_2\widetilde{e}_1 \widehat{\Psi} x \widetilde{e}_1 \widetilde{e}_2
   = \Phi(x)\widetilde{e}_2 
   = \lambda^{-5/2} \bE_{\cM}(e_2e_1 \widehat{\Phi} x e_1e_2) \widetilde{e}_2. 
\end{align*}
Multiplying $e_1e_2$ from the left hand side, we obtain that 
\begin{align*}
   \lambda^{-3} e_1e_2\widetilde{e}_2\widetilde{e}_1 \widehat{\Psi} x \widetilde{e}_1 \widetilde{e}_2
    =& \lambda^{-3/2} e_1e_2e_1 \widehat{\Phi} x e_1e_2 \widetilde{e}_2.
\end{align*}
By the fact that $e_1e_2 \widetilde{e}_2\widetilde{e}_1=\lambda^2 e_1e_4e_3e_2$, we have that 
\begin{align*}
 \lambda^{-1} e_1 x e_4e_3e_2  \widehat{\Psi} \widetilde{e}_1 \widetilde{e}_2
    =& \lambda^{-1/2} e_1 x \widehat{\Phi} e_1e_2 \widetilde{e}_2.
\end{align*}
Applying the Pimsner-Popa basis, we obtain that 
\begin{align*}
  e_4e_3e_2  \widehat{\Psi} \widetilde{e}_1 \widetilde{e}_2
    =& \lambda^{1/2}  \widehat{\Phi} e_1e_2 \widetilde{e}_2.
\end{align*}
Multiplying $\widetilde{e}_1$ from the right hand side, we see that 
\begin{align*}
      e_4e_3e_2  \widehat{\Psi} \widetilde{e}_1
    =& \lambda^{1/2}  \widehat{\Phi} e_1e_4e_3e_2.
\end{align*}
Expanding $\widetilde{e}_1$, we have that 
\begin{align*}
    e_4e_3e_2  \widehat{\Psi} e_2e_1e_3e_2
    =& \lambda^{3/2}  \widehat{\Phi} e_1e_4e_3e_2.
\end{align*}
Multiplying $e_3$ from the right hand side, we see that 
\begin{align*}
      e_4e_3e_2  \widehat{\Psi} e_2e_1e_3 
    =& \lambda^{3/2}  \widehat{\Phi} e_1e_4e_3.
\end{align*}
By taking the conditional expectation $\bE_{\cM'}$, we have that 
\begin{align*}
          e_4e_3e_2  \widehat{\Psi} e_2 e_3
    =& \lambda^{5/2} \widehat{\Phi} e_4e_3.
\end{align*}
Multiplying $e_4$ from the right hand side, we have that 
\begin{align}\label{eq:extension2}
e_4e_3e_2  \widehat{\Psi} e_2 e_3e_4
    =\lambda^{7/2} \widehat{\Phi} e_4.
\end{align}
Multiplying $e_5$ from the both sides, we have that 
\begin{align*}
  e_5 e_4e_3e_2  \widehat{\Psi} e_2 e_3e_4 e_5
    =\lambda^{9/2} \widehat{\Phi} e_5 . 
\end{align*}
By taking the conditional expectation $\bE_{\cM_4}$, we have that Equation \eqref{eq:extension1} is true.

Suppose that Equation \eqref{eq:extension1} is true.
We see that Equation \eqref{eq:extension2} is true.
By the previous computation, we see that $\Psi|_{\cM}=\Phi$.
\end{proof}

\begin{remark}
Suppose that $\bC\subset M_n(\bC)$ is the inclusion and $\widehat{\Phi}=\vcenter{\hbox{\begin{tikzpicture}[scale=0.65]
    \begin{scope}[shift={(0,1.5)}]
    \draw [blue] (-0.5, 0.8)--(-0.5, 0) .. controls +(0, -0.6) and +(0,-0.6).. (0.5, 0)--(0.5, 0.8);    
\begin{scope}[shift={(0.5, 0.3)}]
\draw [fill=white] (-0.3, -0.3) rectangle (0.3, 0.3);
\node at (0, 0) {\tiny $v$};
\end{scope}
    \end{scope}
\draw [blue] (-0.5, -0.8)--(-0.5, 0) .. controls +(0, 0.6) and +(0,0.6).. (0.5, 0)--(0.5, -0.8);
\begin{scope}[shift={(0.5, -0.3)}]
\draw [fill=white] (-0.3, -0.3) rectangle (0.3, 0.3);
\node at (0, 0) {\tiny $v^*$};
\end{scope}
\end{tikzpicture}}}$.
The Fourier multiplier of the extension $\Psi$ of $\Phi$ could be 
\begin{enumerate}[(1)]
    \item $\widehat{\Psi}= \vcenter{\hbox{\begin{tikzpicture}[scale=0.65]
    \begin{scope}[shift={(0,1.5)}]
    \draw [blue] (-0.5, 0.8)--(-0.5, 0) .. controls +(0, -0.6) and +(0,-0.6).. (0.5, 0)--(0.5, 0.8);   \draw [blue] (-1, 0.8)--(-1, 0) .. controls +(0, -0.8) and +(0,-0.8).. (1, 0)--(1, 0.8);  
\begin{scope}[shift={(0.5, 0.3)}]
\draw [fill=white] (-0.3, -0.3) rectangle (0.3, 0.3);
\node at (0, 0) {\tiny $v$};
\end{scope}
    \end{scope}
\draw [blue] (-0.5, -0.8)--(-0.5, 0) .. controls +(0, 0.6) and +(0,0.6).. (0.5, 0)--(0.5, -0.8);
\draw [blue] (-1, -0.8)--(-1, 0) .. controls +(0, 0.8) and +(0,0.8).. (1, 0)--(1, -0.8); 
\begin{scope}[shift={(0.5, -0.3)}]
\draw [fill=white] (-0.3, -0.3) rectangle (0.3, 0.3);
\node at (0, 0) {\tiny $v^*$};
\end{scope}
\end{tikzpicture}}}$.
We shall call this extension as a standard lifting of $\widehat{\Phi}$;
\item $\widehat{\Psi}= \lambda^{1/2}\vcenter{\hbox{\begin{tikzpicture}[scale=0.65]
    \begin{scope}[shift={(0,1.5)}]
    \draw [blue] (-0.8, 0.8)--(-0.8, -0.8) (1, 0.8)--(1, -0.8);
    \draw [blue] (-0.5, 0.8)--(-0.5, 0) .. controls +(0, -0.6) and +(0,-0.6).. (0.5, 0)--(0.5, 0.8);    
\begin{scope}[shift={(0.5, 0.3)}]
\draw [fill=white] (-0.3, -0.3) rectangle (0.3, 0.3);
\node at (0, 0) {\tiny $v$};
\end{scope}
    \end{scope}
    \draw [blue] (-0.8, 0.8)--(-0.8, -0.8) (1, 0.8)--(1, -0.8);
\draw [blue] (-0.5, -0.8)--(-0.5, 0) .. controls +(0, 0.6) and +(0,0.6).. (0.5, 0)--(0.5, -0.8);
\begin{scope}[shift={(0.5, -0.3)}]
\draw [fill=white] (-0.3, -0.3) rectangle (0.3, 0.3);
\node at (0, 0) {\tiny $v^*$};
\end{scope}
\end{tikzpicture}}}$;
\item $\widehat{\Psi}=\displaystyle \sum_{j=1}^m  \vcenter{\hbox{\begin{tikzpicture}[scale=0.65]
    \begin{scope}[shift={(0,1.5)}]
    \draw [blue] (-0.5, 0.8)--(-0.5, 0) .. controls +(0, -0.6) and +(0,-0.6).. (0.5, 0)--(0.5, 0.8);  
    \draw [blue] (-1.2, 0.8)--(-1.2, 0) .. controls +(0, -0.8) and +(0,-0.8).. (1.2, 0)--(1.2, 0.8);  
\begin{scope}[shift={(0.5, 0.3)}]
\draw [fill=white] (-0.3, -0.3) rectangle (0.3, 0.3);
\node at (0, 0) {\tiny $v$};
\end{scope}
\begin{scope}[shift={(1.2, 0.3)}]
\draw [fill=white] (-0.3, -0.3) rectangle (0.3, 0.3);
\node at (0, 0) {\tiny $\overline{w_j}$};
\end{scope}
    \end{scope}
\draw [blue] (-0.5, -0.8)--(-0.5, 0) .. controls +(0, 0.6) and +(0,0.6).. (0.5, 0)--(0.5, -0.8);
\draw [blue] (-1.2, -0.8)--(-1.2, 0) .. controls +(0, 0.8) and +(0,0.8).. (1.2, 0)--(1.2, -0.8); 
\begin{scope}[shift={(0.5, -0.3)}]
\draw [fill=white] (-0.3, -0.3) rectangle (0.3, 0.3);
\node at (0, 0) {\tiny $v^*$};
\end{scope}
\begin{scope}[shift={(1.2, -0.3)}]
\draw [fill=white] (-0.3, -0.35) rectangle (0.3, 0.35);
\node at (0, 0) {\tiny $\overline{w_j^*}$};
\end{scope}
\end{tikzpicture}}}$, where $\displaystyle \sum_{j=1}^m w_jw_j^*=1$;
\item $\widehat{\Psi}= \lambda^{1/2}\vcenter{\hbox{\begin{tikzpicture}[scale=0.65]
    \begin{scope}[shift={(0,1.5)}]
    \draw [blue] (-0.8, 0.8)--(-0.8, -0.8) (1, 0.8)--(1, -0.8);
    \draw [blue] (-0.5, 0.8)--(-0.5, 0) .. controls +(0, -0.6) and +(0,-0.6).. (0.5, 0)--(0.5, 0.8);    
\begin{scope}[shift={(0.5, 0.3)}]
\draw [fill=white] (-0.3, -0.3) rectangle (0.3, 0.3);
\node at (0, 0) {\tiny $v$};
\end{scope}
    \end{scope}
    \draw [blue] (-0.8, 0.8)--(-0.8, -0.8) (1, 0.8)--(1, -0.8);
\draw [blue] (-0.5, -0.8)--(-0.5, 0) .. controls +(0, 0.6) and +(0,0.6).. (0.5, 0)--(0.5, -0.8);
\begin{scope}[shift={(0.5, -0.3)}]
\draw [fill=white] (-0.3, -0.3) rectangle (0.3, 0.3);
\node at (0, 0) {\tiny $v^*$};
\end{scope}
    \begin{scope}[shift={(1, 0.7)}]
\draw [fill=white] (-0.3, -0.3) rectangle (0.3, 0.3);
\node at (0, 0) {\tiny $p_j$};
\end{scope}
\begin{scope}[shift={(-0.8, 0.7)}]
\draw [fill=white] (-0.3, -0.3) rectangle (0.3, 0.3);
\node at (0, 0) {\tiny $q_j$};
\end{scope}
\end{tikzpicture}}}$ where $\displaystyle \sum_{j=1}^m \tau(q_j)p_j=1$ and $p_j, q_j \geq 0$.
\item the convex combination of the above forms.
\end{enumerate}
\end{remark}

\begin{remark}
Suppose that $\cN \subset \cM$ is irreducible and $\widehat{\Phi}=p\in \cM'\cap \cM_2$ is a minimal projection. 
Recall the Jones-Wenzl-Liu formula for the minimal projections is
\begin{align*}
   \vcenter{\hbox{ \begin{tikzpicture} 
  \draw [blue] (-0.2, -0.5)--(-0.2, 0.5) (0.2, -0.5)--(0.2, 0.5) (0.6, -0.5)--(0.6, 0.5);
        \draw [fill=white] (-0.4, -0.25) rectangle (0.4, 0.25);
        \node at (0, 0) {\tiny $p$};
  \end{tikzpicture}}}
  = \sum_{j=1}^{m}\frac{\lambda^{1/2} }{\tau_2(p)} \vcenter{\hbox{ \begin{tikzpicture}
    \draw [blue] (-0.6, 1)--(-0.6, -1);
    \draw [blue] (-0.2, 1)--(-0.2, 0.35).. controls +(0, -0.3) and +(0, -0.3).. (0.2, 0.35)--(0.2, 1);
    \draw [blue] (-0.2, -1)--(-0.2, -0.35).. controls +(0, 0.3) and +(0, 0.3).. (0.2, -0.35)--(0.2, -1);
    \begin{scope}[shift={(-0.4, 0.6)}]
        \draw [fill=white] (-0.4, -0.25) rectangle (0.4, 0.25);
        \node at (0, 0) {\tiny $v_j$};
    \end{scope}
    \begin{scope}[shift={(-0.4, -0.6)}]
        \draw [fill=white] (-0.4, -0.25) rectangle (0.4, 0.25);
        \node at (0, 0) {\tiny $v_j^*$};
    \end{scope}
    \end{tikzpicture}}}
    +    \vcenter{\hbox{ \begin{tikzpicture} 
  \draw [blue] (-0.2, -1.25)--(-0.2, 0.5) (0.2,  -1.25)--(0.2, 0.5) (0.6,  -1.25)--(0.6, 0.5);
        \draw [fill=white] (-0.4, -0.25) rectangle (0.4, 0.25);
        \node at (0, 0) {\tiny $p$};
        \begin{scope}[shift={(0, -0.75)}]
        \draw [fill=white] (-0.4, -0.25) rectangle (0.8, 0.25);
        \node at ( 0.2,0) {\tiny $s_{3, +}$};
        \end{scope}
  \end{tikzpicture}}},
\end{align*}
where $s_{3, +}$ is the central projection in $\cM'\cap \cM_3$ onto the complement of the ideal generated by the Jones projections and $v_j^*v_j=p$, $\{v_jv_j^*\}_{j=1}^m$ is a maximal orthogonal family of projections.
Similarly, 
\begin{align*}
   \vcenter{\hbox{ \begin{tikzpicture} 
  \draw [blue] (-0.2, -0.5)--(-0.2, 0.5) (0.2, -0.5)--(0.2, 0.5) (-0.6, -0.5)--(-0.6, 0.5);
        \draw [fill=white] (-0.4, -0.25) rectangle (0.4, 0.25);
        \node at (0, 0) {\tiny $p$};
  \end{tikzpicture}}}
  = \sum_{j=1}^{m}\frac{\lambda^{1/2} }{\tau_2(p)} \vcenter{\hbox{ \begin{tikzpicture}
    \draw [blue] (-0.6, 1)--(-0.6, -1);
    \draw [blue] (0.2, 1)--(0.2, 0.35).. controls +(0, -0.3) and +(0, -0.3).. (-0.2, 0.35)--(-0.2, 1);
    \draw [blue] (0.2, -1)--(0.2, -0.35).. controls +(0, 0.3) and +(0, 0.3).. (-0.2, -0.35)--(-0.2, -1);
    \begin{scope}[shift={(-0.4, 0.6)}]
        \draw [fill=white] (-0.4, -0.25) rectangle (0.4, 0.25);
        \node at (0, 0) {\tiny $v_j$};
    \end{scope}
    \begin{scope}[shift={(-0.4, -0.6)}]
        \draw [fill=white] (-0.4, -0.25) rectangle (0.4, 0.25);
        \node at (0, 0) {\tiny $v_j^*$};
    \end{scope}
    \end{tikzpicture}}}
    +    \vcenter{\hbox{ \begin{tikzpicture} 
  \draw [blue] (-0.2, -1.25)--(-0.2, 0.5) (0.2,  -1.25)--(0.2, 0.5) (0.6,  -1.25)--(0.6, 0.5);
        \draw [fill=white] (0, -0.25) rectangle (0.8, 0.25);
        \node at (0.4, 0) {\tiny $p$};
        \begin{scope}[shift={(0, -0.75)}]
        \draw [fill=white] (-0.4, -0.25) rectangle (0.8, 0.25);
        \node at ( 0.2,0) {\tiny $s_{3, -}$};
        \end{scope}
  \end{tikzpicture}}},
\end{align*}
where $s_{3, -}$ is the central projection in $\cN'\cap \cM_2$ onto the complement of the ideal generated by the Jones projections and $v_j^*v_j=p$, $\{v_jv_j^*\}_{j=1}^m$ is a maximal orthogonal family of projections.
The Fourier multiplier of the extension $\Psi$ of $\Phi$ could be
\begin{enumerate}
    \item $\widehat{\Psi} =\lambda^{-1/2} \vcenter{\hbox{
    \begin{tikzpicture}
       \draw [blue]   (-0.15, -0.6)--(-0.15, 0.6) (0.15, -0.6)--(0.15, 0.6) (0.45, -0.6)--(0.45, 0.6) (-0.45, -0.6)--(-0.45, 0.6);
       \draw[fill=white] (-0.35, -0.3) rectangle (0.35, 0.3);
       \node at (0, 0) {\tiny $p$};
    \end{tikzpicture}
    }}$;
    \item $\widehat{\Psi} =\vcenter{\hbox{\scalebox{0.8}{
        \begin{tikzpicture}[scale=1.2]
           \draw [blue]  (0.2, 0.7)--(0.2, -1.2) (0.7, 0.7)--(0.7, -1.2);
           \draw [blue] (-0.2, 0.7)--(-0.2, -0.3).. controls +(0, -0.3) and +(0, -0.3) .. (-0.7, -0.3)--(-0.7, 0.7);
           \draw [fill=white] (-0.4, -0.3) rectangle (0.4, 0.3);
           \node at (0, 0) {\tiny $p$};
           \begin{scope}[shift={(0, -1.2)}]
           \draw [blue]  (0.2, -0.7)--(0.2, 0.3) (0.7, 0.7)--(0.7, -0.7);
              \draw [blue] (-0.2, -0.7)--(-0.2, 0.3).. controls +(0, 0.3) and +(0, 0.3) .. (-0.7, 0.3)--(-0.7, -0.7);
               \draw [fill=white] (-0.4, -0.3) rectangle (0.4, 0.3);
           \node at (0, 0) {\tiny $ p$};
           \end{scope}
        \end{tikzpicture}}}} $;
    \item $\displaystyle \widehat{\Psi}=\sum_{j=1}^m \frac{\lambda^{-1/2}}{\tau_2(p)}\vcenter{\hbox{\scalebox{0.8}{
        \begin{tikzpicture}[scale=1.2]
           \draw [blue]  (-0.2, 0.7)--(-0.2, -0.3).. controls +(0, -0.3) and +(0, -0.3).. (-0.6, -0.3)--(-0.6, 0.7) (0.2, 0.7)--(0.2, -0.3).. controls +(0, -0.3) and +(0, -0.3).. (0.6, -0.3)--(0.6, 0.7);
           \draw [fill=white] (-0.4, -0.3) rectangle (0.4, 0.3);
           \node at (0, 0) {\tiny $v_j$};
        \begin{scope}[shift={(0, -1.3)}]
             \draw [blue]  (-0.2, -0.7)--(-0.2, 0.3).. controls +(0, 0.3) and +(0, 0.3).. (-0.6, 0.3)--(-0.6, -0.7) (0.2, -0.7)--(0.2, 0.3).. controls +(0, 0.3) and +(0, 0.3).. (0.6, 0.3)--(0.6, -0.7);
           \draw [fill=white] (-0.4, -0.3) rectangle (0.4, 0.3);
           \node at (0, 0) {\tiny $v_j^*$};          
        \end{scope}
        \end{tikzpicture}}}} + \lambda^{1/2} \vcenter{\hbox{ \begin{tikzpicture} 
  \draw [blue] (-0.2, -1.25)--(-0.2, 0.5) (0.2,  -1.25)--(0.2, 0.5) (0.6,  -1.25)--(0.6, 0.5) (-0.6,  -1.25)--(-0.6, 0.5);
        \draw [fill=white] (-0.4, -0.25) rectangle (0.4, 0.25);
        \node at (0, 0) {\tiny $p$};
        \begin{scope}[shift={(0, -0.75)}]
        \draw [fill=white] (-0.4, -0.25) rectangle (0.8, 0.25);
        \node at ( 0.2,0) {\tiny $s_{3, +}$};
        \end{scope}
  \end{tikzpicture}}}$;
\item $\widehat{\Psi} =\vcenter{\hbox{\scalebox{0.8}{
        \begin{tikzpicture}[scale=1.2]
           \draw [blue]  (0.2, 0.7)--(0.2, -1.2) (-0.2, 0.7)--(-0.2, -1.2);
           \draw [white, line width=3] (0.7, -0.2).. controls +(0, -0.4) and +(0, -0.4) .. (-0.7, -0.2);
           \draw [blue] (0.7, 0.7)--(0.7, -0.2).. controls +(0, -0.4) and +(0, -0.4) .. (-0.7, -0.2)--(-0.7, 0.7);
           \draw [fill=white] (-0.4, -0.25) rectangle (0.4, 0.25);
           \node at (0, 0) {\tiny $p$};
           \begin{scope}[shift={(0, -1.2)}]
           \draw [blue]  (0.2, -0.7)--(0.2, 0) (-0.2, -0.7)--(-0.2, 0);
           \draw [white, line width=3] (0.7, 0.2).. controls +(0, 0.4) and +(0, 0.4) .. (-0.7, 0.2);
              \draw [blue] (0.7, -0.7)--(0.7, 0.2).. controls +(0, 0.4) and +(0, 0.4) .. (-0.7, 0.2)--(-0.7, -0.7);
               \draw [fill=white] (-0.4, -0.25) rectangle (0.4, 0.25);
           \node at (0, 0) {\tiny $ p$};
           \end{scope}
        \end{tikzpicture}}}}$, where the associated planar algebra admits braidings.
\item the convex combination of the above forms.
\end{enumerate}
\end{remark}

\subsection{Bimodule Markov Semigroups}

Suppose $\cN\subset \cM$ is a finite inclusion of finite von Neumann algebras.
A continuous family $\{\Phi_t:\cM\to \cM\}_{t\geq 0}$ of quantum channels is a quantum Markov semigroup if
\begin{enumerate}[(1)]
    \item $\Phi_t\Phi_s=\Phi_{t+s}$ for all $t, s\geq 0$.
    \item $\Phi_0=\id$.
    \item $\Phi_t$ is normal for all $t\geq 0$.
\end{enumerate}
We say $\{\Phi_t\}_{t\geq 0}$ is a bimodule quantum Markov semigroup if $\Phi_t$ is a bimodule quantum channel for $t\geq 0$ with respect to $\cN$.
For more details on bimodule quantum Markov semigroups, we refer to \cite{WuZha25}.

Let $\cL$ be the generator of a bimodule quantum Markov semigroup $\{\Phi_t\}_{t\geq 0}$, i.e. $e^{-t\cL}=\Phi_t$, which is also called Lindbladian.
The Fourier multiplier $\widehat{\cL}$ is related to the Fourier multipliers $\{\widehat{\Phi}_t\}_{t\geq 0}$ as $\displaystyle \widehat{\cL}=\lim_{t\to 0} \frac{\lambda^{-1/2} e_2- \widehat{\Phi}_t}{t}$. 
We are interesting in the following two components of the Fourier multiplier $\widehat{\cL}$:
\begin{align*}
    \widehat{\cL}_0= & -(1-e_2)\widehat{\cL}(1-e_2)\geq 0. \\
    \widehat{\cL}_1 =& e_2\widehat{\cL}(1-e_2).
\end{align*}
The generator $\cL$ can be written in terms of $\widehat{\cL}_0$ and $\widehat{\cL}_1$ as follows:
\begin{align}
 \cL(x)=\lambda^{-1/2}\bE_{\cM}(e_2 \widehat{\cL} e_2)x+\lambda^{-3/2} \bE_{\cM}(e_2e_1\widehat{\cL}_1^*) x+ \lambda^{-3/2} x \bE_{\cM}(\widehat{\cL}_1 e_1 e_2)-  x*\widehat{\cL}_0.   
\end{align}
The Laplacian $\cL_a$ of $\{\Phi_t\}_{t\geq 0}$ is 
\begin{align}
    \cL_a(x)= \frac{1}{2}(1*\widehat{\cL}_0) x
        + \frac{1}{2} x  (1*\widehat{\cL}_0)-x* \widehat{\cL}_0, \quad x\in \cM.
\end{align}
Define 
    \begin{align*}
\cL_w (x)
=&  i[x, \Im\bE_{\cM}(\mathfrak{F}^{-1}(\widehat{\cL}_1))], \\
\Im\bE_{\cM}(\mathfrak{F}^{-1}(\widehat{\cL}_1))
= & \frac{i}{2}\left(\bE_{\cM}(\mathfrak{F}^{-1}(\widehat{\cL}_1))^*- \bE_{\cM}(\mathfrak{F}^{-1}(\widehat{\cL}_1))\right).
    \end{align*}
Then $\mathcal{L}$ is decomposed as $\cL=\cL_a+ \cL_w.$


Suppose that $2\mathbf{y}=1*\widehat{\cL}_0$.
Then the Fourier multiplier of $\cL$ can be written as
\begin{align*}
\widehat{\cL}=   \vcenter{\hbox{\begin{tikzpicture}[scale=0.65]
    \begin{scope}[shift={(0,1.5)}]
    \draw [blue] (-0.5, 0.8)--(-0.5, 0) .. controls +(0, -0.6) and +(0,-0.6).. (0.5, 0)--(0.5, 0.8);    
\begin{scope}[shift={(0.5, 0.3)}]
\draw [fill=white] (-0.3, -0.3) rectangle (0.3, 0.3);
\node at (0, 0) {\tiny $\widetilde{\mathbf{y}}$};
\end{scope}
    \end{scope}
\draw [blue] (-0.5, -0.8)--(-0.5, 0) .. controls +(0, 0.6) and +(0,0.6).. (0.5, 0)--(0.5, -0.8);
\end{tikzpicture}}}
+ \vcenter{\hbox{\begin{tikzpicture}[scale=0.65]
    \begin{scope}[shift={(0,1.5)}]
    \draw [blue] (-0.5, 0.8)--(-0.5, 0) .. controls +(0, -0.6) and +(0,-0.6).. (0.5, 0)--(0.5, 0.8);    
    \end{scope}
\draw [blue] (-0.5, -0.8)--(-0.5, 0) .. controls +(0, 0.6) and +(0,0.6).. (0.5, 0)--(0.5, -0.8);
\begin{scope}[shift={(0.5, -0.3)}]
\draw [fill=white] (-0.3, -0.3) rectangle (0.3, 0.3);
\node at (0, 0) {\tiny $\widetilde{\mathbf{y}}^*$};
\end{scope}
\end{tikzpicture}}} - \vcenter{\hbox{
    \begin{tikzpicture}
       \draw [blue] (-0.15, -0.6)--(-0.15, 0.6) (0.15, -0.6)--(0.15, 0.6);
       \draw[fill=white] (-0.35, -0.3) rectangle (0.35, 0.3);
       \node at (0, 0) {\tiny $ \widehat{\cL}_0$};
    \end{tikzpicture}
    }},
\end{align*}
where $\widetilde{\mathbf{y}}=\mathbf{y}+i \Im\bE_{\cM}(\mathfrak{F}^{-1}(\widehat{\cL}_1))$.

\subsection{Derivations}
The derivation $\partial$ is defined as follows:
\begin{align}
    \partial x=\left[x, \mathfrak{F}^{-1}(\widehat{\cL}_0^{1/2})\right], \quad x\in \cM.
\end{align}
    The Fourier multiplier of $\partial$ is in the 3-box space of the corresponding planar algebra and depicted as follows:
    \begin{align*}
       \widehat{\partial}= \vcenter{\hbox{\scalebox{0.8}{
        \begin{tikzpicture}[scale=1.2]
           \draw [blue] (-0.2, 0.4)--(-0.2, 0.9) (0.2, 0.4)--(0.2, 0.9);
           \draw [blue] (-0.2, -0.4)--(0.7, -1.2);
           \draw [blue] (0.2, -0.4).. controls +(0, -0.3) and +(0, -0.3) .. (0.7, -0.4)--(0.7, 0.9);
           \draw [blue] (-0.2, -1.2).. controls +(0, 0.4) and +(0, 0.4).. (0.3, -1.2);
           \draw [fill=white] (-0.5, -0.4) rectangle (0.5, 0.4);
           \node at (0, 0) {\tiny $\widehat{\cL}_0^{1/2}$};
        \end{tikzpicture}
        }}}
        -
              \vcenter{\hbox{\scalebox{0.8}{
        \begin{tikzpicture}[scale=1.2]
           \draw [blue] (-0.2, -0.4)--(-0.2, -0.9) (0.2, -0.4)--(0.2, -0.9);
           \draw [blue] (-0.2, 0.4)--(0.7, 1.2);
           \draw [blue] (0.2, 0.4).. controls +(0, 0.3) and +(0, 0.3) .. (0.7, 0.4)--(0.7, -0.9);
           \draw [blue] (-0.3, 1.2).. controls +(0, -0.4) and +(0, -0.4).. (0.3, 1.2);
           \draw [fill=white] (-0.5, -0.4) rectangle (0.5, 0.4);
           \node at (0, 0) {\tiny $\overline{\widehat{\cL}_0^{1/2}}$};
        \end{tikzpicture}
        }}}
        =
        \vcenter{\hbox{\scalebox{0.8}{
        \begin{tikzpicture}[scale=1.2]
           \draw [blue] (-0.2, 0.4)--(-0.2, 0.9) (0.2, 0.9)--(0.2, -1.2);
           \draw [blue] (-0.2, -0.4).. controls +(0, -0.3) and +(0, -0.3) .. (-0.7, -0.4)--(-0.7, 0.9);
           \draw [blue] (-0.2, -1.2).. controls +(0, 0.4) and +(0, 0.4).. (-0.7, -1.2);
           \draw [fill=white] (-0.55, -0.4) rectangle (0.55, 0.4);
           \node at (0, 0) {\tiny $\mathfrak{F}^{-1}(\widehat{\cL}_0^{1/2})$};
        \end{tikzpicture}}}}
        -
           \vcenter{\hbox{\scalebox{0.8}{
        \begin{tikzpicture}[scale=1.2]
           \draw [blue] (-0.2, -0.4)--(-0.2, -0.9) (0.2, -0.9)--(0.2, 1.2);
           \draw [blue] (-0.2, 0.4).. controls +(0, 0.3) and +(0, 0.3) .. (-0.7, 0.4)--(-0.7, -0.9);
           \draw [blue] (-0.2, 1.2).. controls +(0, -0.4) and +(0, -0.4).. (-0.7, 1.2);
           \draw [fill=white] (-0.55, -0.4) rectangle (0.55, 0.4);
           \node at (0, 0) {\tiny $\mathfrak{F}^{-1}(\widehat{\cL}_0^{1/2})$};
        \end{tikzpicture}}}},
    \end{align*}
where the Fourier multiplier is obtained from the extension of $\partial$ from $\cM_1$ to $\cM_1$.
Note that $\widehat{\partial}\in \cM'\cap \cM_3$.
Moreover, we have that 
\begin{align*}
       \widehat{\partial^*}=  \lambda^{1/2}\left( 
          \vcenter{\hbox{\scalebox{0.8}{
        \begin{tikzpicture}[scale=1.2]
           \draw [blue] (-0.2, 0.4)--(-0.2, 0.9) (0.2, 0.9)--(0.2, -0.4)--(-0.7, -1.2);
           \draw [blue] (-0.2, -0.4).. controls +(0, -0.3) and +(0, -0.3) .. (-0.7, -0.4)--(-0.7, 0.9);
           \draw [blue] (0, -1.2).. controls +(0, 0.4) and +(0, 0.4).. (0.5, -1.2);
           \draw [fill=white] (-0.55, -0.4) rectangle (0.55, 0.4);
           \node at (0, 0) {\tiny $\overline{ \widehat{\cL}_0^{1/2}}$};
        \end{tikzpicture}}}}
        -\vcenter{\hbox{\scalebox{0.8}{
        \begin{tikzpicture}[scale=1.2]
           \draw [blue] (-0.2, -0.4)--(-0.2, -0.9) (0.2, -0.9)--(0.2, 0.4)--(-0.7, 1.2);
           \draw [blue] (-0.2, 0.4).. controls +(0, 0.3) and +(0, 0.3) .. (-0.7, 0.4)--(-0.7, -0.9);
           \draw [blue] (0, 1.2).. controls +(0, -0.4) and +(0, -0.4).. (0.5, 1.2);
           \draw [fill=white] (-0.55, -0.4) rectangle (0.55, 0.4);
           \node at (0, 0) {\tiny $\widehat{\cL}_0^{1/2}$};
        \end{tikzpicture}}}} \right),
        \quad 
               \widehat{\overline{\partial}}= \vcenter{\hbox{\scalebox{0.8}{
        \begin{tikzpicture}[scale=1.2]
           \draw [blue] (-0.2, 0.4)--(-0.2, 0.9) (0.2, 0.4)--(0.2, 0.9);
           \draw [blue] (-0.2, -0.4)--(0.7, -1.2);
           \draw [blue] (0.2, -0.4).. controls +(0, -0.3) and +(0, -0.3) .. (0.7, -0.4)--(0.7, 0.9);
           \draw [blue] (-0.2, -1.2).. controls +(0, 0.4) and +(0, 0.4).. (0.3, -1.2);
           \draw [fill=white] (-0.5, -0.4) rectangle (0.5, 0.4);
           \node at (0, 0) {\tiny $\overline{\widehat{\cL}_0^{1/2}}$};
        \end{tikzpicture}
        }}}
        -
              \vcenter{\hbox{\scalebox{0.8}{
        \begin{tikzpicture}[scale=1.2]
           \draw [blue] (-0.2, -0.4)--(-0.2, -0.9) (0.2, -0.4)--(0.2, -0.9);
           \draw [blue] (-0.2, 0.4)--(0.7, 1.2);
           \draw [blue] (0.2, 0.4).. controls +(0, 0.3) and +(0, 0.3) .. (0.7, 0.4)--(0.7, -0.9);
           \draw [blue] (-0.3, 1.2).. controls +(0, -0.4) and +(0, -0.4).. (0.3, 1.2);
           \draw [fill=white] (-0.5, -0.4) rectangle (0.5, 0.4);
           \node at (0, 0) {\tiny ${\widehat{\cL}_0^{1/2}}$};
        \end{tikzpicture}
        }}},
\end{align*}
where $\partial^*: \cM_1\to \cM$ is the adjoint of $\partial$ and $\widehat{\partial^*}\in \cM_1'\cap \cM_4$, and $\overline{\partial}: \cM \to \cM_1$ is defined by $\overline{\partial } x=\left[x, \fF^{-1}(\widehat{\cL}_0^{1/2})^*\right]$ for all $x\in \cM$.

Let $\displaystyle \widehat{\cL}_0=\sum_{j=1}^m  \omega_j p_j$ be the spectral decomposition and the directional derivation $\partial_j$ is 
\begin{align*}
    \partial_j x=\omega_j^{1/2} \left [x, \mathfrak{F}^{-1}(p_j)\right], \quad x\in \cM. 
\end{align*}
We have that $\displaystyle \partial=\sum_{j=1}^m  \partial_j$, $\displaystyle \partial^* =\sum_{j=1}^m \partial_j^*$.

\subsection{Gradient Form}
The gradient form of a quantum Markov semigroup $\{\Phi_t\}_{t\geq 0}$ is given by 
\begin{align*}
    \Gamma(x, y)=\frac{1}{2} (x^*\cL(y)+\cL(x)^*y-\cL(x^*y))=\frac{\lambda^{-1/2}}{2}\bE_{\cM}((\partial y)^*(\partial x)),
\end{align*}
where $x, y\in \cM$.
Moreover,
\begin{align*}
\Gamma(x, y)=\frac{1}{2}\left( y^*(1*\widehat{\cL}_0 )x -y^*(x*\widehat{\cL}_0)-(y^**\widehat{\cL}_0)x+(y^*x)* \widehat{\cL}_0\right).
\end{align*}
For convenience, we shall write $\Gamma(x)$ for $\Gamma(x,x)$.
We have that
\begin{align*}
 \frac{d}{ds} \Phi_s(\Phi_{t-s}(x)^*\Phi_{t-s}(x))=2\Phi_s\Gamma(\Phi_{t-s}(x))
\end{align*} 
for all $x\in \cM$ and $s\in [0, t]$.
Let $\widehat{\Gamma}\in \cM_1'\cap \cM_4$ be
\begin{align}\label{eq:gradfourier}
\widehat{\Gamma}=\frac{1}{2}\vcenter{\hbox{\scalebox{0.8}{
        \begin{tikzpicture}[scale=1.2]
           \draw [blue]  (0.3, 0.4)--(0.3, -0.9) (0, 0.9)--(0, -0.9) ;
            \draw [blue] (0.3, 0.3)..controls +(0, 0.45) and +(0, 0.45)..(0.8, 0.3)--(0.8, -1.6) .. controls +(0, -0.45) and +(0, -0.45).. (0.3, -1.6);
        \draw [blue] (-0.3, 0.9)--(-0.3, -0.3)..controls +(0, -0.45) and +(0, -0.45)..(-0.8, -0.3)--(-0.8, 0.9);
           \draw [fill=white] (-0.5, -0.3) rectangle (0.5, 0.3);
           \node at (0, 0) {\tiny $\widehat{\partial}$};
           \begin{scope}[shift={(0, -1.4)}]
        \draw [blue]  (0.3, 0.9)--(0.3, -0.3) (0, 0.9)--(0, -0.9) ;
           \draw [blue] (-0.3, -0.9)--(-0.3, 0.3)..controls +(0, 0.45) and +(0, 0.45)..(-0.8, 0.3)--(-0.8, -0.9);
           \draw [fill=white] (-0.5, -0.3) rectangle (0.5, 0.3);
           \node at (0, 0) {\tiny $\widehat{\partial}^*$};
           \end{scope}
        \end{tikzpicture}}}}
        =\frac{1}{2}\left(     \vcenter{\hbox{\scalebox{0.8}{
        \begin{tikzpicture}[scale=1.2]
        \draw [blue] (-0.6, 0.9)--(-0.6, -0.9);
           \draw [blue]  (-0.2, 0.9)--(-0.2, -0.9) (0.2, 0.9)--(0.2, -0.9);
           \draw [fill=white] (-0.4, -0.3) rectangle (0.4, 0.3);
           \node at (0, 0) {\tiny $ \widehat{\cL}_0$};
        \end{tikzpicture}}}}
        -   \vcenter{\hbox{\scalebox{0.8}{
        \begin{tikzpicture}[scale=1.2]
         \draw [blue] (0.3, 0.9).. controls +(0, -0.3) and +(0, -0.3).. (0.6, 0.9);
           \draw [blue]  (0.2, 0.9)--(0.2, -0.9);
           \draw [blue](-0.2, -0.9)--(-0.2, 0.3).. controls +(0, 0.3) and +(0, 0.3).. (-0.6, 0.3)--(-0.6, -0.9);
           \draw [fill=white] (-0.4, -0.3) rectangle (0.5, 0.3);
           \node at (0, 0) {\tiny $ \widehat{\cL}_0$};
        \end{tikzpicture}}}}
        -
        \vcenter{\hbox{\scalebox{0.8}{
        \begin{tikzpicture}[scale=1.2]
         \draw [blue] (0.3, -0.9).. controls +(0, 0.3) and +(0, 0.3).. (0.6, -0.9);
           \draw [blue]  (0.2, 0.9)--(0.2, -0.9);
           \draw [blue](-0.2, 0.9)--(-0.2, -0.3).. controls +(0, -0.3) and +(0, -0.3).. (-0.6, -0.3)--(-0.6, 0.9);
           \draw [fill=white] (-0.4, -0.3) rectangle (0.4, 0.3);
           \node at (0, 0) {\tiny $ \widehat{\cL}_0$};
        \end{tikzpicture}}}}
        +\vcenter{\hbox{\scalebox{0.8}{
        \begin{tikzpicture}[scale=1.2]
         \draw [blue] (0.3, 0.9).. controls +(0, -0.3) and +(0, -0.3).. (0.6, 0.9);
         \draw [blue] (0.3, -0.9).. controls +(0, 0.3) and +(0, 0.3).. (0.6, -0.9);
           \draw [blue]  (0.2, 0.9)--(0.2, -0.9) ;
           \draw [blue] (-0.2, 0.3).. controls +(0, 0.3) and +(0, 0.3).. (-0.6, 0.3)--(-0.6, -0.3).. controls+(0, -0.3) and +(0, -0.3).. (-0.2, -0.3);
           \draw [fill=white] (-0.4, -0.3) rectangle (0.4, 0.3);
           \node at (0, 0) {\tiny $ \widehat{\cL}_0$};
        \end{tikzpicture}}}} \right),
\end{align}
where we omit two vertical lines on the left hand side for simplicity.
We shall interpret the gradient form in terms of $\widehat{\Gamma}$.

\begin{proposition}\label{prop:gammafourier}
For any $x, y\in \cM$, we have that 
\begin{align*}
\Gamma(x, y)=\lambda^{-11/2} \bE_{\cM}\left( \widetilde{e}_2\widetilde{e}_1 y^* e_1 x \widehat{\Gamma} \widetilde{e}_1\widetilde{e}_2\right).
\end{align*}
\end{proposition}
\begin{proof}
Note that 
\begin{align*}
2\widehat{\Gamma}= \gamma_{2, +}^{-1}\left(\widehat{\cL}_0-\lambda^{-1}\widehat{\cL}_0 e_1 e_2-\lambda^{-1}e_2e_1\widehat{\cL}_0+\lambda^{-1/2}(1*\widehat{\cL}_0)e_2\right).
\end{align*}
We have that 
\begin{align*}
& \bE_{\cM}( \widetilde{e}_2\widetilde{e}_1 y^* e_1 x \gamma_{2, +}^{-1}(\widehat{\cL}_0) \widetilde{e}_1\widetilde{e}_2)\\
=& \lambda^{-8}\bE_{\cM}(e_4e_3e_5e_4 e_2e_1 e_3 e_2 y^* e_1 e_2 e_3 e_4e_5\widehat{\cL}_0 e_5 e_4 e_3 e_2 e_1 x e_2 e_1 e_3e_2 e_4 e_5 e_3 e_4)\\
=& \bE_{\cM}(e_4 e_3 e_5 e_2 e_1 y^* x \widehat{\cL}_0 e_1e_2 e_3 e_4 )\\
=& \lambda^{11/2} (y^*x)*\widehat{\cL}_0.
\end{align*}
The rest can be check similarly and we see that the proposition is true.
\end{proof}

Now we can define a bimodule map $\Phi_{\Gamma}: \cM_1 \to \cM$ as $\Phi_{\Gamma}(x)=\lambda^{-5}\bE_{\cM}(\widetilde{e}_2\widetilde{e}_1 x \widehat{\Gamma} \widetilde{e}_1\widetilde{e}_2)$ which is completely positive.
Moreover $\widehat{\Gamma}=\widehat{\Phi_{\Gamma}}$, which is the Fourier multiplier of $\Phi_{\Gamma}$.
We shall call $\widehat{\Gamma}$ as the Fourier multiplier of the gradient form $\Gamma$ by Proposition \ref{prop:gammafourier}.

\begin{remark}
We have the following properties for $\widehat{\Gamma}$:
\begin{enumerate}
    \item $\widehat{\Gamma}\geq 0$ followed from the first form of $\widehat{\Gamma}$ in Equation \eqref{eq:gradfourier}.
    \item  $\vcenter{\hbox{\scalebox{0.8}{
        \begin{tikzpicture}[scale=1.2]
           \draw [blue] (-0.3, 0.3)--(-0.3, -0.6) (0, 0.3)--(0, -0.6) (0.3, 0.6)--(0.3, -0.6) ;
           \draw [blue] (-0.3, 0.3).. controls +(0, 0.25) and +(0, 0.25).. (0, 0.3);
           \draw [fill=white] (-0.5, -0.3) rectangle (0.5, 0.3);
           \node at (0, 0) {\tiny $\widehat{\Gamma}$};
        \end{tikzpicture}}}}= \vcenter{\hbox{\scalebox{0.8}{
        \begin{tikzpicture}[scale=1.2]
           \draw [blue] (-0.3, -0.3)--(-0.3, 0.6) (0, -0.3)--(0, 0.6) (0.3, 0.6)--(0.3, -0.6) ;
           \draw [blue] (-0.3, -0.3).. controls +(0, -0.25) and +(0, -0.25).. (0, -0.3);
           \draw [fill=white] (-0.5, -0.3) rectangle (0.5, 0.3);
           \node at (0, 0) {\tiny $\widehat{\Gamma}$};
        \end{tikzpicture}}}}=0$, followed from the first form of $\widehat{\Gamma}$ in Equation \eqref{eq:gradfourier}.
\item $\displaystyle \vcenter{\hbox{\scalebox{0.8}{
        \begin{tikzpicture}[scale=1.2]
           \draw [blue] (0, 0.6)--(0, -0.6) (0.3, 0.6)--(0.3, -0.6) ;
           \draw [blue] (-0.3, 0.3).. controls +(0, 0.35) and +(0, 0.35).. (-0.7, 0.3)--(-0.7, -0.3).. controls +(0, -0.35) and +(0, -0.35).. (-0.3, -0.3);
           \draw [fill=white] (-0.5, -0.3) rectangle (0.5, 0.3);
           \node at (0, 0) {\tiny $\widehat{\Gamma}$};
        \end{tikzpicture}}}}=\frac{\lambda^{-1/2}}{2} \vcenter{\hbox{\scalebox{0.8}{
        \begin{tikzpicture}[scale=1.2]
           \draw [blue]  (-0.2, 0.6)--(-0.2, -0.6) (0.2, 0.6)--(0.2, -0.6);
           \draw [fill=white] (-0.4, -0.3) rectangle (0.4, 0.3);
           \node at (0, 0) {\tiny $ \widehat{\cL}_0$};
        \end{tikzpicture}}}}+\frac{1}{2}\lambda^{-1}\tau_2(\widehat{\cL}_0) \vcenter{\hbox{\scalebox{0.8}{
        \begin{tikzpicture}[scale=1.2]
     \draw [blue] (0, -0.5).. controls +(0, 0.4) and +(0, 0.4).. (0.4, -0.5);
     \draw [blue] (0, 0.5).. controls +(0, -0.4) and +(0, -0.4).. (0.4, 0.5);       
        \end{tikzpicture}}}}$, followed from the second form of $\widehat{\Gamma}$ in Equation \eqref{eq:gradfourier}.
\item $\displaystyle \vcenter{\hbox{\scalebox{0.8}{
        \begin{tikzpicture}[scale=1.2]
           \draw [blue] (-0.3, 0.6)--(-0.3, -0.6);
           \draw [blue] (0.3, 0.3).. controls +(0, 0.25) and +(0, 0.25).. (0, 0.3);
           \draw [blue] (0.3, -0.3).. controls +(0, -0.25) and +(0, -0.25).. (0, -0.3);
           \draw [fill=white] (-0.5, -0.3) rectangle (0.5, 0.3);
           \node at (0, 0) {\tiny $\widehat{\Gamma}$};
        \end{tikzpicture}}}}=\frac{\lambda^{-1}}{2}\vcenter{\hbox{\scalebox{0.8}{
        \begin{tikzpicture}[scale=1.2]
           \draw [blue]  (0.2, 0.6)--(0.2, -0.6) ;
           \draw [blue] (-0.2, 0.3).. controls +(0, 0.3) and +(0, 0.3).. (-0.6, 0.3)--(-0.6, -0.3).. controls+(0, -0.3) and +(0, -0.3).. (-0.2, -0.3);
           \draw [fill=white] (-0.4, -0.3) rectangle (0.4, 0.3);
           \node at (0, 0) {\tiny $ \widehat{\cL}_0$};
        \end{tikzpicture}}}}$.
\end{enumerate}
\end{remark}

\begin{remark}
Suppose that 
\begin{align*}
 \widehat{\cL}_0=\frac{1}{\lambda^{-1/2}-\lambda^{1/2}}\left(\vcenter{\hbox{\scalebox{0.8}{
        \begin{tikzpicture}[scale=1.2]
           \draw [blue] (0.2, -0.6)--(0.2, 0.6) (-0.2, -0.6)--(-0.2, 0.6);              
        \end{tikzpicture}}}} -\lambda^{1/2}\vcenter{\hbox{\scalebox{0.8}{
        \begin{tikzpicture}[scale=1.2]
     \draw [blue] (0, -0.5).. controls +(0, 0.4) and +(0, 0.4).. (0.4, -0.5);
     \draw [blue] (0, 0.5).. controls +(0, -0.4) and +(0, -0.4).. (0.4, 0.5);       
        \end{tikzpicture}}}}\right).
\end{align*}   
Then $1*\widehat{\cL}_0=1$ and $\cL=(1-\lambda)^{-1}(\id - \bE_{\cN})$, which is the dephasing semigroup.
The Fourier multiplier of the associated gradient form is
\begin{align*}
    \widehat{\Gamma}=\frac{1}{2(\lambda^{-1/2}-\lambda^{1/2})}\left( \vcenter{\hbox{\scalebox{0.8}{
        \begin{tikzpicture}[scale=1.2]
           \draw [blue] (0.2, -0.6)--(0.2, 0.6) (-0.2, -0.6)--(-0.2, 0.6) (-0.6, -0.6)--(-0.6, 0.6);     
        \end{tikzpicture}}}}  
+\lambda^{-1/2}\vcenter{\hbox{\scalebox{0.8}{
        \begin{tikzpicture}[scale=1.2]
     \draw [blue] (0, -0.5).. controls +(0, 0.4) and +(0, 0.4).. (0.4, -0.5);
     \draw [blue] (0, 0.5).. controls +(0, -0.4) and +(0, -0.4).. (0.4, 0.5);       
     \begin{scope}[shift={(-0.4,0)}]
     \draw [blue] (0,-0.5)--(0, 0.5);
           \end{scope}
        \end{tikzpicture}}}}
-   \vcenter{\hbox{\scalebox{0.8}{
        \begin{tikzpicture}[scale=1.2]
        \draw[blue] (-0.4,0.5) --(0.6, -0.5);
        \begin{scope}[shift={(-0.2, 0)}]
     \draw [blue] (0, -0.5).. controls +(0, 0.4) and +(0, 0.4).. (0.4, -0.5);
     \end{scope}
     \draw [blue] (0, 0.5).. controls +(0, -0.4) and +(0, -0.4).. (0.4, 0.5);       
        \end{tikzpicture}}}} 
    - \vcenter{\hbox{\scalebox{0.8}{
        \begin{tikzpicture}[scale=1.2]
        \draw[blue] (0.6,0.5) --(-0.4, -0.5);
     \draw [blue] (0, -0.5).. controls +(0, 0.4) and +(0, 0.4).. (0.4, -0.5);
  \begin{scope}[shift={(-0.2, 0)}]
     \draw [blue] (0, 0.5).. controls +(0, -0.4) and +(0, -0.4).. (0.4, 0.5);   
          \end{scope}
        \end{tikzpicture}}}} \right).
\end{align*}
\end{remark}

\subsection{Iterated Gradient Form}

The iterated gradient form (which is  also called iterated carr\'{e} du champ) $\Gamma_2$ is defined as 
\begin{align*}
\Gamma_2(x, y) =&  \frac{1}{2} (\Gamma(x, \cL(y))+\Gamma(\cL(x), y))-\cL(\Gamma(x, y)))\\  
=& \frac{\lambda^{-1/2}}{4}\bE_{\cM}((\partial \cL(y))^*(\partial x)) + \frac{\lambda^{-1/2}}{4}\bE_{\cM}((\partial y)^*(\partial \cL (x))) - \frac{\lambda^{-1/2}}{4}\cL\bE_{\cM}((\partial y)^*(\partial x)),
\end{align*}
where $x, y\in \cM$.
For convenience, we shall write $\Gamma_2(x)$ for $\Gamma_2(x,x)$.
We have that 
\begin{align*}
\frac{d}{ds} \Phi_s\Gamma(\Phi_{t-s}(x))=2\Phi_s\Gamma_2(\Phi_{t-s}(x)), \quad x\in \cM, \quad s\in [0, t].
\end{align*}
We denote by $\widehat{\Gamma}_2$ the Fourier multiplier of $\Gamma_2$.
Then 
\begin{align*}
\Gamma_2(x, y)=\lambda^{-11/2}\bE_{\cM}(\widetilde{e}_2 \widetilde{e}_1 y^*e_1x \widehat{\Gamma}_2 \widetilde{e}_1 \widetilde{e}_2)
\end{align*} 
for $x, y\in \cM$. 
The associated bimodule map is denoted by $\Phi_{\Gamma_2}$.
By Proposition \ref{prop:gammafourier}, we have that the Fourier multiplier $\widehat{\Gamma}_2$ of $\Gamma_2$ is 
\begin{align*}
4\widehat{\Gamma}_2
=  \vcenter{\hbox{\scalebox{0.8}{
}}}=(I) +(II)-(III).
\end{align*}
By the pictorial representation of $\widehat{\Gamma}$, we have that
\begin{align*}
 (I)= \vcenter{\hbox{\scalebox{0.8}{
        %
}}}.
\end{align*}
By expanding $\widehat{\cL}$, we have that $2\widehat{\Gamma}_2$ is the sum of the following items:
\begin{align*}
\vcenter{\hbox{\scalebox{0.8}{
        %
}}}  . 
\end{align*}

\begin{remark}
We have the following properties for $\widehat{\Gamma}_2$:
\begin{enumerate}
    \item $\widehat{\Gamma}_2 =\widehat{\Gamma}_2^*$. However we do not have $\widehat{\Gamma}_2\geq 0$ in general.
    \item
    \begin{align*}
   2 \vcenter{\hbox{\scalebox{0.8}{
        %
}}} .
\end{align*}
\end{enumerate}
\end{remark}

\begin{proposition}\label{prop:iteratedgrad}
Suppose that $1*\widehat{\cL}_0=1$ and $\cL_w=0$.
We have that 
\begin{align*}
2\widehat{\Gamma}_2= & \vcenter{\hbox{\scalebox{0.8}{
        %
}}}.
\end{align*}
\end{proposition}
\begin{proof}
By the assumptions, we have that $\displaystyle \widetilde{\mathbf{y}}=\mathbf{y}=\frac{1}{2}$.
By the Fourier multiplier of $\widehat{\Gamma}_2$, we see that the proposition is true.
\end{proof}

\begin{remark}
Under the assumption that $1*\widehat{\cL}_0=1$ and $\cL_w=0$, we have that 
\begin{enumerate}
 \item  $\vcenter{\hbox{\scalebox{0.8}{
        %
}}}\right).
\end{align*} 
The Fourier multiplier of the associated iterated gradient form is
\begin{align*}
    \widehat{\Gamma}_2=\frac{\lambda^{-1/2}}{2(\lambda^{-1/2}-\lambda^{1/2})^2}\left( \frac{3}{2} \vcenter{\hbox{\scalebox{0.8}{
        %
}}} \right)\geq 0.
\end{align*}
The matrix form of $\widehat{\Gamma}_2$ is 
\begin{align*}
\frac{\lambda^{-1/2}}{2(\lambda^{-1/2}-\lambda^{1/2})^2}
 %
\end{align*}
with respect to the system of matrix units 
\begin{align*}
\{e_2,  (\lambda-\lambda^2)^{-1/2} (e_1\vee e_2-e_2)e_1e_2, (\lambda-\lambda^2)^{-1/2} e_2e_1(e_1\vee e_2-e_2),e_1\vee e_2-e_2, 1-e_1\vee e_2\}.
\end{align*}
Note that $1-e_1\vee e_2=0$ if $\displaystyle \lambda=\frac{1}{2}$.
\end{remark}

\section{Bakry-\'{E}mery Estimate}
In this section, we investigate the Bakry-\'{E}mery's $\Gamma_2$-criterion\cite{BakEme85} for bimodule quantum Markov semigroups.

\begin{definition}
Suppose that $\beta\in \bR$ and $\{\Phi_t\}_{t \geq 0}$ is a quantum Markov semigroup.
We say $\{\Phi_t\}_{t \geq 0}$ satisfies the condition $\CBE(\beta)$ if 
\begin{align*}
    \left(\Gamma(\Phi_t(x_j), \Phi_t(x_k))\right)_{j,k=1}^n \leq e^{-2\beta t}\left( \Phi_t\Gamma(x_j, x_k) \right)_{j,k=1}^n,
\end{align*}
for all $x_1, \ldots, x_n\in \cM$ and $n\in \bN$.
\end{definition}

By taking the differentiation, we see that the condition $\CBE(\beta)$ is equivalent to 
\begin{align}\label{eq:beeq1}
\left(\Gamma_2(x_j, x_k) \right)_{j,k=1}^n  \geq \beta \left(\Gamma(x_j, x_k) \right)_{j,k=1}^n.
\end{align}
Recall that $\{\Phi_t\}_{t \geq 0}$ satisfies the condition $\BE(\beta)$ if $ \Gamma(\Phi_t(x), \Phi_t(x)) \leq e^{-2\beta t} \Phi_t\Gamma(x, x)$ for all $x\in \cM$.
This is equivalent to $\Gamma_2(x, x) \geq \beta \Gamma(x, x)$ for all $x\in \cM$.

\begin{theorem}
Suppose that $\{\Phi_t\}_{t\geq 0}$ is a bimodule quantum Markov semigroup.
Then $\{\Phi_t\}_{t \geq 0}$ satisfies $\CBE(\beta)$ if and only if 
\begin{align}\label{eq:cbe2}
    \widehat{\Gamma}_2 \geq \beta \widehat{\Gamma}.
\end{align}
Pictorially, we have that
\begin{align*}
\vcenter{\hbox{\scalebox{0.8}{
        \begin{tikzpicture}[scale=1.2]
           \draw [blue]  (0.3, 0.4)--(0.3, -0.9) (0, 0.9)--(0, -0.9) ;
            \draw [blue] (0.3, 0.3)..controls +(0, 0.45) and +(0, 0.45)..(0.8, 0.3)--(0.8, -1.6) .. controls +(0, -0.45) and +(0, -0.45).. (0.3, -1.6);
           \draw [blue] (-0.3, 0.3)..controls +(0, 0.45) and +(0, 0.45)..(-0.8, 0.3);
        \draw [blue] (-0.3, -0.3)..controls +(0, -0.45) and +(0, -0.45)..(-0.8, -0.3);
           \draw [fill=white] (-0.5, -0.3) rectangle (0.5, 0.3);
           \node at (0, 0) {\tiny $\widehat{\partial}$};
           \begin{scope}[shift={(-1, 0)}]
           \draw [blue] (-0.2, 0.9)--(-0.2, -0.3)..controls +(0, -0.45) and +(0, -0.45)..(-0.6, -0.3)--(-0.6, 0.9);
           \draw [fill=white] (-0.4, -0.3) rectangle (0.4, 0.3);
           \node at (0, 0) {\tiny $ \widehat{\cL}$};               
           \end{scope}
           \begin{scope}[shift={(0, -1.4)}]
        \draw [blue]  (0.3, 0.9)--(0.3, -0.3) (0, 0.9)--(0, -0.9) ;
           \draw [blue] (-0.3, -0.9)--(-0.3, 0.3)..controls +(0, 0.45) and +(0, 0.45)..(-0.8, 0.3)--(-0.8, -0.9);
           \draw [fill=white] (-0.5, -0.3) rectangle (0.5, 0.3);
           \node at (0, 0) {\tiny $\widehat{\partial}^*$};
           \end{scope}
        \end{tikzpicture}}}}
        +
         \vcenter{\hbox{\scalebox{0.8}{
        \begin{tikzpicture}[scale=1.2]
           \draw [blue]  (0.3, 0.9)--(0.3, -0.3) (0, 0.9)--(0, -0.9) ;
        \draw [blue] (0.3, -0.3)..controls +(0, -0.45) and +(0, -0.45)..(0.8, -0.3)--(0.8, 1.6) .. controls +(0, 0.45) and +(0, 0.45).. (0.3, 1.6);
           \draw [blue] (-0.3, 0.3)..controls +(0, 0.45) and +(0, 0.45)..(-0.8, 0.3);
        \draw [blue] (-0.3, -0.3)..controls +(0, -0.45) and +(0, -0.45)..(-0.8, -0.3);
           \draw [fill=white] (-0.5, -0.3) rectangle (0.5, 0.3);
           \node at (0, 0) {\tiny $\widehat{\partial}^*$};
           \begin{scope}[shift={(-1, 0)}]
           \draw [blue] (-0.2, -0.9)--(-0.2, 0.3)..controls +(0, 0.45) and +(0, 0.45)..(-0.6, 0.3)--(-0.6, -0.9);
           \draw [fill=white] (-0.4, -0.3) rectangle (0.4, 0.3);
           \node at (0, 0) {\tiny $ \widehat{\cL}$};               
           \end{scope}
           \begin{scope}[shift={(0, 1.4)}]
        \draw [blue]  (0.3, 0.3)--(0.3, -0.9) (0, 0.9)--(0, -0.9) ;
           \draw [blue] (-0.3, 0.9)--(-0.3, -0.3)..controls +(0, -0.45) and +(0, -0.45)..(-0.8, -0.3)--(-0.8, 0.9);
           \draw [fill=white] (-0.5, -0.3) rectangle (0.5, 0.3);
           \node at (0, 0) {\tiny $\widehat{\partial}$};
           \end{scope}
        \end{tikzpicture}}}}
        - \vcenter{\hbox{\scalebox{0.8}{
        \begin{tikzpicture}[scale=1.2]
           \draw [blue]  (0.3, 0.3)--(0.3, -0.9) (0, 0.3)--(0, -0.9) ;
        \draw [blue] (-0.3, 0.9)--(-0.3, -0.3)..controls +(0, -0.45) and +(0, -0.45)..(-0.8, -0.3)--(-0.8, 0.9);
           \draw [fill=white] (-0.5, -0.3) rectangle (0.5, 0.3);
           \node at (0, 0) {\tiny $\widehat{\partial}$};
           \begin{scope}[shift={(0, -1.4)}]
        \draw [blue]  (0.3, 0.9)--(0.3, -0.3) (0, 0.9)--(0, -0.3) ;
           \draw [blue] (-0.3, -0.9)--(-0.3, 0.3)..controls +(0, 0.45) and +(0, 0.45)..(-0.8, 0.3)--(-0.8, -0.9);
           \draw [fill=white] (-0.5, -0.3) rectangle (0.5, 0.3);
           \node at (0, 0) {\tiny $\widehat{\partial}^*$};
           \end{scope}
         \begin{scope}[shift={(1.4, -0.7)}]
           \draw [blue] (-0.6, 0)--(-0.6, 1)..controls +(0, 0.4) and +(0, 0.4)..(-1.1, 1);
            \draw [blue] (-0.2, 0.3)--(-0.2, 1)..controls +(0, 0.65) and +(0, 0.65)..(-1.4, 1);
           \draw [blue] (-0.6, 0)--(-0.6, -1)..controls +(0, -0.4) and +(0, -0.4)..(-1.1, -1);
            \draw [blue] (-0.2, -0.3)--(-0.2, -1)..controls +(0, -0.65) and +(0, -0.65)..(-1.4, -1);
            \draw [blue] (0.1, -1.6)--(0.1, 1.6) ;
           \draw [fill=white] (-0.4, -0.3) rectangle (0.4, 0.3);
           \node at (0, 0) {\tiny $ \widehat{\cL}$};               
           \end{scope}
        \end{tikzpicture}}}}
        \geq 2\beta \vcenter{\hbox{\scalebox{0.8}{
        \begin{tikzpicture}[scale=1.2]
           \draw [blue]  (0.3, 0.4)--(0.3, -0.9) (0, 0.9)--(0, -0.9) ;
            \draw [blue] (0.3, 0.3)..controls +(0, 0.45) and +(0, 0.45)..(0.8, 0.3)--(0.8, -1.6) .. controls +(0, -0.45) and +(0, -0.45).. (0.3, -1.6);
        \draw [blue] (-0.3, 0.9)--(-0.3, -0.3)..controls +(0, -0.45) and +(0, -0.45)..(-0.8, -0.3)--(-0.8, 0.9);
           \draw [fill=white] (-0.5, -0.3) rectangle (0.5, 0.3);
           \node at (0, 0) {\tiny $\widehat{\partial}$};
           \begin{scope}[shift={(0, -1.4)}]
        \draw [blue]  (0.3, 0.9)--(0.3, -0.3) (0, 0.9)--(0, -0.9) ;
           \draw [blue] (-0.3, -0.9)--(-0.3, 0.3)..controls +(0, 0.45) and +(0, 0.45)..(-0.8, 0.3)--(-0.8, -0.9);
           \draw [fill=white] (-0.5, -0.3) rectangle (0.5, 0.3);
           \node at (0, 0) {\tiny $\widehat{\partial}^*$};
           \end{scope}
        \end{tikzpicture}}}}.
\end{align*}
\end{theorem}
\begin{proof}
 Suppose that $\widehat{\Gamma}_2 \geq \beta \widehat{\Gamma}$.
 We have that the associated bimodule maps satisfies $\Phi_{\Gamma_2}-\beta\Phi_{\Gamma}$ is completely positive. 
 Hence
 \begin{equation}\label{eq:beeq2}
 \begin{aligned}
& \left( ( \Phi_{\Gamma_2}-\beta\Phi_{\Gamma})\otimes \id_n\right)
\left(
 \begin{pmatrix}
 x_1^* \\ \vdots \\ x_n^*
 \end{pmatrix} e_1
 \begin{pmatrix}
 x_1 & \cdots & x_n
 \end{pmatrix}\right)
 \geq & 0.
 \end{aligned}
 \end{equation}
 This implies that Equation \eqref{eq:beeq1} is true and $\{\Phi_t\}_{t\geq 0}$ satisfies $\CBE(\beta)$.
 
Suppose that $\{\Phi_t\}_{t\geq 0}$ satisfies $\CBE(\beta)$.
We see that Inequality \eqref{eq:beeq2} is true.
This implies that 
\begin{align*}
\left((\Phi_{\Gamma_2}-\beta\Phi_{\Gamma})\otimes \id_n\right)\left(
\begin{pmatrix}
x_1^*e_1 y_1^* \\ \vdots \\ x_n^*e_1 y_n^*
\end{pmatrix}
\begin{pmatrix}
y_1e_1x_1 & \cdots & y_n e_1 x_n
\end{pmatrix}\right)
\geq 0.
\end{align*}
Hence $\Phi_{\Gamma_2}-\beta\Phi_{\Gamma}$ is completely positive and $\widehat{\Gamma}_2 - \beta \widehat{\Gamma}\geq 0$.
\end{proof}

In \cite{WirZha21b}, Wirth and Zhang pointed out that $\{\Phi_t\}_{t\geq 0}$ satisfies $\CBE(\beta)$ if $\{\Phi_t\}_{t\geq 0}$ satisfies $\BE(\beta)$ and $\cM$ is commutative.
In the following, we present a different sufficient condition.
\begin{proposition}
Suppose that the inclusion $\cN \subset \cM_1$ admits a downward basic construction and for all $x_1, \ldots, x_n\in \cM$, $n\in \bN$ with $1\leq n \leq \lfloor \lambda^{-1}\rfloor+1$,
\begin{align*}
\sum_{j,k=1}^n \Gamma_2(x_j, x_k) \geq \beta \sum_{j,k=1}^n \Gamma(x_j, x_k).
\end{align*}
Then $\{\Phi_t\}_{t\geq 0}$ satisfies $\CBE(\beta)$.
\end{proposition}
\begin{proof}
Note that any positive element in $\cM_1$ is the form of $\displaystyle \sum_{j,k=1}^n x_j^* e_1 x_k$ for some $x_j\in \cM$, where $n\leq \lfloor \lambda^{-1}\rfloor +1$.
By the assumption, we have that $\Phi_{\Gamma_2}-\beta\Phi_{\Gamma}$ is a positive map.
Note that a positive bimodule map is completely positive if the associated inclusion admits downward basic construction.
Hence $\Phi_{\Gamma_2}-\beta\Phi_{\Gamma}$ is a completely positive map and the desired result is true.
\end{proof}

\begin{proposition}
Suppose that $\{\Phi_t\}_{t\geq 0}$ satisfies $\CBE(\beta)$.
If $\cL_w=0$ and $1*\widehat{\cL}_0=1$, we have that
\begin{align*}
\frac{1}{2}\lambda \vcenter{\hbox{\scalebox{0.8}{
        \begin{tikzpicture}[scale=1.2]
            \draw [blue] (-1.6, 0.8)--(-1.6, -0.8);
           \draw [blue] (-0.2, 0.3)..controls +(0, 0.45) and +(0, 0.45)..(-0.8, 0.3);
           \draw [blue] (0.2, 0.3)..controls +(0, 0.75) and +(0, 0.75)..(-1.2, 0.3);
        \draw [blue] (-0.2, -0.3)..controls +(0, -0.45) and +(0, -0.45)..(-0.8, -0.3);
        \draw [blue] (0.2, -0.3)..controls +(0, -0.75) and +(0, -0.75)..(-1.2, -0.3);
           \draw [fill=white] (-0.4, -0.3) rectangle (0.4, 0.3);
           \node at (0, 0) {\tiny $\widehat{\cL}_0$};
           \begin{scope}[shift={(-1, 0)}]
           \draw [fill=white] (-0.4, -0.3) rectangle (0.4, 0.3);
           \node at (0, 0) {\tiny $ \widehat{\cL}_0$};               
           \end{scope}
        \end{tikzpicture}}}}  
        + \frac{1}{2}\lambda^{1/2}\vcenter{\hbox{\scalebox{0.8}{
        \begin{tikzpicture}[scale=1.2]
           \draw [blue] (0.2, 0.3)--(0.2, -0.9) ;
           \draw [blue] (-0.2, 0.3)..controls +(0, 0.45) and +(0, 0.45)..(-0.8, 0.3);
           \draw [blue] (0.2, 0.3)..controls +(0, 0.75) and +(0, 0.75)..(-1.2, 0.3);
        \draw [blue] (-0.2, -0.3)..controls +(0, -0.45) and +(0, -0.45)..(-0.8, -0.3);
           \draw [fill=white] (-0.4, -0.3) rectangle (0.4, 0.3);
           \node at (0, 0) {\tiny $\widehat{\cL}_0$};
           \begin{scope}[shift={(-1, 0)}]
           \draw [blue] (-0.2, 0.3)--(-0.2, -0.3)..controls +(0, -0.45) and +(0, -0.45)..(-0.6, -0.3)--(-0.6, 0.9);
           \draw [fill=white] (-0.4, -0.3) rectangle (0.4, 0.3);
           \node at (0, 0) {\tiny $ \widehat{\cL}_0$};               
           \end{scope}
        \end{tikzpicture}}}}
        +   \frac{1}{2}\lambda^{1/2} \vcenter{\hbox{\scalebox{0.8}{
        \begin{tikzpicture}[scale=1.2]
           \draw [blue] (0.2, 0.9)--(0.2, -0.3) ;
           \draw [blue] (-0.2, 0.3)..controls +(0, 0.45) and +(0, 0.45)..(-0.8, 0.3);
        \draw [blue] (-0.2, -0.3)..controls +(0, -0.45) and +(0, -0.45)..(-0.8, -0.3);
          \draw [blue] (0.2, -0.3)..controls +(0, -0.75) and +(0, -0.75)..(-1.2, -0.3);
           \draw [fill=white] (-0.4, -0.3) rectangle (0.4, 0.3);
           \node at (0, 0) {\tiny $\widehat{\cL}_0$};
           \begin{scope}[shift={(-1, 0)}]
           \draw [blue] (-0.2, -0.3)--(-0.2, 0.3)..controls +(0, 0.45) and +(0, 0.45)..(-0.6, 0.3)--(-0.6, -0.9);
           \draw [fill=white] (-0.4, -0.3) rectangle (0.4, 0.3);
           \node at (0, 0) {\tiny $ \widehat{\cL}_0$};               
           \end{scope}
        \end{tikzpicture}}}}+ \frac{1}{2} \geq \beta.
\end{align*}
Equivalently,
\begin{align*}
\frac{1}{2}\bE_{\cM}(\widehat{\cL}_0\overline{\widehat{\cL}_0}) +\frac{1}{2}\lambda^{1/2} \left(1*(\widehat{\cL}_0\overline{\widehat{\cL}_0}+ \overline{\widehat{\cL}_0}\widehat{\cL}_0)\right)+\frac{1}{2}\geq \beta.
\end{align*}
Furthermore, if $\cN \subset \mathcal{M}$ is irreducible, we have that $\displaystyle \beta \leq \frac{3}{2} \tau_2(\widehat{\cL}_0\overline{\widehat{\cL}_0})+\frac{1}{2}$.
\end{proposition}
\begin{proof}
Multiplying $\vcenter{\hbox{\scalebox{0.8}{
        \begin{tikzpicture}[scale=1.2]
     \draw [blue] (0, -0.5).. controls +(0, 0.4) and +(0, 0.4).. (0.4, -0.5);
     \draw [blue] (0, 0.5).. controls +(0, -0.4) and +(0, -0.4).. (0.4, 0.5);       
     \begin{scope}[shift={(-0.4,0)}]
     \draw [blue] (0,-0.5)--(0, 0.5);
           \end{scope}
        \end{tikzpicture}}}} $ from the left and the right hand sides of Equation \eqref{eq:cbe2} and applying the Fourier multipliers of $\Gamma$ and $\Gamma_2$, we obtain that 
\begin{align*}
& \lambda^{-1}\vcenter{\hbox{\scalebox{0.8}{
        \begin{tikzpicture}[scale=1.2]
        \begin{scope}[shift={(-0.4,0)}]
     \draw [blue] (0,-0.5)--(0, 0.5);
        \draw [fill=white] (-0.25, -0.3) rectangle (0.25, 0.3);
           \node at (0, 0) {\tiny $ \mathbf{y}\widetilde{\mathbf{y}}^*$};
           \end{scope}
        \end{tikzpicture}}}}    
        + \lambda^{-1} \vcenter{\hbox{\scalebox{0.8}{
        \begin{tikzpicture}[scale=1.2]
       \begin{scope}[shift={(-0.4,0)}]
     \draw [blue] (0,-0.5)--(0, 0.5);
        \draw [fill=white] (-0.25, -0.3) rectangle (0.25, 0.3);
           \node at (0, 0) {\tiny $\widetilde{\mathbf{y}} \mathbf{y}$};
           \end{scope}
        \end{tikzpicture}}}} 
    + \frac{1}{2}\vcenter{\hbox{\scalebox{0.8}{
        \begin{tikzpicture}[scale=1.2]
           \draw [blue] (-0.2, 0.3)..controls +(0, 0.45) and +(0, 0.45)..(-0.8, 0.3);
           \draw [blue] (0.2, 0.3)..controls +(0, 0.75) and +(0, 0.75)..(-1.2, 0.3);
        \draw [blue] (-0.2, -0.3)..controls +(0, -0.45) and +(0, -0.45)..(-0.8, -0.3);
        \draw [blue] (0.2, -0.3)..controls +(0, -0.75) and +(0, -0.75)..(-1.2, -0.3);
           \draw [fill=white] (-0.4, -0.3) rectangle (0.4, 0.3);
           \node at (0, 0) {\tiny $\widehat{\cL}_0$};
           \begin{scope}[shift={(-1, 0)}]
           \draw [fill=white] (-0.4, -0.3) rectangle (0.4, 0.3);
           \node at (0, 0) {\tiny $ \widehat{\cL}_0$};               
           \end{scope}
        \end{tikzpicture}}}}  
        + \frac{\lambda^{-1/2}}{2}\vcenter{\hbox{\scalebox{0.8}{
        \begin{tikzpicture}[scale=1.2]
           \draw [blue] (0.2, 0.3)--(0.2, -0.9) ;
           \draw [blue] (-0.2, 0.3)..controls +(0, 0.45) and +(0, 0.45)..(-0.8, 0.3);
           \draw [blue] (0.2, 0.3)..controls +(0, 0.75) and +(0, 0.75)..(-1.2, 0.3);
        \draw [blue] (-0.2, -0.3)..controls +(0, -0.45) and +(0, -0.45)..(-0.8, -0.3);
           \draw [fill=white] (-0.4, -0.3) rectangle (0.4, 0.3);
           \node at (0, 0) {\tiny $\widehat{\cL}_0$};
           \begin{scope}[shift={(-1, 0)}]
           \draw [blue] (-0.2, 0.3)--(-0.2, -0.3)..controls +(0, -0.45) and +(0, -0.45)..(-0.6, -0.3)--(-0.6, 0.9);
           \draw [fill=white] (-0.4, -0.3) rectangle (0.4, 0.3);
           \node at (0, 0) {\tiny $ \widehat{\cL}_0$};               
           \end{scope}
        \end{tikzpicture}}}}
        +   \frac{\lambda^{-1/2}}{2} \vcenter{\hbox{\scalebox{0.8}{
        \begin{tikzpicture}[scale=1.2]
           \draw [blue] (0.2, 0.9)--(0.2, -0.3) ;
           \draw [blue] (-0.2, 0.3)..controls +(0, 0.45) and +(0, 0.45)..(-0.8, 0.3);
        \draw [blue] (-0.2, -0.3)..controls +(0, -0.45) and +(0, -0.45)..(-0.8, -0.3);
          \draw [blue] (0.2, -0.3)..controls +(0, -0.75) and +(0, -0.75)..(-1.2, -0.3);
           \draw [fill=white] (-0.4, -0.3) rectangle (0.4, 0.3);
           \node at (0, 0) {\tiny $\widehat{\cL}_0$};
           \begin{scope}[shift={(-1, 0)}]
           \draw [blue] (-0.2, -0.3)--(-0.2, 0.3)..controls +(0, 0.45) and +(0, 0.45)..(-0.6, 0.3)--(-0.6, -0.9);
           \draw [fill=white] (-0.4, -0.3) rectangle (0.4, 0.3);
           \node at (0, 0) {\tiny $ \widehat{\cL}_0$};               
           \end{scope}
        \end{tikzpicture}}}} \\
      &   -\lambda^{-1/2}\vcenter{\hbox{\scalebox{0.8}{
        \begin{tikzpicture}[scale=1.2]
           \draw [blue]  (0.2, 0.9)--(0.2, -1.2).. controls +(0, -0.2) and +(0, -0.2) .. (-0.2, -1.2);
           \draw [blue] (-0.2, -1.2)--(-0.2, 0.3).. controls +(0, 0.3) and +(0, 0.3) .. (-0.7, 0.3)--(-0.7, -1.2);
           \draw [fill=white] (-0.4, -0.3) rectangle (0.4, 0.3);
           \node at (0, 0) {\tiny $\widehat{\cL}_0$};
           \begin{scope}[shift={(0.2, -0.7)}]
               \draw [fill=white] (-0.2, -0.3) rectangle (0.2, 0.3);
           \node at (0, 0) {\tiny $ \mathbf{y}$};
           \end{scope}
        \end{tikzpicture}}}} 
        -  \lambda^{-1/2}\vcenter{\hbox{\scalebox{0.8}{
        \begin{tikzpicture}[scale=1.2]
           \draw [blue] (-0.2, 1.2).. controls +(0, 0.2) and +(0, 0.2) .. (0.2, 1.2)--(0.2, -0.9);
           \draw [blue] (-0.2, 1.2)--(-0.2, -0.3).. controls +(0, -0.3) and +(0, -0.3) .. (-0.7, -0.3)--(-0.7, 1.2);
           \draw [fill=white] (-0.4, -0.3) rectangle (0.4, 0.3);
           \node at (0, 0) {\tiny $\widehat{\cL}_0$};
           \begin{scope}[shift={(0.2, 0.7)}]
               \draw [fill=white] (-0.2, -0.3) rectangle (0.2, 0.3);
           \node at (0, 0) {\tiny $\mathbf{y}$};
           \end{scope}
        \end{tikzpicture}}}} 
\geq  \beta \lambda^{-1}\vcenter{\hbox{\scalebox{0.8}{
        \begin{tikzpicture}[scale=1.2]
           \draw [blue]  (0.2, 0.6)--(0.2, -0.6) ;
           \draw [blue] (-0.2, 0.3).. controls +(0, 0.3) and +(0, 0.3).. (-0.6, 0.3)--(-0.6, -0.3).. controls+(0, -0.3) and +(0, -0.3).. (-0.2, -0.3);
           \draw [fill=white] (-0.4, -0.3) rectangle (0.4, 0.3);
           \node at (0, 0) {\tiny $ \widehat{\cL}_0$};
        \end{tikzpicture}}}}.
\end{align*}
By the assumption, we have that $\displaystyle \widetilde{\mathbf{y}}=\mathbf{y}=\frac{1}{2}(1*\widehat{\cL}_0)=\frac{1}{2}$ and the inequality is true.
\end{proof}

\begin{remark}
By Proposition \ref{prop:dephaseem}, we see that the inequality becomes equality for dephasing semigroups.
\end{remark}

\begin{proposition}\label{prop:beeq1}
Suppose that $\cL_w=0$ and $1*\widehat{\cL}_0=1$.    
Then $\{\Phi_t\}_{t \geq 0}$ satisfies $\CBE(\beta)$ if and only if 
\begin{equation}\label{eq:cbe1}
\begin{aligned}
& \vcenter{\hbox{\scalebox{0.8}{
}}}.
\end{aligned}
\end{equation}
\end{proposition}
\begin{proof}
By Proposition \ref{prop:iteratedgrad} and Equation \eqref{eq:gradfourier}, we have that 
\begin{align*}
2\widehat{\Gamma}_2= &  \vcenter{\hbox{\scalebox{0.8}{
        %
}}} \right)=2\beta \widehat{\Gamma}.
\end{align*}
Reformulating it, we see that the proposition is true.

\end{proof}

\begin{corollary}\label{cor:conv1}
Suppose that $\cL_w=0$, $1*\widehat{\cL}_0=1$ and $\{\Phi_t\}_{t\geq 0}$ satisfies $\CBE(\beta)$.
Then 
\begin{align*} 
\widehat{\cL}_0* \widehat{\cL}_0 \geq (2\beta-2)\widehat{\cL}_0 -2\lambda^{1/2} \widehat{\cL}_0 ^2+\lambda^{-1/2}\left(2\beta-1\right) e_2 +\lambda^{1/2} \vcenter{\hbox{\scalebox{0.8}{
        %
}}}.
\end{align*}
If moreover $1*\overline{\widehat{\cL}_0}=1$, then we have that 
\begin{align*}
    |\fF(\widehat{\cL}_0)|^2 +\Re \left(\fF(\widehat{\cL}_0)\right)^2 \geq 2\beta-2 +(4-2\beta) \Re\fF(\widehat{\cL}_0).
\end{align*}
\end{corollary}
\begin{proof}
By taking the conditional expectation $\bE_{\cM'}$ to Inequality \eqref{eq:cbe1}, we have that
\begin{align*}
\lambda^{-1/2}\vcenter{\hbox{\scalebox{0.8}{
        %
}}} .
\end{align*}
This proves the first inequality.

If $1*\overline{\widehat{\cL}_0}=1$, we take the conditional expectation $\bE_{\cM_1}$ to Inequality \eqref{eq:cbe1} and obtain that 
\begin{align*}
& \vcenter{\hbox{\scalebox{0.8}{
        %
}}}.
\end{align*}
This completes the proof of the proposition.
\end{proof}

 \begin{remark}
 We have the following observations for Corollary \ref{cor:conv1}:
 \begin{enumerate}[(1)]
     \item If the inclusion $\cN \subset \cM$ is irreducible, then 
     \begin{align*}
     \widehat{\cL}_0* \widehat{\cL}_0 \geq (2\beta-2)\widehat{\cL}_0 -2\lambda^{1/2} \widehat{\cL}_0 ^2+\lambda^{-1/2}\left(2\beta-1+2\tau_2(\widehat{\cL}_0\overline{\widehat{\cL}_0})\right) e_2.
     \end{align*}
     \item Suppose that $p$ is a minimal projection in $\cM'\cap \cM_2$ such that $p\leq \cR(\widehat{\cL}_0)$.
     Then 
     \begin{align*}
     \beta \leq 1+\frac{\lambda^{1/2}\tau_2(p \widehat{\cL}_0)}{\tau_2(p)} +\frac{1}{2}\frac{\tau_2(p(\widehat{\cL}_0* \widehat{\cL}_0))}{\tau_2(p \widehat{\cL}_0)}.
     \end{align*}
 \end{enumerate} 
 \end{remark}

\begin{proposition}\label{prop:dephaseem}
Suppose that $\displaystyle \widehat{\cL}_0=\frac{1}{\lambda^{-1/2}-\lambda^{1/2}}\left(\vcenter{\hbox{\scalebox{0.8}{
        \begin{tikzpicture}[scale=1.2]
           \draw [blue] (0.2, -0.6)--(0.2, 0.6) (-0.2, -0.6)--(-0.2, 0.6);              
        \end{tikzpicture}}}} -\lambda^{1/2}\vcenter{\hbox{\scalebox{0.8}{
        \begin{tikzpicture}[scale=1.2]
     \draw [blue] (0, -0.5).. controls +(0, 0.4) and +(0, 0.4).. (0.4, -0.5);
     \draw [blue] (0, 0.5).. controls +(0, -0.4) and +(0, -0.4).. (0.4, 0.5);       
        \end{tikzpicture}}}}\right)$.
Then
\begin{align*}
    \widehat{\Gamma}_2 \geq \left(\frac{1}{2}+\lambda\right)\frac{1}{(1-\lambda)} \widehat{\Gamma},
\end{align*}
where the coefficient is optimal.
\end{proposition}
\begin{proof}

By Proposition \ref{prop:beeq1}, we see that the associated semigroup satisfies $\CBE(\beta)$ if and only if 
\begin{align*}
& \frac{\lambda^{-1/2}}{2(\lambda^{-1/2}-\lambda^{1/2})^2}\left( \frac{3}{2} \vcenter{\hbox{\scalebox{0.8}{
        \begin{tikzpicture}[scale=1.2]
           \draw [blue] (0.2, -0.6)--(0.2, 0.6) (-0.2, -0.6)--(-0.2, 0.6) (-0.6, -0.6)--(-0.6, 0.6);     
        \end{tikzpicture}}}}  
+\frac{\lambda^{-1/2}}{2}\vcenter{\hbox{\scalebox{0.8}{
        \begin{tikzpicture}[scale=1.2]
     \draw [blue] (0, -0.5).. controls +(0, 0.4) and +(0, 0.4).. (0.4, -0.5);
     \draw [blue] (0, 0.5).. controls +(0, -0.4) and +(0, -0.4).. (0.4, 0.5);       
     \begin{scope}[shift={(-0.4,0)}]
     \draw [blue] (0,-0.5)--(0, 0.5);
           \end{scope}
        \end{tikzpicture}}}}
- \frac{1}{2}   \vcenter{\hbox{\scalebox{0.8}{
        \begin{tikzpicture}[scale=1.2]
        \draw[blue] (-0.4,0.5) --(0.6, -0.5);
        \begin{scope}[shift={(-0.2, 0)}]
     \draw [blue] (0, -0.5).. controls +(0, 0.4) and +(0, 0.4).. (0.4, -0.5);
     \end{scope}
     \draw [blue] (0, 0.5).. controls +(0, -0.4) and +(0, -0.4).. (0.4, 0.5);       
        \end{tikzpicture}}}} 
    -\frac{1}{2}   \vcenter{\hbox{\scalebox{0.8}{
        \begin{tikzpicture}[scale=1.2]
        \draw[blue] (0.6,0.5) --(-0.4, -0.5);
     \draw [blue] (0, -0.5).. controls +(0, 0.4) and +(0, 0.4).. (0.4, -0.5);
  \begin{scope}[shift={(-0.2, 0)}]
     \draw [blue] (0, 0.5).. controls +(0, -0.4) and +(0, -0.4).. (0.4, 0.5);   
          \end{scope}
        \end{tikzpicture}}}}-\lambda^{1/2} \vcenter{\hbox{\scalebox{0.8}{
        \begin{tikzpicture}[scale=1.2]
     \draw [blue] (0, -0.5).. controls +(0, 0.4) and +(0, 0.4).. (0.4, -0.5);
     \draw [blue] (0, 0.5).. controls +(0, -0.4) and +(0, -0.4).. (0.4, 0.5);       
     \begin{scope}[shift={(0.8,0)}]
     \draw [blue] (0,-0.5)--(0, 0.5);
           \end{scope}
        \end{tikzpicture}}}} \right) \\
\geq &  \frac{\beta}{2(\lambda^{-1/2}-\lambda^{1/2})}\left( \vcenter{\hbox{\scalebox{0.8}{
        \begin{tikzpicture}[scale=1.2]
           \draw [blue] (0.2, -0.6)--(0.2, 0.6) (-0.2, -0.6)--(-0.2, 0.6) (-0.6, -0.6)--(-0.6, 0.6);     
        \end{tikzpicture}}}}  
+\lambda^{-1/2}\vcenter{\hbox{\scalebox{0.8}{
        \begin{tikzpicture}[scale=1.2]
     \draw [blue] (0, -0.5).. controls +(0, 0.4) and +(0, 0.4).. (0.4, -0.5);
     \draw [blue] (0, 0.5).. controls +(0, -0.4) and +(0, -0.4).. (0.4, 0.5);       
     \begin{scope}[shift={(-0.4,0)}]
     \draw [blue] (0,-0.5)--(0, 0.5);
           \end{scope}
        \end{tikzpicture}}}}
-   \vcenter{\hbox{\scalebox{0.8}{
        \begin{tikzpicture}[scale=1.2]
        \draw[blue] (-0.4,0.5) --(0.6, -0.5);
        \begin{scope}[shift={(-0.2, 0)}]
     \draw [blue] (0, -0.5).. controls +(0, 0.4) and +(0, 0.4).. (0.4, -0.5);
     \end{scope}
     \draw [blue] (0, 0.5).. controls +(0, -0.4) and +(0, -0.4).. (0.4, 0.5);       
        \end{tikzpicture}}}} 
    - \vcenter{\hbox{\scalebox{0.8}{
        \begin{tikzpicture}[scale=1.2]
        \draw[blue] (0.6,0.5) --(-0.4, -0.5);
     \draw [blue] (0, -0.5).. controls +(0, 0.4) and +(0, 0.4).. (0.4, -0.5);
  \begin{scope}[shift={(-0.2, 0)}]
     \draw [blue] (0, 0.5).. controls +(0, -0.4) and +(0, -0.4).. (0.4, 0.5);   
          \end{scope}
        \end{tikzpicture}}}} \right).
\end{align*}
This is equivalent to the following inequality:
\begin{align*}
& \begin{pmatrix}
\frac{1}{2}\lambda^{-1}+\frac{1}{2}-\lambda & -\frac{1}{2}\sqrt{\lambda^{-1}-1}-\sqrt{\lambda(1-\lambda)} & 0\\
-\frac{1}{2}\sqrt{\lambda^{-1}-1}-\sqrt{\lambda(1-\lambda)} & \frac{1}{2}+\lambda & 0 \\
0 & 0 & \frac{3}{2}
\end{pmatrix}\\
\geq&  \beta(1-\lambda)
\begin{pmatrix}
\lambda^{-1}-1 & -\sqrt{\lambda^{-1}-1} & 0 \\
-\sqrt{\lambda^{-1}-1} & 1 & 0 \\
0 & 0 & 1
\end{pmatrix}.
\end{align*}
The inequality is true if and only if 
\begin{align*}
\frac{1}{2}\lambda^{-1}+\frac{1}{2}-\lambda \geq & \beta(1-\lambda)(\lambda^{-1}-1),\\
\frac{1}{2}+\lambda\geq &  \beta(1-\lambda),\\
\frac{3}{2} \geq &  \beta(1-\lambda),\\
\left(\frac{1}{2}\lambda^{-1}+\frac{1}{2}-\lambda -\beta(1-\lambda)^2\lambda^{-1}\right)\left(\frac{1}{2}+\lambda- \beta(1-\lambda)\right)&\\
\geq  \left((\frac{1}{2}-\beta(1-\lambda))\sqrt{\lambda^{-1}-1}+\sqrt{\lambda(1-\lambda)}\right)^2. &
\end{align*}
Now we see that the optimal value for $\beta$ is $\displaystyle \left(\frac{1}{2}+\lambda\right)\frac{1}{(1-\lambda)}$.
In this case, we have that $\displaystyle \beta=\frac{3}{2}\tau_2(\widehat{\cL}_0\overline{\widehat{\cL}_0})+\frac{1}{2}$.
\end{proof}

\begin{remark}
Suppose that $\cN \subset \cM$ is an irreducible inclusion with index $<4$.
We have that 
\begin{align*}
    \widehat{\Gamma}_2 \geq \frac{2\cos^2 \frac{\pi}{m}+1}{4\cos^2 \frac{\pi}{m}-1} \widehat{\Gamma}=\frac{\cos\frac{2\pi}{m}+2}{2\cos\frac{2\pi}{m}+1}\widehat{\Gamma},
\end{align*}
where $m=4,5, 6, \ldots.$
\end{remark}

\begin{remark}
By using the Fourier multiplier of the iterated gradient form and the gradient form, we obtain that for any $x\in \cM$ with $\bE_{\cN}(x)=0$, 
\begin{align*}
    \Gamma(x, x) =\frac{\lambda^{-1/2}}{2(\lambda^{-1/2}-\lambda^{1/2})} \left(\bE_{\cN} (x^*x) + x^*x\right)
\end{align*}
and 
\begin{align*}
     \Gamma_2(x, x) =\frac{\lambda^{-1}}{2(\lambda^{-1/2}-\lambda^{1/2})^2} \left(\frac{3}{2}\bE_{\cN} (x^*x) + \frac{1}{2}x^*x\right).
\end{align*}
Let $\lambda_{pp}$ be the associated Pimsner-Popa constant.
Then we see that 
\begin{align*}
   \Gamma_2(x, x) \geq \left(\frac{1}{2} + \frac{\lambda_{pp}}{1+\lambda_{pp}}\right)\frac{1}{1-\lambda} \Gamma(x, x),
\end{align*}
for all $x\in \cM$.
When the Pimsner-Popa constant equals to $\lambda$ (such as the irreducible inclusions), we have that 
\begin{align*}
   \Gamma_2(x, x) \geq \left(\frac{1}{2} + \frac{\lambda}{1+\lambda}\right)\frac{1}{1-\lambda} \Gamma(x, x).
\end{align*}
Note that Proposition \ref{prop:dephaseem} provides a better constant for this case.
\end{remark}

\begin{example}
Suppose that 
\begin{align*}
    \widehat{\cL}_0=\frac{1}{2}\mu^{-1}
    \vcenter{\hbox{\begin{tikzpicture}[scale=0.65]
    \begin{scope}[shift={(0,1.5)}]
    \draw [blue] (-0.5, 0.8)--(-0.5, 0) .. controls +(0, -0.6) and +(0,-0.6).. (0.5, 0)--(0.5, 0.8);    
\begin{scope}[shift={(0.5, 0.3)}]
\draw [fill=white] (-0.3, -0.3) rectangle (0.3, 0.3);
\node at (0, 0) {\tiny $v$};
\end{scope}
    \end{scope}
\draw [blue] (-0.5, -0.8)--(-0.5, 0) .. controls +(0, 0.6) and +(0,0.6).. (0.5, 0)--(0.5, -0.8);
\begin{scope}[shift={(0.5, -0.3)}]
\draw [fill=white] (-0.3, -0.3) rectangle (0.3, 0.3);
\node at (0, 0) {\tiny $v^*$};
\end{scope}
\end{tikzpicture}}}
+\frac{1}{2}\mu\vcenter{\hbox{\begin{tikzpicture}[scale=0.65]
    \begin{scope}[shift={(0,1.5)}]
    \draw [blue] (-0.5, 0.8)--(-0.5, 0) .. controls +(0, -0.6) and +(0,-0.6).. (0.5, 0)--(0.5, 0.8);    
\begin{scope}[shift={(0.5, 0.3)}]
\draw [fill=white] (-0.3, -0.3) rectangle (0.3, 0.3);
\node at (0, 0) {\tiny $v^*$};
\end{scope}
    \end{scope}
\draw [blue] (-0.5, -0.8)--(-0.5, 0) .. controls +(0, 0.6) and +(0,0.6).. (0.5, 0)--(0.5, -0.8);
\begin{scope}[shift={(0.5, -0.3)}]
\draw [fill=white] (-0.3, -0.3) rectangle (0.3, 0.3);
\node at (0, 0) {\tiny $v$};
\end{scope}
\end{tikzpicture}}}
+\frac{1}{2}\gamma \vcenter{\hbox{\begin{tikzpicture}[scale=0.65]
    \begin{scope}[shift={(0,1.5)}]
    \draw [blue] (-0.5, 0.8)--(-0.5, 0) .. controls +(0, -0.6) and +(0,-0.6).. (0.5, 0)--(0.5, 0.8);    
\begin{scope}[shift={(0.5, 0.3)}]
\draw [fill=white] (-0.3, -0.3) rectangle (0.3, 0.3);
\node at (0, 0) {\tiny $v$};
\end{scope}
    \end{scope}
\draw [blue] (-0.5, -0.8)--(-0.5, 0) .. controls +(0, 0.6) and +(0,0.6).. (0.5, 0)--(0.5, -0.8);
\begin{scope}[shift={(0.5, -0.3)}]
\draw [fill=white] (-0.3, -0.3) rectangle (0.3, 0.3);
\node at (0, 0) {\tiny $v$};
\end{scope}
\end{tikzpicture}}}
+
\frac{1}{2}\gamma \vcenter{\hbox{\begin{tikzpicture}[scale=0.65]
    \begin{scope}[shift={(0,1.5)}]
    \draw [blue] (-0.5, 0.8)--(-0.5, 0) .. controls +(0, -0.6) and +(0,-0.6).. (0.5, 0)--(0.5, 0.8);    
\begin{scope}[shift={(0.5, 0.3)}]
\draw [fill=white] (-0.3, -0.3) rectangle (0.3, 0.3);
\node at (0, 0) {\tiny $v^*$};
\end{scope}
    \end{scope}
\draw [blue] (-0.5, -0.8)--(-0.5, 0) .. controls +(0, 0.6) and +(0,0.6).. (0.5, 0)--(0.5, -0.8);
\begin{scope}[shift={(0.5, -0.3)}]
\draw [fill=white] (-0.3, -0.3) rectangle (0.3, 0.3);
\node at (0, 0) {\tiny $v^*$};
\end{scope}
\end{tikzpicture}}}.
\end{align*}
We have that 
\begin{align*}
4\widehat{\Gamma}
=& 4\vcenter{\hbox{\begin{tikzpicture}[scale=0.65]
\begin{scope}[shift={(-1,0.8)}]
\draw [blue] (0, -1.5)--(0, 1.5);
\draw [fill=white] (-0.3, -0.3) rectangle (0.3, 0.3);
\node at (0, 0) {\tiny $\mathbf{y}$};
\end{scope}
    \begin{scope}[shift={(0,1.5)}]
    \draw [blue] (-0.5, 0.8)--(-0.5, 0) .. controls +(0, -0.6) and +(0,-0.6).. (0.5, 0)--(0.5, 0.8);    
    \end{scope}
\draw [blue] (-0.5, -0.8)--(-0.5, 0) .. controls +(0, 0.6) and +(0,0.6).. (0.5, 0)--(0.5, -0.8);
\end{tikzpicture}}}+
\mu^{-1}\vcenter{\hbox{\begin{tikzpicture}[scale=0.65]
\begin{scope}[shift={(-1,0.8)}]
\draw [blue] (0, -1.5)--(0, 1.5);
\end{scope}
    \begin{scope}[shift={(0,1.5)}]
    \draw [blue] (-0.5, 0.8)--(-0.5, 0) .. controls +(0, -0.6) and +(0,-0.6).. (0.5, 0)--(0.5, 0.8);    
\begin{scope}[shift={(0.5, 0.3)}]
\draw [fill=white] (-0.3, -0.3) rectangle (0.3, 0.3);
\node at (0, 0) {\tiny $v$};
\end{scope}
    \end{scope}
\draw [blue] (-0.5, -0.8)--(-0.5, 0) .. controls +(0, 0.6) and +(0,0.6).. (0.5, 0)--(0.5, -0.8);
\begin{scope}[shift={(0.5, -0.3)}]
\draw [fill=white] (-0.3, -0.3) rectangle (0.3, 0.3);
\node at (0, 0) {\tiny $v^*$};
\end{scope}
\end{tikzpicture}}}
+\mu\vcenter{\hbox{\begin{tikzpicture}[scale=0.65]
\begin{scope}[shift={(-1,0.8)}]
\draw [blue] (0, -1.5)--(0, 1.5);
\end{scope}
    \begin{scope}[shift={(0,1.5)}]
    \draw [blue] (-0.5, 0.8)--(-0.5, 0) .. controls +(0, -0.6) and +(0,-0.6).. (0.5, 0)--(0.5, 0.8);    
\begin{scope}[shift={(0.5, 0.3)}]
\draw [fill=white] (-0.3, -0.3) rectangle (0.3, 0.3);
\node at (0, 0) {\tiny $v^*$};
\end{scope}
    \end{scope}
\draw [blue] (-0.5, -0.8)--(-0.5, 0) .. controls +(0, 0.6) and +(0,0.6).. (0.5, 0)--(0.5, -0.8);
\begin{scope}[shift={(0.5, -0.3)}]
\draw [fill=white] (-0.3, -0.3) rectangle (0.3, 0.3);
\node at (0, 0) {\tiny $v$};
\end{scope}
\end{tikzpicture}}}
+
\gamma\vcenter{\hbox{\begin{tikzpicture}[scale=0.65]
\begin{scope}[shift={(-1,0.8)}]
\draw [blue] (0, -1.5)--(0, 1.5);
\end{scope}
    \begin{scope}[shift={(0,1.5)}]
    \draw [blue] (-0.5, 0.8)--(-0.5, 0) .. controls +(0, -0.6) and +(0,-0.6).. (0.5, 0)--(0.5, 0.8);    
\begin{scope}[shift={(0.5, 0.3)}]
\draw [fill=white] (-0.3, -0.3) rectangle (0.3, 0.3);
\node at (0, 0) {\tiny $v$};
\end{scope}
    \end{scope}
\draw [blue] (-0.5, -0.8)--(-0.5, 0) .. controls +(0, 0.6) and +(0,0.6).. (0.5, 0)--(0.5, -0.8);
\begin{scope}[shift={(0.5, -0.3)}]
\draw [fill=white] (-0.3, -0.3) rectangle (0.3, 0.3);
\node at (0, 0) {\tiny $v$};
\end{scope}
\end{tikzpicture}}}
+ \gamma\vcenter{\hbox{\begin{tikzpicture}[scale=0.65]
\begin{scope}[shift={(-1,0.8)}]
\draw [blue] (0, -1.5)--(0, 1.5);
\end{scope}
    \begin{scope}[shift={(0,1.5)}]
    \draw [blue] (-0.5, 0.8)--(-0.5, 0) .. controls +(0, -0.6) and +(0,-0.6).. (0.5, 0)--(0.5, 0.8);    
\begin{scope}[shift={(0.5, 0.3)}]
\draw [fill=white] (-0.3, -0.3) rectangle (0.3, 0.3);
\node at (0, 0) {\tiny $v^*$};
\end{scope}
    \end{scope}
\draw [blue] (-0.5, -0.8)--(-0.5, 0) .. controls +(0, 0.6) and +(0,0.6).. (0.5, 0)--(0.5, -0.8);
\begin{scope}[shift={(0.5, -0.3)}]
\draw [fill=white] (-0.3, -0.3) rectangle (0.3, 0.3);
\node at (0, 0) {\tiny $v^*$};
\end{scope}
\end{tikzpicture}}}\\
& -\mu^{-1}
\vcenter{\hbox{\begin{tikzpicture}[scale=0.65]
\begin{scope}[shift={(-1,0.8)}]
\draw [blue] (0, -1.5)--(0, 1.5);
\draw [fill=white] (-0.3, -0.3) rectangle (0.3, 0.3);
\node at (0, 0) {\tiny $v^*$};
\end{scope}
    \begin{scope}[shift={(0,1.5)}]
    \draw [blue] (-0.5, 0.8)--(-0.5, 0) .. controls +(0, -0.6) and +(0,-0.6).. (0.5, 0)--(0.5, 0.8);    
\begin{scope}[shift={(0.5, 0.3)}]
\draw [fill=white] (-0.3, -0.3) rectangle (0.3, 0.3);
\node at (0, 0) {\tiny $v$};
\end{scope}
    \end{scope}
\draw [blue] (-0.5, -0.8)--(-0.5, 0) .. controls +(0, 0.6) and +(0,0.6).. (0.5, 0)--(0.5, -0.8);
\end{tikzpicture}}}
-\mu
\vcenter{\hbox{\begin{tikzpicture}[scale=0.65]
\begin{scope}[shift={(-1,0.8)}]
\draw [blue] (0, -1.5)--(0, 1.5);
\draw [fill=white] (-0.3, -0.3) rectangle (0.3, 0.3);
\node at (0, 0) {\tiny $v$};
\end{scope}
    \begin{scope}[shift={(0,1.5)}]
    \draw [blue] (-0.5, 0.8)--(-0.5, 0) .. controls +(0, -0.6) and +(0,-0.6).. (0.5, 0)--(0.5, 0.8);    
\begin{scope}[shift={(0.5, 0.3)}]
\draw [fill=white] (-0.3, -0.3) rectangle (0.3, 0.3);
\node at (0, 0) {\tiny $v^*$};
\end{scope}
    \end{scope}
\draw [blue] (-0.5, -0.8)--(-0.5, 0) .. controls +(0, 0.6) and +(0,0.6).. (0.5, 0)--(0.5, -0.8);
\end{tikzpicture}}}
-\gamma
\vcenter{\hbox{\begin{tikzpicture}[scale=0.65]
\begin{scope}[shift={(-1,0.8)}]
\draw [blue] (0, -1.5)--(0, 1.5);
\draw [fill=white] (-0.3, -0.3) rectangle (0.3, 0.3);
\node at (0, 0) {\tiny $v$};
\end{scope}
    \begin{scope}[shift={(0,1.5)}]
    \draw [blue] (-0.5, 0.8)--(-0.5, 0) .. controls +(0, -0.6) and +(0,-0.6).. (0.5, 0)--(0.5, 0.8);    
\begin{scope}[shift={(0.5, 0.3)}]
\draw [fill=white] (-0.3, -0.3) rectangle (0.3, 0.3);
\node at (0, 0) {\tiny $v$};
\end{scope}
    \end{scope}
\draw [blue] (-0.5, -0.8)--(-0.5, 0) .. controls +(0, 0.6) and +(0,0.6).. (0.5, 0)--(0.5, -0.8);
\end{tikzpicture}}}
-\gamma
\vcenter{\hbox{\begin{tikzpicture}[scale=0.65]
\begin{scope}[shift={(-1,0.8)}]
\draw [blue] (0, -1.5)--(0, 1.5);
\draw [fill=white] (-0.3, -0.3) rectangle (0.3, 0.3);
\node at (0, 0) {\tiny $v^*$};
\end{scope}
    \begin{scope}[shift={(0,1.5)}]
    \draw [blue] (-0.5, 0.8)--(-0.5, 0) .. controls +(0, -0.6) and +(0,-0.6).. (0.5, 0)--(0.5, 0.8);    
\begin{scope}[shift={(0.5, 0.3)}]
\draw [fill=white] (-0.3, -0.3) rectangle (0.3, 0.3);
\node at (0, 0) {\tiny $v^*$};
\end{scope}
    \end{scope}
\draw [blue] (-0.5, -0.8)--(-0.5, 0) .. controls +(0, 0.6) and +(0,0.6).. (0.5, 0)--(0.5, -0.8);
\end{tikzpicture}}} \\
& -\mu^{-1}
\vcenter{\hbox{\begin{tikzpicture}[scale=0.65]
\begin{scope}[shift={(-1,0.8)}]
\draw [blue] (0, -1.5)--(0, 1.5);
\draw [fill=white] (-0.3, -0.3) rectangle (0.3, 0.3);
\node at (0, 0) {\tiny $v$};
\end{scope}
    \begin{scope}[shift={(0,1.5)}]
    \draw [blue] (-0.5, 0.8)--(-0.5, 0) .. controls +(0, -0.6) and +(0,-0.6).. (0.5, 0)--(0.5, 0.8);    
    \end{scope}
\draw [blue] (-0.5, -0.8)--(-0.5, 0) .. controls +(0, 0.6) and +(0,0.6).. (0.5, 0)--(0.5, -0.8);
\begin{scope}[shift={(0.5, -0.3)}]
\draw [fill=white] (-0.3, -0.3) rectangle (0.3, 0.3);
\node at (0, 0) {\tiny $v^*$};
\end{scope}
\end{tikzpicture}}}
-\mu
\vcenter{\hbox{\begin{tikzpicture}[scale=0.65]
\begin{scope}[shift={(-1,0.8)}]
\draw [blue] (0, -1.5)--(0, 1.5);
\draw [fill=white] (-0.3, -0.3) rectangle (0.3, 0.3);
\node at (0, 0) {\tiny $v^*$};
\end{scope}
    \begin{scope}[shift={(0,1.5)}]
    \draw [blue] (-0.5, 0.8)--(-0.5, 0) .. controls +(0, -0.6) and +(0,-0.6).. (0.5, 0)--(0.5, 0.8);    
    \end{scope}
\draw [blue] (-0.5, -0.8)--(-0.5, 0) .. controls +(0, 0.6) and +(0,0.6).. (0.5, 0)--(0.5, -0.8);
\begin{scope}[shift={(0.5, -0.3)}]
\draw [fill=white] (-0.3, -0.3) rectangle (0.3, 0.3);
\node at (0, 0) {\tiny $v$};
\end{scope}
\end{tikzpicture}}}
-\gamma
\vcenter{\hbox{\begin{tikzpicture}[scale=0.65]
\begin{scope}[shift={(-1,0.8)}]
\draw [blue] (0, -1.5)--(0, 1.5);
\draw [fill=white] (-0.3, -0.3) rectangle (0.3, 0.3);
\node at (0, 0) {\tiny $v$};
\end{scope}
    \begin{scope}[shift={(0,1.5)}]
    \draw [blue] (-0.5, 0.8)--(-0.5, 0) .. controls +(0, -0.6) and +(0,-0.6).. (0.5, 0)--(0.5, 0.8);    
    \end{scope}
\draw [blue] (-0.5, -0.8)--(-0.5, 0) .. controls +(0, 0.6) and +(0,0.6).. (0.5, 0)--(0.5, -0.8);
\begin{scope}[shift={(0.5, -0.3)}]
\draw [fill=white] (-0.3, -0.3) rectangle (0.3, 0.3);
\node at (0, 0) {\tiny $v$};
\end{scope}
\end{tikzpicture}}}
-\gamma
\vcenter{\hbox{\begin{tikzpicture}[scale=0.65]
\begin{scope}[shift={(-1,0.8)}]
\draw [blue] (0, -1.5)--(0, 1.5);
\draw [fill=white] (-0.3, -0.3) rectangle (0.3, 0.3);
\node at (0, 0) {\tiny $v^*$};
\end{scope}
    \begin{scope}[shift={(0,1.5)}]
    \draw [blue] (-0.5, 0.8)--(-0.5, 0) .. controls +(0, -0.6) and +(0,-0.6).. (0.5, 0)--(0.5, 0.8);    
    \end{scope}
\draw [blue] (-0.5, -0.8)--(-0.5, 0) .. controls +(0, 0.6) and +(0,0.6).. (0.5, 0)--(0.5, -0.8);
\begin{scope}[shift={(0.5, -0.3)}]
\draw [fill=white] (-0.3, -0.3) rectangle (0.3, 0.3);
\node at (0, 0) {\tiny $v^*$};
\end{scope}
\end{tikzpicture}}}.
\end{align*}
The corresponding matrix form is the following:
\begin{align*}
4\widehat{\Gamma} &=\lambda^{-1/2}
\begin{pmatrix}
\mu^{-1} & 0 & \gamma & 0 & 0 & -\sqrt{2}\gamma & 0 & 0 \\
0 & \mu^{-1} & 0 & \gamma & -\sqrt{2}\mu^{-1} & 0 & 0 & 0 \\
\gamma & 0 & \mu & 0 & 0 & -\sqrt{2}\mu & 0 & 0 \\
0 & \gamma & 0 & \mu & -\sqrt{2}\gamma & 0 & 0 & 0 \\
0 & -\sqrt{2}\mu^{-1} & 0 & -\sqrt{2}\gamma & 2\mu^{-1} & 0 & 0 & 0 \\
-\sqrt{2}\gamma & 0 & -\sqrt{2}\mu & 0 & 0 & 2\mu & 0 & 0 \\
0 & 0 & 0 & 0 & 0 & 0 & 0 & 0 \\
0 & 0 & 0 & 0 & 0 & 0 & 0 & 0
\end{pmatrix},
\end{align*}
where the rank-one projections of the system of matrix units are
\begin{align*}
& \frac{\lambda^{1/2}}{2} \leftLineHalfLoops{v^*v}{v}{v^*},
\frac{\lambda^{1/2}}{2} \leftLineHalfLoops{vv^*}{v}{v^*},
\frac{\lambda^{1/2}}{2} \leftLineHalfLoops{v^*v}{v^*}{v},
 \frac{\lambda^{1/2}}{2} \leftLineHalfLoops{vv^*}{v^*}{v}, \\
& \frac{\lambda^{1/2}}{2} \leftLineOnlyBox{v^*v},
  \frac{\lambda^{1/2}}{2}\leftLineOnlyBox{vv^*},
  \frac{\lambda^{1/2}}{2} \leftLineHalfLoops{v^*v}{iPQ}{iPQ},
   \frac{\lambda^{1/2}}{2} \leftLineHalfLoops{vv^*}{iPQ}{iPQ}.
\end{align*}
The matrix form of $\widehat{\Gamma}_2$ is the following:
\begin{equation*}
\resizebox{\textwidth}{!}{%
$4\widehat{\Gamma}_2 = \frac{1}{2\sqrt{\lambda}}
\begin{pmatrix}
2\gamma^2 + \mu^{-2} + 1 & 0 & 2\gamma(\mu + \mu^{-1}) & 0 & 0 & -2\sqrt{2}\gamma\mu & 0 & 0 \\
0 & \mu^{-2} + 3 & 0 & 2\gamma(\mu + \mu^{-1}) & -\sqrt{2}(1+\mu^{-2}) & 0 & 0 & 0 \\
2\gamma(\mu + \mu^{-1}) & 0 & \mu^2 + 3 & 0 & 0 & -\sqrt{2}(1+\mu^2) & 0 & 0 \\
0 & 2\gamma(\mu + \mu^{-1}) & 0 & 2\gamma^2 + \mu^2 + 1 & -2\sqrt{2}\gamma\mu^{-1} & 0 & 0 & 0 \\
0 & -\sqrt{2}(1+\mu^{-2}) & 0 & -2\sqrt{2}\gamma\mu^{-1} & 2\mu^{-2} + 3(1+\gamma^2) & 0 & -(1-\gamma^2) & 0 \\
-2\sqrt{2}\gamma\mu & 0 & -\sqrt{2}(1+\mu^2) & 0 & 0 & 2\mu^2 + 3(1+\gamma^2) & 0 & 1-\gamma^2 \\
0 & 0 & 0 & 0 & -(1-\gamma^2) & 0 & 1-\gamma^2 & 0 \\
0 & 0 & 0 & 0 & 0 & 1-\gamma^2 & 0 & 1-\gamma^2
\end{pmatrix}$%
}.
\end{equation*}
Let $\beta$ be the optimal value satisfies $\widehat{\Gamma}_2 \geq \beta \widehat{\Gamma}$.
By direct computation, we have that $\beta$ satisfies the following constraint:
\begin{align*}
  \det
    \begin{pmatrix}
    \mu^{-2} + 3 - 2\beta\mu^{-1} & 2\gamma(\mu + \mu^{-1}) - 2\beta\gamma & -\sqrt{2}(1+\mu^{-2} - 2\beta\mu^{-1}) \\
    2\gamma(\mu + \mu^{-1}) - 2\beta\gamma & 2\gamma^2 + \mu^2 + 1 - 2\beta\mu & -2\sqrt{2}\gamma(\mu^{-1} - \beta) \\
    -\sqrt{2}(1+\mu^{-2} - 2\beta\mu^{-1}) & -2\sqrt{2}\gamma(\mu^{-1} - \beta) & 2\mu^{-2} + 2 + 4\gamma^2 - 4\beta\mu^{-1}
    \end{pmatrix}=0,
\end{align*}
where $\det$ is the determinant of the matrix.
When $\gamma=0$, we have that $\displaystyle \beta=\frac{\mu+\mu^{-1}}{2}$.
When $\displaystyle \gamma^2 < \frac{1}{2}$, we have that $\beta>0$.
When $\displaystyle \gamma^2> \frac{1}{2}$, we have that $\beta<0$ for $\mu\to 0$ or $\mu\to \infty$.
Hence $\displaystyle \gamma^2=\frac{1}{2}$ is a critical point for the phase transition.
Moreover, for $\mu<1$ and $\displaystyle \gamma^2=\frac{1}{2}$, we have that 
\begin{align*}
\beta= \frac{1}{2} \left( \mu^{-1} + \frac{4}{3}\mu - \sqrt{\mu^{-2} - 2 + \frac{10}{9}\mu^2} \, \right).
 \end{align*}
 
 The following is the graph showing the region of the positivity of $\beta$ which is $\beta_{opt}$ in the picture.
 
\begin{tikzpicture}[
    declare function={
        A(\x) = 1 + \x - 4*(\x^2);
        B(\x) = 2*(1 + 2*\x - (\x^2));
        C(\x) = (1 + \x)*(1 - 2*\x);
        mumax(\x) = sqrt((B(\x) + sqrt(B(\x)^2 - 4*A(\x)*C(\x))) / (-2*C(\x)));
        mumin(\x) = 1 / mumax(\x);
    }
]
\begin{axis}[
    width=15cm,
    height=10cm,
    ymode=log,
    xmin=0, xmax=1.05,
    ymin=1e-3, ymax=1e3,
    xlabel={\textbf{Coherent Scattering / Impurity Strength} ($\gamma^2$)},
    ylabel={\tiny \textbf{Bias Limit} ($\mu$)},
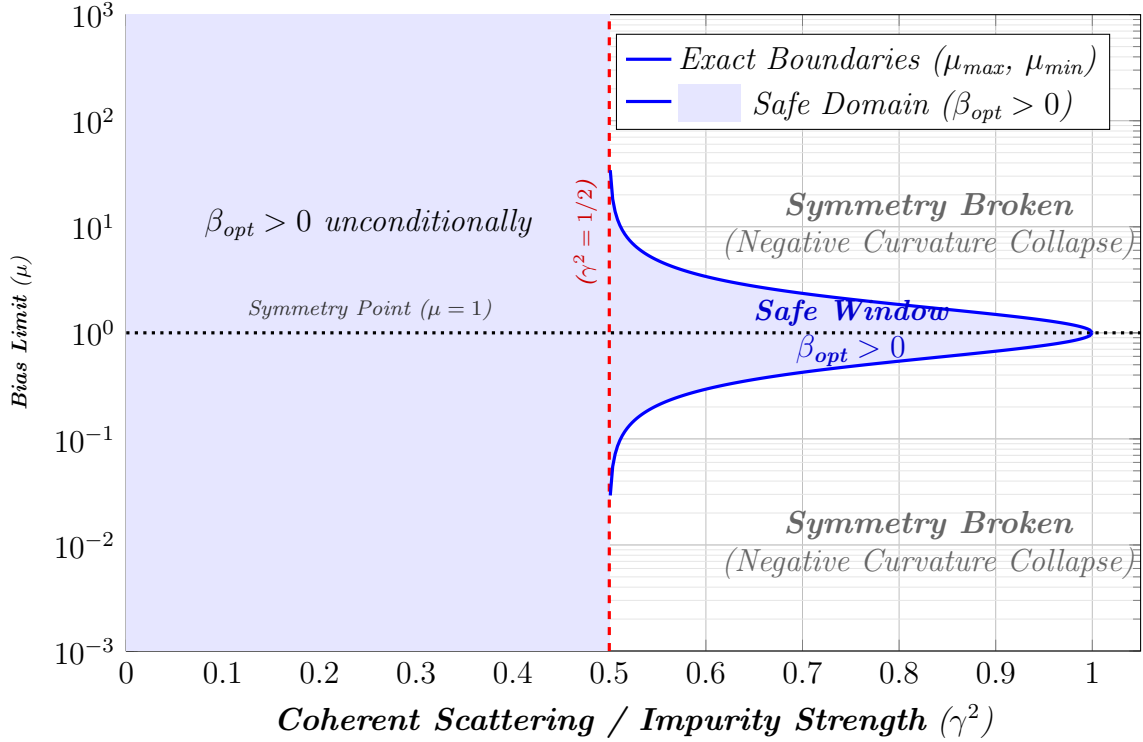
    title={\tiny \textbf{Diagram of the Optimal Bakry-\'Emery Curvature}},
    title style={font=\Large\bfseries, yshift=2ex},
    grid=both,
    major grid style={line width=.2pt,draw=gray!50},
    minor grid style={line width=.1pt,draw=gray!20},
    legend pos=north east,
    legend cell align={left},
    log basis y=10
]

\fill[blue!10] (axis cs:0, 1e-3) rectangle (axis cs:0.5, 1e3);

\draw[dashed, red, very thick] (axis cs:0.5, 1e-3) -- (axis cs:0.5, 1e3) 
    node[pos=0.65, sloped, above, text=red!80!black] {\tiny \textbf{  ($\gamma^2 = 1/2$)}};

\draw[dotted, black, very thick] (axis cs:0, 1) -- (axis cs:1.05, 1)
    node[pos=0.24, above, text=black!80] {\tiny \textit{Symmetry Point ($\mu=1$)}};

\addplot[name path=upper, domain=0.501:1, samples=200, blue, very thick] {mumax(x)};
\addplot[name path=lower, domain=0.501:1, samples=200, blue, very thick] {mumin(x)};

\addplot[blue!10] fill between[of=upper and lower];


\node[anchor=center, align=center] at (axis cs:0.25, 20) 
    {\textbf{}\\ \vspace{1mm}\\ $\beta_{\text{opt}} > 0$ \textit{unconditionally}};
    
\node[anchor=center, align=center, text=blue!80!black] at (axis cs:0.75, 1) 
    {\small \textbf{Safe Window}\\ \textbf{$\beta_{\text{opt}} > 0$}};
    
\node[anchor=center, align=center, text=gray!80!black] at (axis cs:0.83, 10) 
    {\textbf{Symmetry Broken}\\ (Negative Curvature Collapse)};
    
\node[anchor=center, align=center, text=gray!80!black] at (axis cs:0.83, 0.01) 
    {\textbf{Symmetry Broken}\\ (Negative Curvature Collapse)};


\addplot[blue, very thick] coordinates {(-1,-1)}; 
\addlegendentry{Exact Boundaries ($\mu_{\text{max}}$, $\mu_{\text{min}}$)}
\addlegendimage{empty legend}
\addlegendentry{\colorbox{blue!10}{\phantom{XX}} Safe Domain ($\beta_{\text{opt}} > 0$)}
\end{axis}
\end{tikzpicture}
\end{example}

\section{Intertwining Property for Bimodule GNS Symmetry}

Suppose that $\{\Phi_t\}_{t\geq 0}$ is a bimodule quantum Markov semigroup.
We say $\{\Phi_t\}_{t\geq 0}$ is bimodule GNS symmetric if  $\Phi_t$ is equilibrium for all $t\geq 0$ and there exists a strictly positive operator $\widehat{\Delta}\in \cM'\cap \cM_2$ such that $\widehat{\Delta}e_2 =e_2$, and $\overline{\widehat{\cL}}=\overline{\widehat{\Delta}}\widehat{\cL}$.
In this case, we have that $\cL_w=0$.
Recall that $\displaystyle \widehat{\cL}_0=\sum_{j=1}^m \omega_j p_j$ is the decomposition of $\widehat{\cL}_0$ in the subalgebra of $\cM'\cap \cM_2$ generated by $\cR(\widehat{\cL}_0)\widehat{\Delta}$ and $\widehat{\cL}_0$.
The bimodule modular operator $\displaystyle \widehat{\Delta}=\sum_{j=1}^m \mu_j p_j+1-\cR(\widehat{\cL}_0)$ and there is an involution $*$ on $\{1, \ldots, m\}$ such that $\mu_{j^*}=\mu_j^{-1}$ and $p_{j^*}=\overline{p_j}$.
We have that $\displaystyle \overline{\widehat{\cL}_0}=\overline{\widehat{\Delta}} \widehat{\cL}=\sum_{j=1}^m \mu_j^{-1} \omega_j p_j$.

Recall that $\partial x=\left[x, \fF^{-1}(\widehat{\cL}_0^{1/2})\right]$, $\overline{\partial} x=\left[x, \fF^{-1}(\overline{\widehat{\cL}_0}^{1/2})\right]$ and $\partial_j x=\omega_j^{1/2}\left[x, \fF^{-1}(p_j)\right]$ for any $x\in \cM$.

\begin{lemma}
Suppose that $\{\Phi_t\}_{t \geq 0}$ is bimodule GNS symmetric with respect to $\widehat{\Delta}\in \cM'\cap \cM_2$.
Then for any $x\in \cM$,
\begin{align*}
\cL(x)+\cL^*(x)+\frac{1}{2}\left\{x, 1*(\overline{\widehat{\cL}_0}-\widehat{\cL}_0)\right\} =\frac{\lambda^{-1/2}}{2}\sum_{j=1}^m (1+\mu_j^{-1})  \partial_j^*\partial_j x.
\end{align*}
If $1*\widehat{\cL}_0=1*\overline{\widehat{\cL}_0}$, then 
\begin{align*}
\widehat{\cL}=\frac{\lambda^{-1/2}}{2}\left(1+\overline{\widehat{\Delta}}\right)^{-1}\sum_{j=1}^m (1+\mu_j^{-1}) \widehat{ \partial_j^*\partial_j},
\end{align*} 
where $\widehat{ \partial_j^*\partial_j}$ is the Fourier multiplier of $\partial_j^*\partial_j: \cM \to \cM$.
\end{lemma}
\begin{proof}
Recall that 
\begin{align*}
\cL_a+\cL_{\overline{a}}=\frac{\lambda^{-1/2}}{2}(\partial^*\partial+\overline{\partial}^*\overline{\partial}),
\end{align*}
where
\begin{align*}
 \cL_{\overline{a}}(x)=\frac{1}{2}(1*\overline{\widehat{\cL}_0})x +\frac{1}{2}x (1*\overline{\widehat{\cL}_0})-x*\overline{\widehat{\cL}_0}=\cL^*(x)+ \frac{1}{2} \left\{x, 1*(\overline{\widehat{\cL}_0}-\widehat{\cL}_0)\right\} ,
 \end{align*}
 for all $x\in \cM$.
Note that $\displaystyle \overline{\partial}=\sum_{j=1}^m\mu_j^{-1/2}\partial_j $.
We have that 
\begin{align*}
\partial^*\partial+\overline{\partial}^*\overline{\partial}=\sum_{j=1}^m\partial_j^* \partial_j +\overline{\partial}_j^*\overline{\partial}_j=\sum_{j=1}^m (1+\mu_j^{-1})  \partial_j^* \partial_j.
\end{align*}
Combining the two equalities above, we see that the first equality holds.

Suppose that  $1*\widehat{\cL}_0=1*\overline{\widehat{\cL}_0}$.
Then $\widehat{\cL}_{\overline{a}}=\overline{\widehat{\cL}}=\overline{\widehat{\Delta}}\widehat{\cL}$.
We have that 
\begin{align*}
\widehat{\cL}+\overline{\widehat{\Delta}}\widehat{\cL}=\frac{\lambda^{-1/2}}{2}\sum_{j=1}^m (1+\mu_j^{-1}) \widehat{ \partial_j^*\partial_j}.
\end{align*}
This completes the computation.
\end{proof}

Recall that the gradient $\nabla: \cM \to \cM_1^{\oplus( m)}$ is defined as $\displaystyle \nabla x=\lambda^{1/4} (\partial_j x )_{j=1}^m$ for all $x\in \cM$, where $\cM_1^{\oplus( m)}$ is the direct sum of $m$-copies of $\cM_1$.
The divergence $\Div: \cM_1^{\oplus (m)}\to \cM$ is defined as
\begin{align*}
\Div (x_j)_{j=1}^m =\lambda^{1/4} \sum_{j=1}^m \partial_j^* x_j,
\end{align*}
where $x_1, \ldots, x_m\in \cM_1$.

\begin{definition}
Suppose that $\{\Phi_t\}_{t\geq 0}$ is a bimodule GNS symmetric quantum Markov semigroup.
We say $\{\Phi_t\}_{t\geq 0}$ satisfies the intertwining property if there exists a unital completely positive map $\widetilde{\Phi_t}: \cM_1 \to \cM_1$ such that $\widetilde{\Phi_t}^*|_{\cM}=\Phi_t^*$ and 
\begin{align}\label{eq:intertwining}
    \Phi_t^*\Div= \Div \overset{\longrightarrow}{\widetilde{\Phi_t}^*}e^{-\vec{\beta} t} , \quad t\geq 0,
\end{align}
where $ \overset{\longrightarrow}{\widetilde{\Phi_t}^*}=(\widetilde{\Phi_t}^*, \ldots, \widetilde{\Phi_t}^*)$ is the $m$-copies of $\widetilde{\Phi_t}^*$ and $\vec{\beta}=(\beta_1, \ldots, \beta_m)$.
Let $\displaystyle \beta=\min_{1\leq j \leq m} \beta_j$.
We shall denote $\overset{\longrightarrow}{\widetilde{\Phi_t}^*}$ by $\widetilde{\Phi_t}^*$ for simplicity if there is no confusion.
\end{definition}

In the following, we assume that $\left\{\widetilde{\Phi_t}\right\}_{t\geq 0}$ is a bimodule quantum Markov semigroup.
Let $\cJ$ be its generator, i.e. the Lindbladian.
By the fact that $\widetilde{\Phi_t}^*|_{\cM}=\Phi_t^*$ and taking differentiation at $t=0$, we have that $\cJ^*|_{\cM}=\cL^*$.
By Proposition \ref{prop:extension1}, we have that 
\begin{align*}
    \vcenter{\hbox{
    \begin{tikzpicture}
       \draw [blue] (-0.45, -0.3).. controls+(0, -0.3) and +(0, -0.3).. (-0.75, -0.3)--(-0.75, 0.3) .. controls +(0, 0.3) and +(0, 0.3)..(-0.45, 0.3) (-0.15, -0.6)--(-0.15, 0.6) (0.15, -0.6)--(0.15, 0.6) (0.45, -0.6)--(0.45, 0.6);
       \draw[fill=white] (-0.6, -0.3) rectangle (0.6, 0.3);
       \node at (0, 0) {\tiny $\overline{\widehat{\cJ}}$};
    \end{tikzpicture}
    }}
    =    \vcenter{\hbox{
    \begin{tikzpicture}
       \draw [blue]   (-0.15, -0.6)--(-0.15, 0.6) (0.15, -0.6)--(0.15, 0.6) (0.45, -0.6)--(0.45, 0.6);
       \draw[fill=white] (-0.35, -0.3) rectangle (0.35, 0.3);
       \node at (0, 0) {\tiny $\overline{\widehat{\cL}}$};
    \end{tikzpicture}
    }},
\end{align*}
and equivalently,
\begin{align*}
\vcenter{\hbox{
    \begin{tikzpicture}
       \draw [blue] (0.45, -0.3).. controls+(0, -0.3) and +(0, -0.3).. (0.75, -0.3)--(0.75, 0.3) .. controls +(0, 0.3) and +(0, 0.3)..(0.45, 0.3) (-0.15, -0.6)--(-0.15, 0.6) (0.15, -0.6)--(0.15, 0.6) (-0.45, -0.6)--(-0.45, 0.6);
       \draw[fill=white] (-0.6, -0.3) rectangle (0.6, 0.3);
       \node at (0, 0) {\tiny $\widehat{\cJ}$};
    \end{tikzpicture}
    }}
    =  \vcenter{\hbox{
    \begin{tikzpicture}
       \draw [blue]   (-0.15, -0.6)--(-0.15, 0.6) (0.15, -0.6)--(0.15, 0.6) (-0.45, -0.6)--(-0.45, 0.6);
       \draw[fill=white] (-0.35, -0.3) rectangle (0.35, 0.3);
       \node at (0, 0) {\tiny $ \widehat{\cL}$};
    \end{tikzpicture}
    }}.
\end{align*}
This indicates that 
\begin{align*}
\cL \bE_{\cM} =\bE_{\cM} \cJ, \quad \Phi_t \bE_{\cM} =\bE_{\cM} \widetilde{\Phi_t}.
\end{align*}

\begin{remark}
Suppose $\{\Phi_t\}_{ t \geq 0}$ is a GNS symmetric quantum Markov semigroup  with the intertwining property and the extension $\{\widetilde{\Phi_t}\}_{t\geq 0}$ is a quantum Markov semigroup.
We can not conclude that the extension is bimodule GNS symmetric.    
\end{remark}

Let $\widehat{\cJ}_0=-(1-\widetilde{e}_2)\widehat{\cJ}(1-\widetilde{e}_2)$ and recall that 
$\widehat{\cL}=   \vcenter{\hbox{\begin{tikzpicture}[scale=0.65]
    \begin{scope}[shift={(0,1.5)}]
    \draw [blue] (-0.5, 0.8)--(-0.5, 0) .. controls +(0, -0.6) and +(0,-0.6).. (0.5, 0)--(0.5, 0.8);    
\begin{scope}[shift={(0.5, 0.3)}]
\draw [fill=white] (-0.3, -0.3) rectangle (0.3, 0.3);
\node at (0, 0) {\tiny $\mathbf{y}$};
\end{scope}
    \end{scope}
\draw [blue] (-0.5, -0.8)--(-0.5, 0) .. controls +(0, 0.6) and +(0,0.6).. (0.5, 0)--(0.5, -0.8);
\end{tikzpicture}}}
+ \vcenter{\hbox{\begin{tikzpicture}[scale=0.65]
    \begin{scope}[shift={(0,1.5)}]
    \draw [blue] (-0.5, 0.8)--(-0.5, 0) .. controls +(0, -0.6) and +(0,-0.6).. (0.5, 0)--(0.5, 0.8);    
    \end{scope}
\draw [blue] (-0.5, -0.8)--(-0.5, 0) .. controls +(0, 0.6) and +(0,0.6).. (0.5, 0)--(0.5, -0.8);
\begin{scope}[shift={(0.5, -0.3)}]
\draw [fill=white] (-0.3, -0.3) rectangle (0.3, 0.3);
\node at (0, 0) {\tiny $\mathbf{y}$};
\end{scope}
\end{tikzpicture}}} - \vcenter{\hbox{
    \begin{tikzpicture}
       \draw [blue] (-0.15, -0.6)--(-0.15, 0.6) (0.15, -0.6)--(0.15, 0.6);
       \draw[fill=white] (-0.35, -0.3) rectangle (0.35, 0.3);
       \node at (0, 0) {\tiny $ \widehat{\cL}_0$};
    \end{tikzpicture}
    }}$, where $\displaystyle \mathbf{y}=\frac{1}{2} (1*\widehat{\cL}_0)$.

\begin{proposition}\label{prop:semigroupext}
Suppose that $0\leq x\in \cM_1'\cap \cM_5$ with $\vcenter{\hbox{
    \begin{tikzpicture}
       \draw [blue] (0.45, -0.3).. controls+(0, -0.3) and +(0, -0.3).. (0.75, -0.3)--(0.75, 0.3) .. controls +(0, 0.3) and +(0, 0.3)..(0.45, 0.3) (-0.15, -0.6)--(-0.15, 0.6) (0.15, -0.6)--(0.15, 0.6) (-0.45, -0.6)--(-0.45, 0.6);
       \draw[fill=white] (-0.6, -0.3) rectangle (0.6, 0.3);
       \node at (0, 0) {\tiny $x$};
    \end{tikzpicture}
    }}
    =  \vcenter{\hbox{
    \begin{tikzpicture}
       \draw [blue]   (-0.15, -0.6)--(-0.15, 0.6) (0.15, -0.6)--(0.15, 0.6) (-0.45, -0.6)--(-0.45, 0.6);
       \draw[fill=white] (-0.35, -0.3) rectangle (0.35, 0.3);
       \node at (0, 0) {\tiny $ \widehat{\cL}_0$};
    \end{tikzpicture}
    }}$ and     $    \vcenter{\hbox{
    \begin{tikzpicture}
       \draw [blue] (-0.45, -0.3).. controls+(0, -0.3) and +(0, -0.3).. (-0.75, -0.3)--(-0.75, 0.3) .. controls +(0, 0.3) and +(0, 0.3)..(-0.45, 0.3) (0.15, -0.6)--(0.15, 0.6) (0.45, -0.6)--(0.45, 0.6);
       \draw[blue] (-0.15, -0.3).. controls+(0, -0.5) and +(0, -0.5).. (-1.05, -0.3)--(-1.05, 0.3) .. controls +(0, 0.5) and +(0, 0.5)..(-0.15, 0.3);
       \draw[fill=white] (-0.6, -0.3) rectangle (0.6, 0.3);
       \node at (0, 0) {\tiny $x$};
    \end{tikzpicture}
    }}
    =    \vcenter{\hbox{
    \begin{tikzpicture}
       \draw [blue]   (-0.15, -0.3).. controls +(0, -0.3) and +(0, -0.3).. (-0.5, -0.3)--(-0.5, 0.3).. controls +(0, 0.3) and +(0, 0.3).. (-0.15, 0.3) (0.15, -0.6)--(0.15, 0.6) (0.45, -0.6)--(0.45, 0.6);
       \draw[fill=white] (-0.35, -0.3) rectangle (0.35, 0.3);
       \node at (0, 0) {\tiny $\widehat{\cL}_0$};
    \end{tikzpicture}
    }}=2\mathbf{y}$.
Then 
\begin{align*}
\widehat{\cJ}=
\vcenter{\hbox{\begin{tikzpicture}[scale=0.65]
    \begin{scope}[shift={(0,1.5)}]
    \draw [blue] (-0.5, 0.8)--(-0.5, 0) .. controls +(0, -0.6) and +(0,-0.6).. (0.5, 0)--(0.5, 0.8);   \draw [blue] (-1, 0.8)--(-1, 0) .. controls +(0, -0.8) and +(0,-0.8).. (1, 0)--(1, 0.8);  
\begin{scope}[shift={(0.8, 0.3)}]
\draw [fill=white] (-0.5, -0.3) rectangle (-0.1, 0.3);
\node at (-0.3, 0) {\tiny $\mathbf{y}$};
\end{scope}
    \end{scope}
\draw [blue] (-0.5, -0.8)--(-0.5, 0) .. controls +(0, 0.6) and +(0,0.6).. (0.5, 0)--(0.5, -0.8);
\draw [blue] (-1, -0.8)--(-1, 0) .. controls +(0, 0.8) and +(0,0.8).. (1, 0)--(1, -0.8); 
\end{tikzpicture}}}
+\vcenter{\hbox{\begin{tikzpicture}[scale=0.65]
    \begin{scope}[shift={(0,1.5)}]
    \draw [blue] (-0.5, 0.8)--(-0.5, 0) .. controls +(0, -0.6) and +(0,-0.6).. (0.5, 0)--(0.5, 0.8);   \draw [blue] (-1, 0.8)--(-1, 0) .. controls +(0, -0.8) and +(0,-0.8).. (1, 0)--(1, 0.8);  
    \end{scope}
\draw [blue] (-0.5, -0.8)--(-0.5, 0) .. controls +(0, 0.6) and +(0,0.6).. (0.5, 0)--(0.5, -0.8);
\draw [blue] (-1, -0.8)--(-1, 0) .. controls +(0, 0.8) and +(0,0.8).. (1, 0)--(1, -0.8); 
\begin{scope}[shift={(0.8, -0.3)}]
\draw [fill=white] (-0.5, -0.3) rectangle (-0.1, 0.3);
\node at (-0.3, 0) {\tiny $\mathbf{y}$};
\end{scope}
\end{tikzpicture}}} -\vcenter{\hbox{
    \begin{tikzpicture}
       \draw [blue] (0.45, -0.6)--(0.45, 0.6) (-0.15, -0.6)--(-0.15, 0.6) (0.15, -0.6)--(0.15, 0.6) (-0.45, -0.6)--(-0.45, 0.6);
       \draw[fill=white] (-0.6, -0.3) rectangle (0.6, 0.3);
       \node at (0, 0) {\tiny $x$};
    \end{tikzpicture}
    }}
\end{align*}
is an extension of $\widehat{\cL}$.
\end{proposition}
\begin{proof}
By taking the conditional expectation, we see that $\vcenter{\hbox{
    \begin{tikzpicture}
       \draw [blue] (0.45, -0.3).. controls+(0, -0.3) and +(0, -0.3).. (0.75, -0.3)--(0.75, 0.3) .. controls +(0, 0.3) and +(0, 0.3)..(0.45, 0.3) (-0.15, -0.6)--(-0.15, 0.6) (0.15, -0.6)--(0.15, 0.6) (-0.45, -0.6)--(-0.45, 0.6);
       \draw[fill=white] (-0.6, -0.3) rectangle (0.6, 0.3);
       \node at (0, 0) {\tiny $\widehat{\cJ}$};
    \end{tikzpicture}
    }}
    =  \vcenter{\hbox{
    \begin{tikzpicture}
       \draw [blue]   (-0.15, -0.6)--(-0.15, 0.6) (0.15, -0.6)--(0.15, 0.6) (-0.45, -0.6)--(-0.45, 0.6);
       \draw[fill=white] (-0.35, -0.3) rectangle (0.35, 0.3);
       \node at (0, 0) {\tiny $ \widehat{\cL}$};
    \end{tikzpicture}
    }}.$
    By Theorem 5.15 in \cite{WuZha25}, we see that $\widehat{\cJ}$ induces a bimodule quantum Markov semigroup $\{\widetilde{\Phi_t}\}_{t\geq 0}$ on $\cM_1$.
    By Proposition \ref{prop:extension1}, we see that $\widetilde{\Phi_t}$ extends $\Phi_t$ for all $t\geq 0$.
\end{proof}

The derivation $\widetilde{\partial}: \cM_1 \to \cM_3$ is 
\begin{align*}
\widetilde{\partial} x =\left[x, \fF^{-1}(\widehat{\cJ}_0^{1/2}) \right], \quad x\in \cM_1,  
\end{align*}
where the Fourier transform $\mathfrak{F}$ is on the 2-box spaces for the inclusion $\cN \subset \cM_1$ and $\mathfrak{F}^{-1}(\widehat{\cJ}_0^{1/2})\in \cN'\cap \cM_3$.
We see that 
\begin{align*}
 ( \widetilde{\partial} x ) \widetilde{e}_2  =\lambda^{-1} \widehat{\cJ}_0^{1/2} [x, \widetilde{e}_1 ]\widetilde{e}_2. 
\end{align*}
The associated gradient form $\widetilde{\Gamma}$ is 
\begin{align*}
    \widetilde{\Gamma}(x, y)=&\frac{1}{2} (x^*\cJ(y)+\cJ(x)^*y-\cJ(x^*y)) \\
    =& \frac{\lambda^{-1}}{2}\bE_{\cM_1}\left(( \widetilde{\partial} y)^*( \widetilde{\partial} x)\right), \quad x, y\in \cM_1.
\end{align*}

\begin{lemma}
Suppose that $\{\Phi_t\}_{t\geq 0}$ is a bimodule GNS symmetric quantum Markov semigroup and $\{\widetilde{\Phi_t}\}_{t\geq 0}$ is a bimodule quantum Markov semigroup such that $\widetilde{\Phi_t}^*$ extends $\Phi_t^*$.
Then the intertwining property \eqref{eq:intertwining} is equivalent to 
\begin{align}\label{eq:intertwining0}
    \partial_j \cL - \cJ \partial_j =\beta_j \partial_j
\end{align}
for all $j=1, \ldots, m$.
\end{lemma}
\begin{proof}
Suppose that $x\in \cM$ and $X(t)=\partial_j \Phi_t(x)$.
Then $X(0)=\partial_j x$ and 
\begin{align}\label{eq:intertwining2}
    \frac{d}{dt} X(t) =-\partial_j \cL \Phi_t(x)=-\cJ \partial_j \Phi_t(x) -\beta_j \partial_j \Phi_t(x)=(-\cJ-\beta_j )X(t).
\end{align}
Hence $e^{\beta_j t} X(t)$ is the unique solution of $\displaystyle \frac{d}{dt}Y(t)=-\cJ Y(t)$ with $Y(0)=\partial_j x$ followed from the boundedness of $\cJ$.
Note that $\widetilde{\Phi_t}\partial_j x$ is the solution of the equation.
We have that $\widetilde{\Phi_t}\partial_j x=e^{\beta_j t}\partial_j \Phi_t(x)$.
This implies the intertwining property.

Suppose that the intertwining property holds.
Then by differentiating Equation \eqref{eq:intertwining} at $t=0$, we see that Equation \eqref{eq:intertwining2} holds.
\end{proof}

\begin{proposition}
Suppose that $\{\Phi_t\}_{t\geq 0}$ is a bimodule GNS symmetric quantum Markov semigroup and $\{\widetilde{\Phi_t}\}_{t\geq 0}$ is a bimodule quantum Markov semigroup such that $\widetilde{\Phi_t}^*$ extends $\Phi_t^*$.
Then the intertwining property \eqref{eq:intertwining} is equivalent to 
\begin{align}\label{eq:intertwining3}
 \vcenter{\hbox{

    }} 
\end{align}
for all $k=1, \ldots, m$.
\end{proposition}
\begin{proof}
The pictorial representation of the intertwining property (Equation \eqref{eq:intertwining0}) is 
\begin{align*}
    \vcenter{\hbox{\scalebox{0.8}{
        %
}}}.
\end{align*}
By the fact that
\begin{align*}
           \widehat{\partial}_k=\omega_k^{1/2}\vcenter{\hbox{
    %
    }} ,
\end{align*}
we see that the proposition is true.
\end{proof}

\begin{remark}
The intertwining property implies that 
\begin{align*}
       -\vcenter{\hbox{
    %
    }},
\end{align*}
by adding caps to the left at the bottom and turn the right top most string down to Equation \eqref{eq:intertwining3}. 
\end{remark}

\begin{corollary}
Suppose that $\{\Phi_t\}_{t\geq 0}$ is a bimodule GNS symmetric quantum Markov semigroup and $\{\widetilde{\Phi_t}\}_{t\geq 0}$ is a bimodule quantum Markov semigroup such that $\widetilde{\Phi_t}^*$ extends $\Phi_t^*$.
If $1*\widehat{\cJ}_0=1*\widehat{\cL}_0=1$, then the intertwining property \eqref{eq:intertwining} is equivalent to 
\begin{align}\label{eq:intertwining4}
 -\vcenter{\hbox{
%
    }}$ and applying Proposition \ref{prop:semigroupext}, we see that the corollary is true.
\end{proof}

\begin{remark}
The intertwining property implies that 
\begin{align*}
       \vcenter{\hbox{
    %
    }},
\end{align*}
by adding caps to the left at the bottom and turn the right top most string down in Equation \eqref{eq:intertwining4}. 
\end{remark}

\begin{remark}
Suppose that $\vcenter{\hbox{
    %
    }}$.
We have that 
\begin{align*}
\vcenter{\hbox{
    %
    }}$.
We have that 
\begin{align*}
 -\vcenter{\hbox{
%
}}}. 
\end{align*}
We see that the associated semigroup satisfies the intertwining property with $\displaystyle \beta=\frac{\lambda^{1/2}}{\lambda^{-1/2}-\lambda^{1/2}} = \frac{\lambda}{1-\lambda} $, i.e. $ -\vcenter{\hbox{
%
    }}  \right)$.
\end{remark}

\begin{remark}\label{rem:dephaseinter2}
Suppose that $\widehat{\cL}_0$ is defined as previous remark and the inclusion $\cN \subset \cM$ admitting a braiding as follows:
\begin{align*}
 \vcenter{\hbox{\scalebox{0.8}{
        %
    }}\right).
\end{align*}
We have that 
\begin{align*}
2(\lambda^{-1/2}-\lambda^{1/2})\vcenter{\hbox{
    %
    }}.
\end{align*}
By expanding the braiding, we obtain that 
\begin{align*}
& 2(\lambda^{-1/2}-\lambda^{1/2}) \left(\vcenter{\hbox{
    %
}}}\right).
\end{align*}
We see that the associated semigroup satisfies the intertwining property with $\displaystyle \beta = \frac{4-\lambda^{-1}}{2(1-\lambda)} $.

If $\displaystyle \lambda=\frac{1}{2}$, we see that $\beta=2$ which coincides with the optimal coefficient in Proposition \ref{prop:dephaseem}.
This implies the coefficient here is optimal by Proposition \ref{prop:beinter}.

Note that when $m \leq 8$, we have that $ \displaystyle \frac{4-\lambda^{-1}}{2(1-\lambda)} \geq \frac{\lambda}{1-\lambda}$, i.e. the coefficient arising from the braiding is better.
\end{remark}

In the following, we denote by $\widehat{\widetilde{\Gamma}}_{\partial}$ the Fourier multiplier of $\lambda^{-1/2}\bE_{\cM}((\widetilde{\partial}\partial (x))^*(\widetilde{\partial}\partial x) )$, which is depicted as
\begin{align*}
     \widehat{\widetilde{\Gamma}}_{\partial}= \vcenter{\hbox{\scalebox{0.8}{
        \begin{tikzpicture}[scale=1.2]
           \draw [blue] (-0.6, 0.6)--(-0.6, -0.6) (-0.4, 0.6)--(-0.4, -0.6) (0, 0.4)--(0, -0.4) (0.2, 0.4)--(0.2, -0.4) (0.4, 1.5)--(0.4, -1.5);
           \draw [blue] (0.6, -0.3).. controls +(0, -0.3) and +(0, -0.3) .. (0.9, -0.3)--(0.9, 0.3).. controls +(0, 0.3) and +(0, 0.3) ..(0.6, 0.3);
           \draw [fill=white] (-0.7, -0.3) rectangle (0.7, 0.3);
           \node at (0, 0) {\tiny $\widehat{\widetilde{\Gamma}}$};
           \begin{scope}[shift={(-0.6, 0.8)}]
        \draw[blue] (-0.2, 0.7)--(-0.2, -0.3) .. controls +(0, -0.3) and +(0, -0.3) .. (-0.6, -0.3)--(-0.6, 0.7);
           \draw[blue] (0.2, 0.3) .. controls +(0, 0.3) and +(0, 0.3) .. (0.6, 0.3)--(0.6, -0.5);
           \draw[blue] (0, 0.3) .. controls +(0, 0.5) and +(0, 0.5) .. (0.8, 0.3)--(0.8, -0.5);
            \draw [fill=white] (-0.4, -0.3) rectangle (0.4, 0.3);
           \node at (0, 0) {\tiny $\widehat{\partial}$};    
           \end{scope}
        \begin{scope}[shift={(-0.6, -0.8)}]
        \draw[blue] (-0.2, -0.7)--(-0.2, 0.3) .. controls +(0, 0.3) and +(0, 0.3) .. (-0.6, 0.3)--(-0.6, -0.7);
           \draw[blue] (0.2, -0.3) .. controls +(0, -0.3) and +(0, -0.3) .. (0.6, -0.3)--(0.6, 0.5);
           \draw[blue] (0, -0.3) .. controls +(0, -0.5) and +(0, -0.5) .. (0.8, -0.3)--(0.8, 0.5);
            \draw [fill=white] (-0.4, -0.3) rectangle (0.4, 0.3);
           \node at (0, 0) {\tiny $\widehat{\partial}^*$};    
           \end{scope}
        \end{tikzpicture}
        }}}
    =   
    -\vcenter{\hbox{\scalebox{0.8}{
        \begin{tikzpicture}[scale=1.2]
           \draw [blue] (-0.6, 0.4)--(-0.6, -1.4) (-0.3, 0.8)--(-0.3, -1.4) (-0.1, 1.4)--(-0.1, -0.8) (0.2, 1.4)--(0.2, -0.4) (0.4, 1.4)--(0.4, -1.4);
           \draw [blue] (0.6, -0.3).. controls +(0, -0.3) and +(0, -0.3) .. (0.9, -0.3)--(0.9, 0.3).. controls +(0, 0.3) and +(0, 0.3) ..(0.6, 0.3);
           \draw [fill=white] (-0.7, -0.3) rectangle (0.7, 0.3);
           \node at (0, 0) {\tiny $\widehat{\widetilde{\Gamma}}$};
           \begin{scope}[shift={(-0.2, 0.8)}]
           \draw[blue] (-0.1, 0.2) .. controls +(0, 0.2) and +(0, 0.2) .. (-0.4, 0.2)--(-0.4, -0.5);
            \draw [fill=white] (-0.25, -0.25) rectangle (0.25, 0.25);
           \node at (0, 0) {\tiny $\widehat{\cL}_0^{1/2}$};    
           \end{scope}
            \begin{scope}[shift={(-0.2, -0.8)}]
           \draw[blue] (0.1, -0.2) .. controls +(0, -0.2) and +(0, -0.2) .. (0.4,- 0.2)--(0.4, 0.5);
             \draw [fill=white] (-0.25, -0.25) rectangle (0.25, 0.25);
           \node at (0, 0) {\tiny $\widehat{\cL}_0^{1/2}$};    
           \end{scope}
        \end{tikzpicture}
        }}}
        +\vcenter{\hbox{\scalebox{0.8}{
        \begin{tikzpicture}[scale=1.2]
           \draw [blue] (-0.6, 1.4)--(-0.6, -1.4) (-0.3, 1.4)--(-0.3, -1.4) (-0.1, 0.8)--(-0.1, -0.8) (0.2, 0.4)--(0.2, -0.4) (0.4, 1.4)--(0.4, -1.4);
           \draw [blue] (0.6, -0.3).. controls +(0, -0.3) and +(0, -0.3) .. (0.9, -0.3)--(0.9, 0.3).. controls +(0, 0.3) and +(0, 0.3) ..(0.6, 0.3);
           \draw [fill=white] (-0.7, -0.3) rectangle (0.7, 0.3);
           \node at (0, 0) {\tiny $\widehat{\widetilde{\Gamma}}$};
           \begin{scope}[shift={(-0.2, 0.8)}]
           \draw[blue] (0.1, 0.2) .. controls +(0, 0.2) and +(0, 0.2) .. (0.4, 0.2)--(0.4, -0.5);
              \draw [fill=white] (-0.25, -0.25) rectangle (0.25, 0.25);
           \node at (0, 0) {\tiny $\widehat{\cL}_0^{1/2}$};    
           \end{scope}
            \begin{scope}[shift={(-0.2, -0.8)}]
           \draw[blue] (0.1, -0.2) .. controls +(0, -0.2) and +(0, -0.2) .. (0.4,- 0.2)--(0.4, 0.5);
               \draw [fill=white] (-0.25, -0.25) rectangle (0.25, 0.25);
           \node at (0, 0) {\tiny $\widehat{\cL}_0^{1/2}$};    
           \end{scope}
        \end{tikzpicture}
        }}}
        +  \vcenter{\hbox{\scalebox{0.8}{
        \begin{tikzpicture}[scale=1.2]
           \draw [blue] (-0.6, 0.4)--(-0.6, -0.4) (-0.3, 0.8)--(-0.3, -0.8) (-0.1, 1.4)--(-0.1, -1.4) (0.2, 1.4)--(0.2, -1.4) (0.4, 1.4)--(0.4, -1.4);
           \draw [blue] (0.6, -0.3).. controls +(0, -0.3) and +(0, -0.3) .. (0.9, -0.3)--(0.9, 0.3).. controls +(0, 0.3) and +(0, 0.3) ..(0.6, 0.3);
           \draw [fill=white] (-0.7, -0.3) rectangle (0.7, 0.3);
           \node at (0, 0) {\tiny $\widehat{\widetilde{\Gamma}}$};
           \begin{scope}[shift={(-0.2, 0.8)}]
           \draw[blue] (-0.1, 0.2) .. controls +(0, 0.2) and +(0, 0.2) .. (-0.4, 0.2)--(-0.4, -0.5);
             \draw [fill=white] (-0.25, -0.25) rectangle (0.25, 0.25);
           \node at (0, 0) {\tiny $\widehat{\cL}_0^{1/2}$};    
           \end{scope}
            \begin{scope}[shift={(-0.2, -0.8)}]
           \draw[blue] (-0.1, -0.2) .. controls +(0, -0.2) and +(0, -0.2) .. (-0.4,- 0.2)--(-0.4, 0.5);
              \draw [fill=white] (-0.25, -0.25) rectangle (0.25, 0.25);
           \node at (0, 0) {\tiny $\widehat{\cL}_0^{1/2}$};    
           \end{scope}
        \end{tikzpicture}
        }}}
        -\vcenter{\hbox{\scalebox{0.8}{
        \begin{tikzpicture}[scale=1.2]
           \draw [blue] (-0.6, 1.4)--(-0.6, -0.4) (-0.3, 1.4)--(-0.3, -0.8) (-0.1, 0.8)--(-0.1, -1.4) (0.2, 0.4)--(0.2, -1.4) (0.4, 1.4)--(0.4, -1.4);
           \draw [blue] (0.6, -0.3).. controls +(0, -0.3) and +(0, -0.3) .. (0.9, -0.3)--(0.9, 0.3).. controls +(0, 0.3) and +(0, 0.3) ..(0.6, 0.3);
           \draw [fill=white] (-0.7, -0.3) rectangle (0.7, 0.3);
           \node at (0, 0) {\tiny $\widehat{\widetilde{\Gamma}}$};
           \begin{scope}[shift={(-0.2, 0.8)}]
           \draw[blue] (0.1, 0.2) .. controls +(0, 0.2) and +(0, 0.2) .. (0.4, 0.2)--(0.4, -0.5);
            \draw [fill=white] (-0.25, -0.25) rectangle (0.25, 0.25);
           \node at (0, 0) {\tiny $\widehat{\cL}_0^{1/2}$};    
           \end{scope}
            \begin{scope}[shift={(-0.2, -0.8)}]
           \draw[blue] (-0.1, -0.2) .. controls +(0, -0.2) and +(0, -0.2) .. (-0.4,- 0.2)--(-0.4, 0.5);
            \draw [fill=white] (-0.25, -0.25) rectangle (0.25, 0.25);
           \node at (0, 0) {\tiny $\widehat{\cL}_0^{1/2}$};    
           \end{scope}
        \end{tikzpicture}
        }}}.
\end{align*}

\begin{proposition}\label{prop:beinter}
Suppose that $\{\Phi_t\}_{t\geq 0}$ satisfies the intertwining property \eqref{eq:intertwining3}.  
Then for any $x\in \cM$, we have that 
\begin{align*}
   \Gamma_2(x) \geq  \beta \Gamma(x)+\frac{1}{2} \lambda^{-1/2}\sum_{j=1}^m \bE_{\cM}\widetilde{\Gamma}(\partial_j x).
\end{align*}
Moreover, 
\begin{align*}
  \Gamma_2(x)
  \geq & \beta \Gamma(x)+ \frac{ \lambda^{-3/2}}{4m} \bE_{\cM}\left((\widetilde{\partial}\partial (x))^*(\widetilde{\partial}\partial x) \right) 
 =  \beta \Gamma(x)+ \frac{ \lambda^{-1/2}}{2m} \bE_{\cM}\left(\widetilde{\Gamma}(\partial x)\right).
\end{align*}
In terms of the Fourier multipliers, we have that 
\begin{align*}
\widehat{\Gamma}_2 \geq \beta \widehat{\Gamma} +\frac{\lambda^{-1}}{2m} \widehat{\widetilde{\Gamma}}_{\partial}.
\end{align*}
\end{proposition}
\begin{proof}
By the intertwining property (i.e. Equation \eqref{eq:intertwining3}), we have that 
\begin{align*}
2\Gamma_2(x)=& \Gamma(x, \cL(x))+\Gamma(\cL(x), x)-\cL(\Gamma(x, x)) \\
=&  \lambda^{-1/2}\Re \sum_{j=1}^m \bE_{\cM}((\partial_j\cL (x))^*(\partial_j x) )
-\cL(\Gamma(x, x))\\
=& \lambda^{-1/2} \Re  \sum_{j=1}^m\bE_{\cM}((\cJ \partial_j (x))^*(\partial_j x) )
+ \lambda^{-1/2} \Re \sum_{j=1}^m  \beta_j \bE_{\cM}((\partial_j x)^*(\partial_j x) )\\
& -\frac{\lambda^{-1/2}}{2} \sum_{j=1}^m \cL\bE_{\cM}((\partial_j x)^*(\partial_j x))\\
\geq & 2\beta \Gamma(x)+ \lambda^{-1/2}  \sum_{j=1}^m \bE_{\cM}\widetilde{\Gamma}(\partial_j x) \\
=& 2\beta \Gamma(x)+\frac{1}{2}\lambda^{-3/2}  \sum_{j=1}^m \bE_{\cM}\left((\widetilde{\partial}\partial_j (x))^*(\widetilde{\partial}\partial_j x) \right)\\
\geq & 2\beta \Gamma(x)+\lambda^{-3/2}  \frac{1}{2m}\bE_{\cM}((\widetilde{\partial}\partial (x))^*(\widetilde{\partial}\partial x) ).
\end{align*}
By applying the pictorial computation, we see the third inequality is true.
\end{proof}

\begin{proposition}\label{prop:matrixderivation}
Suppose that the inclusion is $\bC\subset M_n(\bC)$ and $\{\Phi_t\}_{t\geq 0}$ is bimodule GNS symmetric for the inclusion whose the lifting is standard.
Then 
\begin{align}\label{eq:be2}
  \Gamma_2(x) \geq \beta \Gamma(x) +\frac{1}{\sum_{j=1}^m (1+\mu_j)^2}  \left| \cL(x)+\cL^*(x)+\frac{1}{2}\left\{x, 1*(\overline{\widehat{\cL}_0}-\widehat{\cL}_0)\right\}\right|^2,
\end{align}
where $m$ is the dimension of the algebra generated by $\cR(\widehat{\cL}_0)\widehat{\Delta}$ and $\widehat{\cL}_0$.
If $1*\overline{\widehat{\cL}_0}=1*\widehat{\cL}_0$, then 
\begin{align*}
  \Gamma_2(x) \geq \beta \Gamma(x) +\frac{1}{\sum_{j=1}^m (1+\mu_j)^2}  \left| \cL(x)+\cL^*(x)\right|^2.
\end{align*}
If the semigroup is tracially symmetric, i.e. $\widehat{\Delta}=1$ or $\widehat{\cL}=\overline{\widehat{\cL}}$, we have that
\begin{align*}
  \Gamma_2(x) \geq \beta \Gamma(x) +\frac{1}{m}  \left| \cL(x)\right|^2.
\end{align*}

\end{proposition}
\begin{proof}
Firstly, we assume that the projection $p_j$ in the decomposition of $\widehat{\cL}_0$ are minimal projection in $\cM'\cap \cM_2$ and then 
\begin{align*}
\widehat{\cL}_0=\sum_{j=1}^m \mu_j^{1/2} \vcenter{\hbox{\begin{tikzpicture}[scale=0.65]
    \begin{scope}[shift={(0,1.5)}]
    \draw [blue] (-0.5, 0.8)--(-0.5, 0) .. controls +(0, -0.6) and +(0,-0.6).. (0.5, 0)--(0.5, 0.8);    
\begin{scope}[shift={(0.5, 0.3)}]
\draw [fill=white] (-0.3, -0.3) rectangle (0.3, 0.3);
\node at (0, 0) {\tiny $v_j$};
\end{scope}
    \end{scope}
\draw [blue] (-0.5, -0.8)--(-0.5, 0) .. controls +(0, 0.6) and +(0,0.6).. (0.5, 0)--(0.5, -0.8);
\begin{scope}[shift={(0.5, -0.3)}]
\draw [fill=white] (-0.3, -0.3) rectangle (0.3, 0.3);
\node at (0, 0) {\tiny $v_j^*$};
\end{scope}
\end{tikzpicture}}},\quad
\widehat{\cJ}_0= \sum_{j=1}^m \mu_j^{1/2} \vcenter{\hbox{\begin{tikzpicture}[scale=0.65]
    \begin{scope}[shift={(0,1.5)}]
    \draw [blue] (-0.5, 0.8)--(-0.5, 0) .. controls +(0, -0.6) and +(0,-0.6).. (0.5, 0)--(0.5, 0.8);  
    \draw [blue] (-1, 0.8)--(-1, 0) .. controls +(0, -0.8) and +(0,-0.8).. (1, 0)--(1, 0.8); 
\begin{scope}[shift={(0.5, 0.3)}]
\draw [fill=white] (-0.3, -0.3) rectangle (0.3, 0.3);
\node at (0, 0) {\tiny $v_j$};
\end{scope}
    \end{scope}
\draw [blue] (-0.5, -0.8)--(-0.5, 0) .. controls +(0, 0.6) and +(0,0.6).. (0.5, 0)--(0.5, -0.8);
\draw [blue] (-1, -0.8)--(-1, 0) .. controls +(0, 0.8) and +(0,0.8).. (1, 0)--(1, -0.8); 
\begin{scope}[shift={(0.5, -0.3)}]
\draw [fill=white] (-0.3, -0.3) rectangle (0.3, 0.3);
\node at (0, 0) {\tiny $v_j^*$};
\end{scope}
\end{tikzpicture}}}.
\end{align*}
subject to $\tau(v_j^*v_k)=0$ for $j\neq k$, $\tau(v_jv_k)=0$ for all $j, k=1, \ldots, m$ and there is an involution $*$ on $\{1, 2, \ldots, m \}$ such that $v_{j^*}=v_j^*$ and $\mu_{j^*}=\mu_j^{-1}$.
Note that the Fourier multiplier of the directional derivations and lifted directional derivations are 
\begin{align*}
\widehat{\partial}_j=\frac{ \mu_j^{1/4}\lambda^{1/4}}{\|v_j\|_2}\left(\vcenter{\hbox{\begin{tikzpicture}[scale=0.65]
     \draw [blue] (1.4, 2.3)--(1.4, -0.8);
    \begin{scope}[shift={(0,1.5)}]
    \draw [blue] (-0.5, 0.8)--(-0.5, 0) .. controls +(0, -0.6) and +(0,-0.6).. (0.5, 0)--(0.5, 0.8); 
\begin{scope}[shift={(0.5, 0.3)}]
\draw [fill=white] (-0.3, -0.3) rectangle (0.3, 0.3);
\node at (0, 0) {\tiny $v_j$};
\end{scope}
    \end{scope}
\draw [blue] (-0.5, -0.8)--(-0.5, 0) .. controls +(0, 0.6) and +(0,0.6).. (0.5, 0)--(0.5, -0.8);
\begin{scope}[shift={(1.4, 0.8)}]
\draw [fill=white] (-0.3, -0.3) rectangle (0.3, 0.3);
\node at (0, 0) {\tiny $\overline{v_j^*}$};  
\end{scope}
\end{tikzpicture}}}  -  \vcenter{\hbox{\begin{tikzpicture}[scale=0.65]
     \draw [blue] (1.4, 2.3)--(1.4, -0.8);
   \begin{scope}[shift={(0,1.5)}]
  \draw [blue] (-0.5, 0.8)--(-0.5, 0) .. controls +(0, -0.6) and +(0,-0.6).. (0.5, 0)--(0.5, 0.8); 
    \end{scope}
\draw [blue] (-0.5, -0.8)--(-0.5, 0) .. controls +(0, 0.6) and +(0,0.6).. (0.5, 0)--(0.5, -0.8);
\begin{scope}[shift={(0.5, -0.3)}]
\draw [fill=white] (-0.3, -0.3) rectangle (0.3, 0.3);
\node at (0, 0) {\tiny $ v_j$};
\end{scope}
\begin{scope}[shift={(1.4, 0.8)}]
\draw [fill=white] (-0.3, -0.3) rectangle (0.3, 0.3);
\node at (0, 0) {\tiny $\overline{v_j^*}$};  
\end{scope}
\end{tikzpicture}}} \right), \quad 
\widehat{\widetilde{\partial}}_j
=\frac{ \mu_j^{1/4}\lambda^{1/2}}{\|v_j\|_2} \left(\vcenter{\hbox{\begin{tikzpicture}[scale=0.65]
     \draw [blue] (1.3, 2.3)--(1.3, -0.8);
     \draw [blue] (1.8, 2.3)--(1.8, -0.8);
    \begin{scope}[shift={(0,1.5)}]
    \draw [blue] (-0.5, 0.8)--(-0.5, 0) .. controls +(0, -0.6) and +(0,-0.6).. (0.5, 0)--(0.5, 0.8); 
     \draw [blue] (-1, 0.8)--(-1, 0) .. controls +(0, -0.8) and +(0,-0.8).. (1, 0)--(1, 0.8);
\begin{scope}[shift={(0.5, 0.3)}]
\draw [fill=white] (-0.3, -0.3) rectangle (0.3, 0.3);
\node at (0, 0) {\tiny $v_j$};
\end{scope}
    \end{scope}
\draw [blue] (-0.5, -0.8)--(-0.5, 0) .. controls +(0, 0.6) and +(0,0.6).. (0.5, 0)--(0.5, -0.8);
\draw [blue] (-1, -0.8)--(-1, 0) .. controls +(0, 0.8) and +(0,0.8).. (1, 0)--(1, -0.8);
\begin{scope}[shift={(1.8, 0.8)}]
\draw [fill=white] (-0.3, -0.3) rectangle (0.3, 0.3);
\node at (0, 0) {\tiny $\overline{v_j^*}$};  
\end{scope}
\end{tikzpicture}}} 
 -  \vcenter{\hbox{\begin{tikzpicture}[scale=0.65]
      \draw [blue] (1.3, 2.3)--(1.3, -0.8);
     \draw [blue] (1.8, 2.3)--(1.8, -0.8);
   \begin{scope}[shift={(0,1.5)}]
  \draw [blue] (-0.5, 0.8)--(-0.5, 0) .. controls +(0, -0.6) and +(0,-0.6).. (0.5, 0)--(0.5, 0.8); 
    \draw [blue] (-1, 0.8)--(-1, 0) .. controls +(0, -0.8) and +(0,-0.8).. (1, 0)--(1, 0.8); 
    \end{scope}
\draw [blue] (-0.5, -0.8)--(-0.5, 0) .. controls +(0, 0.6) and +(0,0.6).. (0.5, 0)--(0.5, -0.8);
\draw [blue] (-1, -0.8)--(-1, 0) .. controls +(0, 0.8) and +(0,0.8).. (1, 0)--(1, -0.8);
\begin{scope}[shift={(0.5, -0.3)}]
\draw [fill=white] (-0.3, -0.3) rectangle (0.3, 0.3);
\node at (0, 0) {\tiny $ v_j$};
\end{scope}
\begin{scope}[shift={(1.8, 0.8)}]
\draw [fill=white] (-0.3, -0.3) rectangle (0.3, 0.3);
\node at (0, 0) {\tiny $\overline{v_j^*}$};  
\end{scope}
\end{tikzpicture}}} \right).
\end{align*}
Hence, for each $j=1, \ldots, m$, the Fourier multiplier of $\widetilde{\partial}_{j^*}\partial_j$ is depicted as
\begin{align*}
\widehat{\widetilde{\partial}_{j^*}\partial_j}
=\lambda^{1/4}\|v_j\|_2^{-2} 
\vcenter{\hbox{\begin{tikzpicture}[scale=0.65]
\draw [blue] (-1.7, 1)--(-1.7, -1);
\draw [blue] (-1.3, 1)--(-1.3, -1);
\draw [blue] (-0.3, 1)--(-0.3, -1);
\draw [blue] (0.3, 1)--(0.3, -1);
 \draw [blue] (0.9, 1)--(0.9, -1);
     \begin{scope}[shift={(-1.5, 0)}]
\draw [fill=white] (-0.6, -0.4) rectangle (0.6, 0.4);
\node at (0, 0) {\tiny $\widehat{\partial_j^*\partial_j}$};  
\end{scope}
    \begin{scope}[shift={(-0.3, 0)}]
\draw [fill=white] (-0.4, -0.4) rectangle (0.4, 0.4);
\node at (0, 0) {\tiny $\overline{v_j^*}$};  
\end{scope}
\begin{scope}[shift={(0.9, 0)}]
\draw [fill=white] (-0.4, -0.4) rectangle (0.4, 0.4);
\node at (0, 0) {\tiny $\overline{v_j}$};  
\end{scope}
\end{tikzpicture}}}.
\end{align*}

Hence
\begin{equation}\label{eq:phigamma0}
\begin{aligned}
 & \sum_{j=1}^m \bE_{\cM}((\widetilde{\partial}\partial_j (x))^*(\widetilde{\partial}\partial_j x) )
 \geq   \sum_{j=1}^m  \bE_{\cM}\left((\widetilde{\partial}_{j^*}\partial_j (x))^*(\widetilde{\partial}_{j^*}\partial_j x) \right)\\
 =&\lambda^{1/2}  \sum_{j=1}^m (\partial_{j}^*\partial_j x)^*(\partial_{j}^*\partial_j x)\\
 \geq &\frac{\lambda^{1/2}}{\sum_{j=1}^m (1+\mu_j^{-1})^2} \left(\sum_{j=1}^m (1+\mu_j^{-1}) \partial_{j}^*\partial_j x\right)^*\left( \sum_{j=1}^m (1+\mu_j^{-1})\partial_{j}^*\partial_j x\right)\\
 =&\frac{4\lambda^{3/2}}{\sum_{j=1}^m (1+\mu_j^{-1})^2}  \left| \cL(x)+\cL^*(x)+\frac{1}{2}\{x, 1*(\overline{\widehat{\cL}_0}-\widehat{\cL}_0)\}\right|^2.
\end{aligned}
\end{equation}
This implies that Equation \eqref{eq:be2} holds.
When $p_j$ is not minimal in $\cM'\cap \cM_2$, we can apply a similar computation to obtain the desired results.
\end{proof}

\begin{remark}
In \cite{WirZha21b}, Wirth and Zhang introduced the Bakry-\'{E}mery estimate $\BE(K, N)$ for tracially symmetric quantum Markov semigroups.
Note that Equation \eqref{eq:be2} does not rely on the bimodule structure.
It could be used to define  the Bakry-\'{E}mery estimate $\BE(\beta, m)$ for (bimodule) GNS symmetric quantum Markov semigroups.

\end{remark}

\begin{theorem}\label{thm:bematrix}
Suppose that $\bC\subset M_n(\bC)$ is the inclusion and $\{\Phi_t\}_{t\geq 0}$ is a bimodule GNS symmetric Markov semigroup.
Then for any $x\in M_n(\bC)$, we have that 
\begin{equation}\label{eq:be4}
\begin{aligned}
& \Gamma(\Phi_t(x)) \leq e^{-2\beta t} \Phi_t(\Gamma(x)) \\
& - \frac{1-e^{-2\beta t}}{\displaystyle \beta \sum_{j=1}^m (1+\mu_j)^2}  \left| \cL(\Phi_t(x))+\int_0^t \Phi_s\cL^*(\Phi_{t-s}(x))+\frac{1}{2}\Phi_s\{\Phi_{t-s}(x), 1*(\overline{\widehat{\cL}_0}-\widehat{\cL}_0) \}d\nu_t(s)\right|^2
\end{aligned}
\end{equation}
is equivalent to Equation \eqref{eq:be2}, where $\displaystyle d\nu_t(s)=\frac{2\beta e^{-2\beta s}}{1-e^{-2\beta t}}ds$ and $\displaystyle \int_0^t d\nu_t(s)=1$.
\end{theorem}
\begin{proof}
We shall follow the lines in \cite{WirZha21b}.
Let 
\begin{align*}
& \phi(t)=  e^{-2\beta t} \Phi_t(\Gamma(x)) -\Gamma(\Phi_t(x)) \\
& - \frac{1-e^{-2\beta t}}{\displaystyle \beta \sum_{j=1}^m (1+\mu_j)^2}  \left| \cL(\Phi_t(x))+\int_0^t \Phi_s\cL^*(\Phi_{t-s}(x))+\frac{1}{2}\Phi_s\{\Phi_{t-s}(x), 1*(\overline{\widehat{\cL}_0}-\widehat{\cL}_0) \}d\nu_t(s)\right|^2.
\end{align*}
We have that $\phi(t)\geq 0$ and $\phi(0)=0$.
This implies that $\phi'(0)\geq 0$.
Note that 
\begin{small}
\begin{align*}
\lim_{t\to 0} \int_0^t \Phi_s\cL^*(\Phi_{t-s}(x))+\frac{1}{2}\Phi_s\{\Phi_{t-s}(x), 1*(\overline{\widehat{\cL}_0}-\widehat{\cL}_0) \} -\cL^*(x) - \frac{1}{2}\{x, 1*(\overline{\widehat{\cL}_0}-\widehat{\cL}_0) \} d\nu_t(s) =0.
\end{align*}
\end{small}
Then Equation \eqref{eq:be2} is true.

Let $\varphi(s)=e^{-2\beta s}\Phi_s\Gamma(\Phi_{t-s}(x))$ for $s\in [0,t]$.
We have that 
\begin{align*}
\varphi'(s)=& 2e^{-2\beta s}\Phi_s(\Gamma_2(\Phi_{t-s}(x))-\beta \Gamma(\Phi_{t-s}(x))) \\
\geq & \frac{2e^{-2\beta s}}{\displaystyle \beta \sum_{j=1}^m (1+\mu_j)^2} \Phi_s \left| \cL(\Phi_{t-s}(x))+ \cL^*(\Phi_{t-s}(x))+\frac{1}{2} \{\Phi_{t-s}(x), 1*(\overline{\widehat{\cL}_0}-\widehat{\cL}_0) \}\right|^2\\
\geq & \frac{2e^{-2\beta s}}{\displaystyle \beta \sum_{j=1}^m (1+\mu_j)^2}  \left| \cL(\Phi_{t}(x))+ \Phi_s\cL^*(\Phi_{t-s}(x))+\frac{1}{2}\Phi_s\{\Phi_{t-s}(x), 1*(\overline{\widehat{\cL}_0}-\widehat{\cL}_0) \}\right|^2.
\end{align*}
Now we see that 
\begin{align*}
& \varphi(t)-\varphi(0)=\int_0^t \varphi'(s) ds \\
=&  \frac{1-e^{-2\beta t}}{\displaystyle\beta \sum_{j=1}^m (1+\mu_j)^2}  \int_0^t\left| \cL(\Phi_{t}(x))+ \Phi_s\cL^*(\Phi_{t-s}(x))+\frac{1}{2}\Phi_s\{\Phi_{t-s}(x), 1*(\overline{\widehat{\cL}_0}-\widehat{\cL}_0) \}\right|^2 d\nu_t(s)\\
\geq &  \frac{1-e^{-2\beta t}}{\displaystyle \beta \sum_{j=1}^m (1+\mu_j)^2}  \left| \cL(\Phi_t(x))+\int_0^t \Phi_s\cL^*(\Phi_{t-s}(x))+\frac{1}{2}\Phi_s\{\Phi_{t-s}(x), 1*(\overline{\widehat{\cL}_0}-\widehat{\cL}_0) \}d\nu_t(s)\right|^2.
\end{align*}
This completes the proof of the proposition.
\end{proof}

\begin{remark}
If $1*\overline{\widehat{\cL}_0}=1*\widehat{\cL}_0$, then Equation \eqref{eq:be4} is reduced to 
\begin{equation}\label{eq:be6}
\begin{aligned}
 \Gamma(\Phi_t(x)) \leq e^{-2\beta t} \Phi_t(\Gamma(x)) -  \frac{1-e^{-2\beta t}}{\displaystyle \beta \sum_{j=1}^m (1+\mu_j)^2}  \left| \cL(\Phi_t(x))+\int_0^t \Phi_s\cL^*(\Phi_{t-s}(x)) d\nu_t(s)\right|^2.
\end{aligned}
\end{equation}
This implies that 
\begin{align*}
 \left\| \cL(\Phi_t(x))+\int_0^t \Phi_s\cL^*(\Phi_{t-s}(x)) d\nu_t(s)\right\|\leq \sqrt{\frac{\beta  \sum_{j=1}^m (1+\mu_j)^2}{e^{2\beta t}-1}},
\end{align*}
for all $x=x^*\in \cM$ with $\Gamma(x)\leq 1$.
\end{remark}

\begin{remark}
Suppose that the bimodule GNS symmetric quantum Markov semigroup $\{\Phi_t\}_{t\geq 0}$ satisfies the intertwining property without assuming that $\{\widetilde{\Phi_t}\}_{t\geq 0}$ is a semigroup.
We still have that 
\begin{equation}\label{eq:gfh}
\begin{aligned}
\Gamma(\Phi_t(x))=&\frac{\lambda^{-1/2}}{2} \sum_{j=1}^m \bE_{\cM}\left((\partial_j\Phi_t(x))^*(\partial_j\Phi_t(x))\right) \\
=& \frac{\lambda^{-1/2}}{2} \sum_{j=1}^m  e^{-2\beta_j t} \bE_{\cM}\left((\widetilde{\Phi_t}(\partial_jx))^*(\widetilde{\Phi_t}(\partial_j x))\right) \\
\geq &\frac{\lambda^{-1/2}}{2} e^{-2\beta t}\sum_{j=1}^m \bE_{\cM}(\widetilde{\Phi_t}((\partial_jx))^*(\partial_j x)) =e^{-2\beta t}\Phi_t(\Gamma(x)).
\end{aligned}
\end{equation}
\end{remark}

In the following, we shall recall the weight transformation $\K_{D, \mu}: \cM_1\to \cM_1$ for a positive element $D\in \cM$ and $\mu>0$ as follows:
\begin{align*}
\K_{D, \mu} x=\int_0^1 (\mu^{-1}D)^s x (\mu D)^{1-s} ds, \quad x\in \cM_1.
\end{align*}
For any $\bfX=(x_j)_{j=1}^m$ and $\bfY=(y_j)_{j=1}^m$ in $\cM_1^{\oplus (m)}$, the inner product is defined as
\begin{align*}
\langle \bfX, \bfY \rangle_{D, \widehat{\Delta}}
=\sum_{j=1}^m \langle \K_{D, \mu_j} x_j, y_j\rangle.
\end{align*}
Suppose that $\{\Phi_t\}_{t\geq 0}$ satisfies the intertwining property.
Then
\begin{align*}
 \|\nabla \Phi_t(x)\|_{D, \widehat{\Delta}}^2
= \langle \K_D \nabla \Phi_t(x), \nabla \Phi_t(x)\rangle
= e^{-2\beta t}\left\langle \K_D \widetilde{ \Phi_t}(\nabla x), \widetilde{ \Phi_t}(\nabla x)\right\rangle\leq e^{-2\beta t}\|\nabla x\|^2_{\Phi^*_t(D), \widehat{\Delta}},
\end{align*}
where the last inequality follows from the Lieb's concavity theorem.

\begin{definition}
Suppose that $\{\Phi_t\}_{t\geq 0}$ is a bimodule GNS symmetric quantum Markov semigroup.
The entropic lower Ricci curvature bound $\ric_{\Phi}$ of $\{\Phi_t\}_{t \geq 0}$ is defined by the following gradient estimate
\begin{align*}
 \ric_\Phi=\sup\left\{ \beta\in\bR:   \|\nabla\Phi_t(x)\|_{D, \widehat{\Delta}}\leq e^{-\beta t} \|\nabla x\|_{\Phi_t^*(D), \widehat{\Delta}}\right\}.
\end{align*}
\end{definition}

Note that for (bimodule) GNS symmetric semigroups with an entropic lower Ricci curvature bound, the semigroup $\{\Phi_t\}_{t\geq 0}$ has modified logarithmic Sobolev inequality.
\begin{align}\label{eq:gnslogs}
H(\Phi_t^*(D)\| D_{\Delta}) -H(\bE_{\cN}(D) \overline{\widehat{\bE}_{\Phi}}\| D_{\Delta}) \leq e^{-2\beta t} (H(D\| D_{\Delta}) -H(\bE_{\cN}(D) \overline{\widehat{\bE}_{\Phi}}\| D_{\Delta})),
\end{align}
where $\displaystyle \widehat{\bE}_{\Phi}=\lim_{t\to\infty}\widehat{\Phi}_t$ and $D_{\Delta}$ is the hidden density obtained from the gradient flow equation.

\begin{theorem}
Suppose that $\{\Phi_t\}_{t\geq 0}$ is a bimodule GNS symmetric quantum Markov semigroup satisfying the intertwining property.
Suppose that the lifting Markov semigroup $\{\widetilde{\Phi_t}\}_{t\geq 0}$ on $\cM_1$ satisfies that there exists a unital completely positive map $\widetilde{\widetilde{\Phi_t}}$ on $\cM_2$ for all $t\geq 0$ such that $\widetilde{\widetilde{\Phi_t^*}}$ extends $\widetilde{\Phi_t^*}$ and
\begin{align*}
\widetilde{\partial} \widetilde{\Phi_t}=e^{-\vec{\beta} t}\widetilde{\widetilde{\Phi_t}}\widetilde{\partial}.
\end{align*}
Then 
\begin{align*}
\|\nabla \Phi_t x\|_{D, \widehat{\Delta}}^2 
\leq  e^{-2\beta t} \|\nabla x\|_{\Phi^*_t(D), \widehat{\Delta}}^2-  \lambda \frac{e^{-2\beta t}- e^{-4\beta t}}{\beta  }  \|\nabla x\|_{\Phi_{\widetilde{\Gamma}}^*(\Phi_t^*(D)), \widehat{\Delta}}^2,
\end{align*}
where $\Phi_{\widetilde{\Gamma}}: \cM_3\to \cM_1$ is the bimodule map associated to $\widetilde{\Gamma}$.
\end{theorem}
\begin{proof}
For any $x\in \cM_1$, by Equation \eqref{eq:gfh}, we have that 
\begin{align*}
\widetilde{\Phi_t }(x^*x) -\widetilde{\Phi_t }(x)^* \widetilde{\Phi_t }(x)
=& 2 \int_0^t  \widetilde{\Phi_s}(\widetilde{\Gamma}(\widetilde{\Phi_{t-s}}(x)))ds\\
\geq & 2 \int_0^t e^{2\beta s}\widetilde{\Gamma}(\widetilde{\Phi_t}(x)) ds \\
=& \frac{e^{2\beta t}-1}{\beta}\widetilde{\Gamma}(\widetilde{\Phi_t}(x))\\
=& \frac{e^{2\beta t}-1}{\beta}\frac{\lambda^{-1}}{2}\bE_{\cM_1}(\widetilde{\partial}\widetilde{\Phi_t}(x))^* (\widetilde{\partial}\widetilde{\Phi_t}(x)) \\
=& \frac{1- e^{-2\beta t}}{\beta}\frac{\lambda^{-1}}{2}\bE_{\cM_1}(\widetilde{\widetilde{\Phi_t}}(\widetilde{\partial}x))^* (\widetilde{\widetilde{\Phi_t}}(\widetilde{\partial}x)) \\
\geq & \frac{1- e^{-2\beta t}}{\beta}\frac{\lambda^{-1}}{2}\bE_{\cM_1}\widetilde{\widetilde{\Phi_t}}((\widetilde{\partial}x)^* (\widetilde{\partial}x)) \\
= & \frac{1- e^{-2\beta t}}{\beta}\frac{\lambda^{-1}}{2} \widetilde{\Phi_t}\bE_{\cM_1}((\widetilde{\partial}x)^* (\widetilde{\partial}x)).
\end{align*}

For any positive element $D\in \cM$ with $\tau(D)=1$ and $x\in \cM_1$, we have that 
\begin{equation}\label{eq:phigamma1}
\begin{aligned}
 \left \langle \widetilde{\Phi_t}\bE_{\cM_1}((\widetilde{\partial}x)^* (\widetilde{\partial}x)), D \right\rangle 
 =&  \left\langle \bE_{\cM_1} (\widetilde{\partial}x)^* (\widetilde{\partial}x), \Phi_t^*(D) \right\rangle\\
 =& 2 \left\langle \Phi_{\widetilde{\Gamma}}(x^*\widetilde{e}_1 x),  \Phi_t^*(D) \right\rangle\\
  =& 2 \left\langle x^*\widetilde{e}_1 x,  \Phi_{\widetilde{\Gamma}}^*(\Phi_t^*(D)) \right\rangle\\
  =& 2\lambda^2 \langle J_1\Phi_{\widetilde{\Gamma}}^*(\Phi_t^*(D))J_1 x\Omega_1, x \Omega_1 \rangle.
\end{aligned}
\end{equation}
This implies that 
\begin{align*}
 \langle x \Phi_t^*(D), x\rangle=&   \langle \widetilde{\Phi}_t(x^*x), D\rangle \\
  \geq & \langle \widetilde{\Phi_t}(x), \widetilde{\Phi_t}(x) D \rangle +\frac{1- e^{-2\beta t}}{\beta}\lambda  \langle J_1\Phi_{\widetilde{\Gamma}}^*(\Phi_t^*(D))J_1 x \Omega_1, x \Omega_1 \rangle,
\end{align*}
i.e.
\begin{align*}
J_1\Phi_t^*(D) J_1
 \geq \fF^{-1}(\widehat{ \widetilde{\Phi_t}}) J_1 DJ_1 \fF^{-1}(\widehat{\widetilde{\Phi_t}})^* +\frac{1- e^{-2\beta t}}{\beta}\lambda  J_1\Phi_{\widetilde{\Gamma}}^*(\Phi_t^*(D))J_1.
\end{align*}
By taking the conjugation $J_1$, we obtain that 
\begin{align*}
\Phi_t^*(D) 
 \geq & J_1\fF^{-1}(\widehat{ \widetilde{\Phi_t}}) J_1 D J_1 \fF^{-1}(\widehat{\widetilde{\Phi_t}})^*J_1 +\frac{1- e^{-2\beta t}}{\beta}\lambda  \Phi_{\widetilde{\Gamma}}^*(\Phi_t^*(D))\\
=& \fF^{-1}(\widehat{ \widetilde{\Phi_t}})  D  \fF^{-1}(\widehat{\widetilde{\Phi_t}})^* +\frac{1- e^{-2\beta t}}{\beta}\lambda  \Phi_{\widetilde{\Gamma}}^*(\Phi_t^*(D)).
\end{align*}
By Lieb's concavity theorem and operator monotonicity of $t^{s}$, where $s\in (0,1)$, we have that 
\begin{align*}
 &\int_0^1 \mu_j^{1-2s} \Phi_t^*(D) ^s J_1\Phi_t^*(D)^{1-s} J_1 ds \\
 \geq&  \int_0^1 \mu_j^{1-2s} \fF^{-1}(\widehat{ \widetilde{\Phi_t}}) D^s J_1 D^{1-s} J_1  \fF^{-1}(\widehat{ \widetilde{\Phi_t}})^*  \\
& +\mu_j^{1-2s}\frac{1- e^{-2\beta t}}{\beta}\lambda \Phi_{\widetilde{\Gamma}}^*(\Phi_t^*(D))^{s} \left(J_1\Phi_{\widetilde{\Gamma}}^*(\Phi_t^*(D))J_1\right)^{1-s} ds.
\end{align*}
Hence for any $x\in \cM$, we have that 
\begin{align*}
& \|\nabla \Phi_t (x)\|_{D, \widehat{\Delta}}^2 
= \left \langle \K_D \nabla \Phi_t(x), \nabla \Phi_t(x)\right\rangle \\
=& e^{-2\beta t} \left\langle  \K_D \widetilde{\Phi_t} \nabla x, \widetilde{\Phi_t} \nabla x \right\rangle \\
=& \lambda^{1/2}e^{-2\beta t} \sum_{j=1}^m \int_0^1\left \langle  (\mu_j^{-1} D)^sJ_1 (\mu_j D)^{1-s}J_1\widetilde{\Phi_t}(  \partial_j x) \Omega_1, \widetilde{\Phi_t}(\partial_j x)\Omega_1 \right \rangle ds, \\
\leq & \lambda^{1/2} e^{-2\beta t} \sum_{j=1}^m \int_0^1\left \langle  (\mu_j^{-1} \Phi_t^*(D))^sJ_1 (\mu_j \Phi_t^*(D))^{1-s}J_1  \partial_j x\Omega_1,  \partial_j x \Omega_1  \right\rangle ds \\
& -\lambda^{3/2} \frac{e^{-2\beta t}- e^{-4\beta t}}{\beta}\sum_{j=1}^m  \int_0^1 \left\langle  (\mu_j^{-1} \Phi_{\widetilde{\Gamma}}^*(\Phi_t^*(D)))^s J_1 (\mu_j \Phi_{\widetilde{\Gamma}}^*(\Phi_t^*(D)))^{1-s}J_1   \partial_j x \Omega_1,  \partial_j x\Omega_1 \right \rangle ds\\
\leq & e^{-2\beta t} \langle  \K_{\Phi_t^*(D)}  \nabla  x,   \nabla x \rangle- \lambda \frac{e^{-2\beta t}- e^{-4\beta t}}{\beta }  \left\langle \K_{\Phi_{\widetilde{\Gamma}}^*(\Phi_t^*(D))}\nabla x, \nabla x \right\rangle \\
=& e^{-2\beta t} \|\nabla x\|_{\Phi_t^*(D), \widehat{\Delta}}^2- \lambda \frac{e^{-2\beta t}- e^{-4\beta t}}{\beta  }  \|\nabla x\|_{\Phi_{\widetilde{\Gamma}}^*(\Phi_t^*(D)), \widehat{\Delta}}^2.
\end{align*}
This completes the proof of the theorem.
\end{proof}

\begin{theorem}\label{thm:ricbd}
Suppose that $\bC\subset M_n(\bC)$ is the inclusion and $\{\Phi_t\}_{t\geq 0}$ is a bimodule GNS symmetric Markov semigroup satisfying the intertwining property.
Then for any $x\in M_n(\bC)$, we have that 
\begin{align*}
 \|\nabla \Phi_t x\|_{D, \widehat{\Delta}}^2 \leq   e^{-2\beta t} \|\nabla x\|_{\Phi_t^*(D), \widehat{\Delta}}^2- (\min_{1\leq j \leq m} \mu_j)\frac{(e^{-2\beta t}- e^{-4\beta t})}{\beta \sum_{j=1}^m (1+\mu_j)^2}  \langle P_Dx, x\rangle,
\end{align*}
where $\displaystyle \xi_D=\cL(\Phi_t^*(D))+\cL^*(\Phi_t^*(D))+\frac{1}{2}\left\{\Phi_t^*(D), 1*(\overline{\widehat{\cL}_0}-\widehat{\cL}_0)\right\}$ and $P_D=\langle \cdot, \xi_D\rangle \xi_D$.
\end{theorem}
\begin{proof}
Firstly, we assume that the projection $p_j$ in the decomposition $\widehat{\cL}_0$ are minimal projection  in $\cM'\cap \cM_2$ and
\begin{align*}
\widehat{\cL}_0=\sum_{j=1}^m \mu_j^{1/2} \vcenter{\hbox{\begin{tikzpicture}[scale=0.65]
    \begin{scope}[shift={(0,1.5)}]
    \draw [blue] (-0.5, 0.8)--(-0.5, 0) .. controls +(0, -0.6) and +(0,-0.6).. (0.5, 0)--(0.5, 0.8);    
\begin{scope}[shift={(0.5, 0.3)}]
\draw [fill=white] (-0.3, -0.3) rectangle (0.3, 0.3);
\node at (0, 0) {\tiny $v_j$};
\end{scope}
    \end{scope}
\draw [blue] (-0.5, -0.8)--(-0.5, 0) .. controls +(0, 0.6) and +(0,0.6).. (0.5, 0)--(0.5, -0.8);
\begin{scope}[shift={(0.5, -0.3)}]
\draw [fill=white] (-0.3, -0.3) rectangle (0.3, 0.3);
\node at (0, 0) {\tiny $v_j^*$};
\end{scope}
\end{tikzpicture}}},\quad
\widehat{\cJ}_0= \sum_{j=1}^m \mu_j^{1/2} \vcenter{\hbox{\begin{tikzpicture}[scale=0.65]
    \begin{scope}[shift={(0,1.5)}]
    \draw [blue] (-0.5, 0.8)--(-0.5, 0) .. controls +(0, -0.6) and +(0,-0.6).. (0.5, 0)--(0.5, 0.8);  
    \draw [blue] (-1, 0.8)--(-1, 0) .. controls +(0, -0.8) and +(0,-0.8).. (1, 0)--(1, 0.8); 
\begin{scope}[shift={(0.5, 0.3)}]
\draw [fill=white] (-0.3, -0.3) rectangle (0.3, 0.3);
\node at (0, 0) {\tiny $v_j$};
\end{scope}
    \end{scope}
\draw [blue] (-0.5, -0.8)--(-0.5, 0) .. controls +(0, 0.6) and +(0,0.6).. (0.5, 0)--(0.5, -0.8);
\draw [blue] (-1, -0.8)--(-1, 0) .. controls +(0, 0.8) and +(0,0.8).. (1, 0)--(1, -0.8); 
\begin{scope}[shift={(0.5, -0.3)}]
\draw [fill=white] (-0.3, -0.3) rectangle (0.3, 0.3);
\node at (0, 0) {\tiny $v_j^*$};
\end{scope}
\end{tikzpicture}}},
\end{align*}
subject to $\tau(v_j^*v_k)=0$ for $j\neq k$, $\tau(v_jv_k)=0$ for all $j, k=1, \ldots, m$ and there is an involution $*$ on $\{1, 2, \ldots, m \}$ such that $v_{j^*}=v_j^*$ and $\mu_{j^*}=\mu_j^{-1}$.

For any $x\in \cM$, we have that 
\begin{align*}
\Phi_t(x^*x)-\Phi_t(x)^*\Phi_t(x)
=& 2\int_0^t \Phi_s(\Gamma(\Phi_{t-s} (x) ))ds
\geq 2\int_0^t e^{2\beta s} \Gamma (\Phi_t(x) )ds \\
=& \frac{e^{2\beta t}-1}{\beta}\Gamma (\Phi_t(x) )\\
=& \frac{e^{2\beta t}-1}{\beta}\frac{\lambda^{-1/2}}{2}\bE_{\cM}((\partial\Phi_t(x))^* (\partial \Phi_t(x))) \\
=& \frac{e^{2\beta t}-1}{\beta}\frac{\lambda^{-1/2}}{2}\sum_{j=1}^m \bE_{\cM}((\partial_j\Phi_t(x))^* (\partial_j\Phi_t(x))) \\
=& \frac{1-e^{-2\beta t}}{\beta}\frac{\lambda^{-1/2}}{2}\sum_{j=1}^m \bE_{\cM}((\widetilde{\Phi_t}\partial_jx)^* (\widetilde{\Phi_t}\partial_j x)) \\
=& \frac{1-e^{-2\beta t}}{\beta}\sum_{j=1}^m (\Phi_t \partial_j^{(0)}x)^* (\Phi_t\partial_j^{(0)} x) \|v_j\|_2^2,
\end{align*}
where $\partial_j x=2^{1/2}\lambda^{1/4}(\partial_j^{(0)}x )\otimes \overline{v_j^*}$ for $j=1, \ldots, m$.

For any $D\in \cM$, we have that 
\begin{align*}
\sum_{j=1}^m \langle (\Phi_t \partial_j^{(0)}x)^* (\Phi_t\partial_j^{(0)} x), D\rangle
\geq  \sum_{j=1}^m \left| \left\langle \Phi_t (\partial_j^{(0)} x), D\right\rangle\right|^2=\sum_{j=1}^m \mu_j \left|\left\langle x, \partial_{j^*}^{(0)} \Phi_t^*(D)\right\rangle\right|^2.
\end{align*}


Let $\displaystyle P_j x\Omega =\|v_j\|_2^2\mu_j \left\langle x, \partial_{j^*}^{(0)} \Phi_t^*(D)\right\rangle \partial_{j^*}^{(0)} \Phi_t^*(D)\Omega$ for $j=1, \ldots, m$ and $x\in \cM$.
We have that 
\begin{align*}
J\Phi_t^*(D) J \geq \fF^{-1}(\widehat{\Phi}) JD J \fF^{-1}(\widehat{\Phi})^* +\frac{1-e^{-2\beta t}}{\beta}(\min_{1\leq j \leq m} \mu_j) \sum_{j=1}^m P_j.
\end{align*}
By taking the adjoint, we obtain that 
\begin{align*}
\Phi_t^*(D)  \geq \fF^{-1}(\widehat{\Phi}) D  \fF^{-1}(\widehat{\Phi})^* + \frac{1-e^{-2\beta t}}{\beta}(\min_{1\leq j \leq m} \mu_j) \sum_{j=1}^m  P_j.
\end{align*}
By Lieb's concavity theorem, we have that  
\begin{align*}
\int_0^1 \mu_j^{1-2s}\Phi_t^*(D)^s J \Phi_t^*(D)^{1-s} J ds 
\geq &  \int_0^1\mu_j^{1-2s} \fF^{-1}(\widehat{\Phi}) D^s J D^{1-s} J \fF^{-1}(\widehat{\Phi})^* ds \\
& +\frac{1-e^{-2\beta t}}{\beta} ( \min_{1\leq k\leq m} \mu_k) \sum_{k=1}^m P_k.
\end{align*}
Hence 
\begin{align*}
\|\nabla\Phi_t(x)\|_{D, \widehat{\Delta}}^2 
=& e^{-2\beta t} \sum_{j=1}^m \langle \K_D \Phi_t \partial_j^{(0)} x, \Phi_t\partial_j^{(0)}x \rangle \|v_j\|_2^2\\
\leq & e^{-2\beta t} \sum_{j=1}^m  \int_0^1\mu_j^{1-2s} \left\langle \fF^{-1}(\widehat{\Phi}) D^s J D^{1-s} J \fF^{-1}(\widehat{\Phi})^* \partial_j^{(0)} x \Omega, \partial_j^{(0)} x\Omega \right\rangle ds  \|v_j\|_2^2\\
& - \frac{(e^{-2\beta t}-e^{-4\beta t})}{\beta} ( \min_{1\leq j \leq m} \mu_j) \sum_{j,k=1}^m  \left \langle P_k  \partial_j^{(0)} x \Omega, \partial_j^{(0)} x\Omega \right\rangle  \|v_j\|_2^2.
\end{align*}
Now we have that 
\begin{align*}
& \sum_{j,k=1}^m   \langle P_k  \partial_j^{(0)} x \Omega, \partial_j^{(0)} x\Omega\rangle  \|v_j\|_2^2  \\
\geq & \sum_{j=1}^m   \langle P_{j^*}  \partial_j^{(0)} x \Omega, \partial_j^{(0)} x\Omega\rangle   \|v_j\|_2^2\\
= &    \sum_{j=1}^m   \left| \langle  \partial_{j^*}^{(0)}\partial_j^{(0)} x \Omega, \Phi_t^*(D)\Omega\rangle \right|^2  \|v_j\|_2^4\\
\geq & \frac{1}{\sum_{j=1}^m (1+\mu_j)^2}  \left| \left\langle \sum_{j=1}^m (1+\mu_j^{-1}) \partial_{j^*}^{(0)}\partial_j^{(0)} x \Omega, \Phi_t^*(D)\Omega \right\rangle \right|^2  \|v_j\|_2^4\\
=& \frac{1}{\sum_{j=1}^m (1+\mu_j)^2} |\langle  \cL(x)+\cL^*(x)+\frac{1}{2}\{x, 1*(\overline{\widehat{\cL}_0}-\widehat{\cL}_0)\}, \Phi_t^*(D)\rangle|^2 \\
=& \frac{1}{\sum_{j=1}^m (1+\mu_j)^2} \langle P_D x, x\rangle.
\end{align*}
This completes the proof of the theorem.
\end{proof}

\section{Intertwining Property for Bimodule KMS Symmetry}

In this section, we shall investigate the intertwining property for bimodule KMS symmetric quantum Makov semigroups.

\begin{definition}
Suppose that $\{\Phi_t\}_{t\geq 0}$ is a bimodule KMS symmetric if $\Phi_t$ is equilibrium for all $t\geq 0$ and there exists strictly positive operator $\widehat{\Delta}\in \cM'\cap \cM_2$ such that $\widehat{\Delta}e_2=e_2$ and $\overline{\widehat{\cL}}=\overline{\widehat{\Delta}}\widehat{\cL}\overline{\widehat{\Delta}}$, $\cR(\widehat{\cL})\overline{\widehat{\Delta}}=\cR(\widehat{\cL})\widehat{\Delta}^{-1}$.
\end{definition}

We assume that $\cR(\overline{\widehat{\cL}_0})=\cR(\widehat{\cL}_0)$ throughout this section.
Let $\widehat{\cL}_\Delta=\overline{\widehat{\Delta}^{1/2}}\widehat{\cL}_0\overline{\widehat{\Delta}^{1/2}}.$
We have that $\overline{\widehat{\cL}_\Delta}=\widehat{\cL}_{\Delta}$.

Let $\displaystyle \cA=\bigoplus_{\ell=1}^m M^{(\ell)}$ be the von Neumann subalgebra generated by $\widehat{\cL}_\Delta$ and $\widehat{\Delta}\cR(\widehat{\cL}_0)$, where $M^{(\ell)}$ is the $d_\ell\times d_\ell$ matrix algebra.
Let $\{E_{j,k}^{(\ell)}\}_{j,k=1}^{d_\ell}$ and $\{F_{j,k}^{(\ell)}\}_{j,k=1}^{d_\ell}$ be systems of matrix units for $M^{(\ell)}$ described in \cite{JWW25}.
Let $\displaystyle U^{(\ell)}=\sum_{j,k=1}^{d_\ell} u_{jk}^{(\ell)} F_{j,k}^{(\ell)}$ be the unitary element in $M^{(\ell)}$ such that $U^{(\ell)} E_{j,k}^{(\ell)}U^{(\ell)*}=F_{j,k}^{(\ell)}$.
We have that $\widehat{\cL}_\Delta$ is diagonal with respect to $\{E_{j,k}^{(\ell)}\}_{j,k, \ell}$ and $\widehat{\Delta}$ is diagonal with respect to the system $\{F_{j,k}^{(\ell)}\}_{j,k, \ell}$.
Let 
\begin{align*}
\widehat{\cL}_{\Delta}=& \sum_{\ell=1}^m \sum_{j=1}^{d_\ell} \omega_j^{(\ell)} E_{j,j}^{(\ell)},
\end{align*}
where $\omega_j^{(\ell)}>0$.
Let $B_\ell=\diag(\omega_1^{(\ell)}, \ldots, \omega_{d_\ell}^{(\ell)})$ and $\displaystyle \widehat{\Delta}=\sum_{\ell=1}^m \widehat{\Delta}^{(\ell)}$.
We have that $\overline{\widehat{\Delta}^{(\ell)}}=\widehat{\Delta}^{(\ell^*)-1}$.
Let $F_\ell=\diag(\mu_1^{(\ell)}, \ldots, \mu_{d_\ell}^{(\ell)})$

Recall that
\begin{align*}
    \cM_1^{\oplus} =\left\{(x_{k}^{(\ell)})_{k, \ell=1}^{d_\ell, m}: x_{k}^{(\ell)}\in \cM_1,  x_{k}^{(\ell)} e_2=F_{k,k}^{(\ell)} x_{k}^{(\ell)} e_2 \right\} \subset \bigoplus_{\ell=1}^m \cM_1^{\oplus d_\ell},
\end{align*}
and 
\begin{align*}
\oM=\left\{(x_{k, r}^{(\ell)})_{k, r, \ell} : x_{k,r}^{(\ell)} \in \cM_1, x_{k,r}^{(\ell)}e_2
= F_{k,r}^{(\ell)} y_{k,r}^{(\ell)} e_2, \text{ for some } y_{k,r}^{(\ell)} \in \cM_1  \right\}.
\end{align*}
The directional derivation $\partial_k^{(\ell)}: \cM\to \cM_1$ is defined by
\begin{align*}
(\partial_k^{(\ell)} x )e_2=\lambda^{-1/2} F_{k,k}^{(\ell)}[x, e_1] e_2.
\end{align*}
The gradient $\nabla_0: \cM \to \cM_1^{\oplus}$ is $\nabla_0 x =(\partial_k^{(\ell)} x)_{k, \ell}$ for all $x\in \cM$, where $\nabla_0 x$ is read as column vector.
Hence $(\nabla_0 x)^*$ is read as row vector.
The divergence $\Div_0=\nabla_0^*: \cM_1^{\oplus} \to \cM$.

We define $\displaystyle \mathbb{L}=\bigoplus_{\ell=1}^m  F_\ell U_\ell^* B_\ell  U_\ell F_\ell$ and $\displaystyle \overline{\mathbb{L}}=\bigoplus_{\ell=1}^m  F_\ell^{-1} U_\ell^* B_\ell  U_\ell F_\ell^{-1}$  .
Let $\mathcal{E}_{\mathbb{L}}: \cM_1^{\oplus}\to \cM_1^{\oplus}$ defined by
\begin{align*}
\mathcal{E}_{\mathbb{L}}(X) =\lambda^{-1} (\bE_{\cM_1} \otimes I) (\mathbb{L}X e_2),
\end{align*}
where $X\in \cM_1^{\oplus}$.

\begin{lemma}\label{lem:dualsum}
For any $x, y\in \cM$, we have that 
\begin{align*}
\Gamma(x, y) = \frac{\lambda^{-1/2}}{2}\bE_{\cM} \left( (\nabla_0 y)^*\mathcal{E}_{\mathbb{L}} (\nabla_0 x)\right),
\end{align*}
and 
\begin{align*}
\partial^*\partial x =   \Div_0 \mathcal{E}_\mathbb{L} \nabla_0 x,\quad
 \overline{\partial}^*\overline{\partial} x =  \Div_0 \mathcal{E}_{\overline{\mathbb{L}}} \nabla_0x.
\end{align*}
\end{lemma}
\begin{proof}
For any $x, y\in \cM$, we have that 
\begin{align*}
\Gamma(x, y) 
=& \frac{\lambda^{-1/2}}{2} \bE_{\cM} ((\partial y)^* \partial x)
=\frac{\lambda^{-3/2}}{2} \bE_{\cM} (e_2(\partial y)^*( \partial x)e_2) \\
=& \frac{\lambda^{-5/2}}{2} \bE_{\cM} (e_2([y, e_1])^*\widehat{\cL}_0( [x, e_1])e_2) \\
=& \frac{\lambda^{-1/2}}{2} \bE_{\cM} \left( (\nabla_0 y)^*\mathcal{E}_{\mathbb{L}} (\nabla_0 x)\right).
\end{align*}

Moreover, we have that 
\begin{align*}
\langle  \partial^*\partial x, y\rangle
=& \langle  \partial x, \partial y\rangle
=\lambda^{-1} \langle ( \partial x)e_2, (\partial y)e_2\rangle \\
=&\lambda^{-2} \langle\widehat{\cL}_0 [x, e_1] e_2, [y, e_1] e_2\rangle \\
=&  \left\langle \mathcal{E}_{\mathbb{L}} \nabla_0 x, \nabla_0 y\right\rangle \\
=&   \left\langle \Div_0 \mathcal{E}_{\mathbb{L}} \nabla_0 x, y\right\rangle.
\end{align*}
Similarly, we have that $\overline{\partial}^*\overline{\partial} x =  \Div_0\mathcal{E}_{ \overline{\mathbb{L} }} \nabla_0 x$.
\end{proof}

\begin{lemma}\label{lem:duallapsum}
Suppose that $\{\Phi_t\}_{t \geq 0}$ is bimodule KMS symmetric with respect to $\widehat{\Delta}\in \cM'\cap \cM_2$.
Then for any $x\in \cM$,
\begin{align*}
\cL(x)+\cL^*(x)+\frac{1}{2}\left\{x, 1*(\overline{\widehat{\cL}_0}-\widehat{\cL}_0)\right\} =\frac{\lambda^{-1/2}}{2}\Div_0 \mathcal{E}_{\mathbb{L}+\overline{\mathbb{L} }} \nabla_0 x .
\end{align*}
If $1*\widehat{\cL}_0=1*\overline{\widehat{\cL}_0}$, then 
\begin{align*}
\cL(x)+\cL^*(x) =\frac{\lambda^{-1/2}}{2}\Div_0 \mathcal{E}_{\mathbb{L}+\overline{\mathbb{L} }} \nabla_0 x.
\end{align*}
\end{lemma}
\begin{proof}
Recall that $\cL_w=-\cL_w^*$ and
\begin{align*}
\cL_a+\cL_{\overline{a}}=\frac{\lambda^{-1/2}}{2}(\partial^*\partial+\overline{\partial}^*\overline{\partial}).
\end{align*}
By Lemma \ref{lem:dualsum}, we see that the lemma is true.
\end{proof}

Recall that $\Div= 2^{-1/2} \lambda^{-1/4}\Div_0\widehat{\iota}^*, \quad \nabla=2^{-1/2} \lambda^{-1/4}\widehat{\iota}\nabla_0,$ where $\widehat{\iota}: \cM_1^{\oplus} \to \oM$ is defined by $\widehat{\iota}(x_k^{(\ell)})_{k, \ell} =\bigoplus_{\ell=1}^m \begin{pmatrix}
x_{1}^{(\ell)}  & \cdots &  x_{1}^{(\ell)}  \\
\vdots  & & \vdots \\
x_{d_\ell}^{(\ell)} & \cdots & x_{d_{\ell}}^{(\ell)}
\end{pmatrix}$.
\begin{definition}
Suppose that $\{\Phi_t\}_{t\geq 0}$ is a bimodule KMS symmetric.
We say that $\{\Phi_t\}_{t\geq 0}$ satisfies the intertwining property if there exist a positive semidefinite matrix $\displaystyle T=\bigoplus_{\ell=1}^m T^{(\ell)}$ with $T^{(\ell)}=(T_{jk}^{(\ell)} F_{j,k}^{(\ell)})_{j,k=1}^{d_\ell}$ and a bimodule quantum channel $\widetilde{\Phi_t}: \cM_1\to \cM_1$ such that $\widetilde{\Phi_t}^*|_{\cM}=\Phi_t^*$ and 
\begin{align}\label{eq:kmsintertwining}
\Phi_t ^* \Div =\Div \bigoplus_{\ell=1}^m   (\widetilde{\Phi_t}^*\otimes I) e^{-T^{(\ell)} t}.
\end{align}
\end{definition}

In the following, we assume that $\widetilde{\Phi_t}$ is a bimodule quantum Markov semigroup.
Let $\cJ$ be its generator.
By the fact that $\widetilde{\Phi_t}^*|_{\cM}=\Phi_t^*$, we have that $\cJ^*|_{\cM}=\cL^*$.

\begin{lemma}
Suppose that $\{\Phi_t\}_{t\geq 0}$ is a bimodule KMS symmetric quantum Markov semigroup and $\{\widetilde{\Phi_t}\}_{t\geq 0}$ is a bimodule quantum Markov semigroup such that $\widetilde{\Phi_t}^*$ extends $\Phi_t^*$ and $\{\widetilde{\Phi_t}^* e^{-T t}\}_{t\geq 0}$ is a semigroup.
Then the intertwining property \eqref{eq:kmsintertwining} is equivalent to 
\begin{align}\label{eq:kmsintertwining0}
    \nabla_0 \cL - \cJ \nabla_0=\mathcal{E}_T\nabla_0.
\end{align}
\end{lemma}
\begin{proof}
Suppose that $x\in \cM$ and $X^{(\ell)}(t)= (\partial_k^{(\ell)} \Phi_t(x))_{k=1}^{d_\ell}$.
Then $X^{(\ell)}(0)= (\partial_k^{(\ell)} x)_{k=1}^{d_\ell}$ and 
\begin{align*}
\frac{d}{dt} X^{(\ell)}(t) 
 =& - (\partial_k^{(\ell)} \cL\Phi_t(x))_{k=1}^{d_\ell} \\
=& - (\cJ \partial_k^{(\ell)} \Phi_t(x))_{k=1}^{d_\ell} -\mathcal{E}_{ T^{(\ell)} }(\partial_k^{(\ell)} \Phi_t(x))_{k=1}^{d_\ell} \\
=& (- \cJ -\mathcal{E}_T) X^{(\ell)}(t).
\end{align*}
Hence $ X^{(\ell)}(t)$ is the unique solution for the equation $\displaystyle \frac{d}{dt} Y(t) = (-\cJ -\mathcal{E}_T) Y(t)$ with $Y(0)=(\partial_k^{(\ell)}x)_{k=1}^{d_\ell}$ followed from the boundedness of $\cJ$.
Note that $\widetilde{\Phi_t} (\partial_k^{(\ell)} x)_{k=1}^{d_\ell} = e^{T^{(\ell)} t} \left(\partial_k^{(\ell)}\Phi_t(x)\right)_{k=1}^{d_\ell}$.
This implies the intertwining property.

Suppose that the intertwining property is true.
Then by differentiating Equation \eqref{eq:kmsintertwining}, we see Equation \eqref{eq:kmsintertwining0} is true.
\end{proof}

\begin{remark}
Expanding Equation \eqref{eq:kmsintertwining0}, we have that 
\begin{align*}
\partial_k^{(\ell)} \cL -\cJ \partial_k^{(\ell)} =\lambda^{-3/2} \sum_{j=1}^{d_\ell} T_{k,j}^{(\ell)} \bE_{\cM_1} \left(F_{k,j}^{(\ell)}[ \cdot, e_1] e_2\right)
\end{align*}
for all $k=1, \ldots, d_\ell$ and $\ell=1, \ldots, m$.
The corresponding pictorial representation is the following:
\begin{align}\label{eq:kmsintertwining3}
 \vcenter{\hbox{
\begin{tikzpicture}
\draw [blue] (-0.15, -0.8)--(-0.15, 0.8);
 \draw [blue] (0.15, -0.8)--(0.15, 0.3) .. controls +(0, 0.35) and +(0, 0.35).. (0.6, 0.3);
 \draw [fill=white] (-0.3, -0.3) rectangle (0.3, 0.3);
 \node at (0, 0) {\tiny $\widehat{\cL}$};
 \begin{scope}[shift={(0.75, 0)}]
  \draw [fill=white] (-0.3, -0.3) rectangle (0.3, 0.3);
 \node at (0, 0) {\tiny $F_{k,k}^{(\ell)}$};
 \draw [blue] (0.15, 0.3)--(0.15, 0.8) (-0.15, -0.3)--(-0.15, -0.8);
  \draw [blue] (0.15, -0.3) .. controls +(0, -0.35) and +(0, -0.35).. (0.6, -0.3)--(0.6, 0.8);
 \end{scope}
\end{tikzpicture}
 }}
 -
\vcenter{\hbox{
    \begin{tikzpicture}
    \draw [blue] (-0.15, -0.8)--(-0.15, 0.8);
 \draw [blue] (0.15, 0.8)--(0.15, -0.3) .. controls +(0, -0.35) and +(0, -0.35).. (0.6, -0.3);
 \draw [fill=white] (-0.3, -0.3) rectangle (0.3, 0.3);
 \node at (0, 0) {\tiny $\widehat{\cL}$};
 \begin{scope}[shift={(0.75, 0)}]
  \draw [fill=white] (-0.3, -0.3) rectangle (0.3, 0.3);
 \node at (0, 0) {\tiny $\overline{F_{k,k}^{(\ell)}}$};
 \draw [blue] (0.15, -0.3)--(0.15, -0.8) (-0.15, 0.3)--(-0.15, 0.8);
  \draw [blue] (0.15, 0.3) .. controls +(0, 0.35) and +(0, 0.35).. (0.6, 0.3)--(0.6, -0.8);
 \end{scope}
    \end{tikzpicture}
    }}
    -
       \vcenter{\hbox{
    \begin{tikzpicture}
       \draw [blue] (-0.75, -0.3)--(-0.75, 0.6) (-0.45, -0.3)--(-0.45, 0.6) (-0.15, -0.8)--(-0.15, 0.6) (0.15, -0.8)--(0.15, 1.6) (0.45, -0.8)--(0.45, 1.6);
       \draw[fill=white] (-0.6, -0.3) rectangle (0.6, 0.3);
       \node at (0, 0) {\tiny $\widehat{\cJ}$};
       \begin{scope}[shift={(-0.6, 0.8)}]
          \draw[fill=white] (-0.3, -0.3) rectangle (0.3, 0.3); 
           \node at (0, 0) {\tiny $F_{k,k}^{(\ell)}$};
        \draw[blue] (0.15, 0.3).. controls+(0, 0.3) and +(0, 0.3).. (0.45, 0.3)--(0.45, -0.3);
        \draw[blue] (-0.15, 0.3)--(-0.15, 0.8);
        \draw [blue] (-0.15, -1.1) .. controls +(0, -0.25) and +(0, -0.25).. (0.15, -1.1);
       \end{scope}
    \end{tikzpicture}
    }}
    +
  \vcenter{\hbox{
    \begin{tikzpicture}
       \draw [blue]   (-0.15, 0.8)--(-0.15, 0.3) (0.15, -0.8)--(0.15, 0.8) (0.45, -0.8)--(0.45, 0.8);
        \draw[fill=white] (-0.6, -0.3) rectangle (0.6, 0.3);
       \node at (0, 0) {\tiny $\widehat{\cJ}$};
       \begin{scope}[shift={(-1.2, 0)}]
       \draw [blue] (0.15, 0.3) .. controls +(0, 0.6) and +(0, 0.6) .. (-0.75, 0.3)--(-0.75, -0.8);
          \draw[fill=white] (-0.3, -0.3) rectangle (0.3, 0.3); 
           \node at (0, 0) {\tiny $F_{k,k}^{(\ell)}$};
        \draw[blue] (0.15, -0.3).. controls+(0, -0.25) and +(0, -0.25).. (0.45, -0.3)--(0.45, 0.3).. controls +(0, 0.3) and +(0, 0.3) .. (0.75, 0.3);
        \draw[blue] (-0.15, 0.3).. controls+(0, 0.3) and +(0, 0.3).. (-0.45, 0.3)--(-0.45, -0.3).. controls +(0, -0.6) and +(0, -0.6) .. (1.05, -0.3);
        \draw [blue] (-0.15, -0.3).. controls +(0, -0.4) and +(0, -0.4) .. (0.75, -0.3);
       \end{scope}
    \end{tikzpicture}
    }}
 =
\sum_{j=1}^m T_{k,j}^{(\ell)} \left( \vcenter{\hbox{
    \begin{tikzpicture}
 \draw [blue] (0.15, 0.8)--(0.15, -0.3) .. controls +(0, -0.35) and +(0, -0.35).. (0.6, -0.3)--(0.6, 0.8);
  \draw [blue] (-0.15, 0.8)--(-0.15, -0.8);
 \draw [fill=white] (-0.3, -0.3) rectangle (0.3, 0.3);
 \node at (0, 0) {\tiny $F_{k,j}^{(\ell)}$};
  \draw [blue] (-0.7, -0.8) .. controls +(0, 0.25) and +(0, 0.25).. (-0.4, -0.8);
    \end{tikzpicture}
    }}  
-
\vcenter{\hbox{
    \begin{tikzpicture}
 \draw [blue] (0.15, -0.8)--(0.15, 0.3) .. controls +(0, 0.35) and +(0, 0.35).. (0.6, 0.3)--(0.6, -0.8);
  \draw [blue] (-0.15, 0.8)--(-0.15, -0.8);
 \draw [fill=white] (-0.3, -0.3) rectangle (0.3, 0.3);
 \node at (0, 0) {\tiny $\overline{F_{k,j}^{(\ell)}}$};
  \draw [blue] (-0.7, 0.8) .. controls +(0, -0.25) and +(0, -0.25).. (-0.4, 0.8);
    \end{tikzpicture}
    }} \right)
\end{align}
for all $k=1, \ldots, d_\ell$ and $\ell=1, \ldots, m$.
\end{remark}

Suppose that $\vec{x}=(x_1, \ldots, x_m)^{\mathsf{T}}$, where $x_1, \ldots, x_m \in \cM_1$.
We denote by $\displaystyle \widetilde{\Gamma}(\vec{x}) =\sum_{j=1}^m \widetilde{\Gamma}(x_j)$.

\begin{proposition}\label{prop:kmsbeinter}
Suppose that $\{\Phi_t\}_{t\geq 0}$ is a bimodule KMS symmetric semigroup satisfying the intertwining property \eqref{eq:kmsintertwining0}, $\cJ\mathcal{E}_{\mathbb{L}}= \mathcal{E}_{\mathbb{L}} \cJ$, and
\begin{align*}
\mathbb{L}^{1/2} T \mathbb{L}^{-1/2} + \mathbb{L}^{-1/2} T \mathbb{L}^{1/2} \geq 2 \beta .
\end{align*}
We have that 
\begin{align}\label{eq:gammakms}
   \Gamma_2(x) \geq  \beta \Gamma(x)+\frac{1}{2}\lambda^{1/2}   \bE_{\cM}\widetilde{\Gamma}(\mathcal{E}_{\mathbb{L}^{1/2}}(\nabla_0 x)).
\end{align}
\end{proposition}
\begin{proof}
By the intertwining property \eqref{eq:kmsintertwining0}, we have that 
\begin{align*}
2\Gamma_2(x)=& \Gamma(x, \cL(x))+\Gamma(\cL(x), x)-\cL(\Gamma(x, x)) \\
=&  \lambda^{-1/2}\Re  \bE_{\cM}((\nabla_0\cL (x))^*\mathcal{E}_{\mathbb{L}}(\nabla_0 x) )
-\cL(\Gamma(x, x))\\
=& \lambda^{-1/2} \Re  \bE_{\cM}((\cJ \nabla_0 (x))^*\mathcal{E}_{\mathbb{L}}(\nabla_0 x) )
+ \lambda^{-1/2}  \Re \bE_{\cM}((\nabla_0 x)^* \mathcal{E}_{T\mathbb{L}}(\nabla_0 x) )\\
& -\frac{\lambda^{-1/2}}{2}   \cL\bE_{\cM}((\nabla_0 x)^*\mathcal{E}_{\mathbb{L}}(\nabla_0 x))\\
=& \lambda^{-1/2} \Re  \bE_{\cM}((\cJ \nabla_0 (x))^*\mathcal{E}_{\mathbb{L}}(\nabla_0 x) )
+ -\frac{\lambda^{-1/2}}{2}   \bE_{\cM}\cJ((\nabla_0 x)^*\mathcal{E}_{\mathbb{L}}(\nabla_0 x))\\
& \frac{1}{2} \lambda^{-1/2}  \bE_{\cM}((\nabla_0 x)^* \mathcal{E}_{T\mathbb{L} +\mathbb{L} T}(\nabla_0 x) ) \\
\geq & 2\beta \Gamma(x) 
 +\frac{1}{2}\lambda^{-1/2}  \Re  \bE_{\cM}(( \mathcal{E}_{\mathbb{L}^{1/2}} \cJ \nabla_0 (x))^*\mathcal{E}_{\mathbb{L}^{1/2}}(\nabla_0 x) )\\
 & -\frac{\lambda^{-1/2}}{2}   \bE_{\cM}\cJ(( \mathcal{E}_{\mathbb{L}^{1/2}} \nabla_0 x)^*\mathcal{E}_{\mathbb{L}^{1/2}}(\nabla_0 x)) \\
=& 2\beta \Gamma(x)+ \lambda^{1/2}   \bE_{\cM}\widetilde{\Gamma}(\mathcal{E}_{\mathbb{L}^{1/2}}(\nabla_0 x)). 
\end{align*}
This completes the proof of the proposition.
\end{proof}

\begin{remark}
Suppose that $\bC \subset M_n(\bC)$ is the inclusion and 
\begin{align*}
\widehat{\cL}_0=\sum_{j=1}^m \vcenter{\hbox{\begin{tikzpicture}[scale=0.65]
    \begin{scope}[shift={(0,1.5)}]
    \draw [blue] (-0.5, 0.8)--(-0.5, 0) .. controls +(0, -0.6) and +(0,-0.6).. (0.5, 0)--(0.5, 0.8);    
\begin{scope}[shift={(0.5, 0.3)}]
\draw [fill=white] (-0.3, -0.3) rectangle (0.3, 0.3);
\node at (0, 0) {\tiny $v_j$};
\end{scope}
    \end{scope}
\draw [blue] (-0.5, -0.8)--(-0.5, 0) .. controls +(0, 0.6) and +(0,0.6).. (0.5, 0)--(0.5, -0.8);
\begin{scope}[shift={(0.5, -0.3)}]
\draw [fill=white] (-0.3, -0.3) rectangle (0.3, 0.3);
\node at (0, 0) {\tiny $v_j^*$};
\end{scope}
\end{tikzpicture}}},\quad
\widehat{\cJ}_0= \sum_{j=1}^m \vcenter{\hbox{\begin{tikzpicture}[scale=0.65]
    \begin{scope}[shift={(0,1.5)}]
    \draw [blue] (-0.5, 0.8)--(-0.5, 0) .. controls +(0, -0.6) and +(0,-0.6).. (0.5, 0)--(0.5, 0.8);  
    \draw [blue] (-1, 0.8)--(-1, 0) .. controls +(0, -0.8) and +(0,-0.8).. (1, 0)--(1, 0.8); 
\begin{scope}[shift={(0.5, 0.3)}]
\draw [fill=white] (-0.3, -0.3) rectangle (0.3, 0.3);
\node at (0, 0) {\tiny $v_j$};
\end{scope}
    \end{scope}
\draw [blue] (-0.5, -0.8)--(-0.5, 0) .. controls +(0, 0.6) and +(0,0.6).. (0.5, 0)--(0.5, -0.8);
\draw [blue] (-1, -0.8)--(-1, 0) .. controls +(0, 0.8) and +(0,0.8).. (1, 0)--(1, -0.8); 
\begin{scope}[shift={(0.5, -0.3)}]
\draw [fill=white] (-0.3, -0.3) rectangle (0.3, 0.3);
\node at (0, 0) {\tiny $v_j^*$};
\end{scope}
\end{tikzpicture}}}.
\end{align*}
Suppose that $y\in \cM'\cap \cM_2$ and $\mathcal{E}_y: \cM_1\to \cM_1$ is $\mathcal{E}_y(x)=\lambda^{-1} \bE_{\cM_1}(y x e_2)$ for any $x\in \cM_1$.
The Fourier multiplier of $\mathcal{E}_y$ is 
$ \vcenter{\hbox{\scalebox{0.8}{
        \begin{tikzpicture}[scale=1.2]
           \draw [blue] (-0.3, 0.9).. controls +(0, -0.4) and +(0, -0.4)..(-0.7, 0.9) ;
           \draw [blue] (0.3, -0.9).. controls +(0, 0.4) and +(0, 0.4)..(0.7, -0.9) ;
           \draw [blue] (-0.2, 0.9)--(-0.2, -0.3)..controls +(0, -0.45) and +(0, -0.45)..(-0.8, -0.3)--(-0.8, 0.9);
           \draw [blue] (0.2, -0.9)--(0.2, 0.3)..controls +(0, 0.45) and +(0, 0.45)..(0.8, 0.3)--(0.8, -0.9);
           \draw [fill=white] (-0.4, -0.3) rectangle (0.4, 0.3);
           \node at (0, 0) {\tiny $y$};
        \end{tikzpicture}}}}$.
Now it is easy to read that $\mathcal{E}_y \cJ =\cJ \mathcal{E}_y$ via the convolution in the planar algebra.
Hence the condition $\cJ\mathcal{E}_{\mathbb{L}}= \mathcal{E}_{\mathbb{L}} \cJ$ in Proposition \ref{prop:kmsbeinter} is ture in general for the inclusion $\bC\subset M_n(\bC)$.
\end{remark}

\begin{theorem}\label{prop:kmsmatrixderivation}
Suppose that the inclusion is $\bC\subset M_n(\bC)$ and $\{\Phi_t\}_{t\geq 0}$ is a bimodule KMS symmetric quantum Markov semigroups admitting standard lifting with intertwining property \eqref{eq:kmsintertwining0} such that
\begin{align*}
\mathbb{L}^{1/2} T \mathbb{L}^{-1/2} + \mathbb{L}^{-1/2} T \mathbb{L}^{1/2} \geq 2 \beta .
\end{align*}
Then 
\begin{align}\label{eq:be2}
  \Gamma_2(x) \geq \beta \Gamma(x) +\frac{1}{\sum_{\ell=1}^m\sum_{j,k=1}^{d_\ell} \left|c_{j,k}^{(\ell)}\right|^2}  \left| \cL(x)+\cL^*(x)+\frac{1}{2}\left\{x, 1*(\overline{\widehat{\cL}_0}-\widehat{\cL}_0)\right\}\right|^2,
\end{align}
where $\displaystyle 1+\mathbb{L}^{-1/2}\overline{\mathbb{L}} \mathbb{L}^{-1/2} =\bigoplus_{\ell=1}^{m} \left(c_{j,k}^{(\ell)} F_{j,k}^{(\ell)}\right)_{j,k=1}^{d_\ell}$ and $m$ is the number of the direct summand of $\cA$.
If $1*\overline{\widehat{\cL}_0}=1*\widehat{\cL}_0$, then 
\begin{align*}
  \Gamma_2(x) \geq \beta \Gamma(x) +\frac{1}{\sum_{\ell=1}^m\sum_{j,k=1}^{d_\ell} \left|c_{j,k}^{(\ell)}\right|^2}  \left| \cL(x)+\cL^*(x)\right|^2.
\end{align*}
\end{theorem}
\begin{proof}
We assume that $F_{j,k}^{(\ell)}$ are rank-one firstly.
In this case, we have that $F_{j,k}^{(\ell)}= \vcenter{\hbox{\begin{tikzpicture}[scale=0.65]
    \begin{scope}[shift={(0,1.5)}]
    \draw [blue] (-0.5, 0.8)--(-0.5, 0) .. controls +(0, -0.6) and +(0,-0.6).. (0.5, 0)--(0.5, 0.8);    
\begin{scope}[shift={(0.5, 0.3)}]
\draw [fill=white] (-0.35, -0.35) rectangle (0.35, 0.35);
\node at (0, 0) {\tiny $v_k^{(\ell)}$};
\end{scope}
    \end{scope}
\draw [blue] (-0.5, -0.8)--(-0.5, 0) .. controls +(0, 0.6) and +(0,0.6).. (0.5, 0)--(0.5, -0.8);
\begin{scope}[shift={(0.5, -0.3)}]
\draw [fill=white] (-0.35, -0.35) rectangle (0.35, 0.35);
\node at (0, 0) {\tiny $v_j^{(\ell)*}$};
\end{scope}
\end{tikzpicture}}}$, where $\displaystyle \tau\left(v_j^{(\ell)*}v_k^{(\ell)}\right)=\delta_{j,k}\lambda^{1/2}$.
Let $\partial_{j,k}^{(\ell)}: \cM \to \cM_1$ defined by $\partial_{j,k}^{(\ell)} x =[x, \fF^{-1}(F_{j,k}^{(\ell)})]$.
Let $\partial_{(\ell), k}: \cM \to \cM$ defined by $\partial_{(\ell), k} x=[x, v_{k}^{(\ell)}]$ for all $x\in \cM$.
We have that 
\begin{align}\label{eq:kms01}
\widetilde{\partial}_{t,s}^{(\ell)} \partial_{k,j} x=\lambda^{-1/4}( \partial_{(\ell), s} \partial_{(\ell), j} x ) \otimes \overline{v_t^{(\ell)*}} \otimes 1 \otimes  \overline{v_k^{(\ell)*}}.
\end{align}

Suppose that $\displaystyle \mathbb{L}^{1/2} =\bigoplus_{\ell=1}^m \left(r_{j,k}^{(\ell)} F_{j,k}^{(\ell)}\right)_{j,k=1}^{d_\ell}$ and $\displaystyle \mathbb{L}=\bigoplus_{\ell=1}^m \left(w_{j,k}^{(\ell)} F_{j,k}^{(\ell)}\right)_{j,k=1}^{d_\ell}$.
This implies that 
\begin{align*}
\mathcal{E}_{\mathbb{L}^{1/2}} (\nabla_0 x)
=\left( \sum_{j=1}^{d_\ell} r_{k,j}^{(\ell)} \partial_{k,j}^{(\ell)} x \right)_{k, \ell} = : (x_{k}^{(\ell)})_{k,\ell} .
\end{align*}
By substituting it into the last item in Equation \eqref{eq:kmsintertwining0}, we see that 
\begin{align*}
\bE_{\cM}\widetilde{\Gamma}(\mathcal{E}_{\mathbb{L}^{1/2}}(\nabla_0 x))
=& \frac{\lambda^{-1}}{2}\sum_{\ell=1}^m \sum_{k=1}^{d_\ell}\bE_{\cM} \left(\left(\widetilde{\partial} x_{k}^{(\ell)}\right)^* \left(\widetilde{\partial} x_{k}^{(\ell)}\right) \right)\\
=&\frac{\lambda^{-1}}{2} \sum_{\ell, \ell'=1}^m \sum_{t, k=1}^{d_\ell}\bE_{\cM} \left(\left(\widetilde{\partial}_{t,t}^{(\ell')} x_{k}^{(\ell)}\right)^* \left( \sum_{s=1}^{d_\ell} w_{t,s}^{(\ell')}\widetilde{\partial}_{t,s}^{(\ell')} x_{k}^{(\ell)}\right) \right),
\end{align*}
where the second equality follows from Lemma \ref{lem:dualsum}.

By Equation \eqref{eq:kms01}, we obtain that 
\begin{align*}
\sum_{s=1}^{d_\ell} w_{t,s}^{(\ell')}\widetilde{\partial}_{t,s}^{(\ell')} x_{k}^{(\ell)}
=& \lambda^{-1/4}\sum_{j,s=1}^{d_\ell} r_{k,j}^{(\ell)} w_{t,s}^{(\ell')}\widetilde{\partial}_{t,s}^{(\ell')} \partial_{k,j}^{(\ell)} x \\
=& \lambda^{-1/4} \sum_{j,s=1}^{d_\ell} r_{k,j}^{(\ell)} w_{t,s}^{(\ell')}( \partial_{(\ell'), s} \partial_{(\ell), j} x) \otimes \overline{v_{t}^{(\ell')*}}\otimes 1 \otimes  \overline{v_{k}^{(\ell)*}}, \\
\widetilde{\partial}_{t,t}^{(\ell')} x_{k}^{(\ell)}
=& \lambda^{-1/4} \sum_{j=1}^{d_\ell} r_{k,j}^{(\ell)} ( \partial_{(\ell'), t} \partial_{(\ell), j} x) \otimes \overline{v_{t}^{(\ell')*}}\otimes 1 \otimes  \overline{v_{k}^{(\ell)*}} .
\end{align*}
Hence
\begin{align*}
\bE_{\cM}\widetilde{\Gamma}(\mathcal{E}_{\mathbb{L}^{1/2}}(\nabla_0 x))
=&\frac{\lambda^{-1/2}}{2}\sum_{\ell, \ell'=1}^m  \sum_{j',j, k, t, s=1}^{d_\ell}\overline{r_{k,j}^{(\ell)} }( \partial_{(\ell'), t} \partial_{(\ell), j} x)^*r_{k,j'}^{(\ell)} w_{t,s}^{(\ell')}( \partial_{(\ell'), s} \partial_{(\ell), j'} x). 
\end{align*}

Let $\displaystyle \partial_{(\ell), k}^{(0)} =\sum_{j'=1}^{d_\ell} r_{k, j'}^{(\ell)} \partial_{(\ell), j'}.$
Then 
\begin{align*}
\bE_{\cM}\widetilde{\Gamma}(\mathcal{E}_{\mathbb{L}^{1/2}}(\nabla_0 x))
=& \frac{\lambda^{-1/2}}{2} \sum_{\ell, \ell'=1}^m  \sum_{k, t=1}^{d_\ell}  ( \partial_{(\ell'), t}^{(0)} \partial_{(\ell), k}^{(0)} x)^*( \partial_{(\ell'), t}^{(0)} \partial_{(\ell), k}^{(0)} x) \\
\geq & \frac{\lambda^{-1/2}}{2} \sum_{\ell=1}^m  \sum_{k, t=1}^{d_\ell}  ( \partial_{(\ell^*), t}^{(0)} \partial_{(\ell), k}^{(0)} x)^*( \partial_{(\ell^*), t}^{(0)} \partial_{(\ell), k}^{(0)} x). 
\end{align*}

Let $\displaystyle \mathbb{L}+\overline{\mathbb{L}} =\bigoplus_{\ell=1}^m \left(z_{j,k}^{(\ell)} F_{j,k}^{(\ell)} \right)_{j,k=1}^{d_\ell}$.
On the other hand, we have that 
\begin{align*}
\Div_0 \mathcal{E}_{\mathbb{L}+\overline{\mathbb{L}}}\nabla_0 x
=& \sum_{\ell=1}^m \sum_{j,k=1}^{d_\ell }z_{j,k}^{(\ell)} \partial_{j,j}^{(\ell)*} \partial_{j,k}^{(\ell)}x \\
=& \lambda^{1/2} \sum_{\ell^* \neq \ell, \ell=1}^m \sum_{j,k=1}^{d_\ell }z_{j,k}^{(\ell)} \partial_{(\ell^*),j} \partial_{(\ell), k}x 
+  \lambda^{1/2} \sum_{\ell^* = \ell, \ell=1}^m \sum_{j,k=1}^{d_\ell }z_{j,k}^{(\ell)} \partial_{(\ell),j^*} \partial_{(\ell), k}x  \\
=& \lambda^{1/2}  \sum_{\ell^*\neq \ell, \ell=1}^m \sum_{j,k=1}^{d_\ell } c_{j,k}^{(\ell)} \partial_{(\ell^*),j}^{(0)} \partial_{(\ell), k}^{(0)}x 
+  \lambda^{1/2}  \sum_{\ell^*=\ell, \ell=1}^m \sum_{j,k=1}^{d_\ell } c_{j,k}^{(\ell)} \partial_{(\ell),j^*}^{(0)} \partial_{(\ell), k}^{(0)}x. 
\end{align*}
Hence
\begin{align*}
\frac{\lambda^{1/2}}{2}\bE_{\cM}\widetilde{\Gamma}(\mathcal{E}_{\mathbb{L}^{1/2}}(\nabla_0 x))
\geq & \frac{1}{4} \sum_{\ell=1}^m  \sum_{k, t=1}^{d_\ell}  ( \partial_{(\ell^*), t}^{(0)} \partial_{(\ell), k}^{(0)} x)^*( \partial_{(\ell^*), t}^{(0)} \partial_{(\ell), k}^{(0)} x) \\
\geq & \frac{\lambda^{-1}}{4}\frac{1}{\sum_{j,k, \ell} |c_{j,k}^{(\ell)}|^2} \left(\Div_0 \mathcal{E}_{\mathbb{L}+\overline{\mathbb{L}}} \nabla_0 x\right)^* (\Div_0 \mathcal{E}_{\mathbb{L}+\overline{\mathbb{L}}} \nabla_0 x)\\
= & \frac{1}{\sum_{j,k, \ell} |c_{j,k}^{(\ell)}|^2} \left| \cL(x)+\cL^*(x)+\frac{1}{2}\{x, 1*(\overline{\widehat{\cL}_0}-\widehat{\cL}_0)\}\right|^2,
\end{align*}
where the last equality follows from Lemma \ref{lem:duallapsum}.
When $F_{j,k}^{(\ell)}$ is not rank-one, a similar computation applies and the theorem is still true.
\end{proof}

\section{Irreducible Inclusion}

\subsection{$\mathbb{Z}_3$-symmetry}
In this section, we shall obtain the modified log-Sobolev inequality for bimodule semigroups with $\bZ_3$-symmetry.
The inclusion is $\cM^G\subset \cM$, where $G=\bZ_3$ is a finite group outerly actiing on a II$_1$ factor $\cM$.
Recall that 
\begin{align*}
\vcenter{\hbox{\scalebox{0.8}{
        \begin{tikzpicture}[scale=1.2]
           \draw [blue]  (0.2, 0.9)--(0.2, -1.2);
           \draw [blue] (-0.2, 0.9)--(-0.2, -0.3).. controls +(0, -0.3) and +(0, -0.3) .. (-0.7, -0.3)--(-0.7, 0.9);
           \draw [fill=white] (-0.4, -0.3) rectangle (0.4, 0.3);
           \node at (0, 0) {\tiny $p_g$};
           \begin{scope}[shift={(0, -1.2)}]
           \draw [blue]  (0.2, -0.9)--(0.2, 0.3);
              \draw [blue] (-0.2, -0.9)--(-0.2, 0.3).. controls +(0, 0.3) and +(0, 0.3) .. (-0.7, 0.3)--(-0.7, -0.9);
               \draw [fill=white] (-0.4, -0.3) rectangle (0.4, 0.3);
           \node at (0, 0) {\tiny $p_g$};
           \end{scope}
        \end{tikzpicture}}}}  
             = \lambda^{1/2}\vcenter{\hbox{\scalebox{0.8}{
        \begin{tikzpicture}[scale=1.2]
        \draw [blue] (-0.7, 0.9)--(-0.7, -0.9);
           \draw [blue]  (-0.2, 0.9)--(-0.2, -0.9) (0.2, 0.9)--(0.2, -0.9);
           \draw [fill=white] (-0.4, -0.3) rectangle (0.4, 0.3);
           \node at (0, 0) {\tiny $p_g$};
        \end{tikzpicture}}}}, \quad 
\vcenter{\hbox{\scalebox{0.8}{
        \begin{tikzpicture}[scale=1.2]
           \draw [blue] (0.2, 0.9)--(0.2, -0.9) ;
           \draw [blue] (-0.2, 0.3)..controls +(0, 0.45) and +(0, 0.45)..(-0.8, 0.3);
        \draw [blue] (-0.2, -0.3)..controls +(0, -0.45) and +(0, -0.45)..(-0.8, -0.3);
           \draw [fill=white] (-0.4, -0.3) rectangle (0.4, 0.3);
           \node at (0, 0) {\tiny $p_g$};
           \begin{scope}[shift={(-1, 0)}]
           \draw [blue] (-0.2, -0.9)--(-0.2, 0.9);
           \draw [fill=white] (-0.4, -0.3) rectangle (0.4, 0.3);
           \node at (0, 0) {\tiny $p_h$};               
           \end{scope}
        \end{tikzpicture}}}}
=\lambda^{1/2}\vcenter{\hbox{\scalebox{0.8}{
        \begin{tikzpicture}[scale=1.2]
           \draw [blue]  (-0.2, 0.9)--(-0.2, -0.9) (0.2, 0.9)--(0.2, -0.9);
           \draw [fill=white] (-0.4, -0.3) rectangle (0.4, 0.3);
           \node at (0, 0) {\tiny $p_{gh}$};
        \end{tikzpicture}}}}, \quad 
 \vcenter{\hbox{\scalebox{0.8}{
        \begin{tikzpicture}[scale=1.2]
           \draw [blue] (0.2, 0.9)--(0.2, -0.9) ;
           \draw [blue] (-0.2, -0.9)--(-0.2, 0.3)..controls +(0, 0.45) and +(0, 0.45)..(-0.8, 0.3)--(-0.8, -0.9);
           \draw [fill=white] (-0.4, -0.3) rectangle (0.4, 0.3);
           \node at (0, 0) {\tiny $p_g$};
           \begin{scope}[shift={(-1, 0)}]
           \draw [blue] (-0.2, 0.9)--(-0.2, -0.3)..controls +(0, -0.45) and +(0, -0.45)..(-0.6, -0.3)--(-0.6, 0.9);
           \draw [fill=white] (-0.4, -0.3) rectangle (0.4, 0.3);
           \node at (0, 0) {\tiny $ p_h$};               
           \end{scope}
        \end{tikzpicture}}}}
 =\vcenter{\hbox{\scalebox{0.8}{
        \begin{tikzpicture}[scale=1.2]
           \draw [blue]  (0.2, 0.9)--(0.2, -1.2);
           \draw [blue] (-0.2, 0.9)--(-0.2, -0.3).. controls +(0, -0.3) and +(0, -0.3) .. (-0.7, -0.3)--(-0.7, 0.9);
           \draw [fill=white] (-0.4, -0.3) rectangle (0.4, 0.3);
           \node at (0, 0) {\tiny $p_{hg}$};
           \begin{scope}[shift={(0, -1.2)}]
           \draw [blue]  (0.2, -0.9)--(0.2, 0.3);
              \draw [blue] (-0.2, -0.9)--(-0.2, 0.3).. controls +(0, 0.3) and +(0, 0.3) .. (-0.7, 0.3)--(-0.7, -0.9);
               \draw [fill=white] (-0.4, -0.3) rectangle (0.4, 0.3);
           \node at (0, 0) {\tiny $p_g$};
           \end{scope}
        \end{tikzpicture}}}}, 
\end{align*}
where $p_g$ is a minimal projection in $\cM'\cap \cM_2$ associated to the group element $g$ and $\displaystyle \lambda=\frac{1}{3}$.
Let 
\begin{align}\label{eq:z3lind}
    \widehat{\cL}_0 = \frac{\lambda^{-1/2}}{\mu+\mu^{-1}}\left(\mu \vcenter{\hbox{\scalebox{0.8}{
        \begin{tikzpicture}[scale=1.2]
           \draw [blue]  (-0.2, 0.9)--(-0.2, -0.9) (0.2, 0.9)--(0.2, -0.9);
           \draw [fill=white] (-0.4, -0.3) rectangle (0.4, 0.3);
           \node at (0, 0) {\tiny $p_{g}$};
        \end{tikzpicture}}}} 
        + \mu^{-1}\vcenter{\hbox{\scalebox{0.8}{
        \begin{tikzpicture}[scale=1.2]
           \draw [blue]  (-0.2, 0.9)--(-0.2, -0.9) (0.2, 0.9)--(0.2, -0.9);
           \draw [fill=white] (-0.4, -0.3) rectangle (0.4, 0.3);
           \node at (0, 0) {\tiny $p_{g^{-1}}$};
        \end{tikzpicture}}}}\right)
\end{align}
where $0< \mu\leq 1$.
We have that $1*\widehat{\cL}_0=1$ and the associated semigroup $\{\Phi_t\}_{t\geq 0}$ is relatively ergodic and bimodule GNS symmetric with respect to the bimodule modular operator $\widehat{\Delta}= e_2 +\mu^{2} p_g +\mu^{-2} p_{g^{-1}}$.
In fact, any bimodule GNS symmetric quantum Markov semigroup arises in this way.
Let 
\begin{small}
\begin{align*}
    \widehat{\cJ}_0=& 
    \frac{\lambda^{-1}\mu}{\mu+\mu^{-1}}\left( (1-\kappa)\vcenter{\hbox{\scalebox{0.8}{
        \begin{tikzpicture}[scale=1.2]
           \draw [blue]  (-0.2, 0.7)--(-0.2, -0.3).. controls +(0, -0.3) and +(0, -0.3).. (-0.6, -0.3)--(-0.6, 0.7) (0.2, 0.7)--(0.2, -0.3).. controls +(0, -0.3) and +(0, -0.3).. (0.6, -0.3)--(0.6, 0.7);
           \draw [fill=white] (-0.4, -0.3) rectangle (0.4, 0.3);
           \node at (0, 0) {\tiny $p_{g}$};
        \begin{scope}[shift={(0, -1.3)}]
             \draw [blue]  (-0.2, -0.7)--(-0.2, 0.3).. controls +(0, 0.3) and +(0, 0.3).. (-0.6, 0.3)--(-0.6, -0.7) (0.2, -0.7)--(0.2, 0.3).. controls +(0, 0.3) and +(0, 0.3).. (0.6, 0.3)--(0.6, -0.7);
           \draw [fill=white] (-0.4, -0.3) rectangle (0.4, 0.3);
           \node at (0, 0) {\tiny $p_{g}$};          
        \end{scope}
        \end{tikzpicture}}}} 
        + \lambda\kappa \vcenter{\hbox{\scalebox{0.8}{
        \begin{tikzpicture}[scale=1.2]
        \draw [blue] (-0.6, 0.9)--(-0.6, -0.9) (0.6, 0.9)--(0.6, -0.9);
           \draw [blue]  (-0.2, 0.9)--(-0.2, -0.9) (0.2, 0.9)--(0.2, -0.9);
           \draw [fill=white] (-0.4, -0.3) rectangle (0.4, 0.3);
           \node at (0, 0) {\tiny $p_g$};
        \end{tikzpicture}}}} \right)
         +\frac{\lambda^{-1}\mu^{-1}}{\mu+\mu^{-1}}  \left( (1-\mu^2\kappa)\vcenter{\hbox{\scalebox{0.8}{
        \begin{tikzpicture}[scale=1.2]
           \draw [blue]  (-0.2, 0.7)--(-0.2, -0.3).. controls +(0, -0.3) and +(0, -0.3).. (-0.6, -0.3)--(-0.6, 0.7) (0.2, 0.7)--(0.2, -0.3).. controls +(0, -0.3) and +(0, -0.3).. (0.6, -0.3)--(0.6, 0.7);
           \draw [fill=white] (-0.4, -0.3) rectangle (0.4, 0.3);
           \node at (0, 0) {\tiny $p_{g^{-1}}$};
        \begin{scope}[shift={(0, -1.3)}]
             \draw [blue]  (-0.2, -0.7)--(-0.2, 0.3).. controls +(0, 0.3) and +(0, 0.3).. (-0.6, 0.3)--(-0.6, -0.7) (0.2, -0.7)--(0.2, 0.3).. controls +(0, 0.3) and +(0, 0.3).. (0.6, 0.3)--(0.6, -0.7);
           \draw [fill=white] (-0.4, -0.3) rectangle (0.4, 0.3);
           \node at (0, 0) {\tiny $p_{g^{-1}}$};          
        \end{scope}
        \end{tikzpicture}}}} 
        +\lambda\mu^2\kappa \vcenter{\hbox{\scalebox{0.8}{
        \begin{tikzpicture}[scale=1.2]
        \draw [blue] (-0.6, 0.9)--(-0.6, -0.9) (0.6, 0.9)--(0.6, -0.9);
           \draw [blue]  (-0.2, 0.9)--(-0.2, -0.9) (0.2, 0.9)--(0.2, -0.9);
           \draw [fill=white] (-0.4, -0.3) rectangle (0.4, 0.3);
           \node at (0, 0) {\tiny $p_{g^{-1}}$};
        \end{tikzpicture}}}} \right),
\end{align*}
\end{small}where $0\leq \kappa\leq 1$.
By Proposition \ref{prop:semigroupext}, we have that $\cJ |_{\cM}=\cL$.
By evaluating the items in the intertwining property, we have that 
\begin{align*}
    \vcenter{\hbox{
\begin{tikzpicture}
\draw [blue] (-0.15, -0.8)--(-0.15, 0.8);
 \draw [blue] (0.15, -0.8)--(0.15, 0.3) .. controls +(0, 0.35) and +(0, 0.35).. (0.6, 0.3);
 \draw [fill=white] (-0.3, -0.3) rectangle (0.3, 0.3);
 \node at (0, 0) {\tiny $\widehat{\cL}_0$};
 \begin{scope}[shift={(0.75, 0)}]
  \draw [fill=white] (-0.3, -0.3) rectangle (0.3, 0.3);
 \node at (0, 0) {\tiny $p_g$};
 \draw [blue] (0.15, 0.3)--(0.15, 0.8) (-0.15, -0.3)--(-0.15, -0.8);
  \draw [blue] (0.15, -0.3) .. controls +(0, -0.35) and +(0, -0.35).. (0.6, -0.3)--(0.6, 0.8);
 \end{scope}
\end{tikzpicture}
 }}
 = \frac{\lambda^{-1/2}\mu}{\mu+\mu^{-1}} \vcenter{\hbox{\scalebox{0.8}{
        \begin{tikzpicture}[scale=1.2]
           \draw [blue]  (-0.2, 0.9)--(-0.2, -1.2);
           \draw [blue] (0.2, 0.9)--(0.2, -0.3).. controls +(0, -0.3) and +(0, -0.3) .. (0.7, -0.3)--(0.7, 0.9);
           \draw [fill=white] (-0.4, -0.3) rectangle (0.4, 0.3);
           \node at (0, 0) {\tiny $p_{g^2}$};
           \begin{scope}[shift={(0, -1.2)}]
           \draw [blue]  (-0.2, -0.9)--(-0.2, 0.3);
              \draw [blue] (0.2, -0.9)--(0.2, 0.3).. controls +(0, 0.3) and +(0, 0.3) .. (0.7, 0.3)--(0.7, -0.9);
               \draw [fill=white] (-0.4, -0.3) rectangle (0.4, 0.3);
           \node at (0, 0) {\tiny $p_g$};
           \end{scope}
        \end{tikzpicture}}}}  
        + \frac{\mu^{-1}}{\mu+\mu^{-1}} \vcenter{\hbox{
    \begin{tikzpicture}
 \draw [blue] (0.15, -0.8)--(0.15, 0.3) .. controls +(0, 0.35) and +(0, 0.35).. (0.6, 0.3)--(0.6, -0.8);
  \draw [blue] (-0.15, 0.8)--(-0.15, -0.8);
 \draw [fill=white] (-0.3, -0.3) rectangle (0.3, 0.3);
 \node at (0, 0) {\tiny $p_{g^{-1}}$};
  \draw [blue] (-0.7, 0.8) .. controls +(0, -0.25) and +(0, -0.25).. (-0.4, 0.8);
    \end{tikzpicture}
    }}, 
\end{align*}
\begin{align*}
    \vcenter{\hbox{
    \begin{tikzpicture}
    \draw [blue] (-0.15, -0.8)--(-0.15, 0.8);
 \draw [blue] (0.15, 0.8)--(0.15, -0.3) .. controls +(0, -0.35) and +(0, -0.35).. (0.6, -0.3);
 \draw [fill=white] (-0.3, -0.3) rectangle (0.3, 0.3);
 \node at (0, 0) {\tiny $\widehat{\cL}_0$};
 \begin{scope}[shift={(0.75, 0)}]
  \draw [fill=white] (-0.3, -0.3) rectangle (0.3, 0.3);
 \node at (0, 0) {\tiny $p_{g^{-1}}$};
 \draw [blue] (0.15, -0.3)--(0.15, -0.8) (-0.15, 0.3)--(-0.15, 0.8);
  \draw [blue] (0.15, 0.3) .. controls +(0, 0.35) and +(0, 0.35).. (0.6, 0.3)--(0.6, -0.8);
 \end{scope}
    \end{tikzpicture}
    }}
=\frac{\lambda^{-1/2}\mu^{-1}}{\mu+\mu^{-1}} \vcenter{\hbox{\scalebox{0.8}{
        \begin{tikzpicture}[scale=1.2]
           \draw [blue]  (-0.2, 0.9)--(-0.2, -1.2);
           \draw [blue] (0.2, 0.9)--(0.2, -0.3).. controls +(0, -0.3) and +(0, -0.3) .. (0.7, -0.3)--(0.7, 0.9);
           \draw [fill=white] (-0.4, -0.3) rectangle (0.4, 0.3);
           \node at (0, 0) {\tiny $p_{g^{-1}}$};
           \begin{scope}[shift={(0, -1.2)}]
           \draw [blue]  (-0.2, -0.9)--(-0.2, 0.3);
              \draw [blue] (0.2, -0.9)--(0.2, 0.3).. controls +(0, 0.3) and +(0, 0.3) .. (0.7, 0.3)--(0.7, -0.9);
               \draw [fill=white] (-0.4, -0.3) rectangle (0.4, 0.3);
           \node at (0, 0) {\tiny $p_{g^{-2}}$};
           \end{scope}
        \end{tikzpicture}}}}  
+ \frac{\mu}{\mu+\mu^{-1}}\vcenter{\hbox{
    \begin{tikzpicture}
 \draw [blue] (0.15, 0.8)--(0.15, -0.3) .. controls +(0, -0.35) and +(0, -0.35).. (0.6, -0.3)--(0.6, 0.8);
  \draw [blue] (-0.15, 0.8)--(-0.15, -0.8);
 \draw [fill=white] (-0.3, -0.3) rectangle (0.3, 0.3);
 \node at (0, 0) {\tiny $p_g$};
  \draw [blue] (-0.7, -0.8) .. controls +(0, 0.25) and +(0, 0.25).. (-0.4, -0.8);
    \end{tikzpicture}
    }} , 
\end{align*}
\begin{align*}
  \vcenter{\hbox{
    \begin{tikzpicture}
       \draw [blue] (-0.75, -0.3)--(-0.75, 0.6) (-0.45, -0.3)--(-0.45, 0.6) (-0.15, -0.8)--(-0.15, 0.6) (0.15, -0.8)--(0.15, 1.6) (0.45, -0.8)--(0.45, 1.6);
       \draw[fill=white] (-0.6, -0.3) rectangle (0.6, 0.3);
       \node at (0, 0) {\tiny $\widehat{\cJ}_0$};
       \begin{scope}[shift={(-0.6, 0.8)}]
          \draw[fill=white] (-0.3, -0.3) rectangle (0.3, 0.3); 
           \node at (0, 0) {\tiny $p_g$};
        \draw[blue] (0.15, 0.3).. controls+(0, 0.3) and +(0, 0.3).. (0.45, 0.3)--(0.45, -0.3);
        \draw[blue] (-0.15, 0.3)--(-0.15, 0.8);
        \draw [blue] (-0.15, -1.1) .. controls +(0, -0.25) and +(0, -0.25).. (0.15, -1.1);
       \end{scope}
    \end{tikzpicture}
    }}
    =
    \frac{\lambda^{-1/2}\mu (1-\kappa)}{\mu+\mu^{-1}} \vcenter{\hbox{\scalebox{0.8}{
        \begin{tikzpicture}[scale=1.2]
           \draw [blue]  (-0.2, 0.9)--(-0.2, -1.2);
           \draw [blue] (0.2, 0.9)--(0.2, -0.3).. controls +(0, -0.3) and +(0, -0.3) .. (0.7, -0.3)--(0.7, 0.9);
           \draw [fill=white] (-0.4, -0.3) rectangle (0.4, 0.3);
           \node at (0, 0) {\tiny $p_{g^2}$};
           \begin{scope}[shift={(0, -1.2)}]
           \draw [blue]  (-0.2, -0.9)--(-0.2, 0.3);
              \draw [blue] (0.2, -0.9)--(0.2, 0.3).. controls +(0, 0.3) and +(0, 0.3) .. (0.7, 0.3)--(0.7, -0.9);
               \draw [fill=white] (-0.4, -0.3) rectangle (0.4, 0.3);
           \node at (0, 0) {\tiny $p_g$};
           \end{scope}
        \end{tikzpicture}}}}  
        + \frac{\mu^{-1}(1-\mu^2\kappa)}{\mu+\mu^{-1}} \vcenter{\hbox{
    \begin{tikzpicture}
 \draw [blue] (0.15, -0.8)--(0.15, 0.3) .. controls +(0, 0.35) and +(0, 0.35).. (0.6, 0.3)--(0.6, -0.8);
  \draw [blue] (-0.15, 0.8)--(-0.15, -0.8);
 \draw [fill=white] (-0.3, -0.3) rectangle (0.3, 0.3);
 \node at (0, 0) {\tiny $p_{g^{-1}}$};
  \draw [blue] (-0.7, 0.8) .. controls +(0, -0.25) and +(0, -0.25).. (-0.4, 0.8);
    \end{tikzpicture}
    }} ,
\end{align*}
and
\begin{align*}
    \vcenter{\hbox{
    \begin{tikzpicture}
       \draw [blue]   (-0.15, 0.8)--(-0.15, 0.3) (0.15, -0.8)--(0.15, 0.8) (0.45, -0.8)--(0.45, 0.8);
        \draw[fill=white] (-0.6, -0.3) rectangle (0.6, 0.3);
       \node at (0, 0) {\tiny $\widehat{\cJ}_0$};
       \begin{scope}[shift={(-1.2, 0)}]
       \draw [blue] (0.15, 0.3) .. controls +(0, 0.6) and +(0, 0.6) .. (-0.75, 0.3)--(-0.75, -0.8);
          \draw[fill=white] (-0.3, -0.3) rectangle (0.3, 0.3); 
           \node at (0, 0) {\tiny $p_{g}$};
        \draw[blue] (0.15, -0.3).. controls+(0, -0.25) and +(0, -0.25).. (0.45, -0.3)--(0.45, 0.3).. controls +(0, 0.3) and +(0, 0.3) .. (0.75, 0.3);
        \draw[blue] (-0.15, 0.3).. controls+(0, 0.3) and +(0, 0.3).. (-0.45, 0.3)--(-0.45, -0.3).. controls +(0, -0.6) and +(0, -0.6) .. (1.05, -0.3);
        \draw [blue] (-0.15, -0.3).. controls +(0, -0.4) and +(0, -0.4) .. (0.75, -0.3);
       \end{scope}
    \end{tikzpicture}
    }}
    =\frac{\lambda^{-1/2}\mu^{-1} (1-\mu^2\kappa)}{\mu+\mu^{-1}} \vcenter{\hbox{\scalebox{0.8}{
        \begin{tikzpicture}[scale=1.2]
           \draw [blue]  (-0.2, 0.9)--(-0.2, -1.2);
           \draw [blue] (0.2, 0.9)--(0.2, -0.3).. controls +(0, -0.3) and +(0, -0.3) .. (0.7, -0.3)--(0.7, 0.9);
           \draw [fill=white] (-0.4, -0.3) rectangle (0.4, 0.3);
           \node at (0, 0) {\tiny $p_{g^{-1}}$};
           \begin{scope}[shift={(0, -1.2)}]
           \draw [blue]  (-0.2, -0.9)--(-0.2, 0.3);
              \draw [blue] (0.2, -0.9)--(0.2, 0.3).. controls +(0, 0.3) and +(0, 0.3) .. (0.7, 0.3)--(0.7, -0.9);
               \draw [fill=white] (-0.4, -0.3) rectangle (0.4, 0.3);
           \node at (0, 0) {\tiny $p_{g^{-2}}$};
           \end{scope}
        \end{tikzpicture}}}}  
+ \frac{\mu (1-\kappa)}{\mu+\mu^{-1}}\vcenter{\hbox{
    \begin{tikzpicture}
 \draw [blue] (0.15, 0.8)--(0.15, -0.3) .. controls +(0, -0.35) and +(0, -0.35).. (0.6, -0.3)--(0.6, 0.8);
  \draw [blue] (-0.15, 0.8)--(-0.15, -0.8);
 \draw [fill=white] (-0.3, -0.3) rectangle (0.3, 0.3);
 \node at (0, 0) {\tiny $p_g$};
  \draw [blue] (-0.7, -0.8) .. controls +(0, 0.25) and +(0, 0.25).. (-0.4, -0.8);
    \end{tikzpicture}
    }}  .
\end{align*}
Hence the difference is the following
\begin{align*}
   -   \frac{\lambda^{-1/2}\mu \kappa}{\mu+\mu^{-1}} \vcenter{\hbox{\scalebox{0.8}{
        \begin{tikzpicture}[scale=1.2]
           \draw [blue]  (-0.2, 0.9)--(-0.2, -1.2);
           \draw [blue] (0.2, 0.9)--(0.2, -0.3).. controls +(0, -0.3) and +(0, -0.3) .. (0.7, -0.3)--(0.7, 0.9);
           \draw [fill=white] (-0.4, -0.3) rectangle (0.4, 0.3);
           \node at (0, 0) {\tiny $p_{g^2}$};
           \begin{scope}[shift={(0, -1.2)}]
           \draw [blue]  (-0.2, -0.9)--(-0.2, 0.3);
              \draw [blue] (0.2, -0.9)--(0.2, 0.3).. controls +(0, 0.3) and +(0, 0.3) .. (0.7, 0.3)--(0.7, -0.9);
               \draw [fill=white] (-0.4, -0.3) rectangle (0.4, 0.3);
           \node at (0, 0) {\tiny $p_g$};
           \end{scope}
        \end{tikzpicture}}}}    
    + \frac{\lambda^{-1/2}\mu \kappa}{\mu+\mu^{-1}} \vcenter{\hbox{\scalebox{0.8}{
        \begin{tikzpicture}[scale=1.2]
           \draw [blue]  (-0.2, 0.9)--(-0.2, -1.2);
           \draw [blue] (0.2, 0.9)--(0.2, -0.3).. controls +(0, -0.3) and +(0, -0.3) .. (0.7, -0.3)--(0.7, 0.9);
           \draw [fill=white] (-0.4, -0.3) rectangle (0.4, 0.3);
           \node at (0, 0) {\tiny $p_{g^{-1}}$};
           \begin{scope}[shift={(0, -1.2)}]
           \draw [blue]  (-0.2, -0.9)--(-0.2, 0.3);
              \draw [blue] (0.2, -0.9)--(0.2, 0.3).. controls +(0, 0.3) and +(0, 0.3) .. (0.7, 0.3)--(0.7, -0.9);
               \draw [fill=white] (-0.4, -0.3) rectangle (0.4, 0.3);
           \node at (0, 0) {\tiny $p_{g^{-2}}$};
           \end{scope}
        \end{tikzpicture}}}} 
       - \frac{\mu\kappa}{\mu+\mu^{-1}} \vcenter{\hbox{
    \begin{tikzpicture}
 \draw [blue] (0.15, -0.8)--(0.15, 0.3) .. controls +(0, 0.35) and +(0, 0.35).. (0.6, 0.3)--(0.6, -0.8);
  \draw [blue] (-0.15, 0.8)--(-0.15, -0.8);
 \draw [fill=white] (-0.3, -0.3) rectangle (0.3, 0.3);
 \node at (0, 0) {\tiny $p_{g^{-1}}$};
  \draw [blue] (-0.7, 0.8) .. controls +(0, -0.25) and +(0, -0.25).. (-0.4, 0.8);
    \end{tikzpicture}
    }} 
    +\frac{\mu \kappa}{\mu+\mu^{-1}}\vcenter{\hbox{
    \begin{tikzpicture}
 \draw [blue] (0.15, 0.8)--(0.15, -0.3) .. controls +(0, -0.35) and +(0, -0.35).. (0.6, -0.3)--(0.6, 0.8);
  \draw [blue] (-0.15, 0.8)--(-0.15, -0.8);
 \draw [fill=white] (-0.3, -0.3) rectangle (0.3, 0.3);
 \node at (0, 0) {\tiny $p_g$};
  \draw [blue] (-0.7, -0.8) .. controls +(0, 0.25) and +(0, 0.25).. (-0.4, -0.8);
    \end{tikzpicture}
    }}  .
\end{align*}
By noting that $g^3=1$ and similar calculation for $p_{g^{-1}}$, we see that $\displaystyle \beta=\frac{\mu \kappa}{\mu+\mu^{-1}}$ and 
\begin{align}\label{eq:interz3}
\partial_j \cL - \cJ \partial_j =\frac{\mu^{-2j+3} \kappa}{\mu+\mu^{-1}} \partial_j, \quad j=1, 2.
\end{align}

\begin{theorem}
Suppose that $\cM^{\bZ_3}\subset \cM$ is the inclusion and $\{\Phi_t\}_{t\geq 0}$ is the quantum Markov semigroup arising from Equation \eqref{eq:z3lind}.
Then for any positive element $D\in \cM$, we have that 
\begin{align*}
H(\Phi_t^*(D)\| D_{\Delta}) -H(\bE_{\cN}(D) \| D_{\Delta}) \leq  e^{-2\frac{\mu t}{\mu+ \mu^{-1}}}  (H(D\| D_{\Delta}) -H( \bE_{\cN}(D)\| D_{\Delta})).
\end{align*}
\end{theorem}
\begin{proof}
It follows from the fact that $\displaystyle \lim_{t\to \infty}\widehat{\Phi}_t=\lambda^{1/2}$.
\end{proof}

\begin{proposition}
Suppose that $\cM^{\bZ_3}\subset \cM$ is the inclusion and $\{\Phi_t\}_{t\geq 0}$ is the quantum Markov semigroup arising from Equation \eqref{eq:z3lind}.
We have that    
\begin{align*}
\widehat{\Gamma}_2 \geq \beta \widehat{\Gamma},
\end{align*}
where
\begin{align}\label{eq:betaz3}
\beta =\frac{1}{4}\left( \mu^2+\mu^{-2} +5 -(\mu^2+\mu^{-2}-1)\sqrt{\frac{\mu^2+\mu^{-2}+14}{\mu^2+\mu^{-2}+2}}\right).
\end{align}
\end{proposition}

\begin{proof}
We have that 
\begin{align*}
2\lambda(\mu+\mu^{-1})^2\widehat{\Gamma}_2=& 
\lambda^{1/2}(2\mu^2+1 +\frac{\mu^{-2}}{2})\vcenter{\hbox{\scalebox{0.8}{
        \begin{tikzpicture}[scale=1.2]
        \draw [blue] (-0.7, 0.9)--(-0.7, -0.9);
           \draw [blue]  (-0.2, 0.9)--(-0.2, -0.9) (0.2, 0.9)--(0.2, -0.9);
           \draw [fill=white] (-0.4, -0.3) rectangle (0.4, 0.3);
           \node at (0, 0) {\tiny $p_g$};
        \end{tikzpicture}}}}
      +\lambda^{1/2}(2\mu^{-2}+1 +\frac{\mu^{2}}{2})\vcenter{\hbox{\scalebox{0.8}{
        \begin{tikzpicture}[scale=1.2]
        \draw [blue] (-0.7, 0.9)--(-0.7, -0.9);
           \draw [blue]  (-0.2, 0.9)--(-0.2, -0.9) (0.2, 0.9)--(0.2, -0.9);
           \draw [fill=white] (-0.4, -0.3) rectangle (0.4, 0.3);
           \node at (0, 0) {\tiny $p_{g^{-1}}$};
        \end{tikzpicture}}}}\\
&+(1-\mu^2-\mu^{-2})\left(\vcenter{\hbox{\scalebox{0.8}{
        \begin{tikzpicture}[scale=1.2]
           \draw [blue]  (0.2, 0.9)--(0.2, -1.2);
           \draw [blue] (-0.2, 0.9)--(-0.2, -0.3).. controls +(0, -0.3) and +(0, -0.3) .. (-0.7, -0.3)--(-0.7, 0.9);
           \draw [fill=white] (-0.4, -0.3) rectangle (0.4, 0.3);
           \node at (0, 0) {\tiny $p_{g}$};
           \begin{scope}[shift={(0, -1.2)}]
           \draw [blue]  (0.2, -0.9)--(0.2, 0.3);
              \draw [blue] (-0.2, -0.9)--(-0.2, 0.3).. controls +(0, 0.3) and +(0, 0.3) .. (-0.7, 0.3)--(-0.7, -0.9);
               \draw [fill=white] (-0.4, -0.3) rectangle (0.4, 0.3);
           \node at (0, 0) {\tiny $p_{g^{-1}}$};
           \end{scope}
        \end{tikzpicture}}}} + \vcenter{\hbox{\scalebox{0.8}{
        \begin{tikzpicture}[scale=1.2]
           \draw [blue]  (0.2, 0.9)--(0.2, -1.2);
           \draw [blue] (-0.2, 0.9)--(-0.2, -0.3).. controls +(0, -0.3) and +(0, -0.3) .. (-0.7, -0.3)--(-0.7, 0.9);
           \draw [fill=white] (-0.4, -0.3) rectangle (0.4, 0.3);
           \node at (0, 0) {\tiny $p_{g^{-1}}$};
           \begin{scope}[shift={(0, -1.2)}]
           \draw [blue]  (0.2, -0.9)--(0.2, 0.3);
              \draw [blue] (-0.2, -0.9)--(-0.2, 0.3).. controls +(0, 0.3) and +(0, 0.3) .. (-0.7, 0.3)--(-0.7, -0.9);
               \draw [fill=white] (-0.4, -0.3) rectangle (0.4, 0.3);
           \node at (0, 0) {\tiny $p_g$};
           \end{scope}
        \end{tikzpicture}}}}  \right) +\frac{\lambda}{2}(\mu^2+\mu^{-2}+8)\vcenter{\hbox{\scalebox{0.8}{
        \begin{tikzpicture}[scale=1.2]
     \draw [blue] (0, -0.5).. controls +(0, 0.4) and +(0, 0.4).. (0.4, -0.5);
     \draw [blue] (0, 0.5).. controls +(0, -0.4) and +(0, -0.4).. (0.4, 0.5);       
     \begin{scope}[shift={(-0.4,0)}]
     \draw [blue] (0,-0.5)--(0, 0.5);
           \end{scope}
        \end{tikzpicture}}}}\\
& +\lambda^{1/2}(-\mu^2+\frac{1}{2}\mu^{-2}-2)\left(   \vcenter{\hbox{\scalebox{0.8}{
        \begin{tikzpicture}[scale=1.2]
         \draw [blue] (0.3, 0.9).. controls +(0, -0.3) and +(0, -0.3).. (0.6, 0.9);
           \draw [blue]  (0.2, 0.9)--(0.2, -0.9);
           \draw [blue](-0.2, -0.9)--(-0.2, 0.3).. controls +(0, 0.3) and +(0, 0.3).. (-0.6, 0.3)--(-0.6, -0.9);
           \draw [fill=white] (-0.4, -0.3) rectangle (0.5, 0.3);
           \node at (0, 0) {\tiny $ p_g$};
        \end{tikzpicture}}}}
        +     \vcenter{\hbox{\scalebox{0.8}{
        \begin{tikzpicture}[scale=1.2]
         \draw [blue] (0.3, -0.9).. controls +(0, 0.3) and +(0, 0.3).. (0.6, -0.9);
           \draw [blue]  (0.2, 0.9)--(0.2, -0.9);
           \draw [blue](-0.2, 0.9)--(-0.2, -0.3).. controls +(0, -0.3) and +(0, -0.3).. (-0.6, -0.3)--(-0.6, 0.9);
           \draw [fill=white] (-0.4, -0.3) rectangle (0.4, 0.3);
           \node at (0, 0) {\tiny $p_g$};
        \end{tikzpicture}}}}\right)\\
  & +\lambda^{1/2}(-\mu^{-2}+\frac{1}{2}\mu^{2}-2)\left(   \vcenter{\hbox{\scalebox{0.8}{
        \begin{tikzpicture}[scale=1.2]
         \draw [blue] (0.3, 0.9).. controls +(0, -0.3) and +(0, -0.3).. (0.6, 0.9);
           \draw [blue]  (0.2, 0.9)--(0.2, -0.9);
           \draw [blue](-0.2, -0.9)--(-0.2, 0.3).. controls +(0, 0.3) and +(0, 0.3).. (-0.6, 0.3)--(-0.6, -0.9);
           \draw [fill=white] (-0.4, -0.3) rectangle (0.5, 0.3);
           \node at (0, 0) {\tiny $ p_{g^{-1}}$};
        \end{tikzpicture}}}}
        +     \vcenter{\hbox{\scalebox{0.8}{
        \begin{tikzpicture}[scale=1.2]
         \draw [blue] (0.3, -0.9).. controls +(0, 0.3) and +(0, 0.3).. (0.6, -0.9);
           \draw [blue]  (0.2, 0.9)--(0.2, -0.9);
           \draw [blue](-0.2, 0.9)--(-0.2, -0.3).. controls +(0, -0.3) and +(0, -0.3).. (-0.6, -0.3)--(-0.6, 0.9);
           \draw [fill=white] (-0.4, -0.3) rectangle (0.4, 0.3);
           \node at (0, 0) {\tiny $p_{g^{-1}}$};
        \end{tikzpicture}}}}\right),
\end{align*}
and 
\begin{align*}
2\lambda(\mu+\mu^{-1})^2\widehat{\Gamma}=& \lambda^{1/2}(\mu^2+1)\vcenter{\hbox{\scalebox{0.8}{
        \begin{tikzpicture}[scale=1.2]
        \draw [blue] (-0.7, 0.9)--(-0.7, -0.9);
           \draw [blue]  (-0.2, 0.9)--(-0.2, -0.9) (0.2, 0.9)--(0.2, -0.9);
           \draw [fill=white] (-0.4, -0.3) rectangle (0.4, 0.3);
           \node at (0, 0) {\tiny $p_g$};
        \end{tikzpicture}}}}
      +\lambda^{1/2}(\mu^{-2}+1)\vcenter{\hbox{\scalebox{0.8}{
        \begin{tikzpicture}[scale=1.2]
        \draw [blue] (-0.7, 0.9)--(-0.7, -0.9);
           \draw [blue]  (-0.2, 0.9)--(-0.2, -0.9) (0.2, 0.9)--(0.2, -0.9);
           \draw [fill=white] (-0.4, -0.3) rectangle (0.4, 0.3);
           \node at (0, 0) {\tiny $p_{g^{-1}}$};
        \end{tikzpicture}}}}     + \lambda(\mu^2+\mu^{-2}+2)\vcenter{\hbox{\scalebox{0.8}{
        \begin{tikzpicture}[scale=1.2]
     \draw [blue] (0, -0.5).. controls +(0, 0.4) and +(0, 0.4).. (0.4, -0.5);
     \draw [blue] (0, 0.5).. controls +(0, -0.4) and +(0, -0.4).. (0.4, 0.5);       
     \begin{scope}[shift={(-0.4,0)}]
     \draw [blue] (0,-0.5)--(0, 0.5);
           \end{scope}
        \end{tikzpicture}}}}\\
        & +\lambda^{1/2}(-\mu^2-1)\left(   \vcenter{\hbox{\scalebox{0.8}{
        \begin{tikzpicture}[scale=1.2]
         \draw [blue] (0.3, 0.9).. controls +(0, -0.3) and +(0, -0.3).. (0.6, 0.9);
           \draw [blue]  (0.2, 0.9)--(0.2, -0.9);
           \draw [blue](-0.2, -0.9)--(-0.2, 0.3).. controls +(0, 0.3) and +(0, 0.3).. (-0.6, 0.3)--(-0.6, -0.9);
           \draw [fill=white] (-0.4, -0.3) rectangle (0.5, 0.3);
           \node at (0, 0) {\tiny $ p_g$};
        \end{tikzpicture}}}}
        +     \vcenter{\hbox{\scalebox{0.8}{
        \begin{tikzpicture}[scale=1.2]
         \draw [blue] (0.3, -0.9).. controls +(0, 0.3) and +(0, 0.3).. (0.6, -0.9);
           \draw [blue]  (0.2, 0.9)--(0.2, -0.9);
           \draw [blue](-0.2, 0.9)--(-0.2, -0.3).. controls +(0, -0.3) and +(0, -0.3).. (-0.6, -0.3)--(-0.6, 0.9);
           \draw [fill=white] (-0.4, -0.3) rectangle (0.4, 0.3);
           \node at (0, 0) {\tiny $p_g$};
        \end{tikzpicture}}}}\right)
        \\
  & +\lambda^{1/2}(-\mu^{-2}-1)\left(   \vcenter{\hbox{\scalebox{0.8}{
        \begin{tikzpicture}[scale=1.2]
         \draw [blue] (0.3, 0.9).. controls +(0, -0.3) and +(0, -0.3).. (0.6, 0.9);
           \draw [blue]  (0.2, 0.9)--(0.2, -0.9);
           \draw [blue](-0.2, -0.9)--(-0.2, 0.3).. controls +(0, 0.3) and +(0, 0.3).. (-0.6, 0.3)--(-0.6, -0.9);
           \draw [fill=white] (-0.4, -0.3) rectangle (0.5, 0.3);
           \node at (0, 0) {\tiny $ p_{g^{-1}}$};
        \end{tikzpicture}}}}
        +     \vcenter{\hbox{\scalebox{0.8}{
        \begin{tikzpicture}[scale=1.2]
         \draw [blue] (0.3, -0.9).. controls +(0, 0.3) and +(0, 0.3).. (0.6, -0.9);
           \draw [blue]  (0.2, 0.9)--(0.2, -0.9);
           \draw [blue](-0.2, 0.9)--(-0.2, -0.3).. controls +(0, -0.3) and +(0, -0.3).. (-0.6, -0.3)--(-0.6, 0.9);
           \draw [fill=white] (-0.4, -0.3) rectangle (0.4, 0.3);
           \node at (0, 0) {\tiny $p_{g^{-1}}$};
        \end{tikzpicture}}}}\right).
\end{align*}
The Bakry-\'{E}mery estimate is equivalent to 
\begin{align*}
& \begin{pmatrix}
2\mu^2+1+\frac{1}{2}\mu^{-2} & 1-\mu^2-\mu^{-2} & -\mu^2+\frac{1}{2}\mu^{-2}-2\\
1-\mu^2-\mu^{-2} & 2\mu^{-2}+1+\frac{1}{2}\mu^{2} & -\mu^{-2}+\frac{1}{2}\mu^2-2 \\
-\mu^2+\frac{1}{2}\mu^{-2}-2 &  -\mu^{-2}+\frac{1}{2}\mu^2-2 & \frac{1}{2}(\mu^2+\mu^{-2})+4
\end{pmatrix}\\
\geq &\beta
\begin{pmatrix}
\mu^2+1 & 0 &  -\mu^2-1 \\
0 & \mu^{-2}+1 & -\mu^{-2}-1 \\
-\mu^2-1 & -\mu^{-2}-1 & \mu^2+\mu^{-2}+2
\end{pmatrix}.
\end{align*}
Solving it for optimal value $\beta$, we obtain Equation \eqref{eq:betaz3}.
Note that we have that $\displaystyle \beta \leq \frac{5}{4}$ in general.
\end{proof}

\begin{remark}
By Equation \eqref{eq:interz3}, we only have that $\displaystyle \widehat{\Gamma}_2 \geq \frac{\mu }{\mu+\mu^{-1}} \widehat{\Gamma}.$
It is necessary to find a better extension of the semigroup to improve the constant $\beta$.
\end{remark}

\subsection{$\bZ_n$-symmetry}

Suppose that $n$ is odd, i.e. $n=2m+1$.
Let 
\begin{align}\label{eq:lapzn}
    \widehat{\cL}_0 = \frac{\lambda^{-1/2}}{\sum_{j=1}^m \mu^j + \mu^{-j}}\sum_{j=1}^m\left(\mu^j \vcenter{\hbox{\scalebox{0.8}{
        \begin{tikzpicture}[scale=1.2]
           \draw [blue]  (-0.2, 0.9)--(-0.2, -0.9) (0.2, 0.9)--(0.2, -0.9);
           \draw [fill=white] (-0.4, -0.3) rectangle (0.4, 0.3);
           \node at (0, 0) {\tiny $p_{g^j}$};
        \end{tikzpicture}}}} 
        + \mu^{-j}\vcenter{\hbox{\scalebox{0.8}{
        \begin{tikzpicture}[scale=1.2]
           \draw [blue]  (-0.2, 0.9)--(-0.2, -0.9) (0.2, 0.9)--(0.2, -0.9);
           \draw [fill=white] (-0.4, -0.3) rectangle (0.4, 0.3);
           \node at (0, 0) {\tiny $p_{g^{-j}}$};
        \end{tikzpicture}}}}\right)
\end{align}
and 
\begin{small}
\begin{align*}
    \widehat{\cJ}_0=& 
    \frac{\lambda^{-1}}{\sum_{j=1}^m \mu^j + \mu^{-j}}\sum_{j=1}^m\left( (1-\kappa)\mu^j \vcenter{\hbox{\scalebox{0.8}{
        \begin{tikzpicture}[scale=1.2]
           \draw [blue]  (-0.2, 0.7)--(-0.2, -0.3).. controls +(0, -0.3) and +(0, -0.3).. (-0.6, -0.3)--(-0.6, 0.7) (0.2, 0.7)--(0.2, -0.3).. controls +(0, -0.3) and +(0, -0.3).. (0.6, -0.3)--(0.6, 0.7);
           \draw [fill=white] (-0.4, -0.3) rectangle (0.4, 0.3);
           \node at (0, 0) {\tiny $p_{g^j}$};
        \begin{scope}[shift={(0, -1.3)}]
             \draw [blue]  (-0.2, -0.7)--(-0.2, 0.3).. controls +(0, 0.3) and +(0, 0.3).. (-0.6, 0.3)--(-0.6, -0.7) (0.2, -0.7)--(0.2, 0.3).. controls +(0, 0.3) and +(0, 0.3).. (0.6, 0.3)--(0.6, -0.7);
           \draw [fill=white] (-0.4, -0.3) rectangle (0.4, 0.3);
           \node at (0, 0) {\tiny $p_{g^j}$};          
        \end{scope}
        \end{tikzpicture}}}}
         +\lambda \kappa \mu^{j} \vcenter{\hbox{\scalebox{0.8}{
        \begin{tikzpicture}[scale=1.2]
        \draw [blue] (-0.6, 0.9)--(-0.6, -0.9) (0.6, 0.9)--(0.6, -0.9);
           \draw [blue]  (-0.2, 0.9)--(-0.2, -0.9) (0.2, 0.9)--(0.2, -0.9);
           \draw [fill=white] (-0.4, -0.3) rectangle (0.4, 0.3);
           \node at (0, 0) {\tiny $p_{g^j}$};
        \end{tikzpicture}}}} \right) \\
   &  +\frac{\lambda^{-1} }{\sum_{j=1}^m \mu^j + \mu^{-j}}\sum_{j=1}^m  \left( (\mu^{-j}-\mu^{j}\kappa)\vcenter{\hbox{\scalebox{0.8}{
        \begin{tikzpicture}[scale=1.2]
           \draw [blue]  (-0.2, 0.7)--(-0.2, -0.3).. controls +(0, -0.3) and +(0, -0.3).. (-0.6, -0.3)--(-0.6, 0.7) (0.2, 0.7)--(0.2, -0.3).. controls +(0, -0.3) and +(0, -0.3).. (0.6, -0.3)--(0.6, 0.7);
           \draw [fill=white] (-0.4, -0.3) rectangle (0.4, 0.3);
           \node at (0, 0) {\tiny $p_{g^{-j}}$};
        \begin{scope}[shift={(0, -1.3)}]
             \draw [blue]  (-0.2, -0.7)--(-0.2, 0.3).. controls +(0, 0.3) and +(0, 0.3).. (-0.6, 0.3)--(-0.6, -0.7) (0.2, -0.7)--(0.2, 0.3).. controls +(0, 0.3) and +(0, 0.3).. (0.6, 0.3)--(0.6, -0.7);
           \draw [fill=white] (-0.4, -0.3) rectangle (0.4, 0.3);
           \node at (0, 0) {\tiny $p_{g^{-j}}$};          
        \end{scope}
        \end{tikzpicture}}}}
        +\lambda\mu^{j}\kappa \vcenter{\hbox{\scalebox{0.8}{
        \begin{tikzpicture}[scale=1.2]
        \draw [blue] (-0.6, 0.9)--(-0.6, -0.9) (0.6, 0.9)--(0.6, -0.9);
           \draw [blue]  (-0.2, 0.9)--(-0.2, -0.9) (0.2, 0.9)--(0.2, -0.9);
           \draw [fill=white] (-0.4, -0.3) rectangle (0.4, 0.3);
           \node at (0, 0) {\tiny $p_{g^{-j}}$};
        \end{tikzpicture}}}} \right).
\end{align*}
\end{small}
By a similar computation, we have that $\displaystyle \beta=\frac{\mu^m\kappa}{\sum_{j=1}^m \mu^j + \mu^{-j}}$.

Suppose that $n$ is even, i.e. $n=2m+2$, $m \geq 1$.
Let 
\begin{align*}
    \widehat{\cL}_0 = \frac{\lambda^{-1/2}}{1+ \sum_{j=1}^m \mu^j + \mu^{-j}}\sum_{j=1}^m\left(\mu^j \vcenter{\hbox{\scalebox{0.8}{
        \begin{tikzpicture}[scale=1.2]
           \draw [blue]  (-0.2, 0.9)--(-0.2, -0.9) (0.2, 0.9)--(0.2, -0.9);
           \draw [fill=white] (-0.4, -0.3) rectangle (0.4, 0.3);
           \node at (0, 0) {\tiny $p_{g^j}$};
        \end{tikzpicture}}}} 
        + \mu^{-j}\vcenter{\hbox{\scalebox{0.8}{
        \begin{tikzpicture}[scale=1.2]
           \draw [blue]  (-0.2, 0.9)--(-0.2, -0.9) (0.2, 0.9)--(0.2, -0.9);
           \draw [fill=white] (-0.4, -0.3) rectangle (0.4, 0.3);
           \node at (0, 0) {\tiny $p_{g^{-j}}$};
        \end{tikzpicture}}}}\right)
        + \frac{\lambda^{-1/2}}{1+\sum_{j=1}^m \mu^j + \mu^{-j}}\vcenter{\hbox{\scalebox{0.8}{
        \begin{tikzpicture}[scale=1.2]
           \draw [blue]  (-0.2, 0.9)--(-0.2, -0.9) (0.2, 0.9)--(0.2, -0.9);
           \draw [fill=white] (-0.4, -0.3) rectangle (0.4, 0.3);
           \node at (0, 0) {\tiny $p_{g^{m+1}}$};
        \end{tikzpicture}}}}
\end{align*}
and
\begin{align*}
    \widehat{\cJ}_0=& 
    \frac{\lambda^{-1}}{1+ \sum_{j=1}^m \mu^j + \mu^{-j}}\sum_{j=1}^m\left(\mu^j (1-\kappa)\vcenter{\hbox{\scalebox{0.8}{
        \begin{tikzpicture}[scale=1.2]
           \draw [blue]  (-0.2, 0.7)--(-0.2, -0.3).. controls +(0, -0.3) and +(0, -0.3).. (-0.6, -0.3)--(-0.6, 0.7) (0.2, 0.7)--(0.2, -0.3).. controls +(0, -0.3) and +(0, -0.3).. (0.6, -0.3)--(0.6, 0.7);
           \draw [fill=white] (-0.4, -0.3) rectangle (0.4, 0.3);
           \node at (0, 0) {\tiny $p_{g^j}$};
        \begin{scope}[shift={(0, -1.3)}]
             \draw [blue]  (-0.2, -0.7)--(-0.2, 0.3).. controls +(0, 0.3) and +(0, 0.3).. (-0.6, 0.3)--(-0.6, -0.7) (0.2, -0.7)--(0.2, 0.3).. controls +(0, 0.3) and +(0, 0.3).. (0.6, 0.3)--(0.6, -0.7);
           \draw [fill=white] (-0.4, -0.3) rectangle (0.4, 0.3);
           \node at (0, 0) {\tiny $p_{g^j}$};          
        \end{scope}
        \end{tikzpicture}}}}
         +\lambda \mu^j \kappa \vcenter{\hbox{\scalebox{0.8}{
        \begin{tikzpicture}[scale=1.2]
        \draw [blue] (-0.6, 0.9)--(-0.6, -0.9) (0.6, 0.9)--(0.6, -0.9);
           \draw [blue]  (-0.2, 0.9)--(-0.2, -0.9) (0.2, 0.9)--(0.2, -0.9);
           \draw [fill=white] (-0.4, -0.3) rectangle (0.4, 0.3);
           \node at (0, 0) {\tiny $p_{g^j}$};
        \end{tikzpicture}}}} \right) \\
      &  +\frac{\lambda^{-1}}{1+ \sum_{j=1}^m \mu^j + \mu^{-j} }\sum_{j=1}^m  \left( (\mu^{-j}-\mu^j\kappa)\vcenter{\hbox{\scalebox{0.8}{
        \begin{tikzpicture}[scale=1.2]
           \draw [blue]  (-0.2, 0.7)--(-0.2, -0.3).. controls +(0, -0.3) and +(0, -0.3).. (-0.6, -0.3)--(-0.6, 0.7) (0.2, 0.7)--(0.2, -0.3).. controls +(0, -0.3) and +(0, -0.3).. (0.6, -0.3)--(0.6, 0.7);
           \draw [fill=white] (-0.4, -0.3) rectangle (0.4, 0.3);
           \node at (0, 0) {\tiny $p_{g^{-j}}$};
        \begin{scope}[shift={(0, -1.3)}]
             \draw [blue]  (-0.2, -0.7)--(-0.2, 0.3).. controls +(0, 0.3) and +(0, 0.3).. (-0.6, 0.3)--(-0.6, -0.7) (0.2, -0.7)--(0.2, 0.3).. controls +(0, 0.3) and +(0, 0.3).. (0.6, 0.3)--(0.6, -0.7);
           \draw [fill=white] (-0.4, -0.3) rectangle (0.4, 0.3);
           \node at (0, 0) {\tiny $p_{g^{-j}}$};          
        \end{scope}
        \end{tikzpicture}}}}
        +\lambda\mu^j\kappa \vcenter{\hbox{\scalebox{0.8}{
        \begin{tikzpicture}[scale=1.2]
        \draw [blue] (-0.6, 0.9)--(-0.6, -0.9) (0.6, 0.9)--(0.6, -0.9);
           \draw [blue]  (-0.2, 0.9)--(-0.2, -0.9) (0.2, 0.9)--(0.2, -0.9);
           \draw [fill=white] (-0.4, -0.3) rectangle (0.4, 0.3);
           \node at (0, 0) {\tiny $p_{g^{-j}}$};
        \end{tikzpicture}}}} \right)\\
     &   +\frac{\lambda^{-1}(1-\mu\kappa)}{1+ \sum_{j=1}^m \mu^j + \mu^{-j}}\vcenter{\hbox{\scalebox{0.8}{
        \begin{tikzpicture}[scale=1.2]
           \draw [blue]  (-0.2, 0.7)--(-0.2, -0.3).. controls +(0, -0.3) and +(0, -0.3).. (-0.6, -0.3)--(-0.6, 0.7) (0.2, 0.7)--(0.2, -0.3).. controls +(0, -0.3) and +(0, -0.3).. (0.6, -0.3)--(0.6, 0.7);
           \draw [fill=white] (-0.4, -0.3) rectangle (0.4, 0.3);
           \node at (0, 0) {\tiny $p_{g^{m+1}}$};
        \begin{scope}[shift={(0, -1.3)}]
             \draw [blue]  (-0.2, -0.7)--(-0.2, 0.3).. controls +(0, 0.3) and +(0, 0.3).. (-0.6, 0.3)--(-0.6, -0.7) (0.2, -0.7)--(0.2, 0.3).. controls +(0, 0.3) and +(0, 0.3).. (0.6, 0.3)--(0.6, -0.7);
           \draw [fill=white] (-0.4, -0.3) rectangle (0.4, 0.3);
           \node at (0, 0) {\tiny $p_{g^{m+1}}$};          
        \end{scope}
        \end{tikzpicture}}}}
        +\frac{ \mu\kappa}{1+ \sum_{j=1}^m \mu^j + \mu^{-j}} \vcenter{\hbox{\scalebox{0.8}{
        \begin{tikzpicture}[scale=1.2]
        \draw [blue] (-0.6, 0.9)--(-0.6, -0.9) (0.6, 0.9)--(0.6, -0.9);
           \draw [blue]  (-0.2, 0.9)--(-0.2, -0.9) (0.2, 0.9)--(0.2, -0.9);
           \draw [fill=white] (-0.4, -0.3) rectangle (0.4, 0.3);
           \node at (0, 0) {\tiny $p_{g^{m+1}}$};
        \end{tikzpicture}}}}.
\end{align*}
By a similar computation, we have that $\displaystyle \beta=\frac{\mu^m \kappa}{1+ \sum_{j=1}^m \mu^j + \mu^{-j}}$.

\begin{theorem}
Suppose that $\cM^{\bZ_n}\subset \cM$ is the inclusion and $\{\Phi_t\}_{t\geq 0}$ is the quantum Markov semigroup arising from Equation \eqref{eq:lapzn}.
Then for any positive element $D\in \cM$, we have that 
\begin{align*}
H(\Phi_t^*(D)\| D_{\Delta}) -H( \bE_{\cN}(D) \| D_{\Delta}) \leq  e^{-2\beta t}  (H(D\| D_{\Delta}) -H(  \bE_{\cN}(D)\| D_{\Delta})),
\end{align*}
where $\displaystyle \beta=\frac{\mu^m \kappa}{1+ \sum_{j=1}^m \mu^j + \mu^{-j}}$ if $n=2m+2$ and $\displaystyle \beta=\frac{\mu^m \kappa}{ \sum_{j=1}^m \mu^j + \mu^{-j}}$ if $n=2m+1$. 
\end{theorem}
\begin{proof}
It follows from the fact that $\displaystyle \lim_{t\to \infty}\widehat{\Phi}_t=\lambda^{1/2}$.
\end{proof}

\subsection{Unitary Fusion Categories}

Suppose that $\mathcal{C}$ is a unitary fusion category and $\mathscr{P}$ is the associated planar algebra (See \cite{WuZha25} for details).
Let $\Irr$ be the set of all simple objects in $\mathcal{C}$ and $\mathbbm{1}$ is the unit object.
We have that $\displaystyle \lambda^{-1}=\sum_{X\in \Irr} d_X^2$, where $d_X$ is the quantum dimension of $X$.
For any simple objects $X, Y \in \mathcal{C}$, we have that 
\begin{align*}
\vcenter{\hbox{\scalebox{0.8}{
        \begin{tikzpicture}[scale=1.2]
           \draw [blue] (0.2, 0.9)--(0.2, -0.9) ;
           \draw [blue] (-0.2, 0.3)..controls +(0, 0.45) and +(0, 0.45)..(-0.8, 0.3);
        \draw [blue] (-0.2, -0.3)..controls +(0, -0.45) and +(0, -0.45)..(-0.8, -0.3);
           \draw [fill=white] (-0.4, -0.3) rectangle (0.4, 0.3);
           \node at (0, 0) {\tiny $1_Y$};
           \begin{scope}[shift={(-1, 0)}]
           \draw [blue] (-0.2, -0.9)--(-0.2, 0.9);
           \draw [fill=white] (-0.4, -0.3) rectangle (0.4, 0.3);
           \node at (0, 0) {\tiny $1_X$};               
           \end{scope}
        \end{tikzpicture}}}}
=\lambda^{1/2}\sum_{W} N_{X, Y}^W \frac{d_Xd_Y}{d_W}
\vcenter{\hbox{\scalebox{0.8}{
        \begin{tikzpicture}[scale=1.2]
           \draw [blue]  (-0.2, 0.9)--(-0.2, -0.9) (0.2, 0.9)--(0.2, -0.9);
           \draw [fill=white] (-0.4, -0.3) rectangle (0.4, 0.3);
           \node at (0, 0) {\tiny $1_{W}$};
        \end{tikzpicture}}}},
\end{align*}
where $1_{X}$ is the minimal projection in $\cM'\cap \cM_2$ associated to the simple object $X$ and $N_{X, Y}^W$ is the fusion coefficients.
The Jones-Wenzl-Liu formula is the following:
\begin{align*}
   \vcenter{\hbox{ \begin{tikzpicture} 
  \draw [blue] (-0.2, -0.5)--(-0.2, 0.5) (0.2, -0.5)--(0.2, 0.5) (0.6, -0.5)--(0.6, 0.5);
        \draw [fill=white] (-0.4, -0.25) rectangle (0.4, 0.25);
        \node at (0, 0) {\tiny $1_X$};
  \end{tikzpicture}}}
  =  \frac{\lambda^{-1/2} }{d_X^2}  \vcenter{\hbox{\scalebox{0.8}{
        \begin{tikzpicture}[scale=1.2]
           \draw [blue]  (-0.2, 0.9)--(-0.2, -1.2);
           \draw [blue] (0.2, 0.9)--(0.2, -0.3).. controls +(0, -0.3) and +(0, -0.3) .. (0.7, -0.3)--(0.7, 0.9);
           \draw [fill=white] (-0.4, -0.3) rectangle (0.4, 0.3);
           \node at (0, 0) {\tiny $1_X$};
           \begin{scope}[shift={(0, -1.2)}]
           \draw [blue]  (-0.2, -0.9)--(-0.2, 0.3);
              \draw [blue] (0.2, -0.9)--(0.2, 0.3).. controls +(0, 0.3) and +(0, 0.3) .. (0.7, 0.3)--(0.7, -0.9);
               \draw [fill=white] (-0.4, -0.3) rectangle (0.4, 0.3);
           \node at (0, 0) {\tiny $1_X$};
           \end{scope}
        \end{tikzpicture}}}}  
    +    \vcenter{\hbox{ \begin{tikzpicture} 
  \draw [blue] (-0.2, -1.25)--(-0.2, 0.5) (0.2,  -1.25)--(0.2, 0.5) (0.6,  -1.25)--(0.6, 0.5);
        \draw [fill=white] (-0.4, -0.25) rectangle (0.4, 0.25);
        \node at (0, 0) {\tiny $1_X$};
        \begin{scope}[shift={(0, -0.75)}]
        \draw [fill=white] (-0.4, -0.25) rectangle (0.8, 0.25);
        \node at ( 0.2,0) {\tiny $s_{3, -}$};
        \end{scope}
  \end{tikzpicture}}},
 \quad 
    \vcenter{\hbox{ \begin{tikzpicture} 
  \draw [blue] (-0.2, -0.5)--(-0.2, 0.5) (0.2, -0.5)--(0.2, 0.5) (-0.6, -0.5)--(-0.6, 0.5);
        \draw [fill=white] (-0.4, -0.25) rectangle (0.4, 0.25);
        \node at (0, 0) {\tiny $1_X$};
  \end{tikzpicture}}}
  =  \frac{\lambda^{-1/2} }{d_X^2}  \vcenter{\hbox{\scalebox{0.8}{
        \begin{tikzpicture}[scale=1.2]
           \draw [blue]  (0.2, 0.9)--(0.2, -1.2);
           \draw [blue] (-0.2, 0.9)--(-0.2, -0.3).. controls +(0, -0.3) and +(0, -0.3) .. (-0.7, -0.3)--(-0.7, 0.9);
           \draw [fill=white] (-0.4, -0.3) rectangle (0.4, 0.3);
           \node at (0, 0) {\tiny $1_X$};
           \begin{scope}[shift={(0, -1.2)}]
           \draw [blue]  (0.2, -0.9)--(0.2, 0.3);
              \draw [blue] (-0.2, -0.9)--(-0.2, 0.3).. controls +(0, 0.3) and +(0, 0.3) .. (-0.7, 0.3)--(-0.7, -0.9);
               \draw [fill=white] (-0.4, -0.3) rectangle (0.4, 0.3);
           \node at (0, 0) {\tiny $1_X$};
           \end{scope}
        \end{tikzpicture}}}}  
    +    \vcenter{\hbox{ \begin{tikzpicture} 
  \draw [blue] (-0.2, -1.25)--(-0.2, 0.5) (0.2,  -1.25)--(0.2, 0.5) (0.6,  -1.25)--(0.6, 0.5);
        \draw [fill=white] (0, -0.25) rectangle (0.8, 0.25);
        \node at (0.4, 0) {\tiny $1_X$};
        \begin{scope}[shift={(0, -0.75)}]
        \draw [fill=white] (-0.4, -0.25) rectangle (0.8, 0.25);
        \node at ( 0.2,0) {\tiny $s_{3, -}$};
        \end{scope}
  \end{tikzpicture}}},
\end{align*}
where $s_{3, \pm}$ is the Jones-Wenzl projection in $\mathscr{P}_{3, \pm}$.
We have that $\displaystyle
 \vcenter{\hbox{ \begin{tikzpicture} 
  \draw [blue] (-0.2, -1.25)--(-0.2, 0.5) (0.2,  -1.25)--(0.2, 0.5) (0.6,  -0.5).. controls +(0, 0.3) and +(0, 0.3).. (1, -0.5)--(1, -1).. controls +(0, -0.3) and +(0, -0.3).. (0.6, -1);
        \draw [fill=white] (-0.4, -0.25) rectangle (0.4, 0.25);
        \node at (0, 0) {\tiny $1_X$};
        \begin{scope}[shift={(0, -0.75)}]
        \draw [fill=white] (-0.4, -0.25) rectangle (0.8, 0.25);
        \node at ( 0.2,0) {\tiny $s_{3, -}$};
        \end{scope}
  \end{tikzpicture}}}=\lambda^{-1/2}\frac{d_X^2-1}{d_X^2}\vcenter{\hbox{\scalebox{0.8}{
        \begin{tikzpicture}[scale=1.2]
           \draw [blue]  (-0.2, 0.9)--(-0.2, -0.9) (0.2, 0.9)--(0.2, -0.9);
           \draw [fill=white] (-0.4, -0.3) rectangle (0.4, 0.3);
           \node at (0, 0) {\tiny $1_{X}$};
        \end{tikzpicture}}}}.$
For any simple objects $X, Y\in \Irr$, let $ \displaystyle p_{X, Y}= \frac{\lambda^{-1/2} }{d_Xd_Y}  \vcenter{\hbox{\scalebox{0.8}{
        \begin{tikzpicture}[scale=1.2]
           \draw [blue]  (-0.2, 0.9)--(-0.2, -1.2);
           \draw [blue] (0.2, 0.9)--(0.2, -0.3).. controls +(0, -0.3) and +(0, -0.3) .. (0.7, -0.3)--(0.7, 0.9);
           \draw [fill=white] (-0.4, -0.3) rectangle (0.4, 0.3);
           \node at (0, 0) {\tiny $1_X$};
           \begin{scope}[shift={(0, -1.2)}]
           \draw [blue]  (-0.2, -0.9)--(-0.2, 0.3);
              \draw [blue] (0.2, -0.9)--(0.2, 0.3).. controls +(0, 0.3) and +(0, 0.3) .. (0.7, 0.3)--(0.7, -0.9);
               \draw [fill=white] (-0.4, -0.3) rectangle (0.4, 0.3);
           \node at (0, 0) {\tiny $1_Y$};
           \end{scope}
        \end{tikzpicture}}}}$.
 Then $p_{X, Y}$  is a partial isometry in $\cM'\cap \cM_3$.
 Furthermore, for any simple object $X\neq \mathbbm{1}$,
 \begin{align*}
    \vcenter{\hbox{
\begin{tikzpicture}
\draw [blue] (-0.15, -0.8)--(-0.15, 0.8);
 \draw [blue] (0.15, -0.8)--(0.15, 0.3) .. controls +(0, 0.35) and +(0, 0.35).. (0.6, 0.3);
 \draw [fill=white] (-0.3, -0.3) rectangle (0.3, 0.3);
 \node at (0, 0) {\tiny $1_Y$};
 \begin{scope}[shift={(0.75, 0)}]
  \draw [fill=white] (-0.3, -0.3) rectangle (0.3, 0.3);
 \node at (0, 0) {\tiny $1_X$};
 \draw [blue] (0.15, 0.3)--(0.15, 0.8) (-0.15, -0.3)--(-0.15, -0.8);
  \draw [blue] (0.15, -0.3) .. controls +(0, -0.35) and +(0, -0.35).. (0.6, -0.3)--(0.6, 0.8);
 \end{scope}
\end{tikzpicture}
 }}
 =&  \sum_{W\in \Irr} \frac{N_{Y, X}^W  d_X}{d_Y d_W} \vcenter{\hbox{\scalebox{0.8}{
        \begin{tikzpicture}[scale=1.2]
           \draw [blue]  (-0.2, 0.9)--(-0.2, -1.2);
           \draw [blue] (0.2, 0.9)--(0.2, -0.3).. controls +(0, -0.3) and +(0, -0.3) .. (0.7, -0.3)--(0.7, 0.9);
           \draw [fill=white] (-0.4, -0.3) rectangle (0.4, 0.3);
           \node at (0, 0) {\tiny $1_W$};
           \begin{scope}[shift={(0, -1.2)}]
           \draw [blue]  (-0.2, -0.9)--(-0.2, 0.3);
              \draw [blue] (0.2, -0.9)--(0.2, 0.3).. controls +(0, 0.3) and +(0, 0.3) .. (0.7, 0.3)--(0.7, -0.9);
               \draw [fill=white] (-0.4, -0.3) rectangle (0.4, 0.3);
           \node at (0, 0) {\tiny $1_Y$};
           \end{scope}
        \end{tikzpicture}}}}  
        -  \vcenter{\hbox{ \begin{tikzpicture} 
 \draw [blue] (-0.15, -1.4)--(-0.15, 0.8);
 \draw [blue] (0.15, -1.4)--(0.15, 0.3) .. controls +(0, 0.35) and +(0, 0.35).. (0.6, 0.3);
 \draw [fill=white] (-0.3, -0.3) rectangle (0.3, 0.3);
 \node at (0, 0) {\tiny $1_Y$};
 \begin{scope}[shift={(0.75, 0)}]
  \draw [fill=white] (-0.3, -0.3) rectangle (0.3, 0.3);
 \node at (0, 0) {\tiny $1_X$};
 \draw [blue] (0.15, 0.3)--(0.15, 0.8) (-0.15, -0.3)--(-0.15, -1.4);
  \draw [blue] (0.15, -0.3) .. controls +(0, -0.35) and +(0, -0.35).. (0.6, -0.3)--(0.6, 0.8);
 \end{scope}       
        \begin{scope}[shift={(0, -0.75)}]
        \draw [fill=white] (-0.4, -0.25) rectangle (0.8, 0.25);
        \node at ( 0.2,0) {\tiny $s_{3, -}$};
        \end{scope}
  \end{tikzpicture}}} \\
  = & \sum_{W\in \Irr}  N_{Y, X}^W \lambda^{1/2} d_X  p_{W, Y} 
  - \vcenter{\hbox{ \begin{tikzpicture} 
 \draw [blue] (-0.15, -1.4)--(-0.15, 0.8);
 \draw [blue] (0.15, -1.4)--(0.15, 0.3) .. controls +(0, 0.35) and +(0, 0.35).. (0.6, 0.3);
 \draw [fill=white] (-0.3, -0.3) rectangle (0.3, 0.3);
 \node at (0, 0) {\tiny $1_Y$};
 \begin{scope}[shift={(0.75, 0)}]
  \draw [fill=white] (-0.3, -0.3) rectangle (0.3, 0.3);
 \node at (0, 0) {\tiny $1_X$};
 \draw [blue] (0.15, 0.3)--(0.15, 0.8) (-0.15, -0.3)--(-0.15, -1.4);
  \draw [blue] (0.15, -0.3) .. controls +(0, -0.35) and +(0, -0.35).. (0.6, -0.3)--(0.6, 0.8);
 \end{scope}       
        \begin{scope}[shift={(0, -0.75)}]
        \draw [fill=white] (-0.4, -0.25) rectangle (0.8, 0.25);
        \node at ( 0.2,0) {\tiny $s_{3, -}$};
        \end{scope}
  \end{tikzpicture}}}.
\end{align*}

Now we shall define the Laplacian part of the quantum Markov semigroup.
Let 
 \begin{align}\label{eq:lapufc}
\lambda^{1/2}\left(\sum_{Y\neq \mathbbm{1}}d_Y\omega_Y\right) \widehat{\cL}_0 
=\sum_{\mathbbm{1}\neq Y\in \Irr} \frac{\omega_Y}{d_Y}
         \vcenter{\hbox{\scalebox{0.8}{
        \begin{tikzpicture}[scale=1.2]
           \draw [blue]  (-0.2, 0.9)--(-0.2, -0.9) (0.2, 0.9)--(0.2, -0.9);
           \draw [fill=white] (-0.4, -0.3) rectangle (0.4, 0.3);
           \node at (0, 0) {\tiny $1_{Y}$};
        \end{tikzpicture}}}},
 \end{align}
 where $\omega_Y \geq 0$.
Now we shall try to obtain a sufficient condition for the intertwining properties of the semigroup arising from Equation \eqref{eq:lapufc}.
 We have that 
 \begin{align*}
&   \lambda^{1/2}\left(\sum_{Y\neq \mathbbm{1}}d_Y\omega_Y\right) \left(-\vcenter{\hbox{
\begin{tikzpicture}
\draw [blue] (-0.15, -0.8)--(-0.15, 0.8);
 \draw [blue] (0.15, -0.8)--(0.15, 0.3) .. controls +(0, 0.35) and +(0, 0.35).. (0.6, 0.3);
 \draw [fill=white] (-0.3, -0.3) rectangle (0.3, 0.3);
 \node at (0, 0) {\tiny $\widehat{\cL}_0$};
 \begin{scope}[shift={(0.75, 0)}]
  \draw [fill=white] (-0.3, -0.3) rectangle (0.3, 0.3);
 \node at (0, 0) {\tiny $1_X$};
 \draw [blue] (0.15, 0.3)--(0.15, 0.8) (-0.15, -0.3)--(-0.15, -0.8);
  \draw [blue] (0.15, -0.3) .. controls +(0, -0.35) and +(0, -0.35).. (0.6, -0.3)--(0.6, 0.8);
 \end{scope}
\end{tikzpicture}
 }}+
 \vcenter{\hbox{
    \begin{tikzpicture}
    \draw [blue] (-0.15, -0.8)--(-0.15, 0.8);
 \draw [blue] (0.15, 0.8)--(0.15, -0.3) .. controls +(0, -0.35) and +(0, -0.35).. (0.6, -0.3);
 \draw [fill=white] (-0.3, -0.3) rectangle (0.3, 0.3);
 \node at (0, 0) {\tiny $\widehat{\cL}_0$};
 \begin{scope}[shift={(0.75, 0)}]
  \draw [fill=white] (-0.3, -0.3) rectangle (0.3, 0.3);
 \node at (0, 0) {\tiny $1_{X^*}$};
 \draw [blue] (0.15, -0.3)--(0.15, -0.8) (-0.15, 0.3)--(-0.15, 0.8);
  \draw [blue] (0.15, 0.3) .. controls +(0, 0.35) and +(0, 0.35).. (0.6, 0.3)--(0.6, -0.8);
 \end{scope}
    \end{tikzpicture}
    }} 
    + \vcenter{\hbox{ \begin{tikzpicture} 
 \draw [blue] (-0.15, -1.4)--(-0.15, 0.8);
 \draw [blue] (0.15, -1.4)--(0.15, 0.3) .. controls +(0, 0.35) and +(0, 0.35).. (0.6, 0.3);
 \draw [fill=white] (-0.3, -0.3) rectangle (0.3, 0.3);
 \node at (0, 0) {\tiny $1_Y$};
 \begin{scope}[shift={(0.75, 0)}]
  \draw [fill=white] (-0.3, -0.3) rectangle (0.3, 0.3);
 \node at (0, 0) {\tiny $1_X$};
 \draw [blue] (0.15, 0.3)--(0.15, 0.8) (-0.15, -0.3)--(-0.15, -1.4);
  \draw [blue] (0.15, -0.3) .. controls +(0, -0.35) and +(0, -0.35).. (0.6, -0.3)--(0.6, 0.8);
 \end{scope}       
        \begin{scope}[shift={(0, -0.75)}]
        \draw [fill=white] (-0.4, -0.25) rectangle (0.8, 0.25);
        \node at ( 0.2,0) {\tiny $s_{3, -}$};
        \end{scope}
  \end{tikzpicture}}}
-
   \vcenter{\hbox{ \begin{tikzpicture} 
     \draw [blue] (-0.15, -1.4)--(-0.15, 0.8);
 \draw [blue] (0.15, 0.8)--(0.15, -0.3) .. controls +(0, -0.35) and +(0, -0.35).. (0.6, -0.3);
 \draw [fill=white] (-0.3, -0.3) rectangle (0.3, 0.3);
 \node at (0, 0) {\tiny $1_Y$};
 \begin{scope}[shift={(0.75, 0)}]
  \draw [fill=white] (-0.3, -0.3) rectangle (0.3, 0.3);
 \node at (0, 0) {\tiny $1_{X^*}$};
 \draw [blue] (0.15, -0.3)--(0.15, -1.4) (-0.15, 0.3)--(-0.15, 0.8);
  \draw [blue] (0.15, 0.3) .. controls +(0, 0.35) and +(0, 0.35).. (0.6, 0.3)--(0.6, -1.4);
 \end{scope}  
        \begin{scope}[shift={(0.1, -0.85)}]
        \draw [fill=white] (-0.4, -0.25) rectangle (1.4, 0.25);
        \node at ( 0.2,0) {\tiny $s_{3, -}$};
        \end{scope}
  \end{tikzpicture}}}
  \right)\\
   = & - \lambda^{1/2}d_X \sum_{Y\neq \mathbbm{1}, W\in \Irr} \frac{\omega_Y}{d_Y} N_{Y, X}^W  p_{W, Y}+ \lambda^{1/2}d_X \sum_{Y\neq \mathbbm{1}, W\in \Irr}  \frac{\omega_Y}{d_Y} N_{Y, X^*}^W  p_{Y, W} \\
   =&  - \lambda^{1/2} d_X \sum_{Y\neq \mathbbm{1}, W\neq \mathbbm{1}\in \Irr} \frac{\omega_Y}{d_Y} N_{Y, X}^W  p_{W, Y}+ \lambda^{1/2}d_X \sum_{Y\neq \mathbbm{1}, W\neq \mathbbm{1}\in \Irr} \frac{ \omega_Y}{d_Y} N_{Y, X^*}^W  p_{Y, W} \\
   & -\lambda^{1/2} d_X\omega_{X^*} p_{\mathbbm{1}, X^*}+\lambda^{1/2} \omega_Xd_X p_{X, \mathbbm{1}}\\
   =&  \lambda^{1/2}d_X \sum_{Y\neq \mathbbm{1}, W\neq \mathbbm{1}\in \Irr} \left(\frac{\omega_W}{d_W}-\frac{\omega_Y}{d_Y}\right) N_{Y, X}^W  p_{W, Y} -\lambda^{1/2} \omega_{X^*} p_{\mathbbm{1}, X^*}+\lambda^{1/2} \omega_X  p_{X, \mathbbm{1}}\\
      =&  \lambda^{1/2} d_X \sum_{Y\neq \mathbbm{1}, W\neq \mathbbm{1}\in \Irr} \left(\frac{\omega_W}{d_W}-\frac{\omega_Y}{d_Y}\right) N_{X, W^*}^{Y^*}  p_{W, Y} -\lambda^{1/2} \omega_{X^*} p_{\mathbbm{1}, X^*}+\lambda^{1/2} \omega_X  p_{X, \mathbbm{1}}.
 \end{align*}
 
When $\vcenter{\hbox{ \begin{tikzpicture} 
 \draw [blue] (-0.15, -1.4)--(-0.15, 0.8);
 \draw [blue] (0.15, -1.4)--(0.15, 0.3) .. controls +(0, 0.35) and +(0, 0.35).. (0.6, 0.3);
 \draw [fill=white] (-0.3, -0.3) rectangle (0.3, 0.3);
 \node at (0, 0) {\tiny $1_Y$};
 \begin{scope}[shift={(0.75, 0)}]
  \draw [fill=white] (-0.3, -0.3) rectangle (0.3, 0.3);
 \node at (0, 0) {\tiny $1_X$};
 \draw [blue] (0.15, 0.3)--(0.15, 0.8) (-0.15, -0.3)--(-0.15, -1.4);
  \draw [blue] (0.15, -0.3) .. controls +(0, -0.35) and +(0, -0.35).. (0.6, -0.3)--(0.6, 0.8);
 \end{scope}       
        \begin{scope}[shift={(0, -0.75)}]
        \draw [fill=white] (-0.4, -0.25) rectangle (0.8, 0.25);
        \node at ( 0.2,0) {\tiny $s_{3, -}$};
        \end{scope}
  \end{tikzpicture}}}
=
   \vcenter{\hbox{ \begin{tikzpicture} 
     \draw [blue] (-0.15, -1.4)--(-0.15, 0.8);
 \draw [blue] (0.15, 0.8)--(0.15, -0.3) .. controls +(0, -0.35) and +(0, -0.35).. (0.6, -0.3);
 \draw [fill=white] (-0.3, -0.3) rectangle (0.3, 0.3);
 \node at (0, 0) {\tiny $1_Y$};
 \begin{scope}[shift={(0.75, 0)}]
  \draw [fill=white] (-0.3, -0.3) rectangle (0.3, 0.3);
 \node at (0, 0) {\tiny $1_{X^*}$};
 \draw [blue] (0.15, -0.3)--(0.15, -1.4) (-0.15, 0.3)--(-0.15, 0.8);
  \draw [blue] (0.15, 0.3) .. controls +(0, 0.35) and +(0, 0.35).. (0.6, 0.3)--(0.6, -1.4);
 \end{scope}  
        \begin{scope}[shift={(0.1, -0.85)}]
        \draw [fill=white] (-0.4, -0.25) rectangle (1.4, 0.25);
        \node at ( 0.2,0) {\tiny $s_{3, -}$};
        \end{scope}
  \end{tikzpicture}}}$,  $\vcenter{\hbox{
    \begin{tikzpicture}
       \draw [blue] (0.45, -0.6)--(0.45, 0.6) (-0.15, -0.6)--(-0.15, 0.6) (0.15, -0.6)--(0.15, 0.6) (-0.45, -0.6)--(-0.45, 0.6);
       \draw[fill=white] (-0.6, -0.3) rectangle (0.6, 0.3);
       \node at (0, 0) {\tiny $\widehat{\cJ}_0$};
    \end{tikzpicture}
    }}
    = \lambda^{1/2} \vcenter{\hbox{
    \begin{tikzpicture}
       \draw [blue]   (-0.15, -0.6)--(-0.15, 0.6) (0.15, -0.6)--(0.15, 0.6) (0.45, -0.6)--(0.45, 0.6) (-0.45, -0.6)--(-0.45, 0.6);
       \draw[fill=white] (-0.35, -0.3) rectangle (0.35, 0.3);
       \node at (0, 0) {\tiny $\widehat{\cL}_0$};
    \end{tikzpicture}
    }}$ and $\omega_Y=1$ for all $\mathbbm{1}\neq Y\in \Irr$.
Then the associated tracially symmetric semigroup satisfies the intertwining property with $\displaystyle \beta=\left(\sum_{Y\neq \mathbbm{1}} d_Y\right)^{-1}$. 
If $\mathscr{P}$ is a commute relation planar algebra of type AA, then the condition $\vcenter{\hbox{ \begin{tikzpicture} 
 \draw [blue] (-0.15, -1.4)--(-0.15, 0.8);
 \draw [blue] (0.15, -1.4)--(0.15, 0.3) .. controls +(0, 0.35) and +(0, 0.35).. (0.6, 0.3);
 \draw [fill=white] (-0.3, -0.3) rectangle (0.3, 0.3);
 \node at (0, 0) {\tiny $1_Y$};
 \begin{scope}[shift={(0.75, 0)}]
  \draw [fill=white] (-0.3, -0.3) rectangle (0.3, 0.3);
 \node at (0, 0) {\tiny $1_X$};
 \draw [blue] (0.15, 0.3)--(0.15, 0.8) (-0.15, -0.3)--(-0.15, -1.4);
  \draw [blue] (0.15, -0.3) .. controls +(0, -0.35) and +(0, -0.35).. (0.6, -0.3)--(0.6, 0.8);
 \end{scope}       
        \begin{scope}[shift={(0, -0.75)}]
        \draw [fill=white] (-0.4, -0.25) rectangle (0.8, 0.25);
        \node at ( 0.2,0) {\tiny $s_{3, -}$};
        \end{scope}
  \end{tikzpicture}}}
=
   \vcenter{\hbox{ \begin{tikzpicture} 
     \draw [blue] (-0.15, -1.4)--(-0.15, 0.8);
 \draw [blue] (0.15, 0.8)--(0.15, -0.3) .. controls +(0, -0.35) and +(0, -0.35).. (0.6, -0.3);
 \draw [fill=white] (-0.3, -0.3) rectangle (0.3, 0.3);
 \node at (0, 0) {\tiny $1_Y$};
 \begin{scope}[shift={(0.75, 0)}]
  \draw [fill=white] (-0.3, -0.3) rectangle (0.3, 0.3);
 \node at (0, 0) {\tiny $1_{X^*}$};
 \draw [blue] (0.15, -0.3)--(0.15, -1.4) (-0.15, 0.3)--(-0.15, 0.8);
  \draw [blue] (0.15, 0.3) .. controls +(0, 0.35) and +(0, 0.35).. (0.6, 0.3)--(0.6, -1.4);
 \end{scope}  
        \begin{scope}[shift={(0.1, -0.85)}]
        \draw [fill=white] (-0.4, -0.25) rectangle (1.4, 0.25);
        \node at ( 0.2,0) {\tiny $s_{3, -}$};
        \end{scope}
  \end{tikzpicture}}}$ holds.
In \cite[Theorem 6.3]{Liu16}, Liu shows that a commute relation planar algebra of type AA is a free product of Temperley-Lieb planar algebras and group planar algebra for abelian groups.

In the following, we shall construct $\widehat{\cJ}_0$ to  obtain the intertwining property for bimodule GNS symmetric semigroups.
Let 
\begin{align*}
\lambda^{1/2}\left(\sum_{Y\neq \mathbbm{1}}d_Y\omega_Y\right)\widehat{\cJ}_0=& 
\sum_{Y\neq \mathbbm{1}}\frac{\omega_Y}{d_Y}\frac{\kappa_Y d_Y^2}{2d_Y^2-1}\left(\frac{\lambda^{-1/2}}{d_Y^2}\vcenter{\hbox{\scalebox{0.8}{
        \begin{tikzpicture}[scale=1.2]
           \draw [blue]  (-0.2, 0.7)--(-0.2, -0.3).. controls +(0, -0.3) and +(0, -0.3).. (-0.6, -0.3)--(-0.6, 0.7) (0.2, 0.7)--(0.2, -0.3).. controls +(0, -0.3) and +(0, -0.3).. (0.6, -0.3)--(0.6, 0.7);
           \draw [fill=white] (-0.4, -0.3) rectangle (0.4, 0.3);
           \node at (0, 0) {\tiny $1_Y$};
        \begin{scope}[shift={(0, -1.3)}]
             \draw [blue]  (-0.2, -0.7)--(-0.2, 0.3).. controls +(0, 0.3) and +(0, 0.3).. (-0.6, 0.3)--(-0.6, -0.7) (0.2, -0.7)--(0.2, 0.3).. controls +(0, 0.3) and +(0, 0.3).. (0.6, 0.3)--(0.6, -0.7);
           \draw [fill=white] (-0.4, -0.3) rectangle (0.4, 0.3);
           \node at (0, 0) {\tiny $1_Y$};          
        \end{scope}
        \end{tikzpicture}}}} 
        +\lambda^{1/2}
        \vcenter{\hbox{ \begin{tikzpicture} 
        \draw [blue] (-0.6, -1.25)--(-0.6, 0.5);
  \draw [blue] (-0.2, -1.25)--(-0.2, 0.5) (0.2,  -1.25)--(0.2, 0.5) (0.6,  -1.25)--(0.6, 0.5);
        \draw [fill=white] (-0.4, -0.25) rectangle (0.4, 0.25);
        \node at (0, 0) {\tiny $1_Y$};
        \begin{scope}[shift={(0, -0.75)}]
        \draw [fill=white] (-0.4, -0.25) rectangle (0.8, 0.25);
        \node at ( 0.2,0) {\tiny $s_{3, -}$};
        \end{scope}
  \end{tikzpicture}}}
  +\lambda^{1/2}
        \vcenter{\hbox{ \begin{tikzpicture} 
        \draw [blue] (-0.6, -1.25)--(-0.6, 0.5);
  \draw [blue] (-0.2, -1.25)--(-0.2, 0.5) (0.2,  -1.25)--(0.2, 0.5) (0.6,  -1.25)--(0.6, 0.5);
        \draw [fill=white] (-0.4, -0.25) rectangle (0.4, 0.25);
        \node at (0, 0) {\tiny $1_Y$};
        \begin{scope}[shift={(-0.4, -0.75)}]
        \draw [fill=white] (-0.4, -0.25) rectangle (0.8, 0.25);
        \node at ( 0.2,0) {\tiny $s_{3, +}$};
        \end{scope}
  \end{tikzpicture}}}
  \right) \\
&   +\sum_{Y \neq \mathbbm{1}} \frac{\omega_Y}{d_Y}(1-\kappa_Y) \lambda^{1/2} \vcenter{\hbox{
    \begin{tikzpicture}
       \draw [blue]   (-0.15, -0.6)--(-0.15, 0.6) (0.15, -0.6)--(0.15, 0.6) (0.45, -0.6)--(0.45, 0.6) (-0.45, -0.6)--(-0.45, 0.6);
       \draw[fill=white] (-0.35, -0.3) rectangle (0.35, 0.3);
       \node at (0, 0) {\tiny $1_Y$};
    \end{tikzpicture}
    }},
\end{align*}
where $0\leq \kappa_Y\leq 1$ for all $\mathbbm{1} \neq Y\in \Irr$.
We see that the semigroup arising from $\widehat{\cJ}_0$ is a bimodule quantum Markov semigroup extending the semigroup.

\begin{align*}
&\lambda^{1/2}\left(\sum_{Y\neq \mathbbm{1}}d_Y\omega_Y\right)\left( \vcenter{\hbox{
    \begin{tikzpicture}
       \draw [blue] (-0.75, -0.3)--(-0.75, 0.6) (-0.45, -0.3)--(-0.45, 0.6) (-0.15, -0.8)--(-0.15, 0.6) (0.15, -0.8)--(0.15, 1.6) (0.45, -0.8)--(0.45, 1.6);
       \draw[fill=white] (-0.6, -0.3) rectangle (0.6, 0.3);
       \node at (0, 0) {\tiny $\widehat{\cJ}_0$};
       \begin{scope}[shift={(-0.6, 0.8)}]
          \draw[fill=white] (-0.3, -0.3) rectangle (0.3, 0.3); 
           \node at (0, 0) {\tiny $1_X$};
        \draw[blue] (0.15, 0.3).. controls+(0, 0.3) and +(0, 0.3).. (0.45, 0.3)--(0.45, -0.3);
        \draw[blue] (-0.15, 0.3)--(-0.15, 0.8);
        \draw [blue] (-0.15, -1.1) .. controls +(0, -0.25) and +(0, -0.25).. (0.15, -1.1);
       \end{scope}
    \end{tikzpicture}
    }}
  -
      \vcenter{\hbox{
    \begin{tikzpicture}
       \draw [blue]   (-0.15, 0.8)--(-0.15, 0.3) (0.15, -0.8)--(0.15, 0.8) (0.45, -0.8)--(0.45, 0.8);
        \draw[fill=white] (-0.6, -0.3) rectangle (0.6, 0.3);
       \node at (0, 0) {\tiny $\widehat{\cJ}_0$};
       \begin{scope}[shift={(-1.2, 0)}]
       \draw [blue] (0.15, 0.3) .. controls +(0, 0.6) and +(0, 0.6) .. (-0.75, 0.3)--(-0.75, -0.8);
          \draw[fill=white] (-0.3, -0.3) rectangle (0.3, 0.3); 
           \node at (0, 0) {\tiny $1_X$};
        \draw[blue] (0.15, -0.3).. controls+(0, -0.25) and +(0, -0.25).. (0.45, -0.3)--(0.45, 0.3).. controls +(0, 0.3) and +(0, 0.3) .. (0.75, 0.3);
        \draw[blue] (-0.15, 0.3).. controls+(0, 0.3) and +(0, 0.3).. (-0.45, 0.3)--(-0.45, -0.3).. controls +(0, -0.6) and +(0, -0.6) .. (1.05, -0.3);
        \draw [blue] (-0.15, -0.3).. controls +(0, -0.4) and +(0, -0.4) .. (0.75, -0.3);
       \end{scope}
    \end{tikzpicture}
    }}\right) \\
  =& \lambda^{1/2} \sum_{Y\neq \mathbbm{1}, W\in \Irr}\frac{\omega_Y\kappa_Y d_Xd_Y N_{X, Y}^W}{2d_Y^2-1}p_{W, Y} -\lambda^{1/2}\sum_{Y\neq \mathbbm{1}, W\in \Irr}\frac{\omega_Y\kappa_Y d_Xd_Y N_{X^*, Y}^W}{2d_Y^2-1}p_{ Y, W} \\
  =& \lambda^{1/2}\sum_{Y\neq \mathbbm{1}, W\neq \mathbbm{1}\in \Irr}\frac{\omega_Y\kappa_Y d_Xd_Y N_{X, Y}^W}{2d_Y^2-1}p_{W, Y} -\lambda^{1/2}\sum_{Y\neq \mathbbm{1}, W\neq \mathbbm{1} \in \Irr}\frac{\omega_Y\kappa_Y d_Xd_Y N_{X^*, Y}^W}{2d_Y^2-1}p_{ Y, W} \\
  & + \frac{\lambda^{1/2}\omega_{X^*}\kappa_{X^*} d_X^2}{2d_X^2-1}p_{\mathbbm{1}, X^*}-\frac{\lambda^{1/2}\omega_{X}\kappa_{X} d_X^2}{2d_X^2-1}p_{ X, \mathbbm{1}}\\
  =&\lambda^{1/2} \sum_{Y\neq \mathbbm{1}, W\neq 0\in \Irr} d_X N_{X, Y}^W\left(\frac{\omega_Y\kappa_Y d_Y }{2d_Y^2-1}- \frac{\omega_W\kappa_W d_W }{2d_W^2-1}\right)p_{W, Y} \\
  & + \frac{\lambda^{1/2}\omega_{X^*}\kappa_{X^*} d_X^2}{2d_X^2-1}p_{\mathbbm{1}, X^*}-\frac{\lambda^{1/2}\omega_{X}\kappa_{X} d_X^2}{2d_X^2-1}p_{ X, \mathbbm{1}}.
\end{align*}

By assuming that the planar algebra is a commute relation planar algebra of type AA,  we see that a sufficient condition to have the intertwining property is that 
\begin{align*}
\omega_Y\left(\frac{1}{d_Y}-\frac{\kappa_Y d_Y}{2d_Y^2-1}\right)=\omega_{W} \left(\frac{1}{d_W}-\frac{\kappa_W d_W}{2d_W^2-1}\right) =\gamma,\quad  N_{Y, X}^W\neq 0.
\end{align*}
In this case the coefficient $\beta$ is 
\begin{align*}
\beta= \left(\frac{\omega_X}{d_X}-\frac{\omega_X\kappa_Xd_X}{2d_X^2-1}\right)\left(\sum_{Y\neq \mathbbm{1}}d_Y\omega_Y\right)^{-1}=\gamma \left(\sum_{Y\neq \mathbbm{1}}d_Y\omega_Y\right)^{-1}.
\end{align*}

An alternative way to obtain the intertwining property is to assume that the Frobenius algebra of a Morita context is commutative.
We refer to \cite[Theorem 4.5]{LMWW26} for the details where the commutativity of the Frobenius algebra is equivalent to the flatness of the $\alpha$-induced biunitary connection \cite{BE98, BE99, BE00, BE99b,BEK99,BEK00}. 
Let $\mathscr{P}$ be the associated planar algebra, i.e. $\vcenter{\hbox{\scalebox{0.6}{
\begin{tikzpicture}[scale=0.9]
\begin{scope}[shift={(0,1)}]
\draw [blue] (1, 1.5)--(-0.5, 0) (1, 2)--(-1, 0);
\path [fill=white] (0, 0.75) circle (0.4cm);
\draw [blue] (-1, 1.5) node [left] {\tiny $J$}--(0.5, 0) (-1, 2) node [left] {\tiny $\overline{J}$} --(1, 0);
\end{scope}
\draw [blue](-1, -0.1)--(-1, 1) (1, -0.1)--(1, 1);
\begin{scope}[shift={(0, 1)}]
\draw [blue](0,0) [partial ellipse=180:360:0.5 and 0.7];
\end{scope}
\end{tikzpicture}}}}=  \vcenter{\hbox{\scalebox{0.6}{
\begin{tikzpicture}[scale=0.9]
\draw [blue](-1, -0.1)--(-1, 1) node [above] {\tiny $J$} (1, -0.1)--(1, 1);
\begin{scope}[shift={(0, 1)}]
\draw [blue](0,0) [partial ellipse=180:360:0.7 and 0.8] node [pos=0, above] {\tiny $\overline{J}$} ;
\end{scope}
\end{tikzpicture}}}}.$
By taking $\displaystyle 
\lambda^{1/2}\left(\sum_{Y\neq \mathbbm{1}}d_Y\omega_Y\right)\widehat{\cJ}_0 =\sum_{Y\neq \mathbbm{1}} \frac{\omega_Y}{d_Y}  \vcenter{\hbox{
    \begin{tikzpicture}
         \draw[blue] (-0.45, 0.9)..controls +(0, -0.45) and +(0, -0.45)..(0.45, 0.9);
         \draw[blue] (-0.45, -0.9)..controls +(0, 0.45) and +(0, 0.45)..(0.45, -0.9);
          \draw [white, line width=0.15cm]   (-0.15, -0.9)--(-0.15, 0.9) (0.15, -0.9)--(0.15, 0.9);
       \draw [blue]   (-0.15, -0.9)--(-0.15, 0.9) (0.15, -0.9)--(0.15, 0.9);
       \draw[fill=white] (-0.4, -0.3) rectangle (0.4, 0.3);
       \node at (0, 0) {\tiny $1_Y$};
    \end{tikzpicture}
    }}$, we have that 
\begin{align*}
-\vcenter{\hbox{
\begin{tikzpicture}
\draw [blue] (-0.15, -0.8)--(-0.15, 0.8);
 \draw [blue] (0.15, -0.8)--(0.15, 0.3) .. controls +(0, 0.35) and +(0, 0.35).. (0.6, 0.3);
 \draw [fill=white] (-0.3, -0.3) rectangle (0.3, 0.3);
 \node at (0, 0) {\tiny $\widehat{\cL}_0$};
 \begin{scope}[shift={(0.75, 0)}]
  \draw [fill=white] (-0.3, -0.3) rectangle (0.3, 0.3);
 \node at (0, 0) {\tiny $1_X$};
 \draw [blue] (0.15, 0.3)--(0.15, 0.8) (-0.15, -0.3)--(-0.15, -0.8);
  \draw [blue] (0.15, -0.3) .. controls +(0, -0.35) and +(0, -0.35).. (0.6, -0.3)--(0.6, 0.8);
 \end{scope}
\end{tikzpicture}
 }}+
 \vcenter{\hbox{
    \begin{tikzpicture}
    \draw [blue] (-0.15, -0.8)--(-0.15, 0.8);
 \draw [blue] (0.15, 0.8)--(0.15, -0.3) .. controls +(0, -0.35) and +(0, -0.35).. (0.6, -0.3);
 \draw [fill=white] (-0.3, -0.3) rectangle (0.3, 0.3);
 \node at (0, 0) {\tiny $\widehat{\cL}_0$};
 \begin{scope}[shift={(0.75, 0)}]
  \draw [fill=white] (-0.3, -0.3) rectangle (0.3, 0.3);
 \node at (0, 0) {\tiny $1_{X^*}$};
 \draw [blue] (0.15, -0.3)--(0.15, -0.8) (-0.15, 0.3)--(-0.15, 0.8);
  \draw [blue] (0.15, 0.3) .. controls +(0, 0.35) and +(0, 0.35).. (0.6, 0.3)--(0.6, -0.8);
 \end{scope}
    \end{tikzpicture}
    }} =-
\vcenter{\hbox{
    \begin{tikzpicture}
       \draw [blue] (-0.75, -0.3)--(-0.75, 0.6) (-0.45, -0.3)--(-0.45, 0.6) (-0.15, -0.8)--(-0.15, 0.6) (0.15, -0.8)--(0.15, 1.6) (0.45, -0.8)--(0.45, 1.6);
       \draw[fill=white] (-0.6, -0.3) rectangle (0.6, 0.3);
       \node at (0, 0) {\tiny $\widehat{\cJ}_0$};
       \begin{scope}[shift={(-0.6, 0.8)}]
          \draw[fill=white] (-0.3, -0.3) rectangle (0.3, 0.3); 
           \node at (0, 0) {\tiny $1_X$};
        \draw[blue] (0.15, 0.3).. controls+(0, 0.3) and +(0, 0.3).. (0.45, 0.3)--(0.45, -0.3);
        \draw[blue] (-0.15, 0.3)--(-0.15, 0.8);
        \draw [blue] (-0.15, -1.1) .. controls +(0, -0.25) and +(0, -0.25).. (0.15, -1.1);
       \end{scope}
    \end{tikzpicture}
    }}
  +
      \vcenter{\hbox{
    \begin{tikzpicture}
       \draw [blue]   (-0.15, 0.8)--(-0.15, 0.3) (0.15, -0.8)--(0.15, 0.8) (0.45, -0.8)--(0.45, 0.8);
        \draw[fill=white] (-0.6, -0.3) rectangle (0.6, 0.3);
       \node at (0, 0) {\tiny $\widehat{\cJ}_0$};
       \begin{scope}[shift={(-1.2, 0)}]
       \draw [blue] (0.15, 0.3) .. controls +(0, 0.6) and +(0, 0.6) .. (-0.75, 0.3)--(-0.75, -0.8);
          \draw[fill=white] (-0.3, -0.3) rectangle (0.3, 0.3); 
           \node at (0, 0) {\tiny $1_X$};
        \draw[blue] (0.15, -0.3).. controls+(0, -0.25) and +(0, -0.25).. (0.45, -0.3)--(0.45, 0.3).. controls +(0, 0.3) and +(0, 0.3) .. (0.75, 0.3);
        \draw[blue] (-0.15, 0.3).. controls+(0, 0.3) and +(0, 0.3).. (-0.45, 0.3)--(-0.45, -0.3).. controls +(0, -0.6) and +(0, -0.6) .. (1.05, -0.3);
        \draw [blue] (-0.15, -0.3).. controls +(0, -0.4) and +(0, -0.4) .. (0.75, -0.3);
       \end{scope}
    \end{tikzpicture}
    }}
\end{align*}
followed from that fact that $\vcenter{\hbox{
    \begin{tikzpicture}
         \draw[blue] (-0.45, 0.9)..controls +(0, -0.45) and +(0, -0.45)..(0.45, 0.9)--(0.45, 2);
         \draw[blue] (-0.45, -0.9)..controls +(0, 0.45) and +(0, 0.45)..(0.45, -0.9);
          \draw [white, line width=0.15cm]   (-0.15, -0.9)--(-0.15, 0.9) (0.15, -0.9)--(0.15, 0.9);
       \draw [blue]   (-0.15, -0.9)--(-0.15, 0.9) (0.15, -0.9)--(0.15, 0.9)--(0.15, 2);
       \draw[fill=white] (-0.4, -0.3) rectangle (0.4, 0.3);
       \node at (0, 0) {\tiny $1_Y$};
              \begin{scope}[shift={(-0.6, 1.2)}]
              \draw[blue] (-0.15, -2.1)--(-0.15, 0.8);
          \draw[fill=white] (-0.3, -0.3) rectangle (0.3, 0.3); 
           \node at (0, 0) {\tiny $1_X$};
        \draw[blue] (0.15, 0.3).. controls+(0, 0.3) and +(0, 0.3).. (0.45, 0.3)--(0.45, -0.3);
        \draw [blue] (-0.15, -2.1) .. controls +(0, -0.25) and +(0, -0.25).. (0.15, -2.1);
       \end{scope}
    \end{tikzpicture}
    }}
    =
  \vcenter{\hbox{
\begin{tikzpicture}
\draw [blue] (-0.15, 0.3)--(-0.15, 0.8);
 \draw [blue] (0.15, 0.3) .. controls +(0, 0.35) and +(0, 0.35).. (0.6, 0.3);
 \draw [blue] (-0.15, -0.3) --(0.9, -1.5)  (0.15, -0.3) --(0.9, -1.15).. controls +(0.4, -0.4) and +(0, -0.4).. (1.4, -0.3)--(1.4, 0.8);
   \draw [white, line width=0.1cm] (0.9, -0.3) --(-0.15, -1.5)  (0.6, -0.3) --(-0.15, -1.15);
  \draw [blue] (0.9, -0.3) --(-0.15, -1.5)  (0.6, -0.3) --(-0.15, -1.15);
 \draw [fill=white] (-0.3, -0.3) rectangle (0.3, 0.3);
 \node at (0, 0) {\tiny $1_X$};
 \begin{scope}[shift={(0.75, 0)}]
  \draw [fill=white] (-0.3, -0.3) rectangle (0.3, 0.3);
 \node at (0, 0) {\tiny $1_Y$};
 \draw [blue] (0.15, 0.3)--(0.15, 0.8) ;
 \end{scope}
\end{tikzpicture}
 }}=     \vcenter{\hbox{
\begin{tikzpicture}
\draw [blue] (-0.15, -0.8)--(-0.15, 0.8);
 \draw [blue] (0.15, -0.8)--(0.15, 0.3) .. controls +(0, 0.35) and +(0, 0.35).. (0.6, 0.3);
 \draw [fill=white] (-0.3, -0.3) rectangle (0.3, 0.3);
 \node at (0, 0) {\tiny $1_Y$};
 \begin{scope}[shift={(0.75, 0)}]
  \draw [fill=white] (-0.3, -0.3) rectangle (0.3, 0.3);
 \node at (0, 0) {\tiny $1_X$};
 \draw [blue] (0.15, 0.3)--(0.15, 0.8) (-0.15, -0.3)--(-0.15, -0.8);
  \draw [blue] (0.15, -0.3) .. controls +(0, -0.35) and +(0, -0.35).. (0.6, -0.3)--(0.6, 0.8);
 \end{scope}
\end{tikzpicture}
 }}$.
This shows that the associated semigroup has intertwining property with $\beta=0$.

\section{Reducible Inclusions}

\subsection{Fermion Algebras-GNS Symmetry}
In this section, we present an alternative interpretation of the twist derivations in \cite{CarMaa17} and obtain a stronger Bakry-\'{E}mery estimate for the GNS symmetric quantum semigroup.
Recall that the Fermion algebra of dimension $2m$ with generators $P_1, \ldots, P_m, Q_1, \ldots, Q_m$ satisfying 
\begin{align*}
Q_jQ_k+Q_kQ_j=P_jP_k+P_kP_j=2\delta_{j,k}, \quad Q_jP_k +P_kQ_j=0,
\end{align*}
where $\delta_{j,k}$ is the Kronecker symbol and $j,k=1, \ldots, m$.
Let $\displaystyle v_j=\frac{1}{\sqrt{2}}w(Q_j+iP_j)$, where $\displaystyle w=i^m \prod_{j=1}^m Q_jP_j$.
We summarize the properties of $v_j, w$ as follows: 
\begin{enumerate}[(1)]
\item $w^*=w$, $w^2=1$, $\tau(w)=0$.
\item $wP_j=-P_jw$, $wQ_j=-Q_j w$ for $j=1, \ldots, m$.
\item $wv_j=-v_jw$, $wv_j^*=-v_j^*w$ for $j=1, \ldots, m$.
\item $v_jv_k+v_kv_j=0, \quad v_jv_k^*+v_k^*v_j=2\delta_{j,k}$ for $j, k=1, \ldots, m$.
\item $\tau(v_jv_k^*)=\delta_{j,k}$, $\tau(v_jv_k)=0$ for all $j,k=1, \ldots, m$.
\item $i \tau(v_jP_k)= -i\tau(v_j^*P_k) =\left\{\begin{array}{ll} 0 & j \neq k, \\ 
0 & j=k, m\geq 2, \\
2^{-1/2} & j=k, m=1. 
\end{array}\right.$
\end{enumerate}
Let
\begin{align*}
    \widehat{\cL}_0=\frac{1}{2}\sum_{j=1}^m e^{\beta a_j/2} \vcenter{\hbox{\begin{tikzpicture}[scale=0.65]
    \begin{scope}[shift={(0,1.5)}]
    \draw [blue] (-0.5, 0.8)--(-0.5, 0) .. controls +(0, -0.6) and +(0,-0.6).. (0.5, 0)--(0.5, 0.8);    
\begin{scope}[shift={(0.5, 0.3)}]
\draw [fill=white] (-0.3, -0.3) rectangle (0.3, 0.3);
\node at (0, 0) {\tiny $v_j$};
\end{scope}
    \end{scope}
\draw [blue] (-0.5, -0.8)--(-0.5, 0) .. controls +(0, 0.6) and +(0,0.6).. (0.5, 0)--(0.5, -0.8);
\begin{scope}[shift={(0.5, -0.3)}]
\draw [fill=white] (-0.3, -0.3) rectangle (0.3, 0.3);
\node at (0, 0) {\tiny $v_j^*$};
\end{scope}
\end{tikzpicture}}}
+ e^{-\beta a_j/2} \vcenter{\hbox{\begin{tikzpicture}[scale=0.65]
    \begin{scope}[shift={(0,1.5)}]
    \draw [blue] (-0.5, 0.8)--(-0.5, 0) .. controls +(0, -0.6) and +(0,-0.6).. (0.5, 0)--(0.5, 0.8);    
\begin{scope}[shift={(0.5, 0.3)}]
\draw [fill=white] (-0.3, -0.3) rectangle (0.3, 0.3);
\node at (0, 0) {\tiny $v_j^*$};
\end{scope}
    \end{scope}
\draw [blue] (-0.5, -0.8)--(-0.5, 0) .. controls +(0, 0.6) and +(0,0.6).. (0.5, 0)--(0.5, -0.8);
\begin{scope}[shift={(0.5, -0.3)}]
\draw [fill=white] (-0.3, -0.3) rectangle (0.3, 0.3);
\node at (0, 0) {\tiny $v_j$};
\end{scope}
\end{tikzpicture}}}, \quad
\widehat{\cJ}_0=\frac{1}{2}\sum_{j=1}^m e^{\beta a_j/2} \vcenter{\hbox{\begin{tikzpicture}[scale=0.65]
    \begin{scope}[shift={(0,1.5)}]
    \draw [blue] (-0.5, 0.8)--(-0.5, 0) .. controls +(0, -0.6) and +(0,-0.6).. (0.5, 0)--(0.5, 0.8);  
    \draw [blue] (-1.2, 0.8)--(-1.2, 0) .. controls +(0, -0.8) and +(0,-0.8).. (1.2, 0)--(1.2, 0.8); 
\begin{scope}[shift={(0.5, 0.3)}]
\draw [fill=white] (-0.3, -0.3) rectangle (0.3, 0.3);
\node at (0, 0) {\tiny $v_j$};
\end{scope}
\begin{scope}[shift={(1.2, 0.3)}]
\draw [fill=white] (-0.3, -0.3) rectangle (0.3, 0.3);
\node at (0, 0) {\tiny $\overline{w}$};
\end{scope}
    \end{scope}
\draw [blue] (-0.5, -0.8)--(-0.5, 0) .. controls +(0, 0.6) and +(0,0.6).. (0.5, 0)--(0.5, -0.8);
\draw [blue] (-1.2, -0.8)--(-1.2, 0) .. controls +(0, 0.8) and +(0,0.8).. (1.2, 0)--(1.2, -0.8); 
\begin{scope}[shift={(0.5, -0.3)}]
\draw [fill=white] (-0.3, -0.3) rectangle (0.3, 0.3);
\node at (0, 0) {\tiny $v_j^*$};
\end{scope}
\begin{scope}[shift={(1.2, -0.3)}]
\draw [fill=white] (-0.3, -0.3) rectangle (0.3, 0.3);
\node at (0, 0) {\tiny $\overline{w}$};
\end{scope}
\end{tikzpicture}}}
+ e^{-\beta a_j/2} \vcenter{\hbox{\begin{tikzpicture}[scale=0.65]
    \begin{scope}[shift={(0,1.5)}]
    \draw [blue] (-0.5, 0.8)--(-0.5, 0) .. controls +(0, -0.6) and +(0,-0.6).. (0.5, 0)--(0.5, 0.8); 
     \draw [blue] (-1.2, 0.8)--(-1.2, 0) .. controls +(0, -0.8) and +(0,-0.8).. (1.2, 0)--(1.2, 0.8);  
\begin{scope}[shift={(0.5, 0.3)}]
\draw [fill=white] (-0.3, -0.3) rectangle (0.3, 0.3);
\node at (0, 0) {\tiny $v_j^*$};
\end{scope}
\begin{scope}[shift={(1.2, 0.3)}]
\draw [fill=white] (-0.3, -0.3) rectangle (0.3, 0.3);
\node at (0, 0) {\tiny $\overline{w}$};
\end{scope}
    \end{scope}
\draw [blue] (-0.5, -0.8)--(-0.5, 0) .. controls +(0, 0.6) and +(0,0.6).. (0.5, 0)--(0.5, -0.8);
\draw [blue] (-1.2, -0.8)--(-1.2, 0) .. controls +(0, 0.8) and +(0,0.8).. (1.2, 0)--(1.2, -0.8); 
\begin{scope}[shift={(0.5, -0.3)}]
\draw [fill=white] (-0.3, -0.3) rectangle (0.3, 0.3);
\node at (0, 0) {\tiny $v_j$};
\end{scope}
\begin{scope}[shift={(1.2, -0.3)}]
\draw [fill=white] (-0.3, -0.3) rectangle (0.3, 0.3);
\node at (0, 0) {\tiny $\overline{w}$};
\end{scope}
\end{tikzpicture}}},
\end{align*}
where $a_j>0$ are distinct.
By a direct computation, we see that $\widehat{\cJ}_0$ is a lifting of $\widehat{\cL}_0$.
We have that 
\begin{align*}
\mathbf{y}=\frac{1}{2}\left(1*\widehat{\cL}_0\right)=\frac{1}{4} \left(\sum_{j=1}^m e^{\beta a_j/2} v_j^*v_j + e^{-\beta a_j/2} v_jv_j^* \right).
\end{align*}
Moreover, we have that 
\begin{align*}
\mathbf{y} v_k-v_k \mathbf{y}=\frac{1}{2} (e^{-\beta a_k/2}-e^{\beta a_k/2}) v_k.
\end{align*}
Let $\cL(x)= \mathbf{y} x+ x\mathbf{y} -x* \widehat{\cL}_0$ for all $x\in \cM$.
By Proposition \ref{prop:semigroupext}, we see that $\widehat{\cJ}$ arising from $\widehat{\cJ}_0$ is a lifting of $\widehat{\cL}$.

The associated bimodule modular operator $\widehat{\Delta}$ is given by 
\begin{align*}
\widehat{\Delta}
=1-\cR(\widehat{\cL}_0)+\lambda^{1/2} \sum_{j=1}^m e^{\beta a_j} \vcenter{\hbox{\begin{tikzpicture}[scale=0.65]
    \begin{scope}[shift={(0,1.5)}]
    \draw [blue] (-0.5, 0.8)--(-0.5, 0) .. controls +(0, -0.6) and +(0,-0.6).. (0.5, 0)--(0.5, 0.8);    
\begin{scope}[shift={(0.5, 0.3)}]
\draw [fill=white] (-0.3, -0.3) rectangle (0.3, 0.3);
\node at (0, 0) {\tiny $v_j$};
\end{scope}
    \end{scope}
\draw [blue] (-0.5, -0.8)--(-0.5, 0) .. controls +(0, 0.6) and +(0,0.6).. (0.5, 0)--(0.5, -0.8);
\begin{scope}[shift={(0.5, -0.3)}]
\draw [fill=white] (-0.3, -0.3) rectangle (0.3, 0.3);
\node at (0, 0) {\tiny $v_j^*$};
\end{scope}
\end{tikzpicture}}}
+ e^{-\beta a_j} \vcenter{\hbox{\begin{tikzpicture}[scale=0.65]
    \begin{scope}[shift={(0,1.5)}]
    \draw [blue] (-0.5, 0.8)--(-0.5, 0) .. controls +(0, -0.6) and +(0,-0.6).. (0.5, 0)--(0.5, 0.8);    
\begin{scope}[shift={(0.5, 0.3)}]
\draw [fill=white] (-0.3, -0.3) rectangle (0.3, 0.3);
\node at (0, 0) {\tiny $v_j^*$};
\end{scope}
    \end{scope}
\draw [blue] (-0.5, -0.8)--(-0.5, 0) .. controls +(0, 0.6) and +(0,0.6).. (0.5, 0)--(0.5, -0.8);
\begin{scope}[shift={(0.5, -0.3)}]
\draw [fill=white] (-0.3, -0.3) rectangle (0.3, 0.3);
\node at (0, 0) {\tiny $v_j$};
\end{scope}
\end{tikzpicture}}}.
\end{align*}
We have that $\overline{\widehat{\cL}_0}=\overline{\widehat{\Delta}} \widehat{\cL}_0$ and $\overline{\widehat{\cL}}=\overline{\widehat{\Delta}} \widehat{\cL}$.
Moreover, the semigroup $\{\Phi_t\}_{t\geq 0}$ arising from $\cL$ is a bimodule GNS symmetric quantum Markov semigroups with respect to $\widehat{\Delta}$.
Note that $\displaystyle \bigvee_{k \geq 1}\cR(\widehat{\cL}_0^{*(k)}) = I$.
We have that $\{\Phi_t\}_{t\geq 0}$ is ergodic.

\begin{lemma}\label{lem:fermiongns}
Let $\displaystyle \rho_0=\prod_{j=1}^m \left(  \frac{1}{2} e^{-\beta a_j/2 } v_j^*v_j +\frac{1}{2} e^{\beta a_j/2}v_jv_j^*\right)$ and $\displaystyle \rho=\frac{\rho_0}{\tau(\rho_0)}$.
We have that $\widehat{\Delta}\cR(\widehat{\cL}_0)=\widehat{\Delta}_\rho\cR(\widehat{\cL}_0)$ and $\{\Phi_t\}_{t\geq 0}$ is an ergodic  GNS symmetric quantum Markov semigroup.
\end{lemma}
\begin{proof}
Note that 
\begin{align*}
\rho_0 v_j\rho_0^{-1}
=& \left(  \frac{1}{2} e^{-\beta a_j /2} v_j^*v_j +\frac{1}{2} e^{\beta a_j/2}v_jv_j^*\right) v_j \left(  \frac{1}{2} e^{\beta a_j /2} v_j^*v_j +\frac{1}{2} e^{-\beta a_j/2}v_jv_j^*\right)
=  e^{\beta a_j}v_j. \\
\rho_0 v_j^*\rho_0^{-1}
=& \left(  \frac{1}{2} e^{-\beta a_j /2} v_j^*v_j +\frac{1}{2} e^{\beta a_j/2}v_jv_j^*\right) v_j^* \left(  \frac{1}{2} e^{\beta a_j /2} v_j^*v_j +\frac{1}{2} e^{-\beta a_j/2}v_jv_j^*\right)
= e^{-\beta a_j}v_j^*. 
\end{align*}
We obtain that $\widehat{\Delta}\cR(\widehat{\cL}_0)=\widehat{\Delta}_\rho\cR(\widehat{\cL}_0)$.
This implies that $\{\Phi_t\}_{t\geq 0}$ is a GNS symmetric quantum Markov semigroup with respect to $\rho$.
\end{proof}

In the following, we shall check the intertwining property pictorially for the semigroup.
\begin{align*}
\lambda^{-1/2}\vcenter{\hbox{
\begin{tikzpicture}
\draw [blue] (-0.15, -0.8)--(-0.15, 0.8);
 \draw [blue] (0.15, -0.8)--(0.15, 0.3) .. controls +(0, 0.35) and +(0, 0.35).. (0.6, 0.3);
 \draw [fill=white] (-0.3, -0.3) rectangle (0.3, 0.3);
 \node at (0, 0) {\tiny $\widehat{\cL}$};
 \begin{scope}[shift={(0.75, 0)}]
  \draw [fill=white] (-0.3, -0.3) rectangle (0.3, 0.3);
 \node at (0, 0) {\tiny $p_k$};
 \draw [blue] (0.15, 0.3)--(0.15, 0.8) (-0.15, -0.3)--(-0.15, -0.8);
  \draw [blue] (0.15, -0.3) .. controls +(0, -0.35) and +(0, -0.35).. (0.6, -0.3)--(0.6, 0.8);
 \end{scope}
\end{tikzpicture}
 }}
 = \vcenter{\hbox{\begin{tikzpicture}[scale=0.65]
     \draw [blue] (1.4, 2.3)--(1.4, -0.8);
    \begin{scope}[shift={(0,1.5)}]
    \draw [blue] (-0.5, 0.8)--(-0.5, 0) .. controls +(0, -0.6) and +(0,-0.6).. (0.5, 0)--(0.5, 0.8); 
\begin{scope}[shift={(0.5, 0.3)}]
\draw [fill=white] (-0.5, -0.3) rectangle (0.5, 0.3);
\node at (0, 0) {\tiny $\mathbf{y} v_k$};
\end{scope}
    \end{scope}
\draw [blue] (-0.5, -0.8)--(-0.5, 0) .. controls +(0, 0.6) and +(0,0.6).. (0.5, 0)--(0.5, -0.8);
\begin{scope}[shift={(1.4, 0.8)}]
\draw [fill=white] (-0.3, -0.3) rectangle (0.3, 0.3);
\node at (0, 0) {\tiny $\overline{v_k^*}$};  
\end{scope}
\end{tikzpicture}}}
+
\vcenter{\hbox{\begin{tikzpicture}[scale=0.65]
     \draw [blue] (1.4, 2.3)--(1.4, -0.8);
    \begin{scope}[shift={(0,1.5)}]
    \draw [blue] (-0.5, 0.8)--(-0.5, 0) .. controls +(0, -0.6) and +(0,-0.6).. (0.5, 0)--(0.5, 0.8); 
\begin{scope}[shift={(0.5, 0.3)}]
\draw [fill=white] (-0.4, -0.3) rectangle (0.4, 0.3);
\node at (0, 0) {\tiny $v_k$};
\end{scope}
    \end{scope}
\draw [blue] (-0.5, -0.8)--(-0.5, 0) .. controls +(0, 0.6) and +(0,0.6).. (0.5, 0)--(0.5, -0.8);
\begin{scope}[shift={(0.5, -0.3)}]
\draw [fill=white] (-0.3, -0.3) rectangle (0.3, 0.3);
\node at (0, 0) {\tiny $\mathbf{y}$};
\end{scope}
\begin{scope}[shift={(1.4, 0.8)}]
\draw [fill=white] (-0.3, -0.3) rectangle (0.3, 0.3);
\node at (0, 0) {\tiny $\overline{v_k^*}$};  
\end{scope}
\end{tikzpicture}}}
-\frac{1}{2}\sum_{j=1}^m e^{\beta a_j/2}\vcenter{\hbox{\begin{tikzpicture}[scale=0.65]
     \draw [blue] (1.4, 2.3)--(1.4, -0.8);
    \begin{scope}[shift={(0,1.5)}]
    \draw [blue] (-0.5, 0.8)--(-0.5, 0) .. controls +(0, -0.6) and +(0,-0.6).. (0.5, 0)--(0.5, 0.8); 
\begin{scope}[shift={(0.5, 0.3)}]
\draw [fill=white] (-0.5, -0.3) rectangle (0.5, 0.3);
\node at (0, 0) {\tiny $v_jv_k$};
\end{scope}
    \end{scope}
\draw [blue] (-0.5, -0.8)--(-0.5, 0) .. controls +(0, 0.6) and +(0,0.6).. (0.5, 0)--(0.5, -0.8);
\begin{scope}[shift={(0.5, -0.3)}]
\draw [fill=white] (-0.3, -0.3) rectangle (0.3, 0.3);
\node at (0, 0) {\tiny $v_j^*$};
\end{scope}
\begin{scope}[shift={(1.4, 0.8)}]
\draw [fill=white] (-0.3, -0.3) rectangle (0.3, 0.3);
\node at (0, 0) {\tiny $\overline{v_k^*}$};  
\end{scope}
\end{tikzpicture}}}
-\frac{1}{2}\sum_{j=1}^m e^{-\beta a_j/2}\vcenter{\hbox{\begin{tikzpicture}[scale=0.65]
     \draw [blue] (1.4, 2.3)--(1.4, -0.8);
    \begin{scope}[shift={(0,1.5)}]
    \draw [blue] (-0.5, 0.8)--(-0.5, 0) .. controls +(0, -0.6) and +(0,-0.6).. (0.5, 0)--(0.5, 0.8); 
\begin{scope}[shift={(0.5, 0.3)}]
\draw [fill=white] (-0.5, -0.3) rectangle (0.5, 0.3);
\node at (0, 0) {\tiny $v_j^*v_k$};
\end{scope}
    \end{scope}
\draw [blue] (-0.5, -0.8)--(-0.5, 0) .. controls +(0, 0.6) and +(0,0.6).. (0.5, 0)--(0.5, -0.8);
\begin{scope}[shift={(0.5, -0.3)}]
\draw [fill=white] (-0.3, -0.3) rectangle (0.3, 0.3);
\node at (0, 0) {\tiny $v_j$};
\end{scope}
\begin{scope}[shift={(1.4, 0.8)}]
\draw [fill=white] (-0.3, -0.3) rectangle (0.3, 0.3);
\node at (0, 0) {\tiny $\overline{v_k^*}$};  
\end{scope}
\end{tikzpicture}}},
\end{align*}

\begin{align*}
 \lambda^{-1/2}   \vcenter{\hbox{
    \begin{tikzpicture}
    \draw [blue] (-0.15, -0.8)--(-0.15, 0.8);
 \draw [blue] (0.15, 0.8)--(0.15, -0.3) .. controls +(0, -0.35) and +(0, -0.35).. (0.6, -0.3);
 \draw [fill=white] (-0.3, -0.3) rectangle (0.3, 0.3);
 \node at (0, 0) {\tiny $\widehat{\cL}$};
 \begin{scope}[shift={(0.75, 0)}]
  \draw [fill=white] (-0.3, -0.3) rectangle (0.3, 0.3);
 \node at (0, 0) {\tiny $\overline{p_k}$};
 \draw [blue] (0.15, -0.3)--(0.15, -0.8) (-0.15, 0.3)--(-0.15, 0.8);
  \draw [blue] (0.15, 0.3) .. controls +(0, 0.35) and +(0, 0.35).. (0.6, 0.3)--(0.6, -0.8);
 \end{scope}
    \end{tikzpicture}
    }}
= \vcenter{\hbox{\begin{tikzpicture}[scale=0.65]
     \draw [blue] (1.4, 2.3)--(1.4, -0.8);
    \begin{scope}[shift={(0,1.5)}]
    \draw [blue] (-0.5, 0.8)--(-0.5, 0) .. controls +(0, -0.6) and +(0,-0.6).. (0.5, 0)--(0.5, 0.8); 
\begin{scope}[shift={(0.5, 0.3)}]
\draw [fill=white] (-0.3, -0.3) rectangle (0.3, 0.3);
\node at (0, 0) {\tiny $\mathbf{y}$};
\end{scope}
    \end{scope}
\draw [blue] (-0.5, -0.8)--(-0.5, 0) .. controls +(0, 0.6) and +(0,0.6).. (0.5, 0)--(0.5, -0.8);
\begin{scope}[shift={(0.5, -0.3)}]
\draw [fill=white] (-0.4, -0.3) rectangle (0.4, 0.3);
\node at (0, 0) {\tiny $ v_k $};
\end{scope}
\begin{scope}[shift={(1.4, 0.8)}]
\draw [fill=white] (-0.3, -0.3) rectangle (0.3, 0.3);
\node at (0, 0) {\tiny $\overline{v_k^*}$};  
\end{scope}
\end{tikzpicture}}}
+
 \vcenter{\hbox{\begin{tikzpicture}[scale=0.65]
     \draw [blue] (1.4, 2.3)--(1.4, -0.8);
    \begin{scope}[shift={(0,1.5)}]
    \draw [blue] (-0.5, 0.8)--(-0.5, 0) .. controls +(0, -0.6) and +(0,-0.6).. (0.5, 0)--(0.5, 0.8); 
    \end{scope}
\draw [blue] (-0.5, -0.8)--(-0.5, 0) .. controls +(0, 0.6) and +(0,0.6).. (0.5, 0)--(0.5, -0.8);
\begin{scope}[shift={(0.5, -0.3)}]
\draw [fill=white] (-0.5, -0.3) rectangle (0.5, 0.3);
\node at (0, 0) {\tiny $v_k\mathbf{y}$};
\end{scope}
\begin{scope}[shift={(1.4, 0.8)}]
\draw [fill=white] (-0.3, -0.3) rectangle (0.3, 0.3);
\node at (0, 0) {\tiny $\overline{v_k^*}$};  
\end{scope}
\end{tikzpicture}}}
-\frac{1}{2}\sum_{j=1}^m e^{\beta a_j/2}\vcenter{\hbox{\begin{tikzpicture}[scale=0.65]
     \draw [blue] (1.4, 2.3)--(1.4, -0.8);
    \begin{scope}[shift={(0,1.5)}]
    \draw [blue] (-0.5, 0.8)--(-0.5, 0) .. controls +(0, -0.6) and +(0,-0.6).. (0.5, 0)--(0.5, 0.8); 
\begin{scope}[shift={(0.5, 0.3)}]
\draw [fill=white] (-0.3, -0.3) rectangle (0.3, 0.3);
\node at (0, 0) {\tiny $v_j$};
\end{scope}
    \end{scope}
\draw [blue] (-0.5, -0.8)--(-0.5, 0) .. controls +(0, 0.6) and +(0,0.6).. (0.5, 0)--(0.5, -0.8);
\begin{scope}[shift={(0.5, -0.3)}]
\draw [fill=white] (-0.5, -0.3) rectangle (0.5, 0.3);
\node at (0, 0) {\tiny $v_kv_j^*$};
\end{scope}
\begin{scope}[shift={(1.4, 0.8)}]
\draw [fill=white] (-0.3, -0.3) rectangle (0.3, 0.3);
\node at (0, 0) {\tiny $\overline{v_k^*}$};  
\end{scope}
\end{tikzpicture}}}
-\frac{1}{2}\sum_{j=1}^m e^{-\beta a_j/2}\vcenter{\hbox{\begin{tikzpicture}[scale=0.65]
     \draw [blue] (1.4, 2.3)--(1.4, -0.8);
    \begin{scope}[shift={(0,1.5)}]
    \draw [blue] (-0.5, 0.8)--(-0.5, 0) .. controls +(0, -0.6) and +(0,-0.6).. (0.5, 0)--(0.5, 0.8); 
\begin{scope}[shift={(0.5, 0.3)}]
\draw [fill=white] (-0.3, -0.3) rectangle (0.3, 0.3);
\node at (0, 0) {\tiny $v_j^*$};
\end{scope}
    \end{scope}
\draw [blue] (-0.5, -0.8)--(-0.5, 0) .. controls +(0, 0.6) and +(0,0.6).. (0.5, 0)--(0.5, -0.8);
\begin{scope}[shift={(0.5, -0.3)}]
\draw [fill=white] (-0.5, -0.3) rectangle (0.5, 0.3);
\node at (0, 0) {\tiny $v_k v_j$};
\end{scope}
\begin{scope}[shift={(1.4, 0.8)}]
\draw [fill=white] (-0.3, -0.3) rectangle (0.3, 0.3);
\node at (0, 0) {\tiny $\overline{v_k^*}$};  
\end{scope}
\end{tikzpicture}}},
\end{align*}

\begin{align*}
\lambda^{-1/2}\vcenter{\hbox{
    \begin{tikzpicture}
       \draw [blue] (-0.75, -0.3)--(-0.75, 0.6) (-0.45, -0.3)--(-0.45, 0.6) (-0.15, -0.8)--(-0.15, 0.6) (0.15, -0.8)--(0.15, 1.6) (0.45, -0.8)--(0.45, 1.6);
       \draw[fill=white] (-0.6, -0.3) rectangle (0.6, 0.3);
       \node at (0, 0) {\tiny $\widehat{\cJ}$};
       \begin{scope}[shift={(-0.6, 0.8)}]
          \draw[fill=white] (-0.3, -0.3) rectangle (0.3, 0.3); 
           \node at (0, 0) {\tiny $p_k$};
        \draw[blue] (0.15, 0.3).. controls+(0, 0.3) and +(0, 0.3).. (0.45, 0.3)--(0.45, -0.3);
        \draw[blue] (-0.15, 0.3)--(-0.15, 0.8);
        \draw [blue] (-0.15, -1.1) .. controls +(0, -0.25) and +(0, -0.25).. (0.15, -1.1);
       \end{scope}
    \end{tikzpicture}
    }} 
    = \vcenter{\hbox{\begin{tikzpicture}[scale=0.65]
     \draw [blue] (1.4, 2.3)--(1.4, -0.8);
    \begin{scope}[shift={(0,1.5)}]
    \draw [blue] (-0.5, 0.8)--(-0.5, 0) .. controls +(0, -0.6) and +(0,-0.6).. (0.5, 0)--(0.5, 0.8); 
\begin{scope}[shift={(0.5, 0.3)}]
\draw [fill=white] (-0.5, -0.3) rectangle (0.5, 0.3);
\node at (0, 0) {\tiny $v_k \mathbf{y} $};
\end{scope}
    \end{scope}
\draw [blue] (-0.5, -0.8)--(-0.5, 0) .. controls +(0, 0.6) and +(0,0.6).. (0.5, 0)--(0.5, -0.8);
\begin{scope}[shift={(1.4, 0.8)}]
\draw [fill=white] (-0.3, -0.3) rectangle (0.3, 0.3);
\node at (0, 0) {\tiny $\overline{v_k^*}$};  
\end{scope}
\end{tikzpicture}}}
+
 \vcenter{\hbox{\begin{tikzpicture}[scale=0.65]
     \draw [blue] (1.4, 2.3)--(1.4, -0.8);
    \begin{scope}[shift={(0,1.5)}]
    \draw [blue] (-0.5, 0.8)--(-0.5, 0) .. controls +(0, -0.6) and +(0,-0.6).. (0.5, 0)--(0.5, 0.8); 
\begin{scope}[shift={(0.5, 0.3)}]
\draw [fill=white] (-0.3, -0.3) rectangle (0.3, 0.3);
\node at (0, 0) {\tiny $v_k$};
\end{scope}
    \end{scope}
\draw [blue] (-0.5, -0.8)--(-0.5, 0) .. controls +(0, 0.6) and +(0,0.6).. (0.5, 0)--(0.5, -0.8);
\begin{scope}[shift={(0.5, -0.3)}]
\draw [fill=white] (-0.3, -0.3) rectangle (0.3, 0.3);
\node at (0, 0) {\tiny $\mathbf{y}$};
\end{scope}
\begin{scope}[shift={(1.4, 0.8)}]
\draw [fill=white] (-0.3, -0.3) rectangle (0.3, 0.3);
\node at (0, 0) {\tiny $\overline{v_k^*}$};  
\end{scope}
\end{tikzpicture}}}
+\frac{1}{2}\sum_{j=1}^m e^{\beta a_j/2}\vcenter{\hbox{\begin{tikzpicture}[scale=0.65]
     \draw [blue] (1.4, 2.3)--(1.4, -0.8);
    \begin{scope}[shift={(0,1.5)}]
    \draw [blue] (-0.5, 0.8)--(-0.5, 0) .. controls +(0, -0.6) and +(0,-0.6).. (0.5, 0)--(0.5, 0.8); 
\begin{scope}[shift={(0.5, 0.3)}]
\draw [fill=white] (-0.5, -0.3) rectangle (0.5, 0.3);
\node at (0, 0) {\tiny $v_kv_j$};
\end{scope}
    \end{scope}
\draw [blue] (-0.5, -0.8)--(-0.5, 0) .. controls +(0, 0.6) and +(0,0.6).. (0.5, 0)--(0.5, -0.8);
\begin{scope}[shift={(0.5, -0.3)}]
\draw [fill=white] (-0.3, -0.3) rectangle (0.3, 0.3);
\node at (0, 0) {\tiny $v_j^*$};
\end{scope}
\begin{scope}[shift={(1.4, 0.8)}]
\draw [fill=white] (-0.3, -0.3) rectangle (0.3, 0.3);
\node at (0, 0) {\tiny $\overline{v_k^*}$};  
\end{scope}
\end{tikzpicture}}}
+\frac{1}{2}\sum_{j=1}^m e^{-\beta a_j/2}\vcenter{\hbox{\begin{tikzpicture}[scale=0.65]
     \draw [blue] (1.4, 2.3)--(1.4, -0.8);
    \begin{scope}[shift={(0,1.5)}]
    \draw [blue] (-0.5, 0.8)--(-0.5, 0) .. controls +(0, -0.6) and +(0,-0.6).. (0.5, 0)--(0.5, 0.8); 
\begin{scope}[shift={(0.5, 0.3)}]
\draw [fill=white] (-0.5, -0.3) rectangle (0.5, 0.3);
\node at (0, 0) {\tiny $v_kv_j^*$};
\end{scope}
    \end{scope}
\draw [blue] (-0.5, -0.8)--(-0.5, 0) .. controls +(0, 0.6) and +(0,0.6).. (0.5, 0)--(0.5, -0.8);
\begin{scope}[shift={(0.5, -0.3)}]
\draw [fill=white] (-0.3, -0.3) rectangle (0.3, 0.3);
\node at (0, 0) {\tiny $v_j$};
\end{scope}
\begin{scope}[shift={(1.4, 0.8)}]
\draw [fill=white] (-0.3, -0.3) rectangle (0.3, 0.3);
\node at (0, 0) {\tiny $\overline{v_k^*}$};  
\end{scope}
\end{tikzpicture}}},   
\end{align*}

\begin{align*}
\lambda^{-1/2}\vcenter{\hbox{
    \begin{tikzpicture}
       \draw [blue]   (-0.15, 0.8)--(-0.15, 0.3) (0.15, -0.8)--(0.15, 0.8) (0.45, -0.8)--(0.45, 0.8);
        \draw[fill=white] (-0.6, -0.3) rectangle (0.6, 0.3);
       \node at (0, 0) {\tiny $\widehat{\cJ}$};
       \begin{scope}[shift={(-1.2, 0)}]
       \draw [blue] (0.15, 0.3) .. controls +(0, 0.6) and +(0, 0.6) .. (-0.75, 0.3)--(-0.75, -0.8);
          \draw[fill=white] (-0.3, -0.3) rectangle (0.3, 0.3); 
           \node at (0, 0) {\tiny $p_k$};
        \draw[blue] (0.15, -0.3).. controls+(0, -0.25) and +(0, -0.25).. (0.45, -0.3)--(0.45, 0.3).. controls +(0, 0.3) and +(0, 0.3) .. (0.75, 0.3);
        \draw[blue] (-0.15, 0.3).. controls+(0, 0.3) and +(0, 0.3).. (-0.45, 0.3)--(-0.45, -0.3).. controls +(0, -0.6) and +(0, -0.6) .. (1.05, -0.3);
        \draw [blue] (-0.15, -0.3).. controls +(0, -0.4) and +(0, -0.4) .. (0.75, -0.3);
       \end{scope}
    \end{tikzpicture}
    }}
    = \vcenter{\hbox{\begin{tikzpicture}[scale=0.65]
     \draw [blue] (1.4, 2.3)--(1.4, -0.8);
    \begin{scope}[shift={(0,1.5)}]
    \draw [blue] (-0.5, 0.8)--(-0.5, 0) .. controls +(0, -0.6) and +(0,-0.6).. (0.5, 0)--(0.5, 0.8); 
\begin{scope}[shift={(0.5, 0.3)}]
\draw [fill=white] (-0.3, -0.3) rectangle (0.3, 0.3);
\node at (0, 0) {\tiny $\mathbf{y}$};
\end{scope}
    \end{scope}
\draw [blue] (-0.5, -0.8)--(-0.5, 0) .. controls +(0, 0.6) and +(0,0.6).. (0.5, 0)--(0.5, -0.8);
\begin{scope}[shift={(0.5, -0.3)}]
\draw [fill=white] (-0.3, -0.3) rectangle (0.3, 0.3);
\node at (0, 0) {\tiny $v_k $};
\end{scope}
\begin{scope}[shift={(1.4, 0.8)}]
\draw [fill=white] (-0.3, -0.3) rectangle (0.3, 0.3);
\node at (0, 0) {\tiny $\overline{v_k^*}$};  
\end{scope}
\end{tikzpicture}}}
+
 \vcenter{\hbox{\begin{tikzpicture}[scale=0.65]
     \draw [blue] (1.4, 2.3)--(1.4, -0.8);
    \begin{scope}[shift={(0,1.5)}]
    \draw [blue] (-0.5, 0.8)--(-0.5, 0) .. controls +(0, -0.6) and +(0,-0.6).. (0.5, 0)--(0.5, 0.8); 
    \end{scope}
\draw [blue] (-0.5, -0.8)--(-0.5, 0) .. controls +(0, 0.6) and +(0,0.6).. (0.5, 0)--(0.5, -0.8);
\begin{scope}[shift={(0.5, -0.3)}]
\draw [fill=white] (-0.5, -0.3) rectangle (0.5, 0.3);
\node at (0, 0) {\tiny $\mathbf{y}v_k$};
\end{scope}
\begin{scope}[shift={(1.4, 0.8)}]
\draw [fill=white] (-0.3, -0.3) rectangle (0.3, 0.3);
\node at (0, 0) {\tiny $\overline{v_k^*}$};  
\end{scope}
\end{tikzpicture}}}
+\frac{1}{2}\sum_{j=1}^m e^{\beta a_j/2}\vcenter{\hbox{\begin{tikzpicture}[scale=0.65]
     \draw [blue] (1.4, 2.3)--(1.4, -0.8);
    \begin{scope}[shift={(0,1.5)}]
    \draw [blue] (-0.5, 0.8)--(-0.5, 0) .. controls +(0, -0.6) and +(0,-0.6).. (0.5, 0)--(0.5, 0.8); 
\begin{scope}[shift={(0.5, 0.3)}]
\draw [fill=white] (-0.3, -0.3) rectangle (0.3, 0.3);
\node at (0, 0) {\tiny $v_j$};
\end{scope}
    \end{scope}
\draw [blue] (-0.5, -0.8)--(-0.5, 0) .. controls +(0, 0.6) and +(0,0.6).. (0.5, 0)--(0.5, -0.8);
\begin{scope}[shift={(0.5, -0.3)}]
\draw [fill=white] (-0.5, -0.3) rectangle (0.5, 0.3);
\node at (0, 0) {\tiny $v_j^*v_k$};
\end{scope}
\begin{scope}[shift={(1.4, 0.8)}]
\draw [fill=white] (-0.3, -0.3) rectangle (0.3, 0.3);
\node at (0, 0) {\tiny $\overline{v_k^*}$};  
\end{scope}
\end{tikzpicture}}}
+\frac{1}{2}\sum_{j=1}^m e^{-\beta a_j/2}\vcenter{\hbox{\begin{tikzpicture}[scale=0.65]
     \draw [blue] (1.4, 2.3)--(1.4, -0.8);
    \begin{scope}[shift={(0,1.5)}]
    \draw [blue] (-0.5, 0.8)--(-0.5, 0) .. controls +(0, -0.6) and +(0,-0.6).. (0.5, 0)--(0.5, 0.8); 
\begin{scope}[shift={(0.5, 0.3)}]
\draw [fill=white] (-0.3, -0.3) rectangle (0.3, 0.3);
\node at (0, 0) {\tiny $v_j^*$};
\end{scope}
    \end{scope}
\draw [blue] (-0.5, -0.8)--(-0.5, 0) .. controls +(0, 0.6) and +(0,0.6).. (0.5, 0)--(0.5, -0.8);
\begin{scope}[shift={(0.5, -0.3)}]
\draw [fill=white] (-0.5, -0.3) rectangle (0.5, 0.3);
\node at (0, 0) {\tiny $ v_jv_k$};
\end{scope}
\begin{scope}[shift={(1.4, 0.8)}]
\draw [fill=white] (-0.3, -0.3) rectangle (0.3, 0.3);
\node at (0, 0) {\tiny $\overline{v_k^*}$};  
\end{scope}
\end{tikzpicture}}}.
\end{align*}
By taking differences, we have that 
\begin{align*}
\partial_k \cL -\cJ \partial_k= \frac{1}{2} (e^{-\beta a_k/2} + e^{\beta a_k/2}) \partial_k, \quad k=1, \ldots, m.
\end{align*}

\begin{theorem}
Suppose that $\{\Phi_t\}_{\ \geq 0}$ is the semigroup described as above.
Then for any $x\in \cM$, we have that 
\begin{align*}
 & \Gamma_2(x) \geq  \frac{1}{2}\min_{1\leq j \leq m} (e^{-\beta a_j/2} + e^{\beta a_j/2}) \Gamma(x) \\
  &+\frac{1}{\sum_{j=1}^m (1+e^{-\beta a_j })^2
  +(1+e^{\beta a_j })^2}  \left| \cL(x)+\cL^*(x)+\sum_{j=1}^m\frac{e^{\beta a_j/2}-e^{-\beta a_j /2}}{4}\{x, v_jv_j^*-v_j^*v_j\}\right|^2.
\end{align*}
\end{theorem}
\begin{proof}
It follows from Proposition \ref{prop:matrixderivation}.
\end{proof}

\begin{theorem}
Suppose that $\{\Phi_t\}_{t\geq 0}$ is the semigroup described as above.
Then for any $x\in \cM$, we have that 
\begin{align*}
 \|\nabla \Phi_t x\|_{D, \widehat{\Delta}}^2 \leq   e^{-2\beta_{\min} t} \|\nabla x\|_{\Phi_t^*(D), \widehat{\Delta}}^2- (\min_{1\leq j \leq m} e^{\pm \beta a_j})\frac{(e^{-2\beta_{\min} t}- e^{-4\beta_{\min} t})}{\beta_{\min}\sum_{j=1}^m (1+e^{-\beta a_j })^2
  +(1+e^{\beta a_j })^2}  \langle P_Dx, x\rangle,
\end{align*}
where $\displaystyle \beta_{\min}= \frac{1}{2}\min_{1\leq j \leq m} (e^{-\beta a_j/2} + e^{\beta a_j/2}) $, $P_D=\langle \cdot, \xi_D\rangle \xi_D$ and 
\begin{align*}
 \xi_D=\cL(\Phi_t^*(D))+\cL^*(\Phi_t^*(D))+\sum_{j=1}^m\frac{e^{\beta a_j/2}-e^{-\beta a_j /2}}{4}\{\Phi_t^*(D), v_jv_j^*-v_j^*v_j\} .
 \end{align*}
\end{theorem}
\begin{proof}
It follows from Theorem \ref{thm:ricbd}.
\end{proof}

\subsection{Fermion Algebras-KMS Symmetry}

In this section, we shall obtain KMS symmetric semigroups for Fermion algebras and characterize its intertwining properties.
Let
\begin{align*}
\widehat{\Delta}
= \sum_{j=1}^m e^{\beta a_j/2} \lambda^{1/2} \vcenter{\hbox{\begin{tikzpicture}[scale=0.65]
    \begin{scope}[shift={(0,1.5)}]
    \draw [blue] (-0.5, 0.8)--(-0.5, 0) .. controls +(0, -0.6) and +(0,-0.6).. (0.5, 0)--(0.5, 0.8);    
\begin{scope}[shift={(0.5, 0.3)}]
\draw [fill=white] (-0.3, -0.3) rectangle (0.3, 0.3);
\node at (0, 0) {\tiny $v_j$};
\end{scope}
    \end{scope}
\draw [blue] (-0.5, -0.8)--(-0.5, 0) .. controls +(0, 0.6) and +(0,0.6).. (0.5, 0)--(0.5, -0.8);
\begin{scope}[shift={(0.5, -0.3)}]
\draw [fill=white] (-0.3, -0.3) rectangle (0.3, 0.3);
\node at (0, 0) {\tiny $v_j^*$};
\end{scope}
\end{tikzpicture}}}
+ \sum_{j=1}^m e^{-\beta a_j/2} \lambda^{1/2} \vcenter{\hbox{\begin{tikzpicture}[scale=0.65]
    \begin{scope}[shift={(0,1.5)}]
    \draw [blue] (-0.5, 0.8)--(-0.5, 0) .. controls +(0, -0.6) and +(0,-0.6).. (0.5, 0)--(0.5, 0.8);    
\begin{scope}[shift={(0.5, 0.3)}]
\draw [fill=white] (-0.3, -0.3) rectangle (0.3, 0.3);
\node at (0, 0) {\tiny $v_j^*$};
\end{scope}
    \end{scope}
\draw [blue] (-0.5, -0.8)--(-0.5, 0) .. controls +(0, 0.6) and +(0,0.6).. (0.5, 0)--(0.5, -0.8);
\begin{scope}[shift={(0.5, -0.3)}]
\draw [fill=white] (-0.3, -0.3) rectangle (0.3, 0.3);
\node at (0, 0) {\tiny $v_j$};
\end{scope}
\end{tikzpicture}}}
+\left( 1- \sum_{j=1}^m \lambda^{1/2}\vcenter{\hbox{\begin{tikzpicture}[scale=0.65]
    \begin{scope}[shift={(0,1.5)}]
    \draw [blue] (-0.5, 0.8)--(-0.5, 0) .. controls +(0, -0.6) and +(0,-0.6).. (0.5, 0)--(0.5, 0.8);    
\begin{scope}[shift={(0.5, 0.3)}]
\draw [fill=white] (-0.3, -0.3) rectangle (0.3, 0.3);
\node at (0, 0) {\tiny $v_j$};
\end{scope}
    \end{scope}
\draw [blue] (-0.5, -0.8)--(-0.5, 0) .. controls +(0, 0.6) and +(0,0.6).. (0.5, 0)--(0.5, -0.8);
\begin{scope}[shift={(0.5, -0.3)}]
\draw [fill=white] (-0.3, -0.3) rectangle (0.3, 0.3);
\node at (0, 0) {\tiny $v_j^*$};
\end{scope}
\end{tikzpicture}}}
-\sum_{j=1}^m \lambda^{1/2}\vcenter{\hbox{\begin{tikzpicture}[scale=0.65]
    \begin{scope}[shift={(0,1.5)}]
    \draw [blue] (-0.5, 0.8)--(-0.5, 0) .. controls +(0, -0.6) and +(0,-0.6).. (0.5, 0)--(0.5, 0.8);    
\begin{scope}[shift={(0.5, 0.3)}]
\draw [fill=white] (-0.3, -0.3) rectangle (0.3, 0.3);
\node at (0, 0) {\tiny $v_j^*$};
\end{scope}
    \end{scope}
\draw [blue] (-0.5, -0.8)--(-0.5, 0) .. controls +(0, 0.6) and +(0,0.6).. (0.5, 0)--(0.5, -0.8);
\begin{scope}[shift={(0.5, -0.3)}]
\draw [fill=white] (-0.3, -0.3) rectangle (0.3, 0.3);
\node at (0, 0) {\tiny $v_j$};
\end{scope}
\end{tikzpicture}}}\right)
\end{align*}
be the associated bimodule modular operator, where $a_j$ are distinct for $j=1, \ldots, m$.
By a direct computation, we see that $\overline{\widehat{\Delta}}=\widehat{\Delta}^{-1}$.

Now we define a self-dual element $H$ in $\cM'\cap \cM_2$.
Let 
\begin{align*}
H=\sum_{j,k=1}^m \gamma_{j,k}^+ \vcenter{\hbox{\begin{tikzpicture}[scale=0.65]
    \begin{scope}[shift={(0,1.5)}]
    \draw [blue] (-0.5, 0.8)--(-0.5, 0) .. controls +(0, -0.6) and +(0,-0.6).. (0.5, 0)--(0.5, 0.8);    
\begin{scope}[shift={(0.5, 0.3)}]
\draw [fill=white] (-0.5, -0.3) rectangle (0.5, 0.3);
\node at (0, 0) {\tiny $iwQ_j$};
\end{scope}
    \end{scope}
\draw [blue] (-0.5, -0.8)--(-0.5, 0) .. controls +(0, 0.6) and +(0,0.6).. (0.5, 0)--(0.5, -0.8);
\begin{scope}[shift={(0.5, -0.3)}]
\draw [fill=white] (-0.5, -0.3) rectangle (0.5, 0.3);
\node at (0, 0) {\tiny $iwQ_k$};
\end{scope}
\end{tikzpicture}}}
+ \gamma_{j,k}^- \vcenter{\hbox{\begin{tikzpicture}[scale=0.65]
    \begin{scope}[shift={(0,1.5)}]
    \draw [blue] (-0.5, 0.8)--(-0.5, 0) .. controls +(0, -0.6) and +(0,-0.6).. (0.5, 0)--(0.5, 0.8);    
\begin{scope}[shift={(0.5, 0.3)}]
\draw [fill=white] (-0.5, -0.3) rectangle (0.5, 0.3);
\node at (0, 0) {\tiny $iwP_j$};
\end{scope}
    \end{scope}
\draw [blue] (-0.5, -0.8)--(-0.5, 0) .. controls +(0, 0.6) and +(0,0.6).. (0.5, 0)--(0.5, -0.8);
\begin{scope}[shift={(0.5, -0.3)}]
\draw [fill=white] (-0.5, -0.3) rectangle (0.5, 0.3);
\node at (0, 0) {\tiny $iwP_k$};
\end{scope}
\end{tikzpicture}}},
\end{align*}
where $(\gamma_{j,k}^{\pm})_{j,k=1}^m$ are real positive definite matrices such that $\gamma_{j,j}^++\gamma_{j,j}^{-}=1$ and $\gamma_{j,j}^+ \neq \gamma_{j,j}^-$ for all $j=1, \ldots, m$.
By a direct computation, we have that  $H=\overline{H}$.

\begin{lemma}
We have that 
\begin{align*}
\cR(H)=\cR(\widehat{\Delta}^{1/2}H \widehat{\Delta}^{1/2})= \cR(\widehat{\Delta}^{-1/2}H \widehat{\Delta}^{-1/2}).
\end{align*}
\end{lemma}
\begin{proof}
By the fact that 
\begin{small}
\begin{align*}
\tau(wQ_jv_k)=-\frac{1}{\sqrt{2}} \delta_{j,k}, \quad \tau(wP_jv_k)=-\frac{i}{\sqrt{2}} \delta_{j,k},\quad
\tau(wQ_jv_k^*)=\frac{1}{\sqrt{2}} \delta_{j,k}, \quad \tau(wP_jv_k^*)=-\frac{i}{\sqrt{2}} \delta_{j,k},
\end{align*}
\end{small}
we see that the projections $\lambda^{1/2}\vcenter{\hbox{\begin{tikzpicture}[scale=0.65]
    \begin{scope}[shift={(0,1.5)}]
    \draw [blue] (-0.5, 0.8)--(-0.5, 0) .. controls +(0, -0.6) and +(0,-0.6).. (0.5, 0)--(0.5, 0.8);    
\begin{scope}[shift={(0.5, 0.3)}]
\draw [fill=white] (-0.3, -0.3) rectangle (0.3, 0.3);
\node at (0, 0) {\tiny $v_j^*$};
\end{scope}
    \end{scope}
\draw [blue] (-0.5, -0.8)--(-0.5, 0) .. controls +(0, 0.6) and +(0,0.6).. (0.5, 0)--(0.5, -0.8);
\begin{scope}[shift={(0.5, -0.3)}]
\draw [fill=white] (-0.3, -0.3) rectangle (0.3, 0.3);
\node at (0, 0) {\tiny $v_j$};
\end{scope}
\end{tikzpicture}}} $ and $\lambda^{1/2}\vcenter{\hbox{\begin{tikzpicture}[scale=0.65]
    \begin{scope}[shift={(0,1.5)}]
    \draw [blue] (-0.5, 0.8)--(-0.5, 0) .. controls +(0, -0.6) and +(0,-0.6).. (0.5, 0)--(0.5, 0.8);    
\begin{scope}[shift={(0.5, 0.3)}]
\draw [fill=white] (-0.3, -0.3) rectangle (0.3, 0.3);
\node at (0, 0) {\tiny $v_j$};
\end{scope}
    \end{scope}
\draw [blue] (-0.5, -0.8)--(-0.5, 0) .. controls +(0, 0.6) and +(0,0.6).. (0.5, 0)--(0.5, -0.8);
\begin{scope}[shift={(0.5, -0.3)}]
\draw [fill=white] (-0.3, -0.3) rectangle (0.3, 0.3);
\node at (0, 0) {\tiny $v_j^*$};
\end{scope}
\end{tikzpicture}}}$ 
are orthogonal to the projections $\lambda^{1/2}\vcenter{\hbox{\begin{tikzpicture}[scale=0.65]
    \begin{scope}[shift={(0,1.5)}]
    \draw [blue] (-0.5, 0.8)--(-0.5, 0) .. controls +(0, -0.6) and +(0,-0.6).. (0.5, 0)--(0.5, 0.8);    
\begin{scope}[shift={(0.5, 0.3)}]
\draw [fill=white] (-0.5, -0.3) rectangle (0.5, 0.3);
\node at (0, 0) {\tiny $iwQ_k$};
\end{scope}
    \end{scope}
\draw [blue] (-0.5, -0.8)--(-0.5, 0) .. controls +(0, 0.6) and +(0,0.6).. (0.5, 0)--(0.5, -0.8);
\begin{scope}[shift={(0.5, -0.3)}]
\draw [fill=white] (-0.5, -0.3) rectangle (0.5, 0.3);
\node at (0, 0) {\tiny $iwQ_k$};
\end{scope}
\end{tikzpicture}}}$
and
$\lambda^{1/2}\vcenter{\hbox{\begin{tikzpicture}[scale=0.65]
    \begin{scope}[shift={(0,1.5)}]
    \draw [blue] (-0.5, 0.8)--(-0.5, 0) .. controls +(0, -0.6) and +(0,-0.6).. (0.5, 0)--(0.5, 0.8);    
\begin{scope}[shift={(0.5, 0.3)}]
\draw [fill=white] (-0.5, -0.3) rectangle (0.5, 0.3);
\node at (0, 0) {\tiny $iwP_k$};
\end{scope}
    \end{scope}
\draw [blue] (-0.5, -0.8)--(-0.5, 0) .. controls +(0, 0.6) and +(0,0.6).. (0.5, 0)--(0.5, -0.8);
\begin{scope}[shift={(0.5, -0.3)}]
\draw [fill=white] (-0.5, -0.3) rectangle (0.5, 0.3);
\node at (0, 0) {\tiny $iwP_k$};
\end{scope}
\end{tikzpicture}}}$ for $j\neq k$.
Note that the projections $\lambda^{1/2}\vcenter{\hbox{\begin{tikzpicture}[scale=0.65]
    \begin{scope}[shift={(0,1.5)}]
    \draw [blue] (-0.5, 0.8)--(-0.5, 0) .. controls +(0, -0.6) and +(0,-0.6).. (0.5, 0)--(0.5, 0.8);    
\begin{scope}[shift={(0.5, 0.3)}]
\draw [fill=white] (-0.5, -0.3) rectangle (0.5, 0.3);
\node at (0, 0) {\tiny $iwQ_j$};
\end{scope}
    \end{scope}
\draw [blue] (-0.5, -0.8)--(-0.5, 0) .. controls +(0, 0.6) and +(0,0.6).. (0.5, 0)--(0.5, -0.8);
\begin{scope}[shift={(0.5, -0.3)}]
\draw [fill=white] (-0.5, -0.3) rectangle (0.5, 0.3);
\node at (0, 0) {\tiny $iwQ_j$};
\end{scope}
\end{tikzpicture}}}$
and
$\lambda^{1/2}\vcenter{\hbox{\begin{tikzpicture}[scale=0.65]
    \begin{scope}[shift={(0,1.5)}]
    \draw [blue] (-0.5, 0.8)--(-0.5, 0) .. controls +(0, -0.6) and +(0,-0.6).. (0.5, 0)--(0.5, 0.8);    
\begin{scope}[shift={(0.5, 0.3)}]
\draw [fill=white] (-0.5, -0.3) rectangle (0.5, 0.3);
\node at (0, 0) {\tiny $iwP_k$};
\end{scope}
    \end{scope}
\draw [blue] (-0.5, -0.8)--(-0.5, 0) .. controls +(0, 0.6) and +(0,0.6).. (0.5, 0)--(0.5, -0.8);
\begin{scope}[shift={(0.5, -0.3)}]
\draw [fill=white] (-0.5, -0.3) rectangle (0.5, 0.3);
\node at (0, 0) {\tiny $iwP_k$};
\end{scope}
\end{tikzpicture}}}$ are orthogonal for all $j,k=1, \ldots, m$.
This implies that $\displaystyle \cR(H)=\sum_{j=1}^m \lambda^{1/2}\vcenter{\hbox{\begin{tikzpicture}[scale=0.65]
    \begin{scope}[shift={(0,1.5)}]
    \draw [blue] (-0.5, 0.8)--(-0.5, 0) .. controls +(0, -0.6) and +(0,-0.6).. (0.5, 0)--(0.5, 0.8);    
\begin{scope}[shift={(0.5, 0.3)}]
\draw [fill=white] (-0.5, -0.3) rectangle (0.5, 0.3);
\node at (0, 0) {\tiny $iwQ_j$};
\end{scope}
    \end{scope}
\draw [blue] (-0.5, -0.8)--(-0.5, 0) .. controls +(0, 0.6) and +(0,0.6).. (0.5, 0)--(0.5, -0.8);
\begin{scope}[shift={(0.5, -0.3)}]
\draw [fill=white] (-0.5, -0.3) rectangle (0.5, 0.3);
\node at (0, 0) {\tiny $iwQ_j$};
\end{scope}
\end{tikzpicture}}}+ \lambda^{1/2}\vcenter{\hbox{\begin{tikzpicture}[scale=0.65]
    \begin{scope}[shift={(0,1.5)}]
    \draw [blue] (-0.5, 0.8)--(-0.5, 0) .. controls +(0, -0.6) and +(0,-0.6).. (0.5, 0)--(0.5, 0.8);    
\begin{scope}[shift={(0.5, 0.3)}]
\draw [fill=white] (-0.5, -0.3) rectangle (0.5, 0.3);
\node at (0, 0) {\tiny $iwP_j$};
\end{scope}
    \end{scope}
\draw [blue] (-0.5, -0.8)--(-0.5, 0) .. controls +(0, 0.6) and +(0,0.6).. (0.5, 0)--(0.5, -0.8);
\begin{scope}[shift={(0.5, -0.3)}]
\draw [fill=white] (-0.5, -0.3) rectangle (0.5, 0.3);
\node at (0, 0) {\tiny $iwP_j$};
\end{scope}
\end{tikzpicture}}}$.

Note that 
\begin{align*}
\lambda^{1/2}\vcenter{\hbox{\begin{tikzpicture}[scale=0.65]
    \begin{scope}[shift={(0,1.5)}]
    \draw [blue] (-0.5, 0.8)--(-0.5, 0) .. controls +(0, -0.6) and +(0,-0.6).. (0.5, 0)--(0.5, 0.8);    
\begin{scope}[shift={(0.5, 0.3)}]
\draw [fill=white] (-0.5, -0.3) rectangle (0.5, 0.3);
\node at (0, 0) {\tiny $iwQ_j$};
\end{scope}
    \end{scope}
\draw [blue] (-0.5, -0.8)--(-0.5, 0) .. controls +(0, 0.6) and +(0,0.6).. (0.5, 0)--(0.5, -0.8);
\begin{scope}[shift={(0.5, -0.3)}]
\draw [fill=white] (-0.5, -0.3) rectangle (0.5, 0.3);
\node at (0, 0) {\tiny $iwQ_j$};
\end{scope}
\end{tikzpicture}}}
+
\lambda^{1/2}\vcenter{\hbox{\begin{tikzpicture}[scale=0.65]
    \begin{scope}[shift={(0,1.5)}]
    \draw [blue] (-0.5, 0.8)--(-0.5, 0) .. controls +(0, -0.6) and +(0,-0.6).. (0.5, 0)--(0.5, 0.8);    
\begin{scope}[shift={(0.5, 0.3)}]
\draw [fill=white] (-0.5, -0.3) rectangle (0.5, 0.3);
\node at (0, 0) {\tiny $iwP_j$};
\end{scope}
    \end{scope}
\draw [blue] (-0.5, -0.8)--(-0.5, 0) .. controls +(0, 0.6) and +(0,0.6).. (0.5, 0)--(0.5, -0.8);
\begin{scope}[shift={(0.5, -0.3)}]
\draw [fill=white] (-0.5, -0.3) rectangle (0.5, 0.3);
\node at (0, 0) {\tiny $iwP_j$};
\end{scope}
\end{tikzpicture}}}
=
\lambda^{1/2}\vcenter{\hbox{\begin{tikzpicture}[scale=0.65]
    \begin{scope}[shift={(0,1.5)}]
    \draw [blue] (-0.5, 0.8)--(-0.5, 0) .. controls +(0, -0.6) and +(0,-0.6).. (0.5, 0)--(0.5, 0.8);    
\begin{scope}[shift={(0.5, 0.3)}]
\draw [fill=white] (-0.3, -0.3) rectangle (0.3, 0.3);
\node at (0, 0) {\tiny $v_j$};
\end{scope}
    \end{scope}
\draw [blue] (-0.5, -0.8)--(-0.5, 0) .. controls +(0, 0.6) and +(0,0.6).. (0.5, 0)--(0.5, -0.8);
\begin{scope}[shift={(0.5, -0.3)}]
\draw [fill=white] (-0.3, -0.3) rectangle (0.3, 0.3);
\node at (0, 0) {\tiny $v_j^*$};
\end{scope}
\end{tikzpicture}}}
+
\lambda^{1/2}\vcenter{\hbox{\begin{tikzpicture}[scale=0.65]
    \begin{scope}[shift={(0,1.5)}]
    \draw [blue] (-0.5, 0.8)--(-0.5, 0) .. controls +(0, -0.6) and +(0,-0.6).. (0.5, 0)--(0.5, 0.8);    
\begin{scope}[shift={(0.5, 0.3)}]
\draw [fill=white] (-0.3, -0.3) rectangle (0.3, 0.3);
\node at (0, 0) {\tiny $v_j^*$};
\end{scope}
    \end{scope}
\draw [blue] (-0.5, -0.8)--(-0.5, 0) .. controls +(0, 0.6) and +(0,0.6).. (0.5, 0)--(0.5, -0.8);
\begin{scope}[shift={(0.5, -0.3)}]
\draw [fill=white] (-0.3, -0.3) rectangle (0.3, 0.3);
\node at (0, 0) {\tiny $v_j$};
\end{scope}
\end{tikzpicture}}} (: = \cR(H)_j).
\end{align*}
We see that $\cR(\widehat{\Delta}^{1/2}H \widehat{\Delta}^{1/2})=\cR(H)$.
Similarly, we have that $\cR(\widehat{\Delta}^{-1/2}H \widehat{\Delta}^{-1/2})=\cR(H)$.
This completes the proof of the lemma.
\end{proof}

The Laplacian part of $\widehat{\cL}$ is given by
\begin{align*}
\widehat{\mathcal{L}}_0
=& \widehat{\Delta}^{1/2} H \widehat{\Delta}^{1/2} \\
=& \sum_{j,k=1}^m
\frac{1}{2} e^{\beta(a_j+a_k)/4} (\gamma_{j,k}^++\gamma_{j,k}^-)
    \vcenter{\hbox{\begin{tikzpicture}[scale=0.65]
    \begin{scope}[shift={(0,1.5)}]
    \draw [blue] (-0.5, 0.8)--(-0.5, 0) .. controls +(0, -0.6) and +(0,-0.6).. (0.5, 0)--(0.5, 0.8);    
\begin{scope}[shift={(0.5, 0.3)}]
\draw [fill=white] (-0.3, -0.3) rectangle (0.3, 0.3);
\node at (0, 0) {\tiny $v_j$};
\end{scope}
    \end{scope}
\draw [blue] (-0.5, -0.8)--(-0.5, 0) .. controls +(0, 0.6) and +(0,0.6).. (0.5, 0)--(0.5, -0.8);
\begin{scope}[shift={(0.5, -0.3)}]
\draw [fill=white] (-0.3, -0.3) rectangle (0.3, 0.3);
\node at (0, 0) {\tiny $v_k^*$};
\end{scope}
\end{tikzpicture}}}
+\sum_{j,k=1}^m
\frac{1}{2} e^{-\beta(a_j+a_k)/4} (\gamma_{j,k}^+ +\gamma_{j,k}^-) \vcenter{\hbox{\begin{tikzpicture}[scale=0.65]
    \begin{scope}[shift={(0,1.5)}]
    \draw [blue] (-0.5, 0.8)--(-0.5, 0) .. controls +(0, -0.6) and +(0,-0.6).. (0.5, 0)--(0.5, 0.8);    
\begin{scope}[shift={(0.5, 0.3)}]
\draw [fill=white] (-0.3, -0.3) rectangle (0.3, 0.3);
\node at (0, 0) {\tiny $v_j^*$};
\end{scope}
    \end{scope}
\draw [blue] (-0.5, -0.8)--(-0.5, 0) .. controls +(0, 0.6) and +(0,0.6).. (0.5, 0)--(0.5, -0.8);
\begin{scope}[shift={(0.5, -0.3)}]
\draw [fill=white] (-0.3, -0.3) rectangle (0.3, 0.3);
\node at (0, 0) {\tiny $v_k$};
\end{scope}
\end{tikzpicture}}}\\
& +\frac{1}{2}\sum_{j,k=1}^m e^{\beta(a_j-a_k)/4} (\gamma_{j,k}^- - \gamma_{j,k}^+) \vcenter{\hbox{\begin{tikzpicture}[scale=0.65]
    \begin{scope}[shift={(0,1.5)}]
    \draw [blue] (-0.5, 0.8)--(-0.5, 0) .. controls +(0, -0.6) and +(0,-0.6).. (0.5, 0)--(0.5, 0.8);    
\begin{scope}[shift={(0.5, 0.3)}]
\draw [fill=white] (-0.3, -0.3) rectangle (0.3, 0.3);
\node at (0, 0) {\tiny $v_j$};
\end{scope}
    \end{scope}
\draw [blue] (-0.5, -0.8)--(-0.5, 0) .. controls +(0, 0.6) and +(0,0.6).. (0.5, 0)--(0.5, -0.8);
\begin{scope}[shift={(0.5, -0.3)}]
\draw [fill=white] (-0.3, -0.3) rectangle (0.3, 0.3);
\node at (0, 0) {\tiny $v_k$};
\end{scope}
\end{tikzpicture}}}
+
\frac{1}{2}\sum_{j,k=1}^m e^{\beta(a_k-a_j)/4}(\gamma_{j,k}^- - \gamma_{j,k}^+)  \vcenter{\hbox{\begin{tikzpicture}[scale=0.65]
    \begin{scope}[shift={(0,1.5)}]
    \draw [blue] (-0.5, 0.8)--(-0.5, 0) .. controls +(0, -0.6) and +(0,-0.6).. (0.5, 0)--(0.5, 0.8);    
\begin{scope}[shift={(0.5, 0.3)}]
\draw [fill=white] (-0.3, -0.3) rectangle (0.3, 0.3);
\node at (0, 0) {\tiny $v_j^*$};
\end{scope}
    \end{scope}
\draw [blue] (-0.5, -0.8)--(-0.5, 0) .. controls +(0, 0.6) and +(0,0.6).. (0.5, 0)--(0.5, -0.8);
\begin{scope}[shift={(0.5, -0.3)}]
\draw [fill=white] (-0.3, -0.3) rectangle (0.3, 0.3);
\node at (0, 0) {\tiny $v_k^*$};
\end{scope}
\end{tikzpicture}}}.
\end{align*}
By taking the cap from the left hand side, we have that 
\begin{align*}
1*\widehat{\cL}_0
=& \frac{1}{2}\sum_{j,k=1}^m  e^{\beta( a_j+a_k)/4}(\gamma_{j,k}^++\gamma_{j,k}^-) v_k^*v_j 
+\frac{1}{2}\sum_{j,k=1}^m  e^{-\beta( a_j+a_k)/4}(\gamma_{j,k}^++\gamma_{j,k}^-) v_kv_j^* \\
& +\frac{1}{2}\sum_{j,k=1}^m  e^{\beta( a_j-a_k)/4}(\gamma_{j,k}^--\gamma_{j,k}^+) v_kv_j 
+\frac{1}{2}\sum_{j,k=1}^m  e^{\beta( a_k-a_j)/4}(\gamma_{j,k}^--\gamma_{j,k}^+) v_k^*v_j^* \\
&(: =2 \mathbf{y}). 
\end{align*}
We have that 
\begin{align*}
\mathbf{y} v_s -v_s \mathbf{y}
=& -\frac{1}{2}\sum_{j=1}^m  e^{\beta( a_j+a_s)/4}(\gamma_{j,s}^++\gamma_{j,s}^-) v_j 
+\frac{1}{2}\sum_{j=1}^m  e^{-\beta( a_s+a_j)/4}(\gamma_{s,j}^++\gamma_{s,j}^-) v_j \\
& -\frac{1}{2}\sum_{j=1}^m  e^{\beta( a_j-a_s)/4}(\gamma_{j,s}^--\gamma_{j,s}^+) v_j^*
+\frac{1}{2}\sum_{j=1}^m  e^{\beta( a_s-a_j)/4}(\gamma_{s,j}^--\gamma_{s,j}^+) v_j^*.
\end{align*}
Let $\cL(x)=\mathbf{y}x + x\mathbf{y}-x*\widehat{\cL}_0$ for $x\in \cM$.

\begin{lemma}
The semigroup $\{\Phi_t\}_{t\geq 0}$ associated to $\cL$ is KMS symmetric with $\cL_w=0$.
\end{lemma}
\begin{proof}
To see $\{\Phi_t\}_{t\geq 0}$ is bimodule KMS symmetric with respect to $\widehat{\Delta}$ and $\cL_w=0$ , it suffices to check that 
$\vcenter{\hbox{\begin{tikzpicture}[scale=0.65]
\draw [blue] (-0.5, -1.5)--(-0.5, 0) .. controls +(0, 0.6) and +(0,0.6).. (0.5, 0)--(0.5, -1.5);
\begin{scope}[shift={(0.5, -0.3)}]
\draw [fill=white] (-0.3, -0.3) rectangle (0.3, 0.3);
\node at (0, 0) {\tiny $\mathbf{y}$};
\end{scope}
\begin{scope}[shift={(0, -1)}]
\draw [fill=white] (-0.7, -0.3) rectangle (0.7, 0.3);
\node at (0, 0) {\tiny $\overline{\widehat{\Delta}}$};    
\end{scope}
\end{tikzpicture}}}= \vcenter{\hbox{\begin{tikzpicture}[scale=0.65]
\draw [blue] (-0.5, -0.8)--(-0.5, 0) .. controls +(0, 0.6) and +(0,0.6).. (0.5, 0)--(0.5, -0.8);
\begin{scope}[shift={(0.5, -0.3)}]
\draw [fill=white] (-0.3, -0.3) rectangle (0.3, 0.3);
\node at (0, 0) {\tiny $\mathbf{y}$};
\end{scope}
\end{tikzpicture}}}.$
This follows from $\tau(\mathbf{y}v_j)=0$ for all $j=1, \ldots, m$.

Let $\displaystyle \rho_0=\prod_{j=1}^m \left(  \frac{1}{2} e^{-\beta a_j/2 } v_j^*v_j +\frac{1}{2} e^{\beta a_j/2}v_jv_j^*\right)$ and $\displaystyle \rho=\frac{\rho_0}{\tau(\rho_0)}$.
By using a similar argument in Lemma \ref{lem:fermiongns}, we have that $\widehat{\Delta}\cR(\widehat{\cL}_0)= \widehat{\Delta}_{\rho,1/2} \cR(\widehat{\cL}_0)$.
This shows that $\{\Phi_t\}_{t\geq 0}$ is a KMS symmetric quantum Markov semigroup.

\end{proof}

Suppose that $A\in M_n(\bC)$.
We denote by $\cR^{\diamond}(A)\in M_n(\bC)$ the range of $A$ with respect to the Schur product, i.e. the $(j,k)$-entry of $\cR^{\diamond}(A)$ is 
\begin{align*}
\big(\cR^{\diamond}(A)\big)_{j,k} =\left\{
\begin{array}{ll} 
1, & A_{j,k} \neq 0\\
0, & A_{j,k}=0
\end{array}\right.,
\end{align*}
where $A_{j,k}$ is the $(j,k)$-entry of $A$.

\begin{proposition}
Suppose that the matrix 
\begin{align*}
 \cR^{\diamond}\left((\gamma_{j,k}^+)_{j,k=1}^m\right)+ \cR^{\diamond}\left((\gamma_{j,k}^-)_{j,k=1}^m\right)
 \end{align*}
is an irreducible matrix.
Then the algebra $\cA$ generated by $\widehat{\Delta}$ and $H$ is isomorphic to $M_{2m}(\bC)$.
\end{proposition}
\begin{proof}
Note that $a_j$ are distinct.
We obtain that $\lambda^{1/2}\vcenter{\hbox{%
}}$ are in $\cA$ for all $j=1, \ldots, m$. 
Hence the projections $\cR(H)_j\in \cA$.

Note that 
\begin{align*}
& 2\cR(H)_j H \left(\lambda^{1/2}\vcenter{\hbox{%
}}.
\end{align*}
By the assumption that $\gamma_{j,j}^+\neq \gamma_{j,j}^-$ for all $j=1, \ldots, m$, we see that 
\begin{align}\label{eq:projection1}
\lambda^{1/2} \vcenter{\hbox{%
}} \in \cA.
\end{align}
Now we have that $\lambda^{1/2} \vcenter{\hbox{%
}}\in \cA$ for any $\gamma_{j,k}^+\neq 0$.
By the result \eqref{eq:projection1}, we have that $\lambda^{1/2} \vcenter{\hbox{%
}}\in \cA$ for any $\gamma_{j,k}^+\neq 0$.
By the assumption that $ \cR^{\diamond}\left((\gamma_{j,k}^+)_{j,k=1}^m\right)+ \cR^{\diamond}\left((\gamma_{j,k}^-)_{j,k=1}^m\right)$ is irreducible, we obtain that $\lambda^{1/2} \vcenter{\hbox{%
}}\in \cA$ for all $j,k=1, \ldots,m$.
This implies that $\lambda^{1/2} \vcenter{\hbox{%
}}\in \cA$ for all $j,k=1, \ldots,m$.
\end{proof}

\begin{remark}
If $ \cR^{\diamond}\left((\gamma_{j,k}^+)_{j,k=1}^m\right)+ \cR^{\diamond}\left((\gamma_{j,k}^-)_{j,k=1}^m\right)$ is not irreducible, then it is a block matrix since it is hermitian.
In this case, the algebra $\cA$ is a multi-matrix algebra.
\end{remark}

The extension of $\cL_0$ is given by the following element $\widehat{\cJ}_0$.
\begin{align*}
\widehat{\cJ}_0
=&   \sum_{j,k=1}^m
\frac{1}{2} e^{\beta(a_j+a_k)/4} (\gamma_{j,k}^++\gamma_{j,k}^-)
\vcenter{\hbox{%
}}.
\end{align*}
It is routine to check that the semigroup arising from $\widehat{\cJ}_0$ is an extension of the semigroup $\{\Phi_t\}_{t\geq 0}$.

Now we shall check the intertwining property for the semigroups.
\begin{align*}
& \lambda^{-1/2} \vcenter{\hbox{
%
}}.
\end{align*}
By taking differences, we obtain that 
\begin{align*}
&  \lambda^{-1/2} \vcenter{\hbox{
%
}}\right).
\end{align*}
This shows the semigroup satisfies the intertwining property.
We do not have the modified logarithmic Sobolev inequality in general. 
But when $\left|\left( \cosh\frac{\beta (a_j-a_s)}{4} \right) (\gamma_{j,s}^--\gamma_{j,s}^+)\right |$ for $j,s=1, \ldots, m$ and $\left( \cosh\frac{\beta (a_j+a_s)}{4} \right) (\gamma_{j,s}^-+\gamma_{j,s}^+) $ for $j\neq s$ are small enough, we see that a modified logarithmic Sobolev inequality holds.

\bibliographystyle{abbrv}
\bibliography{lsi}

\end{document}